\documentclass[11pt]{amsart}
\usepackage{amsmath}
\usepackage{amssymb}
\usepackage{amsthm}
\usepackage{latexsym}
\usepackage{hyperref}
\usepackage{enumerate}
\usepackage[all]{xy}
\usepackage{chngcntr}

\newtheorem{thm}{Theorem}[section]
\newtheorem{cor}[thm]{Corollary}
\newtheorem{lem}[thm]{Lemma}
\newtheorem{prop}[thm]{Proposition}

\newtheorem{thmintro}{Theorem}

\theoremstyle{definition}
\newtheorem{defn}[thm]{Definition}
\newtheorem{rem}[thm]{Remark}

\newtheorem{ex}[thm]{Example}

\usepackage{multirow}
\usepackage{array}
\usepackage{stmaryrd}

\providecommand{\norm}[1]{\left\| #1 \right\|}
\newcommand{\enuma}[1]{\begin{enumerate}[\textup{(}a\textup{)}] {#1} \end{enumerate}}

\newcommand{\mb}{\mathbf}
\newcommand{\bW}{{\mathbf W}}

\newcommand{\mh}{\mathbb}
\newcommand{\mr}{\mathrm}
\newcommand{\mc}{\mathcal}
\newcommand{\mf}{\mathfrak}

\newcommand{\N}{\mathbb N}
\newcommand{\Z}{\mathbb Z}
\newcommand{\Q}{\mathbb Q}
\newcommand{\R}{\mathbb R}
\newcommand{\C}{\mathbb C}
\newcommand{\rN}{\mathrm N}
\newcommand{\rZ}{\mathrm Z}

\newcommand{\inp}[2]{\langle #1 \hspace{1pt},\hspace{1pt} #2 \rangle}

\newcommand{\matje}[4]{\left(\begin{smallmatrix} #1 & #2 \\
#3 & #4 \end{smallmatrix}\right)}

\newcommand{\q}{/\!/}
\newcommand{\vq}{z}

\def\re{{\mathrm e}}
\def\Std{{\rm Std}}
\def\Hom{{\rm Hom}}
\def\End{{\rm End}}

\def\Irr{{\rm Irr}}

\def\Jord{{\rm Jord}}
\def\SO{{\rm SO}}
\def\O{{\rm O}}

\def\Sp{{\rm Sp}}
\def\GL{{\rm GL}}

\def\rs{{\rm rs}}

\def\PGL{{\rm PGL}}
\def\SL{{\rm SL}}

\def\cC{{\mathcal C}}
\def\cG{{\mathcal G}}

\def\cS{{\mathcal S}}

\def\cL{{\mathcal L}}
\def\cH{{\mathcal H}}

\def\cO{{\mathcal O}}
\def\cE{{\mathcal E}}
\def\cF{{\mathcal F}}

\def\cR{{\mathfrak R}}

\def\ffi{{\mathfrak i}}
\def\Fr{{\rm Frob}}

\def\disc{{\rm disc}}

\def\diag{{\rm diag}}

\def\ind{{\rm ind}}

\def\Mod{{\rm Mod}}
\def\nr{{\rm nr}}

\def\rU{{\rm U}}

\def\fs{{\mathfrak s}}
\def\ft{{\mathfrak t}}

\def\Rep{{\rm Rep}}
\def\Mod{{\rm Mod}}
\def\Res{{\rm Res}}

\def\stab{{\rm st}}

\def\Ad{{\rm Ad}}
\def\Lie{{\rm Lie}}

\def\der{{\rm der}}
\def\ad{{\rm ad}}
\def\sc{{\rm sc}}

\def\temp{{\rm temp}}

\def\red{{\rm red}}

\def\cusp{{\rm cusp}}
\def\uni{{\rm un}}

\def\fB{{\mathfrak B}}

\def\IM{{\rm{IM}}}
\def\restriction#1#2{\mathchoice
              {\setbox1\hbox{${\displaystyle #1}_{\scriptstyle #2}$}
              \restrictionaux{#1}{#2}}
              {\setbox1\hbox{${\textstyle #1}_{\scriptstyle #2}$}
              \restrictionaux{#1}{#2}}
              {\setbox1\hbox{${\scriptstyle #1}_{\scriptscriptstyle #2}$}
              \restrictionaux{#1}{#2}}
              {\setbox1\hbox{${\scriptscriptstyle #1}_{\scriptscriptstyle #2}$}
              \restrictionaux{#1}{#2}}}
\def\restrictionaux#1#2{{#1\,\smash{\vrule height .8\ht1 depth .85\dp1}}_{\,#2}}

\def\RS{{\rm RS}}

\def\rB{{\rm B}}
\def\rC{{\rm C}}

\def\rU{{\rm U}}

\def\rBC{{\rm {BC}}}

\def\cM{{\mathcal M}}

\begin{document}

\title[Affine Hecke algebras for Langlands parameters]{Affine Hecke algebras \\
for Langlands parameters}

\author[A.-M. Aubert]{Anne-Marie Aubert}
\address{Sorbonne Universit\'e and Universit\'e Paris Cit\'e, CNRS,
IMJ-PRG, F-75005 Paris, France}
\email{anne-marie.aubert@imj-prg.fr}
\author[A. Moussaoui]{Ahmed Moussaoui}
\address{Laboratoire de Math\'ematiques et Applications, Universit\'e de Poitiers, 
11 Boulevard Marie et Pierre Curie,
B\^{a}timent H3 - TSA 61125, 86073 Poitiers Cedex 9, France}
\email{ahmed.moussaoui@math.univ-poitiers.fr}
\author[M. Solleveld]{Maarten Solleveld}
\address{IMAPP, Radboud Universiteit Nijmegen, Heyendaalseweg 135,
6525AJ Nijmegen, the Netherlands}
\email{m.solleveld@science.ru.nl}
\date{\today}
\thanks{The second author thanks the FMJH for their support. The third author was supported by a
NWO Vidi grant ``A Hecke algebra approach to the local Langlands correspondence" (nr. 639.032.528).}
\subjclass[2010]{20C08,14F43,20G20}
\maketitle

\begin{abstract}
It is well-known that affine Hecke algebras are very useful to describe the smooth representations
of any connected reductive $p$-adic group $\cG$, in terms of the supercuspidal representations
of its Levi subgroups.
The goal of this paper is to create a similar role for affine Hecke algebras on the Galois side
of the local Langlands correspondence.

To every Bernstein component of enhanced Langlands
parameters for $\cG$ we canonically associate an affine Hecke algebra (possibly extended with a
finite $R$-group). We prove that the irreducible representations of this algebra are naturally in
bijection with the members of the Bernstein component, and that the set of central characters of
the algebra is naturally in bijection with the collection of cuspidal supports of these
enhanced Langlands parameters. These bijections send tempered or (essentially) square-integrable
representations to the expected kind of Langlands parameters.

Furthermore we check that for many reductive $p$-adic groups, if a Bernstein component $\mf B$ for
$\cG$ corresponds to a Bernstein component $\mf B^\vee$ of enhanced Langlands parameters via the
local Langlands correspondence, then the affine Hecke algebra that we associate to $\mf B^\vee$ is
Morita equivalent with the Hecke algebra associated to $\mf B$. This constitutes a generalization
of Lusztig's work on unipotent representations. It might be useful to establish a local Langlands
correspondence for more classes of irreducible smooth representations.
\end{abstract}

\tableofcontents

\section*{Introduction}

Let $F$ be a non-archimedean local field and let $\mc G$ be a connected reductive
algebraic group defined over $F$. The conjectural local Langlands correspondence (LLC)
provides a bijection between the set of irreducible smooth $\mc G (F)$-representations
$\Irr (\mc G (F))$ and the set of enhanced L-parameters $\Phi_\re (\mc G (F))$, see
\cite{Bor,Vog,ABPS7}.

Let $\fs$ be an inertial equivalence class for $\mc G (F)$ and let $\Irr (\mc G (F))^\fs$
be the associated Bernstein component. Similarly, inertial equivalence classes $\fs^\vee$
and Bernstein components $\Phi_\re (\mc G (F))^{\fs^\vee}$ for enhanced L-parameters were
developed in \cite{AMS}. It can be expected that every $\fs$ corresponds to a unique
$\fs^\vee$ (an ``inertial Langlands correspondence"), such that the LLC restricts to
a bijection
\begin{equation}\label{eq:1.1}
\Irr (\mc G (F))^\fs \longleftrightarrow \Phi_\re (\mc G (F))^{\fs^\vee} .
\end{equation}
The left hand side can be identified with the space of irreducible representations of
a direct summand $\mc H (\mc G (F))^\fs$ of the full Hecke algebra of $\mc G (F)$. It is
known that in many cases $\mc H (\mc G (F))^\fs$ is Morita equivalent to an affine Hecke
algebra, see \cite[\S 2.4]{ABPS7} and the references therein for an overview.

To improve our understanding of the LLC, we would like to canonically associate to
$\fs^\vee$ an affine Hecke algebra $\mc H (\fs^\vee)$ whose irreducible representations
are naturally parametrized by $\Phi_\re (\mc G (F))^{\fs^\vee}$.
Then \eqref{eq:1.1} could be written as
\begin{equation}\label{eq:1.2}
\Irr (\mc G (F))^\fs \cong \Irr \big( \mc H (\mc G (F))^\fs \big)  \longleftrightarrow
\Irr (\mc H (\fs^\vee)) \cong \Phi_\re (\mc G (F))^{\fs^\vee} ,
\end{equation}
and the LLC for this Bernstein component would become a comparison between two algebras
of the same kind. If moreover $\mc H (\fs^\vee)$ were Morita equivalent to
$\mc H (\mc G (F))^\fs$, then \eqref{eq:1.1} could even be categorified to
\begin{equation}\label{eq:1.3}
\Rep (\mc G (F))^\fs \cong \Mod (\mc H (\fs^\vee)) .
\end{equation}
Such algebras $\mc H (\fs^\vee)$ would also be useful to establish the LLC in new cases.
Suppose one would like to match $\mf s^\vee$ (essentially a set of cuspidal enhanced Langlands 
parameters for a Levi subgroup $\mc L (F)$) with a yet unknown supercuspidal Bernstein block for
$\mc L (F)$. Motivated by some examples, we increase the scope of \eqref{eq:1.3} by considering 
it only for the full subcategories of finite length objects:
\begin{equation}\label{eq:1.5}
\Rep_{\mr{fl}} (\mc G (F))^\fs \cong \Mod_{\mr{fl}} (\mc H (\fs^\vee)) .
\end{equation}
One could compare $\mc H (\fs^\vee)$ with the algebras $\mc H (\mc G (F))^\fs$ 
for various $\fs = [\mc L (F),\sigma]$, and only the Bernstein components $\Irr (\mc G (F))^\fs$ 
for which \eqref{eq:1.5} holds would be good candidates for the image of 
$\Phi_\re (\mc G (F))^{\fs^\vee}$ under the LLC. If one would know a lot about $\mc H (\fs^\vee)$, 
this could substantially reduce the number of possibilities for a LLC for both 
$\Phi_\re (\mc L (F))^{\mf s^\vee}$ and $\Phi_\re (\mc G (F))^{\fs^\vee}$.

This strategy was already employed by Lusztig, for unipotent representations
\cite{Lus6,Lus8}. Bernstein components of enhanced L-parameters had not yet been
defined when the papers \cite{Lus6,Lus8} were written, but the constructions in them
can be interpreted in that way. Lusztig found a bijection between:
\begin{itemize}
\item the set of (``arithmetic") affine Hecke algebras associated to unipotent
Bernstein blocks of adjoint, unramified groups;
\item the set of (``geometric") affine Hecke algebras associated to unramified enhanced
L-parameters for such groups.
\end{itemize}
However, the comparison of these two families of Hecke algebras is not enough to specify 
a canonical bijection between Bernstein components on the $p$-adic and the Galois sides. 
The problem is that one affine Hecke algebra can appear (up to isomorphism) several times on 
either side. This already happens in the unipotent case for exceptional groups, and the issue
seems to be outside the scope of these techniques. In \cite[6.6--6.8]{Lus6} Lusztig wrote
down some remarks about this problem, but he does not work it out completely. \\

The main goal of this paper is the construction of an affine Hecke algebra for
any Bernstein component of enhanced L-parameters, for any $\mc G$. But it quickly
turns out that this is not exactly the right kind of algebra. Firstly, our geometric
construction, which relies on \cite{Lus3,AMS2}, naturally includes some complex
parameters $\mb z_i$, which we abbreviate to $\vec{\mb z}$. Secondly, an affine
Hecke algebra with (indeterminate) parameters is still too simple. In general one must
consider the crossed product of such an object with a twisted group algebra (of some
finite ``$R$-group"). We call this a twisted affine Hecke algebra, see Proposition
\ref{prop:4.2} for a precise definition. Like for reductive groups, there are good
notions of tempered representations and of (essentially) discrete series representations
of such algebras (Definition \ref{defn:4.8}).

\begin{thmintro}\label{thm:A} {\rm [see Theorem \ref{thm:5.13}]}
\enuma{
\item To every Bernstein component of enhanced L-parameters $\fs^\vee$ one can
canonically associate a twisted affine Hecke algebra $\mc H (\fs^\vee,\vec{\mb z})$.
\item For every choice of parameters $z_i \in \R_{>0}$ there exists a natural bijection
\[
\Phi_\re (\mc G (F))^{\fs^\vee} \longleftrightarrow
\Irr \big( \mc H (\fs^\vee,\vec{\mb z}) / ( \{ \mb z_i - z_i\}_i ) \big)
\]
\item For every choice of parameters $z_i \in \R_{\geq 1}$ the bijection from part (b)
matches enhanced bounded L-parameters with tempered irreducible representations.
\item Suppose that $\Phi_\re (\mc G (F))^{\fs^\vee}$ contains enhanced discrete
L-parameters, and that $z_i \in \R_{> 1}$ for all $i$. Then the bijection from part (b)
matches enhanced discrete L-parameters with irreducible essentially discrete series
representations.
\item The bijection in part (b) is equivariant with respect to the canonical actions
of the group of unramified characters of $\mc G (F)$.
}
\end{thmintro}

This can be regarded as a far-reaching generalization of parts of \cite{Lus6,Lus8}:
we allow any reductive group over a non-archimedean local field, and all enhanced
L-parameters for that group. We check (see Section \ref{sec:examples}) that in several
cases where the LLC is known, indeed
\begin{equation}\label{eq:1.4}
\mc H (\mc G (F))^\fs \quad  \text{is Morita equivalent to} \quad
\mc H (\fs^\vee, \vec{\mb z}) / ( \{ \mb z_i - z_i\}_i ) 
\end{equation}
for suitable $z_i \in \R_{>1}$, obtaining \eqref{eq:1.3}. Notice that on the $p$-adic
side the parameters $z_i$ are determined by $\mc H (\mc G (F))^\fs$, whereas on the
Galois side we specify them manually. In fact, in all our examples we can take
$z_i = q_F^{1/2}$. That is a good sign, which indicates that in general $z_i = q_F^{1/2}$
could be the best specialization of the parameters to compare with an affine
Hecke algebra coming from a $p$-adic group.

Yet in general the categorification \eqref{eq:1.3} is asking for too much. We
discovered that for inner twists of $\SL_n (F)$ \eqref{eq:1.4} does not always hold.
Rather, these algebras are equivalent in a weaker sense: the category of finite
length modules of $\mc H (\mc G (F))^\fs$ (i.e. the finite length objects in $\Rep 
(\mc G (F))^\fs$) is equivalent to the category of finite dimensional representations 
of $\mc H (\fs^\vee, \vec{\mb z}) \big/ \big( \{ \mb z_i - q_F^{1/2}\}_i \big)$.\\

Let us describe the contents of the paper more concretely. Our starting point is
a triple $(G,M,q\cE)$ where
\begin{itemize}
\item $G$ is a possibly disconnected complex reductive group,
\item $M$ is a quasi-Levi subgroup of $G$ (the $G$-centralizer of the connected centre 
of a Levi subgroup of $G^\circ$),
\item $q\cE$ is a $M$-equivariant cuspidal local system on a unipotent orbit
$\cC_v^M$ in $M$.
\end{itemize}
To these data we attach a twisted affine Hecke algebra $\mc H (G,M,q\cE,\vec{\mb z})$. This
algebra can be specialized by setting $\vec{\mb z}$ equal to some $\vec{z} \in (\C^\times )^d$.
Of particular interest is the specialization at $\vec{z} = \vec{1}$:
\[
\mc H (G,M,q\cE,\vec{\mb z}) / (\{ \mb z_i - 1 \}_i ) =
\mc O (T) \rtimes \C [ W_{q\cE} , \natural ] ,
\]
where $T = \rZ(M)^\circ$, while the subgroup $W_{q\cE} \subset \rN_G (M) / M$ and the 2-cocycle\\
$\natural \colon W_{q\cE}^2 \to \C^\times$ also come from the data.

The goal of Section \ref{sec:TAHA} is to understand and parametrize representations of the algebra
$\mc H (G,M,q\cE,\vec{\mb z})$. We follow a strategy similar to that in \cite{Lus4}.
The centre naturally contains $\mc O (T)^{W_{q\cE}} = \mc O (T / W_{q\cE})$, so we can
study $\Mod (\mc H (G,M,q\cE,\vec{\mb z}))$ via localization at suitable subsets of
$T / W_{q\cE}$. In Paragraph \ref{par:Rrcc} we reduce to representations with
$\mc O (T)^{W_{q\cE}}$-character in $W_{q\cE} T_{\rs}$, where $T_{\rs}$ denotes the
maximal real split subtorus of $T$. This involves replacing
$\mc H (G,M,q\cE,\vec{\mb z})$ by an algebra of the same kind, but for a smaller $G$.

In Paragraph \ref{par:irreps} we reduce further, to representations of a (twisted)
graded Hecke algebra $\mh H (G,M,q\cE,\vec{\mb r})$. We defined and studied such
algebras in our previous paper \cite{AMS2}. But there we only considered the case with
a single parameter $\mb r$, here we need $\vec{\mb r} = (\mb r_1, \ldots, \mb r_d)$.
The generalization of the results of \cite{AMS2} to a multi-parameter setting is
carried out in Section \ref{sec:0}. With that at hand we can use the construction of
``standard" $\mh H (G,M,q\cE,\vec{\mb r})$-modules and the classification of
irreducible $\mh H (G,M,q\cE,\vec{\mb r})$-modules from \cite{AMS2} to achieve the
same for $\mc H (G,M,q\cE,\vec{\mb z})$. For the parametrization we use triples
$(s,u,\rho)$ where:
\begin{itemize}
\item $s \in G^\circ$ is semisimple,
\item $u \in \rZ_G (s)^\circ$ is unipotent,
\item $\rho \in \Irr \big( \pi_0 (\rZ_G (s,u)) \big)$ such that the quasi-cuspidal
support of $(u,\rho)$, as defined in \cite[\S 5]{AMS}, is $G$-conjugate to
$(M,\cC_v^M,q\cE)$.
\end{itemize}

\begin{thmintro}\label{thm:B} {\rm [see Theorem \ref{thm:4.10}]}
\enuma{
\item Let $\vec{z} \in \R_{\geq 0}^d$. There exists a canonical bijection,
say $(s,u,\rho) \mapsto \bar{M}_{s,u,\rho,\vec{z}}$, between:
\begin{itemize}
\item $G$-conjugacy classes of triples $(s,u,\rho)$ as above,
\item $\Irr \big( \mc H (G,M,q\cE,\vec{\mb z}) / (\{ \mb z_i - z_i \}_i ) \big)$.
\end{itemize}
\item Let $\vec{z} \in \R_{\geq 1}^d$. The module $\bar{M}_{s,u,\rho,\vec{z}}$ is
tempered if and only if $s$ is contained in a compact subgroup of $G^\circ$.
\item Let $\vec{z} \in \R_{> 1}^d$. The module $\bar{M}_{s,u,\rho,\vec{z}}$ is
essentially discrete series if and only if $u$ is distinguished unipotent in $G^\circ$
(i.e. does not lie in a proper Levi subgroup).
}
\end{thmintro}

In the case $M = T ,\, \cC_v^M = \{1\}$ and $q\cE$ trivial, the irreducible
representations in $\mc H (G^\circ, T, q\cE = \mr{triv})$ were already classified in
the landmark paper \cite{KaLu}, in terms of similar triples. In Paragraph \ref{par:KaLu}
we check that the parametrization from Theorem \ref{thm:B} agrees with the
Kazhdan--Lusztig parametrization for these algebras.

Remarkably, our analysis also reveals that \cite{KaLu} does not agree with the classification
of irreducible representations in \cite{Lus6}. To be precise, the difference consists of
a twist with a version of the Iwahori--Matsumoto involution.
Since \cite{KaLu} is widely regarded (see for example \cite{Ree,Vog}) as the correct local
Langlands correspondence for Iwahori-spherical representations, this entails that the
parametrizations obtained by Lusztig in \cite{Lus6,Lus8} can be improved by composition
with a suitable involution. In the special case $G = \Sp_{2n}(\C)$, that already transpired
from work of M\oe glin and Waldspurger \cite{Wal}.\\

Having obtained a good understanding of affine Hecke algebras attached to disconnected
reductive groups, we turn to Langlands parameters. Let
\[
\phi \colon \mb W_F \times \SL_2 (\C) \to {}^L \mc G
\]
be a L-parameter and let $\rho$ be an enhancement of $\phi$. (See Section \ref{sec:Langlands}
for the precise notions.) Let $\cG^\vee_\ad$ be the adjoint group of the complex dual group
$\cG^\vee$ and  let $\cG^\vee_\sc$ be the simply connected cover of $\cG^\vee_\ad$. Let
$\rZ_{\cG^\vee_\ad}(\phi (\mb I_F))$ be the centralizer of $\phi (\mb I_F)$ in
$\cG^\vee_\ad$, and let $J_\phi= Z^1_{\cG^\vee_\sc}(\phi (\mb I_F))$ denote its
inverse image in $\cG^\vee_\sc$. Similarly, we consider the group $G_\phi$ defined to be
inverse image in $\cG^\vee_\sc$ of the centralizer of $\phi (\mb W_F)$ in
$\cG^\vee_\ad$. We emphasize that the complex groups $J_\phi$ and $G_\phi$ can be disconnected --
this is the main reason why we have to investigate Hecke algebras for disconnected reductive groups.

Recall that $\phi$ is determined up to $\mc G^\vee$-conjugacy by $\phi |_{\mb W_F}$ and the
unipotent element $u_\phi = \phi \big( 1, \matje{1}{1}{0}{1} \big)$. As the image of a Frobenius
element is allowed to vary within one Bernstein component, $(\phi|_{\mb I_F},u_\phi)$ contains
almost all information about such a Bernstein component.

The cuspidal support
of $(u_\phi,\rho)$ for $G=G_\phi$ is a triple $(M,\cC_v^M,q\cE)$ as before. Thus we can associate
to $(\phi,\rho)$ the twisted affine Hecke algebra $\mc H (G,M,q\cE,\vec{\mb z})$.
This works quite well in several cases, but in general it is too simple, we encounter
various technical difficulties. The main problem is that the torus $T = \rZ(M)^\circ$ will
not always match up with the torus from which the Bernstein component of $\Phi_\re (\mc G (F))$
containing $(\phi,\rho)$ is built.

Instead we consider the twisted graded Hecke algebra $\mh H (G, M,q\cE,\vec{\mb r})$,
and we tensor it with the coordinate ring of a
suitable vector space to compensate for the difference between $\mc G_\sc^\vee$ and
$\mc G^\vee$. In Paragraph \ref{par:LGHA} we prove that the irreducible representations
of the ensuing algebra are naturally parametrized by a subset of the Bernstein
component $\Phi_\re (\mc G (F))^{\fs^\vee}$ containing $(\phi,\rho)$. In Paragraph
\ref{par:LAHA} we glue fa\-mi\-lies of such algebras together, to obtain the twisted
affine Hecke algebras $\mc H(\fs^\vee, \vec{\mb z})$ featuring in Theorem \ref{thm:A}.
This requires careful analysis of the involved tori and root systems, which we perform
in Paragraph \ref{par:roots}.

We discuss then, in Section \ref{sec:stable}, the relation of the above theory with
the stable Bernstein center on the Galois side of the LLC.
In Section \ref{sec:examples} we explain and work out the examples of general linear,
special linear and classical groups. It turns out that, for general linear groups (and
their inner twists) and classical groups, the extended affine Hecke algebras for enhanced
Langlands parameters (with a suitable specialization of the parameters) are Morita 
equivalent to those obtained from representations of reductive $p$-adic groups. In the case 
of inner twists of special linear groups we establish a slightly weaker result.\\

Let us compare our paper with similar work by other authors. Several mathematicians have noted
that, when two Bernstein components give rise to isomorphic affine Hecke algebras, this often
has to do with the centralizers of the corresponding Langlands parameters. It is known from
the work of Bushnell--Kutzko (see in particular \cite{BuKu2}) that every affine Hecke algebra
associated to a semisimple type for $\GL_n (F)$ is isomorphic to the Iwahori--spherical
Hecke algebra of some $\prod_i \GL_{n_i} (F_i)$, where $\sum_i n_i \leq n$ and $F_i$ is
a finite extension of the field $F$. A similar statement holds for Bernstein components
in the principal series of $F$-split reductive groups \cite[Lemma 9.3]{Roc}.

Dat \cite[Corollary 1.1.4]{Dat} has generalized this to groups of ``GL-type", and in
\cite[Theorem 1.1.2]{Dat} he proves that for such a group $\rZ_{\mc G^\vee}(\phi (\mb I_F))$
determines $\prod_\fs \Rep (\mc G (F))^\fs$, where $\fs$ runs over all Bernstein components
that correspond to extensions of $\phi |_{\mb I_F}$ to $\mb W_F \times \SL_2 (\C)$.
In \cite[\S 1.3]{Dat} Dat discusses possible generalizations of these results to other
reductive groups, but he did not fully handle the cases where
$\rZ_{\mc G^\vee}(\phi (\mb I_F))$ is disconnected. (It is always connected for groups of
GL-type.) Theorem \ref{thm:A}, in combination with the considerations about inner twists
of $\GL_n (F)$ in Paragraph \ref{par:innGLn}, provide explanations for all the equivalences
between Hecke algebras and between categories found by Dat.

Heiermann \cite[\S 1]{Hei} has associated affine Hecke algebras (possibly extended with a
finite $R$-group) to certain collections of enhanced L-parameters for classical groups
(essentially these sets constitute unions of Bernstein components).
Unlike Lusztig he does not base this on geometric constructions in complex groups, rather
on affine Hecke algebras previously found on the $p$-adic side in \cite{Hei1}. In his
setup \eqref{eq:1.2} holds true by construction, but the Hecke algebras are only related
to L-parameters via the LLC, so not in an explicit way.

In \cite[\S 2]{Hei} it is shown that every Bernstein component of enhanced L-parameters for
a classical group is in bijection with a Bernstein component of enhanced unramified 
L-parameters for a product of classical groups of smaller rank. (Some cases require extending
the relevant notions to full orthogonal groups, which is straightforward.)
So in the context of \cite{Hei} the data that we use for affine
Hecke algebras are present, and the algebras appear as well (at least up to Morita
equivalence), but the link between them is not yet explicit.
In Paragraph \ref{par:classical} we discuss how our results clarify this.\\[2mm]

\textbf{Acknowledgements.}\\
We thank referee for his detailed reports, which were very helpful to clarify parts of the
paper.

\renewcommand{\theequation}{\arabic{section}.\arabic{equation}}
\counterwithin*{equation}{section}

\section{Twisted graded Hecke algebras}
\label{sec:0}
We will recall some aspects of the (twisted) graded Hecke algebras studied
in \cite{AMS2}. Let $G$ be a complex reductive group, possibly disconnected.
Let $M$ be a quasi-Levi subgroup of $G$, that is, a group of the form
$M = \rZ_G(\rZ(L)^\circ)$ where $L$ is a Levi subgroup of $G^\circ$. Notice
that $M^\circ = L$ in this case.

We write $T = \rZ(M)^\circ = \rZ(M^\circ)^\circ$, a torus in $G^\circ$.
Let $P^\circ = M^\circ U$ be a parabolic subgroup of $G^\circ$ with Levi
factor $M^\circ$ and unipotent radical $U$. We put $P = M U$.
Let $\mf t^*$ be the dual space of the Lie algebra $\mf t = \Lie (T)$.

Let $v \in \mf m = \Lie (M)$ be nilpotent, and denote its adjoint orbit
by $\cC_v^M$. Let $q\cE$ be an irreducible $M$-equivariant cuspidal local
system on $\cC_v^M$. Then the stalk $q\epsilon = q\cE |_v$ is an irreducible
representation of $A_M (v) = \pi_0 (\rZ_M (v))$. Conversely, $v$ and $q\epsilon$
determine $\cC_v^M$ and $q\cE$. By definition the cuspidality means that
$\Res^{A_M (v)}_{A_{M^\circ}(v)} q\epsilon$ is a direct sum of irreducible
cuspidal $A_{M^\circ}(v)$-representations. Let $\epsilon \in \Irr (A_{M^\circ}(v))$
be one of them, and let $\cE$ be the corresponding $M^\circ$-equivariant cuspidal
local system on $\cC_v^{M^\circ}$. Then $\cE$ is a subsheaf of $q\cE$.
See \cite[\S 5]{AMS} for more background.\\
The triple $(M,\cC_v^M,q\cE)$ (or $(M,v,q\epsilon)$) is called a cuspidal
quasi-support for $G$. We denote its $G$-conjugacy class by $[M,\cC_v^M,q\cE]_G$.
To these data we associate the groups
\begin{equation}\label{eq:0.20}
\begin{array}{l}
W_{q\cE} = \rN_G (q\cE ) / M , \quad \text{where $\rN_G (q\cE) = \mr{Stab}_{\rN_G (M)}(q\cE )$,} \\
W_{q\cE}^\circ =  \mr{Stab}_{\rN_{G^\circ} (M)}(q\cE )/M=\rN_{G^\circ M}(M) / M \\
W_\cE = \mr{Stab}_{\rN_{G^\circ} (M)}(\cE )/M=\rN_{G^\circ}(M^\circ) / M^\circ,\\
\cR_{q\cE} = \rN_G (P,q\cE) / M ,\quad \text{where $\rN_G (P,q\cE) = \rN_G (q\cE) \cap \rN_G (P)$.}
\end{array}
\end{equation}
The group $W_{q\cE}$ acts naturally on the set
\[
R(G^\circ,T) := \{ \alpha \in X^* (T) \setminus \{0\} : \alpha
\text{ appears in the adjoint action of } T \text{ on } \mf g \} .
\]
By \cite[Theorem 9.2]{Lus2} (see also \cite[Lemma 2.1]{AMS2})
$R(G^\circ,T)$ is a root system with Weyl group $W_{q\cE}^\circ$.
The group $\cR_{q\cE}$ is the stabilizer of the set of
positive roots determined by $P$ and
\[
W_{q\cE} = W_{q\cE}^\circ \rtimes \cR_{q\cE} .
\]
We choose semisimple subgroups $G_j \subset G^\circ$, normalized by $\rN_G (q \cE)$,
such that the derived group $G^\circ_\der$ is the almost direct product of the $G_j$.
In other words, every $G_j$ is semisimple, normal in $G^\circ M$, normalized by
$W_{q\cE}$ (which makes sense because it is already normalized by $M$),
and the multiplication map
\begin{equation}\label{eq:0.1}
m_{G^\circ} \colon \rZ(G^\circ )^\circ \times G_1 \times \cdots \times G_d \to G^\circ
\end{equation}
is a surjective group homomorphism with finite central kernel. The number $d$ is not
specified in advance, it indicates the number of independent variables in our
upcoming Hecke algebras.
Of course there are in general many ways to achieve \eqref{eq:0.1}. Two choices are
always canonical:
\begin{equation}\label{eq:0.18}
\begin{array}{ll}
\bullet & G_1 = G^\circ_\der, \text{ with } d = 1; \\
\bullet & \text{every } G_j \text{ is of the form } N_1 N_2 \cdots N_k, \text{ where }
\{ N_1, \ldots, N_k\} \\
& \text{ is a } \rN_G (q\cE) \text{-orbit of simple normal subgroups of } G^\circ.
\end{array}
\end{equation}
In any case, \eqref{eq:0.1} gives a decomposition
\begin{equation}\label{eq:0.6}
\mf g = \rZ(\mf g) \oplus \mf g_1 \oplus \cdots \oplus \mf g_d \quad
\text{where } \rZ(\mf g) = \Lie (\rZ(G^\circ)), \mf g_j = \Lie (G_j) .
\end{equation}
Each root system
\[
R_j := R(G_j T,T) = R(G_j, G_j \cap T)
\]
is a $W_{q\cE}$-stable union of irreducible components of $R(G^\circ,T)$.
Thus we obtain an orthogonal, $W_{q\cE}$-stable decomposition
\begin{equation}\label{eq:0.2}
R(G^\circ,T) = R_1 \sqcup \cdots \sqcup R_d .
\end{equation}
We let $\vec{\mb r} = (\mb r_1, \ldots, \mb r_d)$ be an array of variables,
corresponding to \eqref{eq:0.1} and \eqref{eq:0.2} in the sense that $\mb r_j$
is relevant for $G_j$ and $R_j$ only. We abbreviate
\[
\C [\vec{\mb r}] = \C [\mb r_1,\ldots, \mb r_d] .
\]
Let $\natural \colon (W_{q\cE} / W_{q\cE}^\circ)^2 \to \C^\times$ be a 2-cocycle.
Recall that the twisted group algebra $\C[W_{q\cE},\natural]$ has a $\C$-basis
$\{ N_w : w \in W_{q\cE} \}$ and multiplication rules
\[
N_w \cdot N_{w'} = \natural (w,w') N_{w w'} .
\]
In particular it contains the group algebra of $W_{q\cE}^\circ$.

Let $c \colon R(G^\circ,T)_\red \to \C$ be a $W_{q\cE}$-invariant function.

\begin{prop}\label{prop:0.1}
There exists a unique structure of an associative graded algebra on 
$\C [W_{q\cE},\natural] \otimes S(\mf t^*) \otimes \C[\vec{\mb r}]$, such that:
\begin{itemize}
\item the twisted group algebra $\C[W_{q\cE},\natural]$ is embedded as subalgebra in degree 0;
\item the algebra $S(\mf t^*) \otimes \C[\vec{\mb r}]$ of polynomial functions on
$\mf t \oplus \C^d$ is embedded as a subalgebra, with twice the usual grading on
$S(\mf t^*)$ and each $\mb r_j$ in degree 2;
\item $\C[\vec{\mb r}]$ is central;
\item the braid relation $N_{s_\alpha} \xi - {}^{s_\alpha}\xi N_{s_\alpha} =
c (\alpha) {\mb r}_j (\xi - {}^{s_\alpha} \xi) / \alpha$ \\
holds for all $\xi \in S(\mf t^*)$ and all simple roots $\alpha \in R_j$
\item $N_w \xi N_w^{-1} = {}^w \xi$ for all $\xi \in S (\mf t^*)$ and $w \in \cR_{q\cE}$.
\end{itemize}
\end{prop}
\begin{proof}
For $d = 1, G_1 = G^\circ_\der$ this is \cite[Proposition 2.2]{AMS2}. The general
case can be shown in the same way.
\end{proof}

We denote the algebra just constructed by $\mh H (\mf t, W_{q\cE}, c \vec{\mb r},
\natural)$. When $W_{q\cE}^\circ = W_{q\cE}$, there is no 2-cocycle, and write
simply $\mh H (\mf t, W_{q\cE}^\circ, c \vec{\mb r})$.
It is clear from the defining relations that
\begin{equation}\label{eq:0.3}
S(\mf t^*)^{W_{q\cE}} \otimes \C [\vec{\mb r}] = \mc O (\mf t \times \C^d )^{W_{q\cE}}
\text{ is a central subalgebra of } \mh H (\mf t, W_{q\cE}, c \vec{\mb r}, \natural) .
\end{equation}
By a central character of an $\mh H (\mf t, W_{q\cE}, c \vec{\mb r},\natural)$-module we 
shall mean an element of $\mf t / W_{q\cE} \times \C^d$ by which 
$\mc O (\mf t \times \C^d )^{W_{q\cE}}$ acts on that module.
For $\zeta \in \mf t^{W_{q\cE}} = Z (\mf g )^{\mf R_{q\cE}}$ and $(\pi,V) \in \mr{Mod}
(\mh H (\mf t, W_{q\cE}, c \vec{\mb r}, \natural) )$ we define $(\zeta \otimes \pi,V)
\in \mr{Mod} (\mh H (\mf t, W_{q\cE}, c \vec{\mb r}, \natural) )$ by
\[
(\zeta \otimes \pi )(f_1 f_2 N_w) = f_1 (\zeta) \pi (f_1 f_2 N_w)
\qquad f_1 \in S(\mf t^*), f_2 \in \C[\vec{\mb r}] , w \in W_{q\cE} .
\]
To the cuspidal quasi-support $[M,\cC_v^M,q\cE]_G$ we associated a particular 2-cocycle
\[
\natural_{q\cE} \colon (W_{q\cE} / W_{q\cE}^\circ)^2 \to \C^\times ,
\]
see \cite[Lemma 5.3]{AMS}. The pair $(M^\circ,v)$ also gives rise to a
$W_{q\cE}$-invariant function $c \colon R(G^\circ,T)_\red \to \Z$, see
\cite[Proposition 2.10]{Lus3} or \cite[(12)]{AMS2}. We denote the algebra
$\mh H (\mf t, W_{q\cE}, c \vec{\mb r}, \natural_{q\cE})$, with this particular $c$,
by $\mh H (G,M,q\cE,\vec{\mb r})$.

In \cite{AMS2} we only studied the case $d=1, R_1 = R(G^\circ,T)$, and we denoted
that algebra by $\mh H (G,M,q\cE)$. Fortunately the difference with
$\mh H (G,M,q\cE,\vec{\mb r})$ is so small that almost all properties of
$\mh H (G,M,q\cE)$ discussed in \cite{AMS2} remain valid for $\mh H (\mf t,
W_{q\cE}, c \vec{\mb r}, \natural_{q\cE})$. We will proceed to make this precise.

Write $v = v_1 + \cdots + v_d$ with $v_j \in \mf g_j = \Lie (G_j)$. Then
\[
\cC_v^{M^\circ} = \cC_{v_1}^{M_1} + \cdots + \cC_{v_d}^{M_d} \text{ , where }
M_j = M^\circ \cap G_j .
\]
The $M^\circ$-action on $(\cC_v^{M^\circ},\cE)$ can be inflated to
$\rZ(G^\circ)^\circ \times M_1 \times \cdots \times M_d$, and the pullback of $\cE$
becomes trivial on $\rZ(G^\circ)^\circ$ and decomposes uniquely as
\begin{equation}\label{eq:0.9}
m^*_{G^\circ} \cE = \cE_1 \otimes \cdots \otimes \cE_d
\end{equation}
with $\cE_j$ a $M_j$-equivariant cuspidal local system on $\cC_{v_j}^{M_j}$. From
Proposition \ref{prop:0.1} and \cite[Proposition 2.2]{AMS2} we see that
\begin{equation}\label{eq:0.4}
\mh H (G^\circ,M^\circ,\cE,\vec{\mb r}) = \mh H (G_1,M_1,\cE_1) \otimes \cdots
\otimes \mh H (G_d,M_d,\cE_d) .
\end{equation}
Furthermore the proof of \cite[Proposition 2.2]{AMS2} shows that
\begin{equation}\label{eq:0.5}
\mh H (G,M,q\cE,\vec{\mb r}) = \mh H (G^\circ,M^\circ,\cE,\vec{\mb r}) \rtimes
\C [ \cR_{q\cE} , \natural_{q\cE}] .
\end{equation}

\subsection{Standard modules} \

To parametrize the irreducible representations of the above algebras we use some
elements of the Lie algebras of the involved algebraic groups. Let $\sigma_0 \in
\mf g$ be semisimple and $y \in \rZ_{\mf g}(\sigma_0)$ be nilpotent. We decompose
them along \eqref{eq:0.6}:
\begin{align*}
& \sigma_0 = \sigma_z + \sigma_{0,1} + \cdots + \sigma_{0,d} \quad \text{with }
\sigma_{0,j} \in \mf g_j, \sigma_z \in \rZ( \mf g) , \\
& y = y_1 + \cdots + y_d \quad \text{with } y_j \in \mf g_j .
\end{align*}
Choose algebraic homomorphisms $\gamma_j \colon\SL_2 (\C) \to \rZ_{G_j}(\sigma_{0,j})$
with d$\gamma_j \matje{0}{1}{0}{0} = y_j$. Given $\vec{r} \in \C^d$, we write
$\sigma_j = \sigma_{0,j} + \textup{d}\gamma_j \matje{r_j}{0}{0}{-r_j}$ and
\begin{equation}\label{eq:0.16}
\begin{aligned}
& \textup{d} \vec{\gamma} \matje{\vec r}{0}{0}{-\vec r} = \textup{d}\gamma_1
\matje{r_1}{0}{0}{-r_1} + \cdots + \textup{d}\gamma_d \matje{r_d}{0}{0}{-r_d} , \\
& \sigma = \sigma_0 + \textup{d} \vec{\gamma} \matje{\vec r}{0}{0}{-\vec r} .
\end{aligned}
\end{equation}
Notice that $[\sigma,y_j] = [\sigma_j,y_j] = 2 r_j y_j$. Let us recall
the construction of the standard modules from \cite{Lus3,AMS2}.
We need the groups
\begin{align*}
& M_j (y_j) = \big\{ (g_j,\lambda_j) \in G_j \times \C^\times :
\Ad (g_j) y_j = \lambda_j^2 y_j \big\} , \\
& \vec{M^\circ}(y) =  \big\{ (g,\vec{\lambda}) \in G^\circ \times (\C^\times )^d :
\Ad (g) y_j = \lambda_j^2 y_j \: \forall j = 1,\ldots,d \big\} , \\
& \vec{M}(y) =  \big\{ (g,\vec{\lambda}) \in G^\circ \rN_G (q\cE) \times (\C^\times )^d :
\Ad (g) y_j = \lambda_j^2 y_j \: \forall j = 1,\ldots,d \big\} ,
\end{align*}
and the varieties
\begin{align*}
& \mc P_{y_j} = \big\{ g (P^\circ \cap G_j) \in G_j / (P^\circ \cap G_j) :
\Ad (g^{-1}) y_j \in \cC_{v_j}^{M_j} + \Lie (U \cap G_j) \big\} , \\
& \mc P_y^\circ = \big\{ g P^\circ \in G^\circ / P^\circ :
\Ad (g^{-1}) y \in \cC_v^{M^\circ} + \Lie (U) \big\} , \\
& \mc P_y = \big\{ g P \in G^\circ \rN_G (q\cE) / P :
\Ad (g^{-1}) y \in \cC_v^M + \Lie (U) \big\} .
\end{align*}
The local systems $\cE_j, \cE$ and $q\cE$ give rise to local systems
$\dot{\cE_j}, \dot{\cE}$ and $\dot{q\cE}$ on $\mc P_{y_j}, \mc P_y^\circ$
and $\mc P_y$, respectively. The groups $M_j (y_j), \vec{M^\circ}(y)$ and
$\vec{M}(y)$ act naturally on, respectively, $(\mc P_{y_j},\dot{\mc E_j}),
(\mc P_y^\circ, \dot{\mc E})$ and $(\mc P_y, \dot{q\cE})$.
With the method from \cite{Lus3} and \cite[\S 3.1]{AMS2} we can define
an action of $\mh H (G,M,q\cE,\vec{\mb r}) \times \vec{M}(y)$ on the
equivariant homology $H_*^{\vec{M}(y)^\circ} (\mc P_y,\dot{q\cE})$, and
similarly for $H_*^{\vec{M^\circ}(y)^\circ} (\mc P_y^\circ,\dot{\cE})$ and
$H_*^{M_j(y)^\circ} (\mc P_{y_j},\dot{\cE_j})$. As in \cite{Lus3} we build
\[
E^\circ_{y_j,\sigma_j,r_j} = \C_{\sigma_j,r_j} \underset{H^*_{M_j (y_j)^\circ}
(\{y_j\})}{\otimes} H_*^{M_j(y)^\circ} (\mc P_{y_j},\dot{\cE_j}) .
\]
Similarly we introduce
\begin{align*}
& E^\circ_{y,\sigma,\vec{r}} = \C_{\sigma,\vec{r}} \underset{H^*_{\vec{M^\circ}(y)^\circ}
(\{y\})}{\otimes} H_*^{\vec{M^\circ}(y)^\circ} (\mc P_y^\circ,\dot{\cE}) , \\
& E_{y,\sigma,\vec{r}} = \C_{\sigma,\vec{r}} \underset{H^*_{\vec{M} (y)^\circ}
(\{y\})}{\otimes} H_*^{\vec{M}(y)^\circ} (\mc P_y,\dot{q\cE}) .
\end{align*}
By \cite[Theorem 3.2 and Lemma 3.6]{AMS2} these are modules over, respectively,\\
$\mh H (G_j,M_j,\cE_j) \times \pi_0 (\rZ_{G_j}(\sigma_{0,j},y_j))$,
$\mh H (G^\circ,M^\circ,\cE,\vec{\mb r}) \times \pi_0 (\rZ_{G^\circ}(\sigma_0,y))$ and\\
$\mh H (G,M,q\cE,\vec{\mb r}) \times \pi_0 (\rZ_{G^\circ \rN_G (q\cE)} (\sigma_0 ,y))$.
This last action is the reason to use $G^\circ \rN_G (q\cE)$ instead of $G$ in the
definition of $\mc P_y$.

In terms of \eqref{eq:0.5}, there is a natural module isomorphism
\begin{equation}\label{eq:0.8}
E_{y,\sigma,\vec{r}} \cong \ind_{\mh H (G^\circ,M^\circ, \cE, \vec{\mb r})}^{
\mh H (G,M,q\cE,\vec{\mb r})} E^\circ_{y,\sigma,\vec{r}} .
\end{equation}
It can be proven in the same way as the analogous statement with only
one variable $\mb r$, which is \cite[Lemma 3.3]{AMS2}.

\begin{lem}\label{lem:0.2}
With the identifications \eqref{eq:0.4} there is a natural isomorphism of\\
$\mh H (G^\circ,M^\circ, \cE, \vec{\mb r})$-modules
\[
E^\circ_{y,\sigma,\vec{r}} \cong  \C_{\sigma_z} \otimes E^\circ_{y_1,\sigma_1,r_1}
\otimes \cdots \otimes E^\circ_{y_d,\sigma_d,r_d} ,
\]
which is equivariant for the actions of the appropriate subquotients of $\vec{M^\circ}(y)$.
\end{lem}
\begin{proof}
From \eqref{eq:0.1} and $Z (G^\circ) Z (G_j) \subset P^\circ$ we get natural isomorphisms
\begin{equation}\label{eq:0.10}
\mc P_{y_1} \times \cdots \times \mc P_{y_d} \to \mc P_y^\circ .
\end{equation}
Looking at \eqref{eq:0.9} and the construction of $\dot{\cE}$ in \cite[\S 3.4]{Lus3},
we deduce that
\begin{equation}\label{eq:0.12}
\dot{\cE} \cong \dot{\cE_1} \otimes \cdots \otimes \dot{\cE_d}
\text{ as sheaves on } \mc P_y^\circ .
\end{equation}
From \eqref{eq:0.1} we also get a central extension
\begin{equation}\label{eq:0.11}
1 \to  \ker m_{G^\circ} \to \rZ(G^\circ)^\circ \times M_1 (y_1) \times \cdots
\times M_d (y_d)  \to \vec{M^\circ}(y) \to 1.
\end{equation}
Here $\ker m_{G^\circ}$ refers to the kernel of \eqref{eq:0.1}, a finite central
subgroup which acts trivially on the sheaf $\cE_1 \otimes \cdots \otimes \cE_d \cong
m_{G^\circ}^* \cE$. Restricting to connected components, we obtain a
central extension of $\vec{M^\circ}(y)^\circ$ by
\[
\tilde M := \rZ(G^\circ)^\circ \times M_1 (y_1)^\circ \times \cdots
\times M_d (y_d)^\circ
\]
In fact, equivariant (co)homology is inert under finite central extensions, for all
groups and all varieties. We sketch how this can be deduced from \cite[\S 1]{Lus3}.
By definition
\[
H^*_{\vec{M^\circ}(y)^\circ} (\mc P_y^\circ,\dot{\cE}) = H^* \big(\vec{M^\circ}(y)^\circ
\backslash (\Gamma \times \mc P_y^\circ), {}_\Gamma \dot{\cE} \big)
\]
for a suitable (in particular free) $\vec{M^\circ}(y)^\circ$-variety $\Gamma$ and
a local system derived from $\dot{\cE}$.
On the right hand side we can replace $\vec{M^\circ}(y)^\circ$ by $\tilde M$ without
changing anything. If $\tilde \Gamma$ is a suitable variety for $\tilde M$, then
$\tilde \Gamma \times \Gamma$ is also one. (The freeness is preserved because
\eqref{eq:0.11} is an extension of finite index.) The argument in \cite[p. 149]{Lus3}
shows that
\[
H^* \big(\tilde{M} \backslash (\Gamma \times \mc P_y^\circ), {}_\Gamma \dot{\cE} \big)
\cong H^* \big(\tilde{M} \backslash (\tilde \Gamma \times \Gamma \times \mc P_y^\circ),
{}_{\tilde \Gamma \times \Gamma} \dot{\cE} \big) = H^*_{\tilde{M}} (\mc P_y^\circ,\dot{\cE}) .
\]
In a similar way, using \cite[Lemma 1.2]{Lus3}, one can prove that
\begin{equation}\label{eq:0.13}
H_*^{\vec{M^\circ}(y)^\circ} (\mc P_y^\circ,\dot{\cE}) \cong
H_*^{\tilde M} (\mc P_y^\circ,\dot{\cE}) .
\end{equation}
The upshot of \eqref{eq:0.10}, \eqref{eq:0.12} and \eqref{eq:0.13} is that we can factorize
the entire setting along \eqref{eq:0.4}, which gives
\begin{equation}\label{eq:0.7}
H_*^{M_1 (y)^\circ} (\mc P_{y_1},\dot{\cE_1}) \otimes \cdots \otimes
H_*^{M_d (y)^\circ} (\mc P_{y_d},\dot{\cE_d}) \cong
H_*^{\vec{M^\circ}(y)^\circ} (\mc P_y^\circ,\dot{\cE}) .
\end{equation}
The equivariant cohomology of a point with respect to a connected group depends
only on the Lie algebra \cite[\S 1.11]{Lus3}, so \eqref{eq:0.11}
implies a natural isomorphism
\[
H^*_{Z (G^\circ)^\circ} (\{1\}) \times H^*_{M_1 (y_1)^\circ} (\{y_1\}) \times \cdots
\times H^*_{M_d (y_d)^\circ} (\{y_d\}) \cong
H^*_{\vec{M^\circ}(y)^\circ} (\{y\}) .
\]
Thus we can tensor both sides of \eqref{eq:0.7} with
$\C_{\sigma,\vec{r}}$ and preserve the isomorphism.
\end{proof}

Given $\rho_j \in \Irr \big( \pi_0 (\rZ_{G_j}(\sigma_{0,j},y_j)) \big)$, we can form
the standard $\mh H (G_j,M_j,\cE_j)$-module
\[
E^\circ_{y_j,\sigma_j,r_j,\rho_j} :=
\Hom_{\pi_0 (\rZ_{G_j}(\sigma_{0,j},y_j))} (\rho_j, E^\circ_{y_j,\sigma_j,r_j}) .
\]
Similarly $\rho^\circ \in \Irr \big( \pi_0 (\rZ_{G^\circ}(\sigma_0,y)) \big)$ and
$\rho \in \Irr \big( \pi_0 (\rZ_{G^\circ \rN_G (q\cE)}(\sigma_0,y)) \big)$ give rise to
\begin{equation}\label{eq:0.15}
\begin{aligned}
& E^\circ_{y,\sigma,\vec{r},\rho^\circ} :=
\Hom_{\pi_0 (\rZ_{G^\circ}(\sigma_0 ,y))} (\rho^\circ, E^\circ_{y,\sigma,\vec{r}}) , \\
& E_{y,\sigma,\vec{r},\rho} := \Hom_{\pi_0 (\rZ_{G^\circ \rN_G (q\cE)}(\sigma_0 ,y))}
(\rho, E_{y,\sigma,\vec{r}}) .
\end{aligned}
\end{equation}
We call these standard modules for respectively
$\mh H (G^\circ,M^\circ,\cE,\vec{\mb r})$ and $\mh H (G,M,q\cE,\vec{\mb r})$.

The canonical map \eqref{eq:0.1} induces a surjection
\begin{equation}\label{eq:0.14}
\pi_0 (\rZ_{G_1} (\sigma_{0,1}, y_1)) \times \cdots \times
\pi_0 (\rZ_{G_d} (\sigma_{0,d}, y_d)) \to \pi_0 (\rZ_{G^\circ}(\sigma_0,y)) .
\end{equation}

\begin{lem}\label{lem:0.3}
Let $\rho^\circ \in \Irr \big( \pi_0 (\rZ_{G^\circ}(\sigma_0,y)) \big)$ and let
$\bigotimes_{j=1}^d \rho_j$ be its inflation to \\$\prod_{j=1}^d \pi_0 (\rZ_{G_j}
(\sigma_{0,j}, y_j))$ via \eqref{eq:0.14}. There is a natural isomorphism of
$\mh H (G^\circ,M^\circ,\cE,\vec{\mb r})$-modules
\[
E^\circ_{y,\sigma,\vec{r},\rho^\circ} \cong  \C_{\sigma_z} \otimes E^\circ_{y_1,
\sigma_1,r_1,\rho_1} \otimes \cdots \otimes E^\circ_{y_d,\sigma_d,r_d,\rho_d} .
\]
Every $\bigotimes_{j=1}^d \rho_j \in \Irr \big( \prod_{j=1}^d \pi_0 (\rZ_{G_j}
(\sigma_{0,j}, y_j)) \big)$ for which $\bigotimes_{j=1}^d
E^\circ_{y_j,\sigma_j,r_j,\rho_j}$ is nonzero comes from
$\pi_0 (\rZ_{G^\circ} (\sigma_0,y))$ via \eqref{eq:0.14}.
\end{lem}
\begin{proof}
The module isomorphism follows from the naturality and the equivariance in
Lemma \ref{lem:0.2}.

Suppose that $\bigotimes_{j=1}^d \rho_j \in \Irr \big( \prod_{j=1}^d \pi_0
(\rZ_{G_j} (\sigma_{0,j}, y_j)) \big)$ appears in $\bigotimes_{j=1}^d
E^\circ_{y_j,\sigma_j,r_j}$. By \cite[Proposition 3.7]{AMS2} the cuspidal support
$\Psi_{\rZ_G (\sigma_{0,j})}(y_j,\rho_j)$ is $G_j$-conjugate to
$(M_j,\cC_{y_j}^{M_j},\cE_j)$. In particular $\rho_j$ has the same $\rZ(G_j)$-character
as $\cE_j$, see \cite[Theorem 6.5.a]{Lus2}. Hence $\otimes_j \rho_j$ has the same
central character as $m_{G_0}^* \cE$. That central character factors through the
multiplication map \eqref{eq:0.1} whose kernel is central, so $\otimes_j \rho_j$
also factors through \eqref{eq:0.1}. That is, the map \eqref{eq:0.14} induces
a bijection between the relevant irreducible representations on both sides.
\end{proof}

For some choices of $\rho$ the standard module $E_{y,\sigma,\vec{r},\rho}$ is zero.
To avoid that, we consider triples $(\sigma_0,y,\rho)$ with:
\begin{itemize}
\item $\sigma_0 \in \mf g$ is semisimple,
\item $y \in \rZ_{\mf g}(\sigma_0)$ is nilpotent,
\item $\rho \in \Irr \big( \pi_0 (\rZ_G (\sigma_0,y)) \big)$ is
such that the cuspidal quasi-support $q\Psi_{\rZ_G (\sigma_0)}(y,\rho)$
from \cite[\S 5]{AMS} is $G$-conjugate to $(M,\cC_v^M,q\cE)$.
\end{itemize}
Given in addition $\vec{r} \in \C^d$, we construct $\sigma = \sigma_0 +
\textup{d} \vec{\gamma} \matje{\vec r}{0}{0}{-\vec r} \in \mf g$ as in
\eqref{eq:0.16}. Although this depends on the choice of $\vec{\gamma}$, the
conjugacy class of $\sigma$ does not.

By definition
\[
\mh H (G^\circ \rN_G (q\cE),M,q\cE,\vec{\mb r}) = \mh H (G,M,q\cE,\vec{\mb r}),
\]
but of course $\pi_0 (\rZ_{G^\circ \rN_G (q\cE)}(\sigma_0,y))$ can be a proper
subgroup of $\pi_0 (\rZ_G (\sigma_0,y))$. As shown in the proof of
\cite[Lemma 3.21]{AMS2}, the functor $\ind_{\pi_0 (\rZ_{G^\circ \rN_G (q\cE)}
(\sigma_0,y))}^{\pi_0 (\rZ_G (\sigma_0,y))}$ provides a bijection between the
$\tilde \rho$ in the triples for $G^\circ \rN_G (q\cE)$ and the $\rho$ in the
triples for $G$. For $\rho = \ind_{\pi_0 (\rZ_{G^\circ \rN_G (q\cE)}(\sigma_0,y))
}^{\pi_0 (\rZ_G (\sigma_0,y))} \tilde{\rho}$ we define, in terms of \eqref{eq:0.15},
\begin{equation}\label{eq:0.17}
E_{y,\sigma,\vec{r},\rho} = E_{y,\sigma,\vec{r},\tilde{\rho}} .
\end{equation}
We would like to exhibit the central characters of these standard 
$\mh H (G,M,q\cE,\vec{\mb r})$-modules. It has turned out that the treatment of this aspect 
in \cite{AMS2} was flawed, we correct that here. We fix a homomorphism of algebraic groups
\[
\gamma_v \colon \SL_2 (\C) \to M \quad \text{with} \quad \textup{d} \gamma_v \matje{0}{1}{0}{0} = v.
\]
We write
\begin{equation}\label{eq:0.24}
\textup{d} \gamma_v \matje{1}{0}{0}{-1} = \sigma_v = \sigma_{v,1} + \cdots + \sigma_{v,d}
\quad \text{where} \quad \sigma_{v,j} \in \mr{Lie}(M \cap G_j) . 
\end{equation}
For $\vec{r} \in \C^d$ we put
\[
\vec{r} \sigma_v = r_1 \sigma_{v,1} + \cdots + r_d \sigma_{v,d} \in \mf m. 
\]
We record the linear bijection
\[
\begin{array}{cccc}
\Sigma_v \colon& \mf t \oplus \C^d & \to & \mf t \oplus \C^d (\sigma_v,1) \\
& (\sigma_0, \vec{r}) & \mapsto & (\sigma_0 + \vec{r} \sigma_v , \vec{r} )
\end{array} .
\]
Here the target is a linear subspace of $\mf m \oplus \C^d$ and the inverse map is
\[
\Sigma_v^{-1} \colon(\sigma,\vec{r}) \mapsto (\sigma - \vec{r} \sigma_v, \vec{r}) .
\]
The next result is a correction of \cite[Proposition 3.5]{AMS2}, which was based on a wrong
interpretation of \cite[\S 8]{Lus5}. Our improvement consists mainly of adding
$\Sigma_v^{\pm 1}$ at the right places.

\begin{prop}\label{prop:0.7}
Let $(y,\sigma,\vec{r})$ be as in \eqref{eq:0.16} and assume that $\mc P_y$ is nonempty.
\enuma{
\item $(\mr{Ad}(\rN_G (P,q\cE) G^\circ) \sigma - \vec{r} \sigma_v) \cap \mf t$ is a single 
$W_{q\cE}$-orbit in $\mf t$.
\item The $\mh H (G,M,q\cE,\vec{\mb r})$-module $E_{y,\sigma,\vec{r}}$ admits the central
character\\ $((\mr{Ad}(\rN_G (P,q\cE) G^\circ) \sigma - \vec{r} \sigma_v) \cap \mf t, \vec{r})$.
\item The pair $(y,\sigma)$ is $G^\circ$-conjugate to one with $\sigma_0$ and 
$\vec{r} \sigma_v + \textup{d} \vec{\gamma} \matje{-\vec{r}}{0}{0}{\vec{r}}$ in $\mf t$.
\item Suppose $(y,\sigma)$ has the properties as in (c). Then $\sigma_0, \sigma_v$ and 
$\textup{d} \vec{\gamma} \matje{1}{0}{0}{-1}$ commute, and 
$\sigma_v + \textup{d} \vec{\gamma} \matje{-1}{0}{0}{1} \in \mf t_\R$.
\item Suppose $(y,\sigma)$ has the properties as in (c). Then the central character of
$E_{y,\sigma,\vec{r},\rho}$ can be expressed as $W_{q\cE} \big( \sigma_0 \pm (\vec{r} \sigma_v 
+ \textup{d}\vec{\gamma} \matje{-\vec{r}}{0}{0}{\vec{r}}), \vec{r}\big)$.
}
\end{prop}
\begin{proof}
(a) By \cite[Theorem 8.11]{Lus5}, $\mh H (G_j,M_j,\cE_j,\mb r_j)$ is canonically isomorphic to
the endomorphism algebra of a certain perverse sheaf $K_j^*$, in the 
$G_j \times \C^\times$-equivariant bounded derived category of constructible sheaves on
$\mf g_j$. According to \cite[\S 8.13.a]{Lus5}, there exists a canonical surjection 
\begin{equation}\label{eq:0.21}
H_{G_j \times \C^\times}^* (\mr{point}) \cong \mc O (\mf g_j \oplus \C)^{G_j \times \C^\times}
= \mc O (\mf g_j)^{G_j} \otimes \C [\mb r_j] \to \rZ( \mr{End} (K_j^*) ).
\end{equation}
By \cite[Lemma 2.3]{AMS2} the right hand side is
\begin{equation}\label{eq:0.22}
\rZ( \mr{End} (K_j^*)) \cong \rZ( \mh H (G_j,M_j,\cE_j,\mb r_j)) \cong 
\mc O (\mf t \cap \mf g_j)^{W_{\cE_j}} \otimes \C [\mb r_j] .
\end{equation}
By \cite[\S 8.13.b]{Lus5}, the composition of \eqref{eq:0.21} and \eqref{eq:0.22} corresponds
to an injection like $\Sigma_v$, namely 
\begin{equation}\label{eq:0.23}
\begin{array}{ccc}
(\mf t \cap \mf g_j) / W_{\cE_j} \oplus \C & \to & 
\Irr \big( \mc O (\mf g_j \oplus \C)^{G_j \times \C^\times} \big) \\
(\sigma_{0,j}, r_j) & \mapsto & (\sigma_{0,j} + r_j \sigma_{v,j}, r_j)
\end{array} ,
\end{equation}
where the right hand side is the variety of semisimple adjoint orbits in $\mf g_j \oplus \C$.
Hence 
\[
(\mr{Ad}(G_j) \sigma_j - r_j \sigma_{v,j}) \cap \mf t \cap \mf g_j = 
\mr{Ad}(G_j) \sigma_{0,j} \cap \mf t \cap \mf g_j
\]
is either empty or a single $W_{\cE_j}$-orbit. We will see in the proof of part (b) that it is
nonempty. Combining these statements for all $j = 1,\ldots, d$, we find that 
$(\mr{Ad}(G^\circ) \sigma - \vec{r} \sigma_v) \cap \mf t$ is a single $W_\cE$-orbit 
$W_\cE \sigma' \subset \mf t$. As $M$ stabilizes $\mf m$ and centralizes $\mf t$:
\begin{equation}\label{eq:0.26}
\begin{aligned} \mr{Ad}(G^\circ M) \sigma \cap (\mf t + \vec{r} \sigma_v) & = 
\mr{Ad}(M) (W_\cE \sigma'  + \vec{r} \sigma_v)  \cap (\mf t + \vec{r} \sigma_v) \\  
& = (W_\cE \sigma' + \mr{Ad}(M) \vec{r} \sigma_v ) \cap (\mf t + \vec{r} \sigma_v) .
\end{aligned}
\end{equation}
Here $\mr{Ad}(M) (\vec{r} \sigma_v)$ lies in the derived subalgebra of $\mf m$, so the right hand 
side of \eqref{eq:0.26} equals $W_\cE \sigma' + \vec{r} \sigma_v$. In other words,
\[
(\mr{Ad}(G^\circ M) \sigma - \vec{r} \sigma_v) \cap \mf t = W_\cE \sigma' .
\]
As $\rN_G (P,q\cE) G^\circ / G^\circ M \cong W_{q\cE} / W_\cE$, we can pass from 
Ad$(G^\circ M)$-orbits to\\ Ad$(\rN_G (P,q\cE) G^\circ)$-orbits in the required way.\\
(b) The assumption $\mc P_y \neq \emptyset$ implies that 
$H_*^{\vec{M} (y)^\circ}( \mc P_y, \dot{q\cE})$ is nonzero. By \cite[Proposition 8.6.c]{Lus3}
and because the semisimple adjoint orbits in Lie$(\vec{M} (y)^\circ)$ form an irreducible
variety, $E^\circ_{y,\sigma,\vec{r}}$ is nonzero for all eligible $(\sigma,\vec{r})/ \sim$.

The adjoint action of $\mc O (\mf t \cap \mf g_j)^{W_{\cE_j}} \otimes \C [\mb r_j]$ on
$E_{y_j,\sigma_j,r_j}$ can be realized as 
\[
Z (\mh H (G_j,M_j,\cE_j,\mb r_j)) \leftarrow H^*_{G_j \times \C^\times}(\mr{point}) \to
H^*_{M_j (y_j)^\circ}( y_j) \to H^*_{M_j (y_j)^\circ} (\mc P_{y_j}) 
\]
and then the product in equivariant homology. By construction $H^*_{M_j (y_j)^\circ}( y_j)$
acts on $E_{y_j,\sigma_j,r_j}$ via the character $(\sigma_j,r_j) / \sim$. Hence 
$H^*_{G_j \times \C^\times}(\mr{point})$ acts via the character 
$\mr{Ad}(G_j \times \C^\times)(\sigma_j,r_j)$. In view of \eqref{eq:0.21}--\eqref{eq:0.23},
$\rZ(\mh H (G_j,M_j,\cE_j,\mb r_j))$ acts via 
\[
( (\mr{Ad}(G_j) \sigma_j - r_j \sigma_{v,j}) \cap \mf g_j \cap \mf t, r_j) .
\]
For all $j = 1,\ldots,d$ together, this shows that $\rZ( \mh H (G^\circ,M^\circ,\cE,\vec{\mb r})$
acts on $E^\circ_{y,\sigma,\vec{r}}$ as
$( (\mr{Ad}(G^\circ) \sigma - \vec{r} \sigma_v) \cap \mf t, \vec{r})$.
Now we use that $\rN_G (P,q\cE) G^\circ / G^\circ \cong W_{q\cE} / W_\cE$ and 
\[
\rZ(\mh H (G,M,q\cE,\mb r)) = \rZ( \mh H (G^\circ,M^\circ,\cE,\vec{\mb r}))^{W_{q\cE} / W_\cE}, 
\]
and we conclude with \eqref{eq:0.8}.\\
(c) By part (b) with $r = 0$ we may assume that $\sigma_0 \in \mf t$. Then $\exp (y)$ is 
contained in the reductive group $\rZ_G (\sigma_0)^\circ$, so we can arrange that the image of
$\vec{\gamma}$ lies in there. Applying part (b) to this group, we find $g \in \rZ_G (\sigma_0)^\circ$
such that 
\[
\mr{Ad}(g) \sigma = \sigma_0 + \mr{Ad}(g) \textup{d} \vec{\gamma} \matje{\vec{r}}{0}{0}{-\vec{r}}
\quad \text{lies in} \quad \mf t + \vec{r} \sigma_v .
\]
Then $\mr{Ad}(g) \textup{d} \vec{\gamma} \matje{\vec{r}}{0}{0}{-\vec{r}} - \vec{r} \sigma_v 
\in \mf t$, so $(\mr{Ad}(g) y, \mr{Ad}(g) \sigma)$ has the required properties. \\
(d) The assumption and $\sigma_v \in \mf m$ imply that $\textup{d}\gamma \matje{1}{0}{0}{-1}
\in \mf m$. As $\sigma_0 \in \mf t = \rZ(\mf m)$, it commutes with both $\sigma_v$ and  
$\textup{d}\gamma \matje{1}{0}{0}{-1}$. The latter two differ by an element of $\mf t = \rZ(\mf m)$, 
so they commute as well. It follows that 
\begin{equation}\label{eq:0.25}
\chi_{y,v} \colon z \mapsto \gamma \matje{z}{0}{0}{z^{-1}} \gamma_v \matje{z^{-1}}{0}{0}{z}
\end{equation}
is an algebraic cocharacter of $T$. By definition of $\mf t_\R$, the derivative 
\[
\textup{d} \chi_{y,v} \colon r \mapsto \textup{d}\gamma \matje{1}{0}{0}{-1} - r \sigma_v
\]
evaluates to an element of $X_* (T) \otimes_\Z \R = \mf t_\R$ for every $r \in \R$.\\
(e) As in the proof of part (c), we may assume that
$\mr{im}(\vec{\gamma}) \subset \rZ_G (\sigma_0)^\circ$. Put $s_y = \gamma \matje{0}{1}{-1}{0}$
and consider the parameter $(\mr{Ad}(s_y) \sigma, \mr{Ad}(s_y) y, \vec{r})$. We have
\[
\mr{Ad}(s_y) \sigma + \vec{r} s_v = \sigma_0 + \textup{d} \vec{\gamma} \matje{-\vec{r}}{0}{0}{\vec{r}}
- \vec{r} (-s_v) \in \mf t.
\]
Here $-s_v$ is the semisimple element in the $\mf{sl}_2$-triple $\textup{d}\gamma_v \matje{0}{0}{-1}{0},
-s_v, -v$, which is conjugate to $v, s_v, \textup{d}\gamma_v \matje{0}{0}{1}{0}$ by
$\gamma_v \matje{0}{1}{-1}{0} \in M$. Thus part (c) says that 
\[
W_{q\cE} \big( \sigma_0 + \vec{r} \sigma_v + \textup{d}\vec{\gamma} \matje{-\vec{r}}{0}{0}{\vec{r}}, 
\vec{r}\big)
\] 
is also the central character of $E_{y,\sigma,\vec{r},\rho}$. 
\end{proof}

\subsection{Irreducible modules} \

The standard modueles and the irreducible modules which are annihilated by $\vec{\mb r}$ must be 
treated somewhat differently from the others. We need an improvement on the analysis in 
\cite[Lemma 3.9 and Proposition 3.10]{AMS2}.

\begin{thm}\label{thm:0.8}
Let $E^\circ_{y,\sigma_0,0,\rho^\circ}$ be a nonzero standard 
$\mh H (G^\circ,M^\circ,\cE,\vec{\mb r})$-module on which each $\mb r_j$ acts as zero.
\enuma{
\item $E^\circ_{y,\sigma_0,0,\rho^\circ} = \Hom_{\pi_0 (\rZ_{G^\circ} (\sigma_0,y))}
(\rho^\circ, H_* (\mc P_y^\circ, \dot{\cE}))$ as graded $\C[W_\cE]$-modules, where we
use the homological grading from $H_* (\mc P_y^\circ, \dot{\cE})$.
\item For each $d \in \Z$, 
\[
\bigoplus\nolimits_{n \geq d} \Hom_{\pi_0 (\rZ_{G^\circ} (\sigma_0,y))} \big( \rho^\circ, 
H_n (\mc P_y^\circ, \dot{\cE})\big)
\]
is an $\mh H (G^\circ,M^\circ,\cE,\vec{\mb r})$-submodule of $E^\circ_{y,\sigma_0,0,\rho^\circ}$.
\item $E^\circ_{y,\sigma_0,0,\rho^\circ}$ has a unique irreducible quotient isomorphic
to the module $M^\circ_{y,\sigma_0,0,\rho^\circ}$ from \cite[Proposition 3.8]{AMS2}.
}
\end{thm}
In the proof of \cite[Lemma 3.9]{AMS2} it was assumed incorrectly that $S(\mf t^*)$ acts
semisimply on $E^\circ_{y,\sigma_0,0,\rho^\circ}$, which lead to the wrong claim that it is
always completely reducible as $\mh H (G^\circ,M^\circ,\cE,\vec{\mb r})$-module.
\begin{proof}
(a) By \cite[10.13]{Lus5}, $E^\circ_{y,\sigma_0,0,\rho^\circ}$ can be identified with
$H_* (\mc P_y^\circ, \dot{\cE})$ as graded $W_\cE \times \pi_0 (\rZ_{G^\circ}(\sigma_0,y))
$-representations. In fact these finite groups both preserve the homological grading. Now
apply $\Hom_{\pi_0 (\rZ_{G^\circ} (\sigma_0,y))}(\rho^\circ,?)$ to both modules.\\
(b) First we assume that $\sigma_0$ is central in $\mf g$. Write 
\[
\mh H (G^\circ,M^\circ,\cE,\vec{\mb r}) = S(\rZ(\mf g)^*) \otimes \mh H (G_\der^\circ,M \cap
G_\der^\circ,\cE,\vec{\mb r})
\]
as in \eqref{eq:0.4}. Then $S(\rZ(\mf g)^*)$ acts on $E^\circ_{y,\sigma_0,0,\rho^\circ}$ by
the character $\sigma_0$ and the restriction of $E^\circ_{y,\sigma_0,0,\rho^\circ}$ to
$\mh H (G_\der^\circ,M \cap G_\der^\circ,\cE,\vec{\mb r})$ is
\[
E^\circ_{y,0,0} = \C_{0,0} \underset{H^*_{\vec{M^\circ}(y)^\circ}
(\{y\})}{\otimes} H_*^{\vec{M^\circ}(y)^\circ} (\mc P_y^\circ,\dot{\cE}) .
\]
Here $H_*^{\vec{M^\circ}(y)^\circ} (\mc P_y^\circ,\dot{\cE})$ is a graded 
$\mh H (G_\der^\circ,M \cap G_\der^\circ,\cE,\vec{\mb r})$-module by \cite[Theorem 8.13]{Lus3},
although typically not semisimple. As $\C_{0,0}$ is a graded $H^*_{\vec{M^\circ}(y)^\circ}(\{y\})
$-module it follows that $E^\circ_{y,0,0}$ is a graded 
$\mh H (G_\der^\circ,M \cap G_\der^\circ,\cE,\vec{\mb r})$-module. For $E^\circ_{y,\sigma_0,0}$
as $\mh H (G^\circ,M^\circ,\cE,\vec{\mb r})$-module, we can still say that the action of 
$\mh H (G^\circ,M^\circ,\cE,\vec{\mb r})$ only raises degrees, and that an element 
$x \in \mh H (G^\circ,M^\circ,\cE,\vec{\mb r})$ of degree $n$ can only raise the degree of
an element of $E^\circ_{y,\sigma_0,0}$ by at most $n$.

Now we lift the condition on $\sigma_0$ and we consider the Levi subgroup 
$Q^\circ = \rZ_{G^\circ}(\sigma_0)$ of $G^\circ$. By Proposition \ref{prop:0.7}.a, 
$\rZ(\mr{Lie}(M^\circ))$ contains a $G^\circ$-conjugate of $\sigma_0$. Upon replacing $(y,\sigma_0)$
by a suitable $G^\circ$-conjugate, we may assume that $M^\circ$ centralizes $\sigma_0$,
so that $Q^\circ \supset M^\circ$. By \cite[Corollary 1.18]{Lus7}, there is a natural
isomorphism of $\mh H (G^\circ,M^\circ,\cE,\vec{\mb r})$-modules
\begin{equation}\label{eq:0.27} 
W_\cE \ltimes S(\mf t^*) \underset{W_\cE^{Q^\circ} \ltimes S (\mf t^* )}{\otimes} 
E^{Q^\circ}_{y,\sigma_0,0} =
 \mh H (G^\circ,M^\circ,\cE) \underset{\mh H ({Q^\circ},M^\circ,\cE)}{\otimes} 
E^{Q^\circ}_{y,\sigma_0,0} \longrightarrow E^\circ_{y,\sigma_0,0} .
\end{equation}
We note that \cite[Corollary 1.18]{Lus7} is applicable because $r = 0$ and $\ad (\sigma_0)$ 
is an invertible linear transformation of Lie$(U_{Q^\circ})$, where $U_{Q^\circ}$ is the unipotent 
radical of a parabolic subgroup of $G^\circ$ with Levi factor ${Q^\circ}$. The map \eqref{eq:0.27} 
comes from a morphism $\mc P_y^{Q^\circ} \to \mc P_y^\circ$, which entails that it changes all 
homological degrees (see part a) by the same amount, namely 
$\dim \mc P_y^\circ - \dim \mc P_y^{Q^\circ}$. From \eqref{eq:0.27} we deduce that, as 
$S(\mf t^*)$-modules and as graded vector spaces,
\[
E^\circ_{y,\sigma_0,0} [\dim \mc P_y^{Q^\circ} - \dim \mc P_y^\circ] = 
\bigoplus_{w \in W_\cE / W_\cE^{Q^\circ}} \C w \otimes_\C E^{Q^\circ}_{y,\sigma_0,0} =
\bigoplus_{w \in W_\cE / W_\cE^{Q^\circ}} (w)^* E^{Q^\circ}_{y,\sigma_0,0} ,
\]
where $[\cdots]$ denotes a degree shift.

We denote the ideal generated by the $\mb r_j$ by $(\vec{\mb r})$. As $(\vec{\mb r})$ is divided 
out in \eqref{eq:0.27}, the action of $w \in W_\cE$ on $S(\mf t^*)$ reduces to the usual action, 
induced from the action of $W_\cE$ on $\mf t$. Thus the property established above in the case 
$\sigma_0$ central remains valid here: the action of an element $x \in S(\mf t^*)$ of degree $n$ 
on $E^\circ_{y,\sigma_0,0}$ can only raise degrees, and raises them by at most $n$. Since 
\[
\mh H (G^\circ,M^\circ,\cE,\vec{\mb r}) / (\vec{\mb r}) = S(\mf t^*) \otimes \C [W_\cE] 
\]
as vector spaces and $\C[W_\cE]$ preserves the degrees (because it sits in degree 0, see also
part a), the degree properties of $E^\circ_{y,\sigma_0,0}$ also hold when we regard it as 
$\mh H (G^\circ,M^\circ,\cE,\vec{\mb r})$-module. 

In particular, for any $d \in \Z$ the sum of the subspaces in degree $\geq d$ is a submodule.
The action of $\pi_0 (\rZ_{G^\circ}(\sigma_0,y))$ preserves the degrees, so this remains valid for 
\[
\Hom_{\pi_0 (\rZ_{G^\circ} (\sigma_0,y))}(\rho^\circ,E^\circ_{y,\sigma_0,0}) = 
E^\circ_{y,\sigma_0,0,\rho^\circ} .
\]
Combine that with part (a) to obtain the stated form.\\
(c) With parts (a,b) instead of \cite[Lemma 3.9]{AMS2}, the proof of \cite[Proposition 3.10]{AMS2}
still works. It shows that $E^\circ_{y,\sigma_0,0,\rho^\circ}$ has a unique irreducible 
subquotient isomorphic to $M^\circ_{y,\sigma_0,0,\rho^\circ}$ and that the part of
$E^\circ_{y,\sigma_0,0,\rho^\circ}$ in one specific homological degree projects bijectively
onto this subquotient. More precisely, the argument \cite[(40)--(42)]{AMS2} shows that this 
subquotient appears as 
\[
\Hom_{\pi_0 (\rZ_{G^\circ} (\sigma_0,y))}(\rho^\circ, H_d (\mc P_y^\circ, \dot{\cE})) \subset 
E^\circ_{y,\sigma_0,0,\rho^\circ},
\]
where $d$ is the minimal degree for which this space is nonzero. By part (b),
\[ 
E^{\circ,>d}_{y,\sigma_0,0,\rho^\circ} := \bigoplus\nolimits_{n > d} \Hom_{\pi_0 (\rZ_{G^\circ} 
(\sigma_0,y))} \big( \rho^\circ, H_n (\mc P_y^\circ, \dot{\cE})\big)
\]
is a $\mh H (G^\circ,M^\circ,\cE,\vec{\mb r})$-submodule. Further 
\[
E^{\circ,d}_{y,\sigma_0,0,\rho^\circ} := 
E^\circ_{y,\sigma_0,0,\rho^\circ} / E^{\circ,>d}_{y,\sigma_0,0,\rho^\circ}   
\cong \Hom_{\pi_0 (\rZ_{G^\circ} (\sigma_0,y))}(\rho^\circ, H_d (\mc P_y^\circ, \dot{\cE}))
\] 
as $\C [W_\cE]$-modules. The proof of part (b) can be modified to study the 
$\mh H (G^\circ,M^\circ,\cE,\vec{\mb r})$-module $E^{\circ,d}_{y,\sigma_0,0,\rho^\circ}$.
In the case that $\sigma_0$ is central, it shows that $E^{\circ,d}_{y,\sigma_0,0,\rho^\circ}$
is a semisimple module, on which the $\mh H (G^\circ,M^\circ,\cE,\vec{\mb r})$-action factors
through 
\[
\begin{array}{cccll}
\mr{ev}_{\sigma_0,0} \colon & \mh H (G^\circ,M^\circ,\cE,\vec{\mb r}) / (\vec{\mb r}) & \to & \C[W_\cE] \\
& f x & \mapsto & f(\sigma_0) x & \quad f \in S(\mf t^*), x \in \C [W_\cE] .
\end{array}
\]
When $\sigma_0$ is not central, the functor 
$\mr{ind}^{W_\cE \ltimes S(\mf t^*)}_{W_\cE^{Q^\circ} \ltimes S (\mf t^* )}$ from
\eqref{eq:0.27} preserves irreduci\-bi\-li\-ty of modules that admit the 
$S(\mf t^* \oplus \C)$-character $(\sigma_0,0)$, because $W_\cE^{Q^\circ} = (W_\cE)_{\sigma_0}$. 
Hence $\mr{ind}^{W_\cE \ltimes S(\mf t^*)}_{W_\cE^{Q^\circ} \ltimes S (\mf t^* )}$ preserves the
semisimplicity of $E^{Q^\circ}_{y,\sigma_0,0,\rho^\circ}$.

This shows that the distinguished irreducible subquotient of $E^\circ_{y,\sigma_0,0,\rho^\circ}$
is a direct summand of $E^{\circ,d}_{y,\sigma_0,0,\rho^\circ}$. Since 
$E^{\circ,d}_{y,\sigma_0,0,\rho^\circ}$ is a quotient of $E^\circ_{y,\sigma_0,0,\rho^\circ}$,
so is our distinguished subquotient.
\end{proof}

The next result generalizes \cite[Theorem 3.20]{AMS2} to several variables $r_j$.
We define $\Irr_{\vec{r}}(\mh H (G,M,q\cE,\vec{\mb r}))$ to be the set of
equivalence classes of those irreducible representations of
$\mh H (G,M,q\cE,\vec{\mb r})$ on which $\mb r_j$ acts as $r_j$.

\begin{thm}\label{thm:0.4}
Fix $\vec{r} \in \C^d$. The standard $\mh H (G,M,q\cE,\vec{\mb r})$-module
$E_{y,\sigma,\vec{r},\rho}$ is nonzero if and only if
$q\Psi_{\rZ_G (\sigma_0)}(y,\rho) = (M,\cC_v^M,q\cE)$ up to $G$-conjugacy.
In that case it has a distinguished irreducible quotient $M_{y,\sigma,\vec{r},\rho}$,
which appears with multiplicity one in $E_{y,\sigma,\vec{r},\rho}$.

The map $M_{y,\sigma,\vec{r},\rho} \longleftrightarrow (\sigma_0,y,\rho)$
sets up a canonical bijection between \\ $\Irr_{\vec{r}}(\mh H (G,M,q\cE,\vec{\mb r}))$
and $G$-conjugacy classes of triples as above.
\end{thm}
\begin{proof}
For $\mh H (G_j,M_j,\cE_j)$ this is \cite[Proposition 3.7 and Theorem 3.11]{AMS2}, where
we use Theorem \ref{thm:0.8} to replace the input from the flawed \cite[Lemma 3.9]{AMS2}.
With \eqref{eq:0.4} and Lemma \ref{lem:0.3} we can generalize that to
$\mh H (G^\circ,M^\circ,\cE,\vec{\mb r})$. The method to go from there to
$\mh H (G^\circ \rN_G (q\cE),M,q\cE,\vec{\mb r})$ is exactly the same as in
\cite[\S 3.3--3.4]{AMS2} (for $\mh H (G^\circ,M^\circ,\cE)$ and
$\mh H (G^\circ \rN_G (q\cE),M,q\cE)$). That is, the proof of \cite[Theorem 3.20]{AMS2}
applies and establishes the theorem for $\mh H (G^\circ \rN_G (q\cE),M,q\cE,\vec{\mb r})$.
In view of \eqref{eq:0.17} we can replace $G^\circ \rN_G (q\cE)$ by $G$.
\end{proof}

The irreducible module $M_{y,\sigma,\vec{r},\rho}$ has the same central character as
the standard module $E_{y,\sigma,\vec{r},\rho}$ of which it is a quotient. It can be
found in Proposition \ref{prop:0.7}.

The above modules are compatible with parabolic induction, in a suitable sense
and under a certain condition. Let $Q \subset G$ be an algebraic subgroup containing $M$,
such that $Q^\circ$ is a Levi subgroup of $G^\circ$. Let
$y,\sigma,\vec r,\rho$ be as in Theorem \ref{thm:0.4}, with
$\sigma,y \in \mf q = \Lie (Q)$. By \cite[\S 3.2]{Ree} the natural map
\begin{equation}
\pi_0 (\rZ_Q (\sigma,y)) = \pi_0 ( \rZ_{Q \cap \rZ_G (\sigma_0)}(y)) \to
\pi_0 (\rZ_{\rZ_G (\sigma_0)}(y)) = \pi_0 (\rZ_G (\sigma,y))
\end{equation}
is injective, so we can consider the left hand side as a subgroup of the right
hand side. Let $\rho^Q \in \Irr \big( \pi_0 (\rZ_Q (\sigma,y)) \big)$ be such that
$q\Psi_{\rZ_Q (\sigma_0)} (y,\rho^Q) = (M,\cC_v^M,q\cE)$. Then $E_{y,\sigma,r,\rho},
M_{y,\sigma,r,\rho}, E^Q_{y,\sigma,r,\rho^Q}$ and $M^Q_{y,\sigma,r,\rho^Q}$ are
defined.

Further, $P Q^\circ$ is a parabolic subgroup of $G^\circ$ with $Q^\circ$ as Levi factor.
The unipotent radical $\mc R_u (P Q^\circ)$ is normalized by $Q^\circ$, so its Lie
algebra $\mf u_Q = \mr{Lie} (\mc R_u (P Q^\circ))$ is stable under the adjoint actions
of $Q^\circ$ and $\mf q$. By \eqref{eq:0.6} $\mf u_Q$ decomposes as the direct sum of
the subspaces $\mf u_{Q,j} = \mf u_Q \cap \mf g_j$. In particular ad$(y_j)$ acts on
$\mf u_{Q,j}$. We denote the cokernel of ad$(y_j) \colon\mf u_{Q,j} \to \mf u_{Q,j}$ by
$_y \mf u_{Q,j}$. From $[\sigma_j,y_j] = 2 r_j y_j$ we see that $\mr{ad}(\sigma_j)$
descends to a linear map ${}_y \mf u_{Q,j} \to {}_y \mf u_{Q,j}$.

Following Lusztig \cite[\S 1.16]{Lus7}, we define
\[
\begin{array}{cccc}
\epsilon_{y,j} \colon& \mr{Lie}(M^Q (y)^\circ) & \to & \C \\
& (\sigma,r) & \mapsto & \det( \mr{ad}(\sigma_j) - 2 r_j :
{}_y \mf u_{Q,j} \to {}_y \mf u_{Q,j})
\end{array}.
\]
All parameters for which parabolic induction could behave problematically
are zeros of a function $\epsilon_{y,j}$.

\begin{prop}\label{prop:0.6}
Let $y,\sigma, \vec r,\rho$ be as in Theorem \ref{thm:0.4}, and assume that
$\epsilon_{y,j} (\sigma,r) \neq 0$ for each $j = 1,\ldots,d$.
\enuma{
\item There is a natural isomorphism of $\mh H (G,M,q\cE,\vec{\mb r})$-modules
\[
\mh H (G,M,q\cE,\vec{\mb r}) \underset{\mh H (Q,M,q\cE,\vec{\mb r})}{\otimes}
E^Q_{y,\sigma,\vec{r},\rho^Q} \cong
\bigoplus\nolimits_\rho \Hom_{\pi_0 (\rZ_Q (\sigma,y))} (\rho^Q ,\rho) \otimes
E_{y,\sigma,r,\rho} ,
\]
where the sum runs over all $\rho \in \Irr \big( \pi_0 (\rZ_G (\sigma,y)) \big)$ with\\
$q\Psi_{\rZ_G (\sigma_0)}(y,\rho) = (M,\cC_v^M,q\cE)$.
\item For $\vec{r} = \vec{0}$ part (a) contains an isomorphism of
$S(\mf t^*) \rtimes \C[W_{q\cE} ,\natural_{q\cE}]$-modules
\[
\mh H (G,M,q\cE,\vec{\mb r}) \underset{\mh H (Q,M,q\cE,\vec{\mb r})}{\otimes}
M^Q_{y,\sigma,\vec{0},\rho^Q} \cong
\bigoplus\nolimits_\rho \Hom_{\pi_0 (\rZ_Q (\sigma,y))} (\rho^Q ,\rho) \otimes
M_{y,\sigma,\vec{0},\rho} .
\]
\item The multiplicity of $M_{y,\sigma,\vec{r},\rho}$ in $\mh H (G,M,q\cE,\vec{\mb r})
\underset{\mh H (Q,M,q\cE,\vec{\mb r})}{\otimes} E^Q_{y,\sigma,\vec{r},\rho^Q}$ is\\
$[\rho^Q \colon\rho]_{\pi_0 (\rZ_Q (\sigma,y))}$.
It already appears that many times as a quotient, via \\
$E^Q_{y,\sigma,\vec{r},\rho^Q} \to M^Q_{y,\sigma,\vec{r},\rho^Q}$.
More precisely, there is a natural isomorphism
\[
\Hom_{\mh H (Q,M,q\cE,\vec{\mb r})} (M^Q_{y,\sigma,\vec{r},\rho^Q},
M_{y,\sigma,\vec{r},\rho}) \cong \Hom_{\pi_0 (\rZ_Q (\sigma,y))} (\rho^Q ,\rho)^* .
\]
}
\end{prop}
\begin{proof}
For twisted graded Hecke algebras with only one parameter $\mb r$ this is
\cite[Proposition 3.22]{AMS2}, as corrected in \cite[Theorem A.1]{AMS2} and in
the version with quasi-Levi subgroups as discussed on \cite[p. 47]{AMS2}. Using
Theorem \ref{thm:0.4}, the proof of that result also works in the present setting.
\end{proof}

For an improved parametrization we use the Iwahori--Matsumoto involution, whose
definition we will now generalize to $\mh H (G,M,q\cE,\vec{\mb r})$. Extend the
sign representation of the Weyl group $W_{q\cE}^\circ$ to a character of $W_{q\cE}$
which is trivial on $\cR_{q\cE}$. Then we define
\begin{equation}\label{eq:0.19}
\IM (N_w) = \text{sign}(w) N_w ,\; \IM (\mb r_j) = \mb r_j ,\;
\IM (\xi) = -\xi \; (\xi \in \mf t^*) .
\end{equation}
Notice that IM is canonically determined by $G,P,M$ and $q\cE$, precisely the data
that are needed to define $\mh H (G,M,q\cE,\vec{\mb r})$.
Twisting representations with this involution is useful in relation with the
properties temperedness and (essentially) discrete series, see \cite[\S 3.5]{AMS2}.

\begin{prop}\label{prop:0.5}
\enuma{
\item Fix $\vec{r} \in \C^d$. There exists a canonical bijection
\[
(\sigma_0,y,\rho) \longleftrightarrow \IM^*
M_{y,\textup{d} \vec{\gamma} \matje{\vec r}{0}{0}{-\vec r}-\sigma_0,\vec{r},\rho}
\]
between conjugacy classes triples as in Theorem \ref{thm:0.4} and
$\Irr_{\vec{r}}(\mh H (G,M,q\cE,\vec{\mb r}))$.
\item Suppose that $\Re (\vec{r}) \in \R_{\geq 0}^d$. Then $\IM^*
M_{y,\textup{d} \vec{\gamma} \matje{\vec r}{0}{0}{-\vec r}-\sigma_0,\vec{r},\rho}$
is tempered if and only if $\sigma_0 \in i \mf t_\R = i \R \otimes_\Z X_* (T)$.
\item Suppose that $\Re (\vec{r}) \in \R_{>0}^d$. Then $\IM^*
M_{y,\textup{d} \vec{\gamma} \matje{\vec r}{0}{0}{-\vec r}-\sigma_0,\vec{r},\rho}$
is essentially discrete series if and only if $y$ is distinguished in $\mf g$.
In this case $\sigma_0 \in Z (\mf g)$.
\item Let $\zeta \in \mf g^G = \rZ(\mf g )^{G / G^\circ} \!. \!$ Then part (a) maps
$(\zeta + \sigma_0, y, \rho)$ to \\  $\zeta \otimes \IM^*
M_{y,\textup{d} \vec{\gamma} \matje{\vec r}{0}{0}{-\vec r}-\sigma_0,\vec{r},\rho}$.
\item Suppose that $\Re (\vec{r}) \in \R_{> 0}^d$ and that $\sigma_0 \in
i \mf t_\R + \rZ(\mf g)$. Then
\[
\IM^* M_{y,\textup{d} \vec{\gamma}
\matje{\vec r}{0}{0}{-\vec r}-\sigma_0,\vec{r},\rho} = \IM^* E_{y,\textup{d}
\vec{\gamma} \matje{\vec r}{0}{0}{-\vec r}-\sigma_0,\vec{r},\rho}.
\]
\item Suppose that $\sigma_0, \sigma_v + \textup{d}\gamma \matje{-1}{0}{0}{1} \in \mf t$
(which can always be arranged by Proposition \ref{prop:0.7}.c). Both 
$\IM^* M_{y,\textup{d} \vec{\gamma} \matje{\vec r}{0}{0}{-\vec r}-\sigma_0,\vec{r},\rho}$ 
and $\IM^* E_{y,\textup{d} \vec{\gamma} \matje{\vec r}{0}{0}{-\vec r}-\sigma_0,\vec{r},\rho}$ 
admit the central character $W_{q\cE} \big( \sigma_0 \pm (\vec{r} \sigma_v + 
\textup{d} \vec{\gamma} \matje{-\vec r}{0}{0}{\vec r} ) , \vec{r} \big)$.
}
\end{prop}
\begin{proof}
Part (a) follows immediately from Theorem \ref{thm:0.4}. Parts (b) and (c) are
consequences of \cite[\S 3.5]{AMS2}, see in particular (82) and (83) therein.\\
(d) From \eqref{eq:0.15} and Lemma \ref{lem:0.3} we see that
\[
E_{y,\sigma' - \zeta,\vec{r},\rho} = -\zeta \otimes E_{y,\sigma',\vec{r},\rho}
\]
whenever both sides are defined. By Theorem \ref{thm:0.4} the analogous equation for
$M_{y,\sigma',\vec{r},\rho}$ holds. Apply this with
$\sigma' = \textup{d} \vec{\gamma} \matje{\vec r}{0}{0}{-\vec r}-\sigma_0$
and use that $\IM^*$ turns $-\zeta$ into $\zeta$.\\
(e) Notice that $\sigma_0 - \sigma_z \in i \mf t_\R$. Write $\rho = \tau^* \ltimes
\rho^\circ$ as in \cite[Lemma 3.13]{AMS2}. By Lemma \ref{lem:0.3} and
\cite[Theorem 1.21]{Lus7}(for the simple factors of $G^\circ_\der$)
\begin{multline*}
M^\circ_{y,\textup{d} \vec{\gamma} \matje{\vec r}{0}{0}{-\vec r}-\sigma_0,\vec{r},
\rho^\circ} = M^{G_\der^\circ}_{y,\textup{d} \vec{\gamma} \matje{\vec r}{0}{0}{-\vec r} +
(\sigma_z -\sigma_0),\vec{r},\rho^\circ} \otimes \C_{-\sigma_z} = \\
E^{G_\der^\circ}_{y,\textup{d} \vec{\gamma} \matje{\vec r}{0}{0}{-\vec r} +
(\sigma_z -\sigma_0),\vec{r},\rho^\circ} \otimes \C_{-\sigma_z} = E^\circ_{y,\textup{d}
\vec{\gamma} \matje{\vec r}{0}{0}{-\vec r}-\sigma_0,\vec{r},\rho^\circ}.
\end{multline*}
By \cite[Lemma 3.16]{AMS2}
\[
M_{y,\textup{d} \vec{\gamma}
\matje{\vec r}{0}{0}{-\vec r}-\sigma_0,\vec{r},\rho}  =
\tau \ltimes M^\circ_{y,\textup{d} \vec{\gamma}
\matje{\vec r}{0}{0}{-\vec r}-\sigma_0,\vec{r},\rho^\circ} ,
\]
while \cite[Lemma 3.18]{AMS2} says that
\[
E_{y,\textup{d} \vec{\gamma}
\matje{\vec r}{0}{0}{-\vec r}-\sigma_0,\vec{r},\rho}  =
\tau \ltimes E^\circ_{y,\textup{d} \vec{\gamma}
\matje{\vec r}{0}{0}{-\vec r}-\sigma_0,\vec{r},\rho^\circ} .
\]
Applying $\IM^*$ to both these modules, we obtain the desired statement. \\
(f) Since the first module is a quotient of the second, it suffices to consider the
latter. From \eqref{eq:0.19} we see that the effect of $\IM^*$ on central characters is
$W_{q\cE}(\sigma,\vec{r}) \mapsto W_{q\cE}(-\sigma,\vec{r})$. 
Combine that with Proposition \ref{prop:0.7}.e.
\end{proof}

\section{Twisted affine Hecke algebras}
\label{sec:TAHA}

We would like to push the results of \cite{AMS2} and the previous section to
affine Hecke algebras, because these appear more directly in the representation theory
of reductive $p$-adic groups. This can be achieved with Lusztig's reduction theorems
\cite{Lus4}. The first reduces to representations with a ``real'' central
character (to be made precise later), and the second reduction theorem relates
representations of affine Hecke algebras with representations of graded Hecke algebras.

Our goal is a little more specific though, we want to consider not just one
(twisted) graded Hecke algebra, but a family of those, parametrized by a torus.
We want to find a (twisted) affine Hecke algebra which contains all members of
this family as some kind of specialization. Let us mention here that, although
we phrase this section with quasi-Levi subgroups and cuspidal quasi-supports,
all the results are equally valid for Levi subgroups and cuspidal supports.

Let $G$ be a possibly disconnected complex reductive group and let $(M,\cC_v^M,q\cE)$
be a cuspidal quasi-support for $G$. For any $t \in T = Z (M)^\circ$ the reductive
group $G_t = \rZ_G (t)$ contains $M$, and we can consider the twisted graded Hecke algebra
\[
\mh H (G_t,M,q\cE,\vec{\mb r}) =
\mh H (\mf t, \rN_{G_t}(q\cE) / M, c_t \vec{\mb r}, \natural_{q \cE,t}) .
\]
Here $\vec{\mb r} = (\mb r_1,\ldots,\mb r_d)$ refers to the almost direct
factorization of $G_t^\circ$ induced by \eqref{eq:0.6}.
Let us investigate how this family of algebras depends on $t$. For any $t \in T$, the
2-cocycle $\natural_{q \cE,t}$ of $\rN_{G_t}(q\cE) / M$ is just the restriction of
$\natural_{q\cE} \colon W_{q\cE}^2 \to \C^\times$. This can be seen from \cite[\S 3]{Lus2}
and the proofs of \cite[Proposition 4.5 and Lemma 5.4]{AMS}. More concretely, the
perverse sheaves $(\mr{pr}_1)_! \dot{q\cE}$ and $(\mr{pr}_1)_! \dot{q\cE}^*$
on Lie$(G)$ from \cite[(90)]{AMS2} and \cite[\S 3.4]{Lus3} extend the perverse sheaves
$q\pi_* (\widetilde{q\cE})$ and $q\pi_* (\widetilde{q\cE}^*)$ on Lie$(G)_{RS}$ from
\cite[\S 5]{AMS}. The latter naturally contain the corresponding objects
$q\pi_{t,*} (\widetilde{q\cE})$ and $q\pi_{t,*} (\widetilde{q\cE}^*)$ for $G_t$.
We denote the category of $G$-equivariant perverse sheaves on a $G$-variety $X$ by
$\mc P_G (X)$. The algebra
\[
\C [\rN_{G_t}(q\cE) / M, \natural_{q\cE,t}] \cong
\End_{\mathcal P_{G_t} \Lie (G_t)_{\RS}} \big( q\pi_{t,*} (\widetilde{q\cE}) \big)
\]
from \cite[Proposition 4.5 and Lemma 5.4]{AMS} is canonically embedded in
\[
\C [W_{q\cE},\natural_{q\cE}] \cong
\End_{\mathcal P_G \Lie (G)_{\RS}} \big( q\pi_* (\widetilde{q\cE}) \big) .
\]
We will simply write $W_{q\cE,t}$ for $\rN_{G_t}(q\cE) / M$, and $\natural_{q\cE}$
for $\natural_{q\cE,t}$.

On the other hand, the parameter function $c_t \colon R (\rZ_G (t)^\circ,T)_\red \to \C$
could depend on $t$, we have to specify which $t$ we use for a given root $\alpha$. 
Recall that $c_t (\alpha)$ was defined in \cite[\S 2]{Lus3}.
For any root $\alpha \in R (G^\circ,T)$:
\[
\mf g_\alpha \subset \Lie (G_t) \; \Longleftrightarrow \; \alpha (t) = 1.
\]
From \cite[Proposition 2.2]{Lus3} we know that $R(G^\circ,T)$ is a root system, so\\
$R(G^\circ,T) \cap \R \alpha \subset \{ \alpha , 2 \alpha, -\alpha, -2\alpha \}$
for every nondivisible root $\alpha$.

\begin{prop}\label{prop:4.1} \textup{\cite[Propositions 2.8, 2.10 and 2.12]{Lus3}} \\
Let $y \in \mathfrak{m}$ be an element of the nilpotent orbit defined by the
cuspidal quasi-support $(M,\cC_v^M,q\cE)$.
\enuma{
\item Suppose that $R(G^\circ,T) \cap \R \alpha = \{\alpha, -\alpha\}$.
Then $c_t (\alpha)$ satisfies
\begin{equation}\label{eq:4.1}
0 = \ad (y)^{c_t (\alpha) - 1} \colon \mf g_\alpha \to \mf g_\alpha \quad \text{and}
\quad 0 \neq  \ad (y)^{c_t (\alpha) - 2} \colon \mf g_\alpha \to \mf g_\alpha .
\end{equation}
This condition is independent of $t$, as long as $\mf g_\alpha \subset \Lie (G_t)$.
So we can unambiguously write $c (\alpha)$ for $c_t (\alpha)$ in this case.
Moreover $c (\alpha) \in \N$ is even.
\item Suppose that $R(G^\circ,T) \cap \R \alpha = \{ \alpha , 2 \alpha, -\alpha, -2\alpha \}$.

When $\alpha (t) = 1$, $\{\alpha ,2\alpha\} \subset R (\rZ_G (t)^\circ,T)$. Then $c_t (\alpha)$
is again given by \eqref{eq:4.1}, and it is odd.
We write $c (\alpha) = c_t (\alpha)$ for such a $t \in T$.
Furthermore $c_t (2 \alpha)$ is given by \eqref{eq:4.1} with $2 \alpha$ instead of $\alpha$,
and it equals 2.

When $\alpha (t) = - 1$, still $2 \alpha \in R (\rZ_G (t)^\circ,T)$, and $c_t (2 \alpha)$
is given by \eqref{eq:4.1} with $2 \alpha$ instead of $\alpha$. It equals 2, and we
write $c (2 \alpha) = 2$.
}
\end{prop}
With the conventions from Proposition \ref{prop:4.1}, $c_t$ is always the restriction of\\
$c \colon R(G^\circ,T) \to \C$ to $R(\rZ_G (t)^\circ,T)_\red$.

Now we construct the algebras that we need.

\begin{prop}\label{prop:4.2}
Consider the following data:
\begin{itemize}
\item the root datum $\mathcal R = (R(G^\circ,T),X^* (T),R(G^\circ,T)^\vee, X_* (T))$,
with simple roots determined by $P$;
\item the group $W_{q\cE} = W_{q\cE}^\circ \rtimes \cR_{q\cE}$;
\item a 2-cocycle $\natural \colon (W_{q\cE} / W_{q\cE}^\circ)^2 \to \C^\times$;
\item $W_{q\cE}$-invariant functions $\lambda : R(G^\circ,T)_\red \to \Z_{\geq 0}$ and\\
$\lambda^* \colon\{ \alpha \in R(G^\circ,T)_\red \colon\alpha^\vee \in 2 X_* (T) \} \to \Z_{\geq 0}$;
\item an array of invertible variables $\vec{\mb z} = (\mb z_1, \ldots, \mb z_d)$,
corresponding to the decomposition \eqref{eq:0.6} of $\mf g$.
\end{itemize}
The vector space
\[
\cO  (T \times (\C^\times)^d ) \otimes \C [W_{q\cE}] = \C [X^* (T)] \otimes
\C[\vec{\mb z},\vec{\mb z}^{-1}] \otimes \C [W_{q\cE}^\circ] \otimes \C [\cR_{q\cE},\natural]
\]
admits a unique algebra structure such that:
\begin{itemize}
\item $\C [X^* (T)], \C[\vec{\mb z},\vec{\mb z}^{-1}]$ and $\C[\cR_{q\cE},\natural]$
are embedded as subalgebras;
\item $\C[\vec{\mb z}, \vec{\mb z}^{-1}] = \C [\mb z_1, \mb z_1^{-1}, \ldots, \mb z_d,
\mb z_d^{-1}]$ is central;
\item the span of $W_{q\cE}^\circ$ is the Iwahori--Hecke algebra $\mc H (W_{q\cE}^\circ,
\vec{\mb z}^{2 \lambda})$ of $W_{q\cE}^\circ$ with parameters $\vec{\mb z}^{2 \lambda (\alpha)}$.
That is, it has a basis $\{ N_w : w \in W_{q\cE}^\circ \}$ such that
\[
\begin{array}{cccl}
N_w N_v & = & N_{wv} & \text{if } \ell (w) + \ell (v) = \ell (wv), \\
(N_{s_\alpha} + \mb{z}_j^{-\lambda (\alpha)}) (N_{s_\alpha} - \mb{z}_j^{\lambda (\alpha)}) &
= & 0 & \text{if } \alpha \in R(G_j T,T)_\red \text{ is a simple root.}
\end{array}
\]
\item for $\gamma \in \cR_{q\cE}, w \in W_{q\cE}^\circ$ and $x \in X^* (T)$:
\[
N_\gamma N_w \theta_x N_\gamma^{-1} = N_{\gamma w \gamma^{-1}} \theta_{\gamma (x)} .
\]
\item for a simple root $\alpha \in R(G_j T,T)$ and $x \in X^* (T)$
corresponding to $\theta_x \in \cO (T)$:
\end{itemize}
\begin{multline*}
\theta_x N_{s_\alpha} - N_{s_\alpha} \theta_{s_\alpha (x)} = \\
\left\{ \!\! \begin{array}{ll}
\big( \mb{z}_j^{\lambda (\alpha)} - \mb{z}_j^{-\lambda (\alpha)} \big) (\theta_x -
\theta_{s_\alpha (x)}) / (\theta_0 - \theta_{-\alpha}) & \alpha^\vee \notin 2 X_* (T) \\
\big( \mb{z}_j^{\lambda (\alpha)} \! - \mb{z}_j^{-\lambda (\alpha)} \! + \theta_{-\alpha}
(\mb{z}_j^{\lambda^* (\alpha)} \! - \mb{z}_j^{-\lambda^* (\alpha)}) \big) (\theta_x -
\theta_{s_\alpha (x)}) / (\theta_0 - \theta_{\! -2\alpha}) & \alpha^\vee \in 2 X_* (T)
\end{array}\right. \!\!
\end{multline*}
\end{prop}
\begin{proof}
In the case $\cR_{q\cE} = 1$, the existence and uniqueness of such an algebra is
well-known. It follows for instance from \cite[\S 3]{Lus4}, once we identify
$T_{s_\alpha}$ from \cite{Lus4} with $\mb{z}_j^{\lambda (\alpha)} N_{s_\alpha}$.
It is called an affine Hecke algebra and denoted by $\cH (\mathcal R,\lambda,
\lambda^*, \vec{\mb z})$.

Since $\lambda$ and $\lambda^*$ are $W_{q\cE}$-invariant,
\[
A_\gamma \colon N_w \theta_x \mapsto N_{\gamma w \gamma^{-1}} \theta_{\gamma (x)}
\]
defines an automorphism of $\cH (\mathcal R,\lambda,\lambda^*,\vec{\mb z})$. Clearly
\[
\cR_{q\cE} \to \mathrm{Aut}(\cH (\mathcal R,\lambda,\lambda^*,\vec{\mb z})) :
\gamma \mapsto A_\gamma
\]
is a group homomorphism. Pick a central extension $\cR_{q\cE}^+$ of $\cR_{q\cE}$
and a central idempotent $p_\natural$ such that $\C[\cR_{q\cE},\natural] \cong
p_\natural \C[\cR_{q\cE}^+]$. Now the same argument as in the proof of
\cite[Proposition 2.2]{AMS2} shows that the algebra
\begin{equation}\label{eq:4.5}
\C[\cR_{q\cE},\natural] \ltimes \cH (\mathcal R,\lambda, \lambda^*, \vec{\mb z}) \cong
p_\natural \C[\cR_{q\cE}^+] \ltimes \cH (\mathcal R,\lambda,
\lambda^*, \vec{\mb z}) \subset \cR_{q\cE}^+ \ltimes \cH (\mathcal R,\lambda,
\lambda^*, \vec{\mb z})
\end{equation}
has the required properties.
\end{proof}

When $\cR_{q\cE} = 1$, specializations of $\cH (\mathcal R,\lambda, \lambda^*,
\vec{\mb z})$ at $\vec{\mb r} = \vec{r} \in \R_{>0}^d$ figure for example in
\cite{Opd}. In relation with $p$-adic groups one should think of the variables
$\vec{\mb z}$ as as $(q_j^{1/2})_{j=1}^d$, where $q_j$ is the cardinality of
some finite field.

We define, for $\alpha \in R(G^\circ,T)_\red$:
\begin{equation}\label{eq:4.2}
\begin{array}{lll@{\qquad}l}
\lambda (\alpha) & = & c (\alpha) / 2 & 2 \alpha \notin R(G^\circ,T) \\
\lambda^* (\alpha) & = & c (\alpha) / 2 &
2 \alpha \notin R(G^\circ,T), \alpha^\vee \in 2 X_* (T) \\
\lambda (\alpha) & = & c(\alpha)/2 + c(2\alpha)/4 & 2\alpha \in R(G^\circ,T) \\
\lambda^* (\alpha) & = & c(\alpha)/2 - c(2\alpha)/4 & 2\alpha \in R(G^\circ,T) .
\end{array}
\end{equation}
By Proposition \ref{prop:4.1} $\lambda (\alpha) \in \Z_{\geq 0}$ in all cases.

For $\natural = \natural_{q\cE}$ \cite[(91)]{AMS2} says that
\[
\C [\cR_{q\cE},\natural_{q\cE}] \cong
\End^+_{\mathcal P_G \Lie (G)_{\RS}} \big( q\pi_* (\widetilde{q\cE}) \big) .
\]
We denote the algebra constructed in Proposition \ref{prop:4.1}, with these extra
data, by $\cH (G,M,q\cE,\vec{\mb z})$. Since it is built from an affine Hecke
algebra $\cH (\mathcal R,\lambda,\lambda^*,\vec{\mb z})$ and a twisted group algebra
$\C[W_{q\cE},\natural_{q\cE}]$, we refer to it as a twisted affine Hecke algebra.
When $d=1$ we simply write $\cH (G,M,q\cE)$. We record that
\begin{equation}\label{eq:4.35}
\cH (G,M,q\cE) = \cH (G,M,q\cE,\vec{\mb z}) / ( \{ \mb z_i - \mb z_j : 1 \leq i,j \leq d \} ) .
\end{equation}
The same argument as for \cite[Lemma 2.8]{AMS2} shows that
\begin{equation}\label{eq:4.34}
\mc H (G,M,q\cE,\vec{\mb z}) = \cH (\mc R,\lambda,\lambda^*, \vec{\mb z}) \rtimes
\End^+_{\mc P_G \mr{Lie}(G)_\RS}\big( q\pi_* (\widetilde{q\cE}) \big) .
\end{equation}
If we are in one of the cases \eqref{eq:0.18}, then with this interpretation
$\cH (G,M,q\cE,\vec{\mb z})$ depends canonically on $(G,M,q\cE)$. In general the algebra
$\mc H (G,M,q\cE,\vec{\mb z})$ is not entirely canonical, since it involves
the choice of a decomposition \eqref{eq:0.6}.

\begin{lem}\label{lem:4.3}
$\cO (T \times (\C^\times)^d )^{W_{q\cE}} = \cO (T)^{W_{q\cE}} \otimes
\C[\vec{\mb z},\vec{\mb z}^{-1}]$ is a central subalgebra of $\cH (G,M,q_\cE, \vec{\mb z})$.
It equals $\rZ(\cH (G,M,q\cE, \vec{\mb z}))$ if $W_{q\cE}$ acts faithfully on $T$.
\end{lem}
\begin{proof}
The case $W_{q\cE} = 1 , d = 1$ is \cite[Proposition 3.11]{Lus4}.
The general case follows readily from that, as observed in \cite[\S 1.2]{Sol3}.
\end{proof}

For $\zeta \in \rZ(G) \cap G^\circ$ and $(\pi,V) \in \mr{Mod}(\cH (G,M,q_\cE, \vec{\mb z}))$
we define $(\zeta \otimes \pi,V) \in \mr{Mod}(\cH (G,M,q_\cE, \vec{\mb z}))$ by
\[
(\zeta \otimes \pi) (f_1 f_2 N_w) = f_1 (\zeta) \pi (f_1 f_2 N_w) \qquad
f_1 \in \mc O (T), f_2 \in \C [\vec{\mb z}, \vec{\mb z}^{-1}], w \in W_{q\cE} .
\]

\subsection{Reduction to real central character} \
\label{par:Rrcc}

Let $T = T_\uni \times T_{\rs}$ be the polar decomposition of the complex torus
$T$, in a unitary and a real split part:
\begin{equation}\label{eq:4.11}
\begin{aligned}
T_\uni = \Hom (X^* (T),S^1) = \exp (i \mf t_\R) , \\
T_{\rs} = \Hom (X^* (T),\R_{>0}) = \exp (\mf t_\R) .
\end{aligned}
\end{equation}
We write the polar decomposition of an arbitrary element $t \in T$ as 
\[
t = (t |t|^{-1} ) \, |t| \in T_\uni \times T_{\rs}.
\]
By Lemma \ref{lem:4.3} every irreducible representation of $\cH (G,M,q\cE,\vec{\mb z})$
admits a \\ $\cO (T \times (\C^\times)^d)^{W_{q\cE}}$-character, an element of
$T / W_{q\cE} \times (\C^\times)^d$. We will refer to this as the central character.
Following \cite[Definition 2.2]{BaMo} we say that a central character
$(W_\cE t,\vec{z})$ is ``real'' if $\vec{z} \in \R_{>0}^d$ and the unitary part
$t |t|^{-1}$ is fixed by $W_\cE^\circ$.

For $t \in T$ we define $\tilde{\rZ_G}(t)$ to be the subgroup of $G$ generated by
$\rZ_G (t)$ and the root subgroups for $\alpha \in R(G^\circ,T)$ with $\alpha^\vee \in
2 X_* (T)$ and $\alpha (t) = -1$. Thus $R(\tilde{\rZ_G}(t)^\circ,T)$ consists of the roots
$\alpha \in R (G^\circ,T)$ with $s_\alpha (t) = t$.
The analogue of $\cR_{q\cE}$ for $\tilde{\rZ_G}(t)$ is $\cR_{q\cE,t}$, the stabilizer
of $R(\tilde{\rZ_G}(t)^\circ,T) \cap R(P,T)$ in $W_{q\cE,t}$.

Our first reduction theorem will relate modules of $\cH (G,M,q\cE,\vec{\mb z})$ and of\\
$\cH (\tilde{\rZ_G}(t),M,q\cE,\vec{\mb z})$. Assuming that every $\mb z_j$ acts via a
positive real number, we end up with representations admitting a real central character.
To describe the effect on $\cO (T \times (\C^\times)^d)$-weights, we need some
preparations. Consider the set
\[
W_\cE^{t} = \big\{ w \in W_\cE : w \big( R(\tilde{\rZ_G}(t)^\circ,T) \cap
R(P,T) \big) \subset R(P,T) \big\} .
\]
Recall that the parabolic subgroup $P \subset G^\circ$ determines a set of simple
reflections and a length function on the Weyl group $W_{q\cE}^\circ = W_\cE$.
We use this to define two cones in $\mf t_\R = X_* (T) \otimes_\Z \R$:
\begin{align*}
& \mf t_\R^+ := \{ x \in \mf t_\R : \inp{x}{\alpha} \geq 0 \; \forall \alpha \in R(P,T) \} ,\\
& \mf t_\R^- := \big\{ \sum\nolimits_{\alpha \in R(P,T)} x_\alpha \alpha^\vee :
x_\alpha \leq 0 \big\}.
\end{align*}

\begin{lem}\label{lem:4.4}
\enuma{
\item $W_\cE^{t}$ is the unique set of shortest length representatives for \\
$W_\cE / W(\tilde{\rZ_G}(t)^\circ,T)$ in $W_\cE$.
\item $\bigcup_{w \in W_\cE^{t}} w^{-1} \mf t_\R^+$ equals $\mf t_\R^{+,t}$,
the analogue of $\mf t_\R^+$ for the group $\tilde{\rZ_G}(t)^\circ$.
The same holds for $\mf t_\R^{*,+}$.
\item $\{ x \in \mf t_\R : W_\cE^{t} x \subset \mf t_\R^- \}$ equals
$\mf t_\R^{-,t}$, the analogue of $\mf t_\R^-$ for $\tilde{\rZ_G}(t)^\circ$.
}
\end{lem}
\begin{proof}
(a) This well-known when $R(\tilde{\rZ_G}(t)^\circ,T)$ is parabolic subsystem
\cite[Proposition 1.10.c and \S 1.6]{Hum}, and the general case is mentioned in 
\cite[proof of Lemma 3.4]{Lus8}. For completeness, we provide a proof.

For any $w \in W_\cE$, $w^{-1} W(P,T) \cap R(\tilde{\rZ_G}(t)^\circ,T)$ is a positive
system in $R(\tilde{\rZ_G}(t)^\circ,T)$. By \cite[Theorem 1.8]{Hum} there exists a
unique $v \in W(\tilde{\rZ_G}(t)^\circ,T)$ such that 
\[
v^{-1} \big( w^{-1} W(P,T) \cap R(\tilde{\rZ_G}(t)^\circ,T) \big) = 
W(P,T) \cap R(\tilde{\rZ_G}(t)^\circ,T) .
\]
This is also the unique $v \in W(\tilde{\rZ_G}(t)^\circ,T)$ such that 
\[
wv \big( R(\tilde{\rZ_G}(t)^\circ,T) \cap R(P,T) \big) \subset R(P,T) .
\]
Hence $W_\cE^t$ is a set of representatives for $W_\cE / W(\tilde{\rZ_G}(t)^\circ,T)$.

Consider $w \in W_\cE$ of minimal length in $w W(\tilde{\rZ_G}(t)^\circ,T)$. By
\cite[Proposition 5.7]{Hum}, $w(\alpha) \in R(P,T)$ for all $\alpha \in 
R(\tilde{\rZ_G}(t)^\circ,T) \cap R(P,T)$, so $w \in W_\cE^t$. We deduce that every left 
coset of $W(\tilde{\rZ_G}(t)^\circ,T)$ contains a unique element of minimal length, 
namely its representative in $W_\cE^t$.\\
(b) Suppose that $x \in \mf t_\R^+$ and $\alpha \in R(\tilde{\rZ_G}(t)^\circ,T)
\cap R(P,T)$. For all $w \in W_\cE^{t}$ we have $w \alpha \in R(P,T)$, so
\[
\inp{\alpha}{w^{-1} x} = \inp{w \alpha}{x} \geq 0 .
\]
Hence $\bigcup_{w \in W_\cE^t} w^{-1} \mf t_\R^+ \subset \mf t_\R^{+,t}$.
Let $S$ be a sphere in $\mf t_\R$ centred in 0. Then
\[
\text{vol}(S) / \text{vol}(S \cap \mf t_\R^+) = |W_\cE| \quad \text{and} \quad
\text{vol}(S) / \text{vol}(S \cap \mf t_\R^{+,t}) = |W (\tilde{\rZ_G}(t)^\circ,T)| .
\]
With part (a) it follows that
\begin{equation}\label{eq:4.3}
|W_\cE^{t}| \text{ vol}(S \cap \mf t_\R^+) = |W_\cE| \text{ vol}
(S \cap \mf t_\R^+) / |W (\tilde{\rZ_G}(t)^\circ,T)| = \text{vol}(S \cap \mf t_\R^{+,t}) .
\end{equation}
Since $\mf t_\R^+$ is a Weyl chamber for $W_\cE$, the translates $w \mf t_\R^+$
intersect $\mf t_\R^+$ only in a set of measure zero. Hence the left hand side
of \eqref{eq:4.3} is the volume of $S \cap \bigcup_{w \in W_\cE^{t}} w^{-1}
\mf t_\R^+$. As $\bigcup_{w \in W_\cE^{t}} w^{-1} \mf t_\R^+ \subset
\mf t_\R^{+,t}$ and both are cones defined by linear equations coming from roots,
the equality \eqref{eq:4.3} shows that they coincide.

The same reasoning applies to $\mf t_\R^*$ and the dual root systems.\\
(c) The definition of $W_\cE^{t}$ entails $W_\cE^{t} \mf t_\R^{-,t}
\subset \mf t_\R^-$. Conversely, suppose that $x \in \mf t_\R$ and that
$W_\cE^{t} x \subset \mf t_\R^-$. For every $w \in W_\cE^{t}$ and
every $\lambda \in \mf t_\R^{*,+}$:
\[
\inp{x}{w^{-1} \lambda} = \inp{w x}{\lambda} \leq 0 .
\]
In view of part (b) for $\mf t_\R^{*,+}$, this means that $x \in \mf t_\R^{-,t}$.
\end{proof}

\begin{thm}\label{thm:4.5}
Let $t \in T_\uni$.
\enuma{
\item There is a canonical equivalence between the following categories:
\begin{itemize}
\item finite dimensional $\cH (\tilde{\rZ_G}(t),M,q\cE,\vec{\mb z})$-modules with
$\cO (T \times (\C^\times)^d)$-weights in $t T_{\rs} \times \R_{>0}^d$;
\item finite dimensional $\cH (G,M,q\cE,\vec{\mb z})$-modules with
$\cO (T \times (\C^\times)^d)$-weights in\\ $W_{q\cE} t T_{\rs} \times \R_{>0}^d$.
\end{itemize}
It is given by localization of the centre and induction, and we denote it (suggestively)
by $\ind_{\cH (\tilde{\rZ_G}(t),M,q\cE,\vec{\mb z})}^{\cH (G,M,q\cE,\vec{\mb z})}$.
\item The above equivalences are compatible with parabolic
induction, in the following sense. Let $Q \subset G$ be an algebraic subgroup
such that $Q \cap G^\circ$ is a Levi subgroup of $G^\circ$ and $Q \supset M$. Then
\[
\ind_{\cH (\tilde{\rZ_G}(t),M,q\cE,\vec{\mb z})}^{\cH (G,M,q\cE,\vec{\mb z})} \circ
\ind^{\cH (\tilde{\rZ_G}(t),M,q\cE,\vec{\mb z})}_{\cH (\tilde{\rZ_Q}(t),M,q\cE,\vec{\mb z})}
= \ind_{\cH (Q,M,q\cE,\vec{\mb z})}^{\cH (G,M,q\cE,\vec{\mb z})} \circ
\ind^{\cH (Q,M,q\cE,\vec{\mb z})}_{\cH (\tilde{\rZ_Q}(t),M,q\cE,\vec{\mb z})} .
\]
\item The set of $\cO (T \times (\C^\times)^d)$-weights of
$\ind_{\cH (\tilde{\rZ_G}(t),M,q\cE,\vec{\mb z})}^{\cH (G,M,q\cE,\vec{\mb z})} (V)$ is
\[
\big\{ (w x,\vec{z}) : w \in \cR_{q\cE} W_\cE^{\circ,t}, (x,\vec{z}) \text{ is a }
\mc O (T \times (\C^\times)^d)\text{-weight of } V \big\} .
\]
}
\end{thm}
\begin{proof}
(a) The case $d=1 , \cR_{q\cE} = 1$ was proven in \cite[Theorem 8.6]{Lus4}.

Let $\cR_{q\cE}^+ \to \cR_{q\cE}$ be a central extension as in \eqref{eq:4.5}. Extend it
trivially to a central extension $\cR_{q\cE}^+ W_{q\cE}^\circ \to W_{q\cE}$ and
let $\cR_{q\cE,t}^+$ be the inverse image of $\cR_{q\cE,t} \subset W_{q\cE,t}$ in
$\cR_{q\cE}^+ W_{q\cE}^\circ$. Then
\begin{equation}\label{eq:4.26}
\begin{aligned}
& \cH (G,M,q\cE,\vec{\mb z}) = \cH (G^\circ,M^\circ,\cE,\vec{\mb z}) \rtimes
p_\natural \C[\cR_{q\cE}^+], \\
& \cH (\tilde{\rZ_G}(t),M,q\cE,\vec{\mb z}) = \cH (\tilde{\rZ_{G^\circ}}(t),
M^\circ,\cE,\vec{\mb z}) \rtimes p_\natural \C[\cR_{q\cE,t}^+] .
\end{aligned}
\end{equation}
As $p_\natural \in \C [\ker (\cR_{q\cE}^+ \to \cR_{q\cE})]$ is a central idempotent,
we may just as well establish the analogous result for the algebras
\[
\cH (G^\circ,M^\circ,\cE,\vec{\mb z}) \rtimes \cR_{q\cE}^+ \quad \text{and} \quad
\cH (\tilde{\rZ_{G^\circ}}(t), M^\circ,\cE,\vec{\mb z}) \rtimes \cR_{q\cE,t}^+ .
\]
Since we are dealing with finite dimensional representations only, we can decompose
them according to the (generalized) weights for the action of the centre. Fix
$(x,\vec{z}) \in T_{\rs} \times \R_{>0}^d$.
Denote the category of finite dimensional $A$-modules with weights in $U$ by
$\Mod_{f,U}(A)$. We compare the categories
\begin{equation}\label{eq:4.25}
\begin{aligned}
& \Mod_{f,W_{q\cE,t} t x \times \{ \vec{z} \}} \big( \cH (\tilde{\rZ_{G^\circ}}(t),
M^\circ,\cE,\vec{\mb z}) \rtimes \cR_{q\cE,t}^+ \big) ,\\
& \Mod_{f,W_{q\cE,t} t x \times \{ \vec{z} \}} \big(
\cH (G^\circ,M^\circ,\cE,\vec{\mb z}) \rtimes \cR_{q\cE}^+ \big) .
\end{aligned}
\end{equation}
The most appropriate technique to handle the general case is analytic localization,
as in \cite[\S 4]{Opd} (but there with fixed parameters $z_1,\ldots,z_d$). For a
submanifold $U \subset T \times (\C^\times )^d$, let $C^{an}(U)$ be the algebra
of complex analytic functions on $U$. We assume that $U$ is $W_{q\cE}$-stable and
Zariski-dense. Then the restriction map \\ $\cO (T \times (\C^\times)^d) \to
C^{an}(U)$ is injective, and we can form the algebra
\begin{equation}\label{eq:4.30}
\cH^{an}(U) := C^{an}(U)^{W_{q\cE}} \underset{\cO (T \times (\C^\times)^d)^{W_{q\cE}}
}{\otimes} \cH (G^\circ,M^\circ,\cE,\vec{\mb z}) \rtimes \cR_{q\cE}^+ .
\end{equation}
As observed in \cite[Proposition 4.3]{Opd}, the finite dimensional modules of
$\cH^{an}(U)$ can be identified with the finite dimensional modules of
$\cH (G^\circ,M^\circ,\cE,\vec{\mb z}) \rtimes \cR_{q\cE}^+$ with
$\cO (T \times (\C^\times)^d)$-weights in $U$.

In \cite[Conditions 2.1]{Sol3} it is described how one can find an open neighborhood
$U_0 \subset T \times (\C^\times)^d$ of $(x,\vec{z})$, which is so small that
localization to $U_0$ is more or less equivalent to localization at $(x,\vec{z})$.
We take $U = W_{q\cE} U_0$ and $\tilde{U} = W_{q\cE,t} U_0$. By Lusztig's
first reduction theorem, in the version \cite[Theorem 2.1.2]{Sol3}, there is
a natural inclusion of
\[
\cH_t^{an}(\tilde{U}) := C^{an}(\tilde{U})^{W_{q\cE,t}} \underset{\cO (T \times
(\C^\times)^d)^{W_{q\cE,t}}}{\otimes} \cH (\tilde{\rZ_{G^\circ}}(t),M^\circ,\cE,
\vec{\mb z}) \rtimes \cR_{q\cE,t}^+ \vspace{-1mm}
\]
in $\cH^{an}(U)$, which moreover is a Morita equivalence. Hence the composed functor
\begin{multline*}
\ind^{\cH (G^\circ,M^\circ,\cE,\vec{\mb z}) \rtimes \cR_{q\cE}^+}_{\cH
(\tilde{\rZ_{G^\circ}}(t), M^\circ,\cE,\vec{\mb z}) \rtimes \cR_{q\cE,t}^+} :
\Mod_{f,\tilde{U}} \big( \cH (\tilde{\rZ_{G^\circ}}(t),M^\circ,\cE,
\vec{\mb z}) \rtimes \cR_{q\cE,t}^+ \big) \to \\
\Mod_f \big( \cH_t^{an}(\tilde{U}) \big) \to \Mod_f \big( \cH^{an}(U) \big) \to
\Mod_{f,U} \big( \cH (G^\circ,M^\circ,\cE,\vec{\mb z}) \rtimes \cR_{q\cE}^+ \big)
\end{multline*}
is an equivalence of categories. We specialize this at $W_{q\cE,t} t x \times
\{ \vec{z} \} \subset U$ and we restrict to modules on which $p_\natural$ acts
as the identity. Via \eqref{eq:4.25} and \eqref{eq:4.26} this gives the
required equivalence of categories $\ind_{\cH (\tilde{\rZ_G}(t),M,q\cE,
\vec{\mb z})}^{\cH (G,M,q\cE,\vec{\mb z})}$.\\
(b) We just showed that the above functor is really induction between localizations
of the indicated algebras. Similar remarks apply to the functor
$\ind_{\cH (Q,M,q\cE,\vec{\mb z})}^{\cH (G,M,q\cE,\vec{\mb z})}$.
Thus the acclaimed compatibility with parabolic induction is just an
instance of the transitivity of induction.\\
(c) Lemma \ref{lem:4.4}.a and the constructions in \cite[\S 2.1]{Sol3} entail that
\begin{equation}\label{eq:4.6}
\ind_{\cH (\tilde{\rZ_G}(t),M,q\cE,\vec{\mb z})}^{\cH (G,M,q\cE,\vec{\mb z})} (V)
\cong \C [\cR_{q\cE} W_\cE^{t},\natural_{q\cE}]
\underset{\C [\cR_{q\cE,t},\natural_{q\cE}]}{\otimes} V
\end{equation}
as $\cO (T \times (\C^\times)^d)$-modules. Notice that the group $\cR_{q\cE,t}$
acts from the right on $\cR_{q\cE} W_\cE^{t}$, because it stabilizes
$R(\tilde{\rZ_G}(t)^\circ,T) \cap R(P,T)$. Since
\[
\cH (\tilde{\rZ_G}(t),M,q\cE,\vec{\mb z}) \cong \cH (\tilde{\rZ_G}(t)^\circ M,
M,q\cE,\vec{\mb z}) \rtimes \C [\cR_{q\cE,t},\natural_{q\cE}] ,
\]
the $\cO (T \times (\C^\times)^d)$-weights of $V$ come in full $\cR_{q\cE,t}$-orbits.
It was observed in the proof of \cite[Proposition 4.20]{Opd} that the
$\cO (T \times \C^\times)^d)$-weights of $\C w \otimes V \; (w \in W_\cE^{\circ,t})$
are precisely $(wx,\vec{z})$ with $(x,\vec{z})$ a $\cO (T \times (\C^\times)^d)$-weight
of $V$. Multiplication by $N_\gamma \; (\gamma \in \cR_{q\cE})$ just changes a weight
$(x,\vec{z})$ to $(\gamma x,\vec{z})$. These observations and \eqref{eq:4.6} prove
that the $\cO (T \times (\C^\times)^d)$-weights of $\ind_{\cH (\tilde{\rZ_G}(t),M,
q\cE,\vec{\mb z})}^{\cH (G,M,q\cE,\vec{\mb z})} (V)$ are as stated.
\end{proof}

In our reduction process we would like to preserve the analytic properties from
\cite[\S 3.5]{AMS2}. Just as in \cite[(79)]{AMS2}, we can define
$\cO (T)$-weights for modules of affine Hecke algebras or extended versions such
as $\cH (\tilde{\rZ_G}(t),M,q\cE,\vec{\mb z})$. We denote the set of $\cO (T)$-weights of
a module $V$ for such an algebra by Wt$(V)$. We can apply the polar
decomposition \eqref{eq:4.11} to it, which gives a set $|\mathrm{Wt}(V)|
\subset T_{\rs}$.

Let us recall the definitions of temperedness and discrete series
from \cite[\S 2.7]{Opd}.

\begin{defn}\label{defn:4.8}
Let $V$ be a finite dimensional $\cH (G,M,q\cE,\vec{\mb z})$-module. We say that $V$ is
tempered (respectively anti-tempered) if $|\mathrm{Wt} (V)| \subset
\exp (\mf t_\R^-)$, respectively $\subset \exp (-\mf t_\R^-)$.

Let $\mf t_\R^{--}$ be the interior of $\mf t_\R^-$ in $\mf t_\R$.
We call $V$ discrete series (resp. anti-discrete series) if $|\mathrm{Wt} (V)|
\subset \exp (\mf t_\R^{--})$, respectively $\subset \exp (-\mf t_\R^{--})$.
The module $V$ is essentially discrete series if its restriction to
$\cH (G / \rZ(G^\circ)^\circ ,M / \rZ( G^\circ)^\circ,q\cE,\vec{\mb z})$ is discrete series,
or equivalently if $|\mathrm{Wt}(V)| \subset \exp (\rZ(\mf g) \oplus \mf t_\R^{--})$.
\end{defn}

The next result fills a gap in \cite[Theorem 2.3.1]{Sol3}, where it was used
between the lines. Similar results, for $G^\circ_\der$ only and with somewhat
different notions of temperedness and discrete series, were proven in
\cite[Lemmas 3.4 and 3.5]{Lus8}.

\begin{prop}\label{prop:4.9}
The equivalence from Theorem \ref{thm:4.5}.a, and its inverse, preserve:
\enuma{
\item (anti-)temperedness,
\item the discrete series.
\item The $\cH (\tilde{\rZ_G} (t),M,q\cE,\vec{\mb z})$-module
$\ind_{\cH (\tilde{\rZ_G} (t),M,q\cE,\vec{\mb z})}^{\cH (G,M,q\cE,\vec{\mb z})} (V)$
is essentially discrete series if and only if $V$ is essentially discrete series
and $R(\tilde{\rZ_G}(t)^\circ,T)$ has full rank in $R(G^\circ,T)$.
}
\end{prop}
\begin{rem} The extra condition for essentially
discrete series representations is necessary, for the centre of
$\tilde{\rZ_G}(t)^\circ$ can be of higher dimension than that of $G^\circ$.
\end{rem}
\begin{proof}
Let $V$ be a finite dimensional $\cH (\tilde{\rZ_G}(t),M,q\cE,\vec{\mb z})$-module with
$\cO (T \times (\C^\times)^d)$-weights in $t T_{\rs} \times \R_{>0}^d$.\\
(a) The $\cO (T)$-weights of
$\ind_{\cH (\tilde{\rZ_G}(t),M,q\cE,\vec{\mb z})}^{\cH (G,M,q\cE,\vec{\mb z})} (V)$
were given in Theorem \ref{thm:4.5}.c. As $\log = \exp^{-1} : T_{\rs} \to \mf t_\R$ is
$W_{q \cE}$-equivariant, it entails that
\[
\log \big| \mathrm{Wt} \big( \ind_{\cH (\tilde{\rZ_G}(t),M,q\cE,\vec{\mb z})}^{\cH (G,M,
q\cE,\vec{\mb z})} V \big) \big| = \cR_{q\cE} W_\cE^{t} \log | \mathrm{Wt} (V) | .
\]
Recall from Lemma \ref{lem:4.4}.c that
\[
\mf t_\R^{-,t} = \{ x \in \mf t_\R : W_\cE^{t} x \subset \mf t_\R^{-} \} =
\{ x \in \mf t_\R : \cR_{q\cE} W_\cE^{t} x \subset \mf t_\R^{-} \} .
\]
Comparing these with the definition of (anti-)temperedness for $G$ and for $\tilde{\rZ_G}(t)$,
we see that $V$ is (anti-)tempered if and only if
$\ind_{\cH (\tilde{\rZ_G}(t),M,q\cE,\vec{\mb z}) }^{\cH (G,M,q\cE,\vec{\mb z})} (V)$ is so. \\
(b) We have to assume that $\rZ(G^\circ)$ is finite, for otherwise $\exp (\mf t_\R^{--})$
is empty and there are no discrete series representations on any side of the equivalences.

Suppose that $V$ is discrete series. Then $\tilde{\rZ_G}(t)^\circ$ is semisimple, so
$R(\tilde{\rZ_G}(t)^\circ,T)$ is of full rank in $R(G^\circ,T)$. This implies that
$\mf t_\R^{--,t}$ is an open subset of $\mf t_\R^{--}$. The same argument as for
part (a) shows that $\ind_{\cH (\tilde{\rZ_G}(t),M,q\cE,\vec{\mb z})
}^{\cH (G,M,q\cE,\vec{\mb z})} (V)$ is discrete series.

Conversely, suppose that
$\ind_{\cH (\tilde{\rZ_G}(t),M,q\cE,\vec{\mb z})}^{\cH (G,M,q\cE,\vec{\mb z})} (V)$
is discrete series. It is tempered, so $V$ is tempered and $|\mathrm{Wt}(V)| \subset
\exp (\mf t_\R^{-,t})$. Assume that $\tilde{\rZ_G}(t)^\circ$ is not 
semisimple. Then
\[
\mf t_Z := \Lie \big( Z (\tilde{\rZ_G}(t)^\circ) \big) =
\bigcap\nolimits_{\alpha \in R (\tilde{\rZ_G}(t)^\circ,T)} \ker \alpha
\]
has positive dimension. In particular $\mf t_Z^*$ contains nonzero elements $\lambda
\in \mf t_\R^{*,+}$, for example the sum of the fundamental weights for simple roots
not in $\R R (\tilde{\rZ_G}(t)^\circ,T)$. Let $t' \in T$ be any weight of $V$. Then
$\log |t'| \in \mf t_\R^{-,t} \subset \Lie (\tilde{\rZ_G}(t)_\der^\circ)$. Hence
$\inp{\log |t'|}{\lambda} = 0$, which means that $\log |t'| \in \mf t_\R^-
\setminus \mf t_\R^{--}$. But $t'$ is also a weight of $\ind_{\cH (\tilde{\rZ_G}(t),
M,q\cE)}^{\cH (G,M,q\cE)} (V)$, and that is a discrete series representation, so
$\log |t'| \in \mf t_\R^{--}$. This contractiction shows that $\tilde{\rZ_G}(t)^\circ$
is semisimple.

Suppose now that $\log |t'|$ does not lie in the interior of $\mf t_\R^{-,t}$.
Then it is orthogonal to a nonzero element $\lambda'$ in the boundary of
$\mf t_\R^{*,+,t}$. By Lemma \ref{lem:4.4}.b we can choose a $w \in W_\cE^{t}$
such that $w \lambda' \in \mf t_\R^{*,+}$. Theorem \ref{thm:4.5}.c $w t'$ is a
weight of $\ind_{\cH (\tilde{\rZ_G}(t),M,q\cE,\vec{\mb z})}^{
\cH (G,M,q\cE,\vec{\mb z})} (V)$, and it satisfies
\[
\inp{\log |w t'|}{w \lambda'} = \inp{\log |t'|}{\lambda'} = 0 .
\]
This shows that $\log |w t'| \notin \mf t_\R^{--}$, which contradicts that
$\ind_{\cH (\tilde{\rZ_G}(t),M,q\cE,\vec{\mb z})}^{\cH (G,M,q\cE,\vec{\mb z})} (V)$
is discrete series. Therefore $\log |t'|$ belongs to $\mf t_\R^{--,t}$.
As $t'$ was an arbitrary weight of $V$, this proves that $V$ is discrete series.\\
(c) Suppose that $\ind_{\cH (\tilde{\rZ_G} (t),M,q\cE,\vec{\mb z})}^{\cH (G,M,q\cE,
\vec{\mb z})} (V)$ is essentially discrete series. Its restriction to
$\cH (G / \rZ(G^\circ)^\circ ,M / \rZ(G^\circ)^\circ,q\cE,\vec{\mb z})$ is discrete
series, so by what we have just proven $V$ is discrete series as a module for
$\cH (\widetilde{\rZ_{G / \rZ(G^\circ)^\circ}}(t), M / \rZ(G^\circ)^\circ, q\cE, \vec{\mb z})$,
and $\widetilde{\rZ_{G / \rZ(G^\circ)^\circ}}(t)^\circ$ is semisimple.
Then  $R(\tilde{\rZ_G}(t)^\circ,T)$ has full rank in $R(G^\circ,T)$ and the restriction of
$V$ to the smaller algebra
$\cH (\tilde{\rZ_G}(t)^\circ / \rZ(G^\circ)^\circ , M / \rZ(G^\circ)^\circ ,q\cE,\vec{\mb z})$
is also discrete series, so $V$ is essentially discrete series.

Conversely, suppose that $V$ is essentially discrete series and that
$R(\tilde{\rZ_G} (t)^\circ,T)$ has full rank in $R(G^\circ,T)$. The second assumption
implies that $\rZ(G^\circ)^\circ$ is also the connected centre of $\tilde{\rZ_G}(t)^\circ$.
The same argument as in the tempered and the discrete series case shows that
\[
\big| \mathrm{Wt}\big( \ind_{\cH (\tilde{\rZ_G} (t),M,q\cE,\vec{\mb z})}^{
\cH (G,M,q\cE,\vec{\mb z})} V \big) \big| \subset \exp (\mf t_\R^{--} \oplus \rZ(\mf g)) .
\]
This means that $\ind_{\cH (\tilde{\rZ_G} (t),M,q\cE,\vec{\mb z})}^{
\cH (G,M,q\cE,\vec{\mb z})} (V)$ is essentially discrete series.
\end{proof}

Suppose that $t' \in W_{q\cE} t$. Then we can apply Theorem \ref{thm:4.5}.a also
with $t'$ instead of $t$, and that should give essentially the same equivalence of
categories. We check this in a slightly more general setting, which covers all
$t' \in T \cap  \Ad (G) t$. (Recall that for $g,h \in G$ we write Ad$(g)(h) = 
g h g^{-1}$.) By \cite[\S 8.13.b]{Lus5} and Proposition \ref{prop:0.7}.a
\begin{equation}\label{eq:TAdG}
T \cap \Ad (G) t \text{ equals } T \cap \Ad (\rN_G (T)) t \supset W_{q\cE} t.
\end{equation}
Let $g \in \rN_G (M) = \rN_G (T)$, with image $\bar g$ in $\rN_G (M) / M$. Conjugation
with $g$ yields an algebra isomorphism
\begin{equation}\label{eq:4.7}
\begin{aligned}
& \Ad (g) \colon\cH (\tilde{\rZ_G}(t),M,q\cE,\vec{\mb z}) \to \cH (\tilde{\rZ_G}(g t g^{-1}),M,
\Ad (g^{-1})^* q\cE,\vec{\mb z}) , \\
& \Ad (g) (N_w ) = N_{\bar g w \bar{g}^{-1}}, \quad
\Ad (g) \theta_x = \theta_{x \circ \Ad (g^{-1})} = \theta_{\bar g x}, \quad
\Ad (g) \mb z_j = \mb z_j ,
\end{aligned}
\end{equation}
where $w \in W_{q\cE}$ and $x \in X^* (T)$. Notice that this depends only on $g$
through its class in $\rN_G (M) / M$.

\begin{lem}\label{lem:4.6}
Let $t \in T_\uni$ and $g \in \rN_G (M)$. Then
\[
\ind_{\cH (\tilde{\rZ_G}(t),M,q\cE,\vec{\mb z})}^{\cH (G,M,q\cE,\vec{\mb z})} =
\Ad (g)^* \circ \ind_{\cH (\tilde{\rZ_G}(gtg^{-1}),M,\Ad (g^{-1})^* q\cE,\vec{\mb z})
}^{\cH (G,M,\Ad (g^{-1})^* q\cE,\vec{\mb z})} \circ \Ad (g^{-1})^*
\]
as functors between the appropriate categories of modules of these algebras
(as specified in Theorem \ref{thm:4.5}).
\end{lem}
\begin{rem} This result was used, but not proven, in
\cite[\S 4.9 and \S 5.20]{Lus6} and \cite[Theorem 2.3.1]{Sol3}.\end{rem}
\begin{proof}
Our argument for Theorem \ref{thm:4.5}.a, with \eqref{eq:4.5}, shows how several
relevant results can be extended from $\cH (G^\circ M,M,q\cE,\vec{\mb z})$ to
$\cH (G,M,q\cE,\vec{\mb z})$. This justifies the below use of some results from
\cite{Lus4}, which were formulated only for $\cH (G^\circ M,M,q\cE)$.

Let $(\pi,V)$ be a finite dimensional $\cH (G,M,q\cE)$-module with $\cO (T \times
\C^\times)$-weights in $W_{q\cE} t T_{\rs} \times \R_{>0}$. In \cite[\S 8]{Lus4}
$V$ is decomposed canonically as $\bigoplus_{t' \in W_{q\cE} t} V_{t' T_{\rs}}$,
where $V_{t' T_{\rs}}$ is the sum of all generalized $\cO (T)$-weight spaces
with weights in $t' T_{\rs}$. Then $V_{t' T_{\rs}}$ is a module for
$\cH (\tilde{\rZ_G} (t'),M,q\cE)$ and
\begin{equation}\label{eq:4.8}
V = \ind_{\cH (\tilde{\rZ_G}(t'),M,q\cE,\vec{\mb z})}^{\cH (G,M,q\cE,\vec{\mb z})}
(V_{t' T_{\rs}} ) .
\end{equation}
Assume that $g \in \rN_G (M,q\cE)$, so $\bar g \in W_{q\cE}$. Then $V_{t T_{\rs}}$ and
$V_{g t g^{-1} T_{\rs}}$ are related via multiplication with an element $\tau_{\bar g}$,
which lives in a suitable localization of $\cH (G,M,q\cE,\vec{\mb z})$
\cite[\S 5]{Lus4}. We can rewrite the right hand side of \eqref{eq:4.8} as
\begin{equation}\label{eq:4.9}
\ind_{\cH (\tilde{\rZ_G}(t),M,q\cE,\vec{\mb z})}^{\cH (G,M,q\cE,\vec{\mb z})}
\big( \tau_{\bar g} V_{g^{-1} t g T_{\rs}} \big) =
\tau_{\bar g} \big( \ind_{\cH (\tilde{\rZ_G}(gtg^{-1}),M,q\cE,\vec{\mb z})
}^{\cH (G,M,q\cE,\vec{\mb z})} ( V_{g^{-1} t g T_{\rs}}) \big) .
\end{equation}
From \cite[\S 8.8]{Lus4} and \cite[Lemma 4.2]{Sol2} we see that the effect of
conjugation by $\tau_{\bar g}$ on $\cH (G,M,q\cE,\vec{\mb z})$ and
$\cH (\tilde{\rZ_G}(t),M,q\cE,\vec{\mb z})$ boils down to
the algebra isomorphism \eqref{eq:4.7}. The right hand side of \eqref{eq:4.9} becomes
\[
\Ad (g)^* \circ \ind_{\cH (\tilde{\rZ_G}(gtg^{-1}),M,q\cE,\vec{\mb z})}^{
\cH (G,M,q\cE,\vec{\mb z}} \circ \Ad (g^{-1})^* \big( V_{t T_{\rs}}\big) ,
\]
which proves the lemma for such $g$.

Now we consider a general $g \in \rN_G (M)$. We will analyse
\begin{equation}\label{eq:4.10}
\Ad (g)^* \circ \ind_{\cH (\tilde{\rZ_G}(gtg^{-1}),M,\Ad (g^{-1})^* q\cE,\vec{\mb z})}^{
\cH (G,M,\Ad (g^{-1})^* q\cE,\vec{\mb z})} \circ \Ad (g^{-1})^* \big( V_{t T_{\rs}} \big) .
\end{equation}
From the above we see that the underlying vector space is
\[
\bigoplus_{w \in g W_{q\cE} g^{-1} / g W_{q\cE,t} g^{-1}} \hspace{-1cm} \tau_w
\big( \Ad (g^{-1})^* V_{t T_{\rs}} \big) = \bigoplus_{w \in W_{q\cE} / W_{q\cE,t}}
\Ad (g^{-1})^* \tau_w V_{t T_{\rs}} = \Ad (g^{-1})^* V .
\]
The action of $\cH (\tilde{\rZ_G}(gtg^{-1}),M, \Ad (g^{-1})^* q\cE,\vec{\mb z}) =
\Ad (g) \cH (\tilde{\rZ_G}(t),M,q\cE,\vec{\mb z})$ works out to
\[
(\Ad (g) h) \cdot (\Ad (g^{-1})^* v) = \Ad (g^{-1})^* (h \cdot v) .
\]
Thus \eqref{eq:4.10} can be identified with $V$.
\end{proof}

\subsection{Parametrization of irreducible representations} \
\label{par:irreps}

Next we want to reduce from $\cH (\tilde{\rZ_G}(t),M,q\cE,\vec{\mb z})$-modules to
modules over\\ $\mh H (G_t,M,q\cE,\vec{\mb r})$. The exponential map for
$T \times \C^\times$ gives a $W_{q\cE,t}$-equivariant map
\[
\exp_t \colon\mf t \oplus \C^d \to T \times (\C^\times)^d ,\qquad
\exp_t (x,r_1,\ldots,r_d) = (t \exp (x),\exp r_1,\dots,\exp r_d) .
\]
Notice that the restriction $\exp_t \colon\mf t_\R \oplus \R^d \to
t T_{\rs} \times \R_{>0}^d$ is a diffeomorphism.

\begin{thm}\label{thm:4.7}
Let $t \in T_\uni$.
\enuma{
\item There is a canonical equivalence between the following categories:
\begin{itemize}
\item finite dimensional $\mh H (G_t,M,q\cE,\vec{\mb r})$-modules with
$\mc O (\mf t \oplus \C^d)$-weights in \\ $\mf t_\R \oplus \R^d$;
\item finite dimensional $\cH (\tilde{\rZ_G}(t),M,q\cE,\vec{\mb z})$-modules with
$\cO (T \times (\C^\times)^d)$-weights in $t T_{\rs} \times \R_{>0}^d$.
\end{itemize}
It is given by localization with respect to central ideals in combination with
the map $\exp_t$. We denote this equivalence by $(\exp_t)_*$.
\item The functor $(\exp_t)_*$ is compatible with parabolic induction, in the
following sense. Let $Q \subset G$ be an algebraic subgroup such that $Q \cap G^\circ$
is a Levi subgroup of $G^\circ$ and $Q \supset M$. Then
\[
\ind^{\cH (\tilde{\rZ_G}(t),M,q\cE,\vec{\mb z})}_{\cH (\tilde{\rZ_Q}(t),M,q\cE,\vec{\mb z})}
\circ (\exp_t^Q )_* = (\exp_t )_* \circ
\ind^{\mh H (G_t,M,q\cE,\vec{\mb z})}_{\mh H (Q_t,M,q\cE,\vec{\mb z})} .
\]
\item The functor $(\exp_t)_*$ preserves the underlying vector space of a
representation, and it transforms a $S(\mf t^* \oplus \C^d)$-weight $(x,\vec{r})$ into
a $\cO (T \times (\C^\times)^d)$-weight $\exp_t (x,\vec{r})$.
\item The functors $(\exp_t)_*$ and $(\exp_t)_*^{-1}$ preserve (anti-)temperedness
and (essentially) discrete series.
}
\end{thm}
\begin{proof}
(a) The case $d=1,\cR_{q\cE} = 1$ was proven in \cite[Theorem 9.3]{Lus4}.

For the general case we use the similar techniques and notations as in the proof of
Theorem \ref{thm:4.5}.a. By the same argument as over there, it suffices to compare
the categories
\begin{equation}\label{eq:4.32}
\begin{aligned}
& \Mod_{f,W_{q\cE,t} t x \times \{\vec{z} \}} \big( \cH (\tilde{\rZ_G}(t)^\circ,
M^\circ,\cE,\vec{\mb z}) \rtimes \cR_{q\cE,t}^+ \big) , \\
& \Mod_{f,W_{q\cE,t} \log (x) \times \{ \log (\vec{z}) \}} \big( \mh H (\rZ_G (t)^\circ,
M^\circ,\cE,\vec{\mb r}) \rtimes \cR_{q\cE,t}^+ \big) .
\end{aligned}
\end{equation}
Recall from \eqref{eq:4.2} that the parameter functions for these algebras are
related by
\begin{equation}\label{eq:4.38}
\begin{array}{lll@{\qquad}l}
c_t (\alpha) & = & 2 \lambda (\alpha) & 2 \alpha \notin R (\tilde{\rZ_G}(t)^\circ,T) ,\\
c_t (\alpha) & = & \lambda (\alpha) + \lambda^* (\alpha) &
2 \alpha \in R (\tilde{\rZ_G}(t)^\circ,T) , \alpha (t) = 1 , \\
c_t (2 \alpha) / 2 & = & \lambda (\alpha) - \lambda^* (\alpha) &
2 \alpha \in R (\tilde{\rZ_G}(t)^\circ,T) , \alpha (t) = -1 .
\end{array}
\end{equation}
Let us define $k \colon R (\tilde{\rZ_G}(t)^\circ,T)_\red \to \R$ by
\begin{equation}\label{eq:4.37}
\begin{array}{lll@{\qquad}l}
k (\alpha) & = & 2 \lambda (\alpha) & 2 \alpha \notin R (\tilde{\rZ_G}(t)^\circ,T) ,\\
k (\alpha) & = & \lambda (\alpha) + \alpha (t) \lambda^* (\alpha) &
2 \alpha \in R (\tilde{\rZ_G}(t)^\circ,T) .
\end{array}
\end{equation}
The only difference between $\mh H (\mf t,W(\tilde{\rZ_G}(t)^\circ,T), k \vec{\mb r})$
and $\mh H (\rZ_G (t)^\circ,M^\circ,\cE,\vec{\mb r})$ arises from roots $\alpha \in
R (\tilde{\rZ_G}(t)^\circ,T) \setminus R (\rZ_G (t)^\circ,T)$ with
$\alpha (t) = -1$. The corresponding braid relations are
\[
\begin{array}{lll@{\qquad}l}
N_{s_\alpha} \xi - {}^{s_\alpha} \xi N_{s_\alpha} & = &
(\lambda (\alpha) - \lambda^* (\alpha)) \mb r_j (\xi - {}^{s_\alpha} \xi) / \alpha &
\text{in } \mh H (\mf t,W(\tilde{\rZ_G}(t)^\circ,T), k \vec{\mb r}), \\
N_{s_{2 \alpha}} \xi - {}^{s_{2 \alpha}} \xi N_{s_{2\alpha}} & = &
c_t (2\alpha) \mb r_j (\xi - {}^{s_{2 \alpha}} \xi) / (2 \alpha) &
\text{in } \mh H (\rZ_G (t)^\circ,M^\circ,\cE,\vec{\mb r}) .
\end{array}
\]
Since $s_\alpha = s_{2 \alpha}$ and $c_t (2 \alpha) = 2 (\lambda (\alpha) -
\lambda^* (\alpha))$, these two braid relations are equivalent, and we may identify
\begin{equation}\label{eq:4.33}
\mh H (\mf t,W(\tilde{\rZ_G}(t)^\circ,T), k \vec{\mb r}) \rtimes \cR_{q\cE,t}^+ =
\mh H (\rZ_G (t)^\circ,M^\circ,\cE,\vec{\mb r}) \rtimes \cR_{q\cE,t}^+ .
\end{equation}
Let $V \subset \mf t \times \C^d$ be a $W_{q\cE,t}$-stable, Zariski-dense submanifold.
Like in \eqref{eq:4.30} we can form the algebra
\[
\mh H_t^{an}(V) := C^{an}(V)^{W_{q\cE,t}} \underset{\mc O (\mf t \oplus \C^d)^{W_{q\cE,t}}
}{\otimes} \mh H (\mf t,W(\tilde{\rZ_G}(t)^\circ,T), k \vec{\mb r}) \rtimes \cR_{q\cE,t}^+ .
\]
The argument for \cite[Proposition 4.3]{Opd} shows that its finite dimensional modules
are precisely the finite dimensional $\mh H (\mf t,W(\tilde{\rZ_G}(t)^\circ,T),
k \vec{\mb r}) \rtimes \cR_{q\cE,t}^+$-modules with \\ $\mc O (\mf t \oplus \C^d)$-weights
in $V$. If $\exp_t$ is injective on $V$, it induces an algebra isomorphism
\begin{equation}\label{eq:4.31}
\exp_t^* \colon C^{an}(\exp_t (V))^{W_{q\cE,t}} \to C^{an}(V)^{W_{q\cE,t}} .
\end{equation}
We suppose in addition that $V$ is contained in a sufficiently small open neighborhood of
$\mf t_\R \oplus \R^d$. In view of the relations between the parameters \eqref{eq:4.38}
and \eqref{eq:4.37}, we can apply \cite[Theorem 2.1.4.b]{Sol3}. It shows that \eqref{eq:4.31}
extends to an isomorphism of $C^{an}(V)^{W_{q\cE,t}}$-algebras
\[
\Phi_t \colon C^{an}(\exp_t (V))^{W_{q\cE,t}} \underset{\cO (T \times
(\C^\times)^d)^{W_{q\cE,t}}}{\otimes} \cH (\tilde{\rZ_G}(t)^\circ,M^\circ,\cE,
\vec{\mb z}) \rtimes \cR_{q\cE,t}^+ \to \mh H_t^{an}(V) ,
\]
which is the identity on $\C[\cR_{q\cE,t}^+]$.

Choosing for $V$ a small neighborhood of
$W_{q\cE,t} \log (x) \times \{ \log (\vec{z}) \}$ in $\mf t \oplus \C^d$, $\Phi_t$
induces an equivalence between the categories of modules with weights in, respectively,
$W_{q\cE,t} t x \times \{ \vec{z} \}$ and $W_{q\cE,t} \log (x) \times \{ \log (\vec{z}) \}$.
In view of \cite[Proposition 4.3]{Opd} and \eqref{eq:4.33}, this provides the equivalence
between the categories \eqref{eq:4.32}.

Since $\Phi_t$ fixes $p_\natural \in \C[\cR_{q\cE,t}^+]$, we can restrict that equivalence
to modules on which $p_\natural$ acts as the identity.\\
(b) For $G^\circ$ this is shown in \cite[Theorem 6.2]{BaMo} and
\cite[Proposition 6.4]{Sol1}. Extending $G^\circ$ to a disconnected group boils
down to extending the involved algebras by $\C[\cR_{q\cE,t},\natural_{q\cE}]$ or
$\C[\cR^Q_{q\cE,t},\natural_{q\cE}]$. As we noted in proof of part (a), the algebra
homomorphism $\Phi_t$ used to define $(\exp_t )_*$ is the identity on
$\C[\cR_{q\cE,t},\natural_{q\cE}] \subset \C [\cR_{q\cE,t}^+]$. Hence this extension
works the same on both sides of the equivalence, and the argument given in
\cite[\S 6]{Sol1} generalizes to the current setting.\\
(c) By construction \cite[\S 2.1]{Sol3} $(\exp_t)_* \pi = \pi \circ \exp_t^*$ as
$\mc O (T \times (\C^\times)^d )$-representations.
(For $f \in \mc O (T \times (\C^\times)^d )$ the action of $f \circ \exp_t$ on the
vector space underlying $\pi$ is defined via a suitable localization.)
This immediately implies that $(\exp_t)_*$ has the effect of $\exp_t$ on weights.\\
(d) This result generalizes the observations made in \cite[(2.11)]{Slo}.
Let $V$ be a finite dimensional $\cH (\tilde{\rZ_G}(t),M,q\cE)$-module with
$\cO (T \times (\C^\times)^d )$-weights in $t T_{\rs} \times \R_{>0}$. By part (b)
\[
\mathrm{Wt}((\exp_t)_*^{-1} V) = \exp_t^{-1} (\mathrm{Wt}(V)) \subset \mf t_\R .
\]
By assumption $t \in T_\uni$, so we get
\[
| \mathrm{Wt}(V)| = \exp \big( \Re \big( \mathrm{Wt} ((\exp_t)_*^{-1} V) \big) \big) .
\]
Comparing \cite[Definition 3.24]{AMS2} and Definition \ref{defn:4.8}, we see that
$(\exp_t )_*$ and $(\exp_t )_*^{-1}$ preserve (anti-)temperedness and the discrete
series. With \cite[Definition 3.27]{AMS2} we see that "essentially discrete series"
is also respected.
\end{proof}

Theorems \ref{thm:4.5} and \ref{thm:4.7} together provide an equivalence between
$\mh H (G_t,M,q\cE,\vec{\mb r})$-modules with central character in
$\mf t_\R / W_{q\cE,t} \times \R^d$ and $\cH (G,M,q\cE,\vec{\mb z})$-modules with central
character in $W_{q\cE} t T_{\rs} / W_{q\cE} \times \R_{>0}^d$, where $t \in T_\uni$.

Recall from \cite[Corollary 3.23]{AMS2} and Theorem \ref{thm:0.4} that we can parametrize
$\Irr_{\vec{r}} (\mh H (G_t,M,q\cE,\vec{\mb r}))$ with $\rN_{G_t}(M)/M$-orbits of triples
$(\sigma_0,\cC,\cF)$, where $\sigma_0 \in \mf t,\; \cC$ is a nilpotent
$\rZ_{G_t}(\sigma_0)$-orbit in $\rZ_{\mf g}(\sigma_0)$ and $\cF$ is an irreducible
$\rZ_{G_t} (\sigma_0)$-equivariant local system on $\cC$ such that
$\Psi_{\rZ_{G_t}(\sigma_0)}(\cC,\cF) = (M,\cC_v^M,q\cE)$, up to $\rZ_{G_t}(\sigma_0)$-conjugacy.

To find all irreducible representations with $S(\mf t^*)^{W_{q\cE}}$-character in
$\mf t_\R$ (those are all we need for the relation with affine Hecke algebras)
it suffices to consider such triples $(\sigma_0,\cC,\cF)$ with $\sigma_0 \in \mf t_\R$.
To phrase things more directly in terms of the group $G$, we allow $t$ to vary in
$T_\uni$ and we replace $\sigma_0$ by $t' = t \exp (\sigma_0) \in t T_{\rs}$.
In other words, we consider triples $(t',\cC,\cF)$ such that:
\begin{itemize}
\item $t' \in T$ with unitary part $t = t' |t'|^{-1}$;
\item $\cC$ is a nilpotent $\rZ_G (t')$-orbit in $\rZ_{\mf g}(t') = \Lie (G_{t'})$.
\item $\cF$ is an irreducible $\rZ_G (t')$-equivariant local system on $\cC$ with\\
$q\Psi_{\rZ_G (t')}(\cC,\cF) = (M,\cC_v^M,q\cE)$, up to $\rZ_G (t')$-conjugacy.
\end{itemize}
To such a triple we can associate the standard $\mh H (G_t,M,q\cE,\vec{\mb r})$-modules
\begin{equation}\label{eq:2.64}
E_{y, \log |t'| + \textup{d}\vec{\gamma} \matje{\vec{r}}{0}{0}{-\vec{r}}, \vec{r},\rho}
\quad \text{and} \quad \IM^*
E_{y, -\log |t'| + \textup{d}\vec{\gamma} \matje{\vec{r}}{0}{0}{-\vec{r}}, \vec{r},\rho} ,
\end{equation}
where $y \in \cC$ and $\rho$ is the representation of $\pi_0 (\rZ_G (t',y))$ on $\cF_y$.
Furthermore\\ $\gamma\colon \SL_2 (\C) \to \rZ_G (t')^\circ$ is an algebraic homomorphism with
\begin{equation}\label{eq:4.12}
\textup{d}\gamma \matje{0}{1}{0}{0} = y \quad \text{and} \quad
\textup{d}\gamma \matje{1}{0}{0}{-1} \in \mf t + \sigma_v ,
\end{equation}
where $\sigma_v$ is as in \eqref{eq:0.24} and 
$\textup{d}\vec{\gamma} \matje{\vec{r}}{0}{0}{-\vec{r}}$ is given by \eqref{eq:0.16}. 
The modules \eqref{eq:2.64} have distinguished irreducible quotients
\[
M_{y, \log |t'| + \textup{d}\vec{\gamma} \matje{\vec{r}}{0}{0}{-\vec{r}}, \vec{r},\rho}
\quad \text{and} \quad \IM^*
M_{y, -\log |t'| + \textup{d}\vec{\gamma} \matje{\vec{r}}{0}{0}{-\vec{r}}, \vec{r},\rho} .
\]
By \cite[Corollary 3.23]{AMS2} all these representations depend
only on the $\rN_{G_t}(M)/M$-orbit of $(t',\cC,\cF)$, not on the additional choices.

For $\vec{z} \in \R_{>0}^d$ we consider the irreducible $\cH (G,M,q\cE,\vec{\mb z})$-module
\begin{equation}\label{eq:2.65}
\ind^{\cH (G,M,\cE,\vec{\mb z})}_{\cH (\tilde{\rZ_G}(t),M,q\cE,\vec{\mb z})} (\exp_t)_*
\, \IM^* \, M_{y, \textup{d}\vec{\gamma} \matje{\log \vec{z}}{0}{0}{-\log \vec{z}}
-\log |t'|, \log \vec{z},\rho} .
\end{equation}

\begin{lem}\label{lem:4.12}
Fix $\vec{z} \in \R_{>0}^d$. The representations \eqref{eq:2.65}
provide a bijection between $\Irr_{\vec{z}} (\cH (G,M,q\cE,\vec{\mb z})$ and
$\rN_G (M)/M$-orbits of triples $(t',\cC,\cF)$ as above.
\end{lem}
\begin{proof}
For irreducible $\cH (\tilde{\rZ_G} (t),M,q\cE,\vec{\mb z})$-representations with central
character in $W_{q\cE,t} t T_{\rs} \times \R_{>0}$ this follows from \cite[Corollary 3.23]{AMS2}
and Theorems \ref{thm:4.7} and \ref{thm:4.5}. We note that at this point we still have
to consider $\rN_{G_t}(M) / M$-conjugacy classes of parameters $(t',\cC,\cF)$.

With Theorem \ref{thm:4.5}.a we extend this to the whole of
$\Irr_{\vec{z}} (\cH (G,M,q\cE,\vec{\mb z}))$. In view of \eqref{eq:TAdG}, this involves the 
choice of a unitary element $t$ in a $\rN_G (M)$-orbit in $T$. But by Lemma \ref{lem:4.6} 
the parametrization does not depend on that choice. Hence the representation \eqref{eq:2.65}
depends, up to isomorphism, only on the $\rN_G (M)/M$-orbit of $(t',\cC,\cF)$.
\end{proof}

To simplify the parameters, we would like to get rid of the restriction $t' \in T$ --
we would rather allow any semisimple element of $G^\circ$. It is also convenient to
replace $\cC$ by a single unipotent element (contained in $\exp \cC$) in $G^\circ$,
and $\cF$ by the associated representation of the correct component group.

As new parameters we take triples $(s,u,\rho)$ such that:
\begin{itemize}
\item $s \in G^\circ$ is semisimple;
\item $u \in \rZ_G (s)^\circ$ is unipotent;
\item $\rho \in \Irr \big( \pi_0 (\rZ_G (s,u)) \big)$ with $q\Psi_{\rZ_G (s)}(u,\rho) =
(M,\cC_v^M,q\cE)$ up to $G$-conjugacy.
\end{itemize}
Assume that $s \in T$ and choose an algebraic homomorphism
$\gamma_u \colon \SL_2 (\C) \to \rZ_G (s)^\circ$ with
\begin{equation}\label{eq:4.14}
\gamma_u \matje{1}{1}{0}{1} = u \quad \text{and} \quad
\textup{d}\gamma_u \matje{1}{0}{0}{-1} \in \mf t + \sigma_v .
\end{equation}
Using the decomposition \eqref{eq:0.6} of $\mf g$ we write, like in \eqref{eq:0.16},
\begin{equation}\label{eq:4.36}
\vec{\gamma_u} \matje{\vec{z}}{0}{0}{\vec{z}^{-1}} = \exp \Big( \textup{d}
\vec{\gamma_u} \matje{\log \vec{z}}{0}{0}{- \log \vec{z}} \Big) \in M.
\end{equation}
For $\vec{z} \in \R_{>0}^d$ we define the standard $\cH (G,M,q\cE,\vec{\mb z})$-module
\[
\bar{E}_{s,u,\rho,\vec{z}} = \ind^{\cH (G,M,q\cE,\vec{\mb z})}_{
\cH (\tilde{\rZ_G}(s |s|^{-1}),M,q\cE,\vec{\mb z})}
(\exp_{s |s|^{-1}})_* \, \IM^* \, E_{\log u, \textup{d}\vec{\gamma_u}
\matje{\log \vec{z}}{0}{0}{-\log \vec{z}} -\log |s|, \log \vec{z},\rho} .
\]
and its irreducible quotient
\[
\bar{M}_{s,u,\rho,\vec{z}} = \ind^{\cH (G,M,q\cE,\vec{\mb z})}_{
\cH (\tilde{\rZ_G}(s |s|^{-1}),M,q\cE,\vec{\mb z})}
(\exp_{s |s|^{-1}})_* \, \IM^* \, M_{\log u, \textup{d}\vec{\gamma_u}
\matje{\log \vec{z}}{0}{0}{-\log \vec{z}} -\log |s|, \log \vec{z},\rho} .
\]
Even when $s \notin T$, the condition on $\rho$ and \cite[Propositions 3.5.a and
3.7]{AMS2} guarantee the existence of a $g_0 \in G^\circ$ such that
$g_0 s g_0^{-1} \in T$. In this case we put
\begin{equation}\label{eq:4.15}
\bar{E}_{s,u,\rho,\vec{z}} := \bar{E}_{g_0 s g_0^{-1}, g_0 u g_0^{-1},
g_0 \cdot \rho, \vec{z}} \quad \text{and} \quad \bar{M}_{s,u,\rho,\vec{z}} :=
\bar{M}_{g_0 s g_0^{-1}, g_0 u g_0^{-1}, g_0 \cdot \rho, \vec{z}} .
\end{equation}
We extend the polar decomposition \eqref{eq:4.11} to this setting by
\[
|s| := g_0^{-1} \, |g_0 s g_0^{-1}| \, g_0.
\]
With the Jordan decomposition in $G^\circ$ it is possible to combine $s$ and $u$
in a single element $g = s u \in G^\circ$. Then $s$ equals the semisimple part
$g_S$, $u$ becomes the unipotent part $g_U$ and
$\rho \in \Irr \big( \pi_0 (\rZ_G (g)) \big)$.

Now we come to our main result about affine Hecke algebras. In the case that
$G$ is connected, it is almost the same parametrization as in \cite[\S 5.20]{Lus6}
and \cite[Theorems 10.4]{Lus8}. The only difference is that we twist by the
Iwahori--Matsumoto involution. This is necessary to improve the unsatisfactory
notions of $\zeta$-tempered and $\zeta$-square integrable in
\cite[Theorem 10.5]{Lus8}.

\begin{thm}\label{thm:4.10}
Let $\vec{z} \in \R_{>0}^d$.
\enuma{
\item The maps
\[
(g,\rho) \; \mapsto \; (s = g_S,u = g_U,\rho) \; \mapsto \bar{M}_{s,u,\rho,\vec{z}}
\]
provide canonical bijections between the following sets:
\begin{itemize}
\item $G$-conjugacy classes of pairs $(g,\rho)$ with $g \in G^\circ$ and
$\rho \in \Irr \big( \pi_0 (\rZ_G (g)) \big)$ such that $q\Psi_{\rZ_G (g_S)} (g_U,\rho)
= (M,\cC_v^M,q\cE)$ up to $G$-conjugacy;
\item $G$-conjugacy classes of triples $(s,u,\rho)$ as above;
\item $\Irr_{\vec{z}} (\cH (G,M,q\cE,\vec{\mb z}))$.
\end{itemize}
\item Suppose that $s \in T$. The representations $\bar{E}_{s,u,\rho,\vec{z}}$ and
$\bar{M}_{s,u,\rho,\vec{z}}$ admit the $\cO (T)^{W_{q\cE}}$-character $W_{q\cE} s \,
\vec{\chi}_{u,v} (\vec{z})^{\pm 1}$, with $\chi_{u,v}$ as in \eqref{eq:0.25}
and the arrow defined as in \eqref{eq:0.16}.
\item Suppose that $\vec{z} \in \R_{\geq 1}^d$. The following are equivalent:
\begin{itemize}
\item $s$ is contained in a compact subgroup of $G^\circ$;
\item $|s| = 1$;
\item $\bar{M}_{s,u,\rho,\vq}$ is tempered;
\item $\bar{E}_{s,u,\rho,\vq}$ is tempered.
\end{itemize}
\item When $\vec{z} \in \R_{>1}^d ,\; \bar{M}_{s,u,\rho,\vec{z}}$ is essentially discrete
series if and only if $u$ is distinguished in $G^\circ$. In this case $|s| \in \rZ(G^\circ)$.

There are no essentially discrete series representations on which at least one $\mb z_j$
acts as 1.
\item Let $\zeta \in \rZ(G) \cap G^\circ$. Then
\[
\bar{M}_{\zeta s,u,\rho,\vec{z}} = \zeta \otimes \bar{M}_{s,u,\rho,\vec{z}} \quad \text{and}
\quad \bar{E}_{\zeta s,u,\rho,\vec{z}} = \zeta \otimes \bar{E}_{s,u,\rho,\vec{z}} ,
\]
where $\zeta \otimes$ is as defined after Lemma \ref{lem:4.3}.
\item Suppose that $\vec{z} \in \R_{>1}^d$ and $|s| \in \rZ(G^\circ)$. Then
$\bar{E}_{s,u,\rho,\vec{z}} = \bar{M}_{s,u,\rho,\vec{z}}$.
}
\end{thm}
\begin{proof}
(a) The uniqueness in the Jordan decomposition entails that the first map is a
canonical bijection.

We already noted in \eqref{eq:4.15} that, for every eligible triple $(s,u,\rho)$,
$s$ lies in $\Ad (G^\circ) T$. Therefore we may restrict to triples with $s \in T$.
Consider the map
\[
(s,u,\rho) \mapsto (s, \cC_{\log u}^{\rZ_G (s)},\cF),
\]
where $\cF$ is determined by $\cF_{\log u} = \rho$. As in the proof of
\cite[Corollary 3.23]{AMS2}, this gives a canonical bijection between $G$-conjugacy classes
of triples $(s,u,\rho)$ and the parameters used in Lemma \ref{lem:4.12}. Furthermore
\eqref{eq:4.14} just reflects \eqref{eq:4.12}, so Lemma \ref{lem:4.12} yields the
desired canonical bijection with $\Irr_{\vec{z}} (\cH (G,M,q\cE,\vec{\mb z}))$.\\
(b) By Proposition \ref{prop:0.5}.f the $\mh H (\rZ_G (s |s|^{-1}),M,q\cE,\vec{\mb r})$-representation
\begin{equation}
\IM^* \, E_{\log u, \textup{d}\vec{\gamma_u} \matje{\log \vec{z}}{0}{0}{-\log \vec{z}}
-\log |s|,\log \vec{z},\rho}
\end{equation}
admits the central character $W_{q\cE,s |s|^{-1}} \big( \log |s| \pm
\textup{d} \vec{\chi}_{u,v} (\log \vec{z}), \log \vec{z} \big)$.

By Theorems \ref{thm:4.7}.c and \ref{thm:4.5}.c the central character of
$\bar{E}_{s,u,\rho,\vec{z}}$ becomes
\[
\big( W_{q\cE} s \, \vec{\chi}_{u,v}(\vec{z}),\vec{z} \big) =
W_{q\cE} \big( s \, \vec{\chi}_{u,v} (\vec{z})^{-1},\vec{z} \big).
\]
The same holds for the quotient $\bar{M}_{s,u,\rho,\vec{z}}$.\\
(c) Suppose that $s \in T$. By \cite[(84)]{AMS2} the representation
\eqref{eq:4.13} and its quotient
\begin{equation}\label{eq:4.13}
\IM^* \, M_{\log u, \textup{d}\vec{\gamma_u} \matje{\log \vec{z}}{0}{0}{-\log \vec{z}}
-\log |s|,\log \vec{z},\rho}
\end{equation}
are tempered if and only if $\log |s| \in i \mf t_\R$.
By definition $\log |s| \in \mf t_\R$, so this condition is equivalent to $\log |s| = 0$.
This is turn is equivalent to $|s| = 1$ and to $s \in T_\uni$. By Theorem \ref{thm:4.7}.d
and Proposition \ref{prop:4.9}.b this is also equivalent to temperedness of
$\bar{E}_{s,u,\rho,\vec{z}}$ or $\bar{M}_{s,u,\rho,\vec{z}}$.

The proof of part (a) shows that also for general $s$, temperedness is equivalent to
$|s| = 1$. This happens if and only if $s$ lies in the unitary part of a torus conjugate
to $T$, which in turn is equivalent to $s$ lying in a compact subgroup of $G^\circ$.\\
(d) As in part (c), it suffices to consider the case $s \in T$.

Suppose that $\bar{M}_{s,u,\rho,\vec{z}}$ is essentially discrete series. By Proposition
\ref{prop:4.9}.c and Theorem \ref{thm:4.7}.d the representation \eqref{eq:4.13} has
the same property. Moreover we saw in the proof of Proposition \ref{prop:4.9}.c that
$\widetilde{\rZ_{G^\circ_\der}}(s |s|^{-1})^\circ$ is semisimple. Up to doubling some
roots (with respect to $T$), $\rZ_{G^\circ_\der} (s |s|^{-1})^\circ$ has the same root system,
so that group is semisimple as well.

By assumption $\log \vec{z} \in \R_{>0}^d$. Now \cite[(85)]{AMS2} says that
$\log u$ is distinguished in
$\Lie \big( \rZ_G (s |s|^{-1})^\circ \big)$. In view of the aforementioned semisimplicity,
this is the same as distinguished in $\mf g$. So $u$ is distinguished in $G^\circ$.

Conversely, suppose that $u$ is distinguished in $G^\circ$, or equivalently that
$\log u$ is distinguished in $\mf g$. As $u$ commutes with $s$, it also commutes with
$|s|$ and with $s |s|^{-1}$. This implies that $R(\rZ_G (s |s|^{-1})^\circ,T)$ and
$R(\tilde{\rZ_G} (s |s|^{-1})^\circ,T)$ have full rank in $R(G^\circ,T)$. By
\cite[(85)]{AMS2}, Theorem \ref{thm:4.7}.d and Proposition \ref{prop:4.9}.c
$\bar{M}_{s,u,\rho,\vec{z}}$ is essentially discrete series.

Suppose that either of the above two conditions holds. Then $|s| \in T_{\rs}$ commutes
with the distinguished unipotent element $u \in G^\circ$. This implies that the
semisimple subalgebra $\C \log |s| \subset \mf g$ is contained in $\rZ(\mf g)$. Hence
$|s| \in \rZ(G^\circ)$. Moreover \cite[Theorem 3.26.b]{AMS2} and Lemma \ref{lem:0.3}
imply that $\bar{E}_{s,u,\rho,\vec{z}} = \bar{M}_{s,u,\rho,\vec{z}}$.

Finally, suppose that $\cH (G,M,q\cE,\vec{\mb z})$ has an essentially discrete series
representation on which $\mb \rZ_j$ acts as 1. Its dimension is finite, so it has an
irreducible subquotient, say $\bar{M}_{s,u,\rho,\vec{z}}$. Then $\IM^* M_{\log u, -\log |s|,
\log \vec{z},\rho}$ restricts to an essentially discrete series representation of
$\mh H (\rZ_G ({s |s|^{-1}})^\circ,M^\circ,\cE)$, which is annihilated by $\mb r_j$.
By \eqref{eq:0.4} and \eqref{eq:0.5} it contains a $\mh H (G_j,M_j,\cE_j)$-representation
with the same properties. But \cite[Theorem 3.26.c]{AMS2} says that this is impossible.\\
(e) By Proposition \ref{prop:0.5}.d
\begin{multline*}
(\exp_{\zeta s |\zeta s|^{-1}})_* \, \IM^* \, M_{\log u, \textup{d}\vec{\gamma_u}
\matje{\log \vec{z}}{0}{0}{-\log \vec{z}} -\log |\zeta s|, \log \vec{z},\rho} = \\
(\exp_{\zeta |\zeta|^{-1} s |s|^{-1}})_* \, \log |\zeta| \otimes \IM^* \,
M_{\log u, \textup{d}\vec{\gamma_u}
\matje{\log \vec{z}}{0}{0}{-\log \vec{z}} -\log |\zeta s|, \log \vec{z},\rho} .
\end{multline*}
From Theorem \ref{thm:4.7}.a and the definitions of $\zeta \otimes , \log |\zeta| \otimes$
we see that this equals
\[
\zeta \otimes (\exp_{s |s|^{-1}})_* \, \IM^* \, M_{\log u, \textup{d}\vec{\gamma_u}
\matje{\log \vec{z}}{0}{0}{-\log \vec{z}} -\log |s|, \log \vec{z},\rho} .
\]
Since $\zeta$ is central in $G$, $\cH (\tilde{\rZ_G}(s |s|^{-1}),M,q\cE,\vec{\mb z})$ does
not change upon replacing $s$ by $\zeta s$, and $\zeta \otimes$ is preserved by
$\ind^{\cH (G,M,q\cE,\vec{\mb z})}_{\cH (\tilde{\rZ_G}(s |s|^{-1}),M,q\cE,\vec{\mb z})}$.
This proves the claim for $\bar{M}_{s,u,\rho,\vec{z}}$, while the argument for
$\bar{E}_{s,u,\rho,\vec{z}}$ is analogous.\\
(f) We use Theorems \ref{thm:4.5}.a and
\ref{thm:4.7}.a to translate the statement to modules over
$\mh H (G,M,q \cE, \vec{\mb r})$, with $\vec{\mb r}$ acting as $\log (\vec{z}) \in
\R_{>0}^d$. Then we apply Proposition \ref{prop:0.5}.e.
\end{proof}

Let us discuss the relation between the parametrization from Theorem \ref{thm:4.10}.a
and parabolic induction. Suppose that $Q \subset G$ is an algebraic subgroup such
that $Q \cap G^\circ$ is a Levi subgroup of $G^\circ$ and $M \subset Q$. Let
$(s,u,\rho)$ be as above, with $s,u \in Q^\circ$. Also take $\rho^Q \in \Irr \big(
\pi_0 (\rZ_Q (s,u)) \big)$ with $q\Psi_{\rZ_Q (s)} (u,\rho^Q) = (M,\cC_v^M,q\cE)$ up
to $Q$-conjugation.

Recall $\epsilon_j$ from \pageref{prop:0.6}. We extend it to the current setting by
defining
\[
\epsilon_{u,j} (s,\vec{z}) = \epsilon_{\log u,j} \Big( \textup{d}\vec{\gamma_u}
\matje{\log \vec{z}}{0}{0}{-\log \vec{z}} -\log |s|, \log \vec{z} \Big) .
\]

\begin{cor}\label{cor:4.15}
Assume that $\epsilon_{u,j} (s,\vec{z}) \neq 0$ for each $j = 1,\ldots,d$. 
\enuma{
\item There is a natural isomorphism of $\mc H (G,M,q\cE,\vec{\mb z})$-modules
\[
\mc H (G,M,q\cE,\vec{\mb z}) \underset{\mc H (Q,M,q\cE,\vec{\mb z})}{\otimes}
\bar{E}^Q_{s,u,\rho^Q,\vec{z}} \cong
\bigoplus\nolimits_\rho \Hom_{\pi_0 (\rZ_Q (s,u))} (\rho^Q ,\rho) \otimes
\bar{E}_{s,u,\rho,\vec{z}} ,
\]
where the sum runs over all $\rho \in \Irr \big( \pi_0 (\rZ_G (s,u)) \big)$
with $q\Psi_{\rZ_G (s)}(u,\rho) = \\ (M,\cC_v^M,q\cE)$ up to $G$-conjugation.
For $\vec{z} = \vec{1}$ this isomorphism contains
\[
\mc H (G,M,q\cE,\vec{\mb z}) \underset{\mc H (Q,M,q\cE,\vec{\mb z})}{\otimes}
\bar{M}^Q_{s,u,\rho^Q,\vec{z}} \cong
\bigoplus\nolimits_\rho \Hom_{\pi_0 (\rZ_Q (s,u))} (\rho^Q ,\rho) \otimes
\bar{M}_{s,u,\rho,\vec{z}} .
\]
\item The multiplicity of $\bar{M}_{s,u,\rho,\vec{z}}$ in $\mc H (G,M,q\cE,\vec{\mb z})
\underset{\mc H (Q,M,q\cE,\vec{\mb z})}{\otimes} \bar{E}^Q_{s,u,\rho^Q,\vec{z}}$ is\\
$[\rho^Q \colon\rho]_{\pi_0 (\rZ_Q (s,u))}$. It already appears that many times as a
quotient, via \\ $\bar{E}^Q_{s,u,\rho^Q,\vec{z}} \to \bar{M}^Q_{s,u,\rho^Q,\vec{z}}$.
More precisely, there is a natural isomorphism
\[
\Hom_{\mc H (Q,M,q\cE,\vec{\mb z})} (\bar{M}^Q_{s,u,\rho^Q,\vec{z}},
\bar{M}_{s,u,\rho,\vec{z}}) \cong \Hom_{\pi_0 (\rZ_Q (s,u))} (\rho^Q ,\rho)^* .
\]
}
\end{cor}
\begin{proof}
Recall that the analogous statement for twisted graded Hecke algebras is
Proposition \ref{prop:0.6}. To that we can apply the Iwahori--Matsumoto
involution, supported by \cite[(83)]{AMS2}. Next, part (b) of Theorem \ref{thm:4.7}
allows us to apply part (a) while retaining the desired properties. The same goes for
Theorem \ref{thm:4.5}. Then we have transferred Proposition \ref{prop:0.6}
to the representations $\bar{E}_{s,u,\rho,\vec{z}}$ and $\bar{M}_{s,u,\rho,\vec{z}}$.
\end{proof}

Notice that the parameters in Theorem \ref{thm:4.10}.a do not depend on $\vec{z}$.
This enables us to relate $\Irr_{\vec{z}} (\cH (G,M,q\cE,\vec{\mb z}))$ to an extended
quotient of $T$ by $W_{q\cE}$, as in \cite[\S 2.3]{ABPS7} and \cite[(87)]{AMS2}.
The 2-cocycle $\natural_{q\cE}$ of $W_{q\cE}$ gives rise to a twisted version of the 
extended quotient $T \q W_{q\cE}$, see \cite[\S 2.1]{ABPS7}.

\begin{thm}\label{thm:4.11}
Let $\vec{z} \in \R_{>0}^d$. There exists a canonical bijection
\[
\mu_{G,M,q\cE} \colon(T \q W_{q\cE} )_{\natural_{q\cE}} \to
\Irr_{\vec{z}} (\cH (G,M,q\cE,\vec{\mb z}))
\]
such that:
\begin{itemize}
\item $\mu_{G,M,q\cE} (T_\uni \q W_{q\cE} )_{\natural_{q\cE}} =
\Irr_{\vec{z},\temp} (\cH (G,M,q\cE,\vec{\mb z}))$ when $\vec{z} \in \R_{\geq 1}^d$;
\item the central character of $\mu_{G,M,q\cE} (t,\pi_t)$ is
$\big( W_{q\cE} t \, \vec{\chi} (\vec{z}), \vec{z} \big)$,
for some algebraic cocharacter $\chi$ of $\rZ_G (t)^\circ$.
\end{itemize}
\end{thm}
\begin{rem}Together with \cite[Theorem 5.4.2]{Sol3} this proves a substantial
part of the ABPS conjectures \cite[\S 15]{ABPS2} for the twisted affine Hecke algebra\\ 
$\cH (G,M,q\cE,\vec{\mb z})$.
For $\vec{z} \in (0,1]^d ,\; \mu_{G,M,q\cE} (T_\uni \q W_{q\cE} )_{\natural_{q\cE}}$ is
the anti-tempered part of $\Irr_{\vec{z}} (\cH (G,M,q\cE,\vec{\mb z}))$,
compare with \cite[Theorem 3.29]{AMS2}.
\end{rem}
\begin{proof}
From Proposition \ref{prop:4.2} we see that
\[
\cH (G,M,q\cE,\vec{\mb z}) / (\mb \rZ_1 - 1, \ldots, \mb \rZ_d - 1) \cong
\cO (T) \rtimes \C [W_{q\cE},\natural_{q\cE}] .
\]
By \cite[Lemma 2.3]{ABPS7} there exists a canonical bijection
\[
\begin{array}{ccc}
(T \q W_{q\cE} )_{\natural_{q\cE}} & \to
& \Irr (\cO (T) \rtimes \C [W_{q\cE},\natural_{q\cE}]) \\
(t,\pi_t) & \mapsto & \C_t \rtimes \pi_t = \ind_{\cO (T) \rtimes \C [W_{q\cE,t},
\natural_{q\cE}]}^{\cO (T) \rtimes \C [W_{q\cE},\natural_{q\cE}]} (\C_t \otimes V_{\pi_t})
\end{array} .
\]
We consider $\C_t \rtimes \pi_t$ as an irreducible $\cH (G,M,q\cE,\vec{\mb z})$-representation
with central character $(W_{q\cE} t,1)$. By Theorem \ref{thm:4.11} there exist $u$
and $\rho$, unique up to $\rZ_G (t)$-conjugation, such that $\C_t \rtimes \pi_t \cong
\bar{M}_{t,u,\rho,1}$. Now we define
\[
\mu_{G,M,q\cE}(t,\pi_t) = \bar{M}_{t,u,\rho,\vec{z}} .
\]
This is canonical because Theorem \ref{thm:4.10}.a is.
The properties involving temperedness and the central character follow from parts (c)
and (b) of Theorem \ref{thm:4.10}.
\end{proof}

\subsection{Comparison with the Kazhdan--Lusztig parametrization} \
\label{par:KaLu}

Irreducible representations of affine Hecke algebras were also classified in
\cite{KaLu,Ree}, in terms of equivariant K-theory. This concerns the cases with
only one complex parameter $q = {\mb z}^2$, which is not a root of unity.
In terms of Proposition \ref{prop:4.2} this means that $\lambda = \lambda^* = 1$.
In view of \eqref{eq:4.2} and \cite[Proposition 2.8]{Lus3}, this happens if
and only if $T = M^\circ$ is a maximal torus of $G^\circ$ and $v = 1$.
For the upcoming comparison we assume that $M = \rZ_G (T)$ equals $T$. Then
$\pi_0 (\rZ_M (v)) = 1$, $q\cE$ is the trivial representation and
\[
\cR_{q\cE} = \rN_G (T,B) / T \cong G / G^\circ ,
\]
where $B$ is a Borel subgroup of $G^\circ$ containing $T$ (called $P$ before).
The Kazhdan--Lusztig parametrization was extended to algebras of the form
\[
\cH (G,T,q\cE = \mr{triv}) =
\cH (\mc R (G^\circ,T), \lambda = 1, \lambda^* = 1, \mb z) \rtimes \cR_{q\cE}
\]
in \cite[\S 9]{ABPS5}. The parameters are triples $(t_q,u,\rho)$, where
\begin{itemize}
\item $t_q \in T$ is semisimple;
\item $u \in G^\circ$ is unipotent and $t_q u t_q^{-1} = u^q$;
\item $\mc B_{G^\circ}^{t_q,u}$ is the variety of Borel subgroups of $G^\circ$
containing $t_q$ and $u$;
\item $\rho \in \Irr \big( \pi_0 (\rZ_G (t_q,u)) \big)$ such that every irreducible
component of $\rho |_{\pi_0 (\rZ_{G^\circ}(t_q,u))}$ appears in
$H_* (\mc B_{G^\circ}^{t_q,u},\C)$.
\end{itemize}
Two triples of this kind are considered equivalent if they are $G$-conjugate.
The representation $\bar{M} (t_q,u,\rho)$ attached to these data is the unique irreducible
quotient of the standard module
\begin{equation}\label{eq:4.16}
\bar{E}_{t_q,u,\rho} :=  \Hom_{\pi_0 (\rZ_G (t_q,u))} \big( \rho,
H_* (\mc B_{G^\circ}^{t_q,u} \times \cR_{q\cE},\C) \big) .
\end{equation}
The classification of $\cH (G^\circ,T,\cE = \mr{triv})$ with $q = \mb z = 1$ goes
back to Kato \cite[Theorem 4.1]{Kat}, see also \cite[\S 8]{ABPS5}. With
\cite[Remark 9.2]{ABPS5} and the subsequent argument (which underlies the above for
$q \neq 1$) it can be extended to $\cH (G,T,q\cE = \mr{triv})$. The parameters are
the same as above (only with $q = 1$), and the irreducible module is
\begin{equation}\label{eq:4.17}
\bar{M} (t_1,u,\rho) =
\Hom_{\pi_0 (\rZ_G (t_1,u))} \big( \rho, H_{d(u)} (\mc B_{G^\circ}^{t_1,u} \times \cR_{q\cE},\C) \big) ,
\end{equation}
where $d(u)$ refers to the dimension of $\mc B_{G^\circ}^{t_1,u}$ as a real variety.
Clearly $\bar{M} (t_1,u,\rho)$ is again a quotient of $\bar{E}_{t_1,u,\rho}$, but for $q=1$
\eqref{eq:4.16} has other irreducible quotients as well, in lower homological degree.

\begin{lem}\label{lem:4.14}
The above set of parameters $(t_q,u,\rho)$ is naturally in bijection with the sets
of parameters used in Theorem \ref{thm:4.10}.a.
\end{lem}
\begin{proof}
By \cite[Lemma 7.1]{ABPS5}, we obtain the same $G$-conjugacy classes of parameters
if we replace the above $t_q$ by a semisimple element $s \in \rZ_{G^\circ}(u)$.
In Theorem \ref{thm:4.10} we also have parameters $(s,u,\rho)$, but with a
different condition on $\rho$, namely that
\[
q \Psi_{\rZ_G (s)} (u,\rho) =  (T,v=1,q\epsilon = \mr{triv}) .
\]
By definition this is equivalent to
\begin{equation}\label{eq:4.18}
\Psi_{\rZ_G (s)^\circ}(u,\rho_s) = (T,v=1,\epsilon = \mr{triv}) ,
\end{equation}
for any irreducible constituent $\rho_s$ of $\rho |_{\pi_0 (\rZ_{\rZ_G (s)^\circ}(u))}$.
Write $r = \log z \in \R$ and $y = \log (u) \in \Lie (\rZ_G (s))$. According to
\cite[Proposition 3.7]{AMS2} for the group $\rZ_G (s)^\circ$, \eqref{eq:4.18}
is equivalent to $\rho_s$ appearing in
\[
E^\circ_{y,0,r} = \C_{0,r} \underset{H_*^{M(y)^\circ} (\{y\})}{\otimes}
H_*^{M(y)^\circ} (\mc P_y^\circ ,\C) = H_* (\mc P_y^\circ,\C) .
\]
To make this more explicit, we assume (as we may) that $s \in T$. Then
$\rZ_B (s) = \rZ_G (s)^\circ \cap B$ is a Borel subgroup of $\rZ_G (s)^\circ$ and
\begin{multline}\label{eq:4.19}
\mc P_y^\circ = \{ g \rZ_B (s) \in \rZ_G (s)^\circ / \rZ_B (s) : \Ad (g^{-1}) y \in
\Lie (\rZ_B (s)) \} = \\
\{ g \rZ_B (s) \in \rZ_G (s)^\circ / \rZ_B (s) : u \in g \rZ_B (s) g^{-1} \} =
\mc B_{\rZ_G (s)^\circ}^u .
\end{multline}
Hence \eqref{eq:4.18} is equivalent to $\rho_s$ appearing in
$H_* (\mc B_{\rZ_G (s)^\circ}^u, \C)$. Let $\rho^\circ$ be a
$\pi_0 (\rZ_{G^\circ}(s,u))$-constituent of $\rho$ containing $\rho_s$. By
\cite[Proposition 6.2]{ABPS5} there are isomorphisms of
$\rZ_{G^\circ}(s,u)$-varieties
\begin{equation}\label{eq:4.29}
\mc B_{G^\circ}^{t_q,u} \cong \mc B_{G^\circ}^{s,u} \cong
\mc B_{\rZ_G (s)^\circ}^u \times \rZ_{G^\circ}(s,u) / \rZ_{\rZ_G (s)^\circ}(u) .
\end{equation}
With this and Frobenius reciprocity we see that the condition on $\rho_s$
is also equivalent to $\rho^\circ$ appearing in $H_* (\mc B_{G^\circ}^{s,u},\C)$.
We conclude that the parameters $(s,u,\rho)$ in Theorem \ref{thm:4.10} are
equivalent to those in \cite[\S 9]{ABPS5}, the only change being
$s \leftrightarrow t_q$.
\end{proof}

\begin{prop}\label{prop:4.13}
The parametrization of $\Irr_z (\cH (G,T,q\cE = \mr{triv}))$
obtained in Theorem \ref{thm:4.10}.a agrees with the above parametrization by
the representations \\ $\bar{M} (t_q,u,\rho)$, when we set $q = z^2 \in \R_{>0}$
and take Lemma \ref{lem:4.14} into account. Moreover the standard modules
$\bar{E}_{s,u,\rho,z}$ and $\bar{E}_{t_q,u,\rho}$ are isomorphic.

In other words, our classification of irreducible representations of affine Hecke
algebras agrees with that of Kazhdan--Lusztig and the extended versions thereof.
\end{prop}
\begin{rem} Our parametrization differs from the one used by Lusztig
in \cite[\S 5.20]{Lus6} and \cite[Theorem 10.4]{Lus8}, namely by the
Iwahori--Matsumoto involution. Thus Proposition \ref{prop:4.13} shows that the
classification of unipotent representations of adjoint simple groups in
\cite{Lus6,Lus8} does not agree with the earlier classification of
Iwahori--spherical representations in \cite{KaLu}.\end{rem}
\begin{proof}
Let $(s,u,\rho)$ be a triple as above, and choose an algebra homomorphism \\
$\gamma_u \colon \SL_2 (\C) \to \rZ_G (s)^\circ$ with $\gamma_u \matje{1}{1}{0}{1} = u$.
Then we can take $t_q = s \gamma_u \matje{z}{0}{0}{z^{-1}}$, where $z^2 = q$.
Recall that $\bar{M} (t_q,u,\rho)$ is a quotient of $\bar{E}_{t_q,u,\rho}$ from
\eqref{eq:4.16}. Write
$\rho = \rho^\circ \rtimes \tau^*$, where
\[
\tau^* \in \Irr (\cR_{q\cE,u,s,\rho^\circ}) \quad \text{with} \quad
\cR_{q\cE,u,s,\rho^\circ} = \pi_0 (\rZ_G (s,u))_{\rho^\circ} / \pi_0 (\rZ_{G^\circ}(s,u)) .
\]
From \cite[(72)]{ABPS5} we see that $\bar{E}_{t_q,u,\rho}$ equals
\begin{equation}\label{eq:4.20}
\Hom_{\pi_0 (\rZ_{G^\circ}(s,u))} \big(\rho^\circ,
H_* (\mc B_{G^\circ}^{s,u},\C) \big) \rtimes \tau .
\end{equation}
To the part without $\rtimes \tau$ we can apply \cite{EvMi}, which compares the
two parametrizations.
In \cite{EvMi} both the Iwahori--Matsumoto involution and a related ``shift'' are
mentioned. This involution is necessary to get temperedness for the same parameters
in both classifications. Unfortunately, it is not entirely clear what Evens and
Mirkovich mean by a ``shift'', for signs can be inserted at various places. In any case
their argument is based on temperedness and a comparison of weights
\cite[Theorem 5.5]{EvMi}, and it will work once we arrange the modules such that
these two aspects match. With this in mind, \cite[Theorem 6.10]{EvMi} says that
the $\mh H (\rZ_{G^\circ}(s |s|^{-1}), T,\mr{triv})$-module obtained from
$\Hom_{\pi_0 (\rZ_{G^\circ}(s,u))} \big(\rho^\circ, H_* (\mc B_{G^\circ}^{s,u},\C)
\big)$ via Theorems \ref{thm:4.5} and \ref{thm:4.7} is
$\IM^* E_{y, \textup{d} \gamma_u \matje{r}{0}{0}{-r} - \log |s| , r, \rho^\circ}$.
The extension with the group $\cR_{q\cE}$ is handled in the same way for all
algebras under consideration here, namely with Clifford theory. It follows that
applying Theorems \ref{thm:4.5} and \ref{thm:4.7} to \eqref{eq:4.20} yields
\begin{equation}\label{eq:4.21}
\Big( \IM^* E_{y, \textup{d} \gamma_u \matje{r}{0}{0}{-r} - \log |s| , r, \rho^\circ}
\Big) \rtimes \tau .
\end{equation}
Moreover IM is the identity on $\C [\cR_{q\cE}]$, so the large brackets are actually
superfluous here. Notice that the subgroup of $\Gamma$ appearing in $\rZ_G (s |s|^{-1})$
is $\Gamma_{\mr{Ad}(G^\circ) s |s|^{-1}}$, the stabilizer of the Ad$(G^\circ)$-orbit of
$s |s|^{-1}$. The action of $\cR_{q\cE,u,s,\rho^\circ}$ underlying $\rtimes \tau$
in \eqref{eq:4.20} comes from the action of $\pi_0 (\rZ_G (s,u))$ on
$H_* (\mc B_{\rZ_G (s |s|^{-1})^\circ}^u \times \Gamma_{\mr{Ad}(G^\circ) s |s|^{-1}} ,\C)$.
By \eqref{eq:4.19} for the group $\rZ_G (s |s|^{-1})$:
\[
\mc B_{\rZ_G (s |s|^{-1})^\circ}^u \times \Gamma_{\mr{Ad}(G^\circ) s |s|^{-1}} = \mc P_y.
\]
Via this equality the $\pi_0 (\rZ_G (s,u))$-action on $H_* (\mc B_{\rZ_G (s |s|^{-1}
)^\circ}^u \times \Gamma_{\mr{Ad}(G^\circ) s |s|^{-1}},\C)$ agrees with the action on
\[
H_* (\mc P_y,\C) \cong \C_{|s|,r} \underset{H_*^{M(y)^\circ} (\{y\})}{\otimes}
H_*^{M(y)^\circ}(\mc P_y,\C)
\]
from \cite[Theorem 3.2.d]{AMS2}. Hence
\begin{multline*}
\big( \IM^* E_{y, \textup{d} \gamma_u \matje{r}{0}{0}{-r} - \log |s| , r, \rho^\circ}
\big) \rtimes \tau =
\IM^* \big( E_{y, \textup{d} \gamma_u \matje{r}{0}{0}{-r} - \log |s| , r, \rho^\circ}
\rtimes \tau \big) \\
= \IM^* \big( E_{y, \textup{d} \gamma_u \matje{r}{0}{0}{-r} - \log |s| , r, \rho^\circ
\rtimes \tau*} \big) = \IM^* E_{y, \textup{d} \gamma_u \matje{r}{0}{0}{-r} - \log |s| ,
r, \rho} .
\end{multline*}
We see that the standard modules $\bar{E}_{t_q,u,\rho}$ and $\bar{E}_{s,u,\rho,z}$
give the same module upon applying Theorems \ref{thm:4.5} and \ref{thm:4.7}.
Hence they are isomorphic.

From here on we have to assume that $q = z^2 \in \R_{> 0}$ is not a root of unity.
We recognize the unique irreducible quotient of the right hand side as \eqref{eq:4.13},
a part of the definition of $\bar{M}_{s,u,\rho,z}$. Using Theorems \ref{thm:4.7} and
\ref{thm:4.5} again, but now in the opposite direction, we see that both
$\bar{M}_{s,u,\rho,z}$ and $\bar{M} (t_q,u,\rho)$ are the unique irreducible quotient of
\[
\ind^{\cH (G,M,\cE)}_{\cH (\tilde{\rZ_G}(s |s|^{-1}),M,q\cE)}
(\exp_{s |s|^{-1}})_* \, \IM^* \, E_{\log u, \textup{d}\gamma_u
\matje{\log z}{0}{0}{-\log z} -\log |s|, \log z,\rho} .
\]
Thus the two parametrizations agree when $q = z^2 \neq 1$.

For $q = z = 1$ a different argument is needed. We note that \eqref{eq:4.20} still
applies, which enables us to write
\[
\bar{M} (t_1 = s,u,\rho) = \Hom_{\pi_0 (\rZ_{G^\circ} (s,u))} \big( \rho^\circ,
H_{d(u)} (\mc B_{G^\circ}^{s,u},\C) \big) \rtimes \tau .
\]
From the definition of the $X^* (T)$-action in \cite[\S 3]{Kat} we see that
$H_* (\mc B_{G^\circ}^{s,u},\C)$ is completely reducible as a $X^* (T)$-module.
With \cite[Theorem 8.2]{ABPS5} we deduce that the weight space for $s \in T$ is,
as $(W_{q\cE})_s$-representation, equal to
\begin{multline*}
\Hom_{\pi_0 (\rZ_G (s,u))} \big( \rho, H_{d(u)} (\mc B_{\rZ_G (s)^\circ}^u \times
\Gamma_{\mr{Ad}(G^\circ)s},\C) \big) = \\
\Hom_{\pi_0 (\rZ_{\rZ_G (s)^\circ} (u))}
\big( \rho^\circ, H_{d(u)} (\mc B_{\rZ_G (s)^\circ}^u,\C) \big) \rtimes \tau .
\end{multline*}
From \cite[(34)]{AMS2} we can also determine the $X^* (T)$-weight space for
$s$ in $\bar{M}_{s,u,\rho,1}$. First we look at the $S(\mf t^*)$-weight $-\log |s|$ in
$M^\circ_{y,-\log |s|,0,\rho^\circ}$, that gives
$M^{Q^\circ}_{y,-\log |s|,0,\rho^\circ}$. As in \cite[Section 3.2]{AMS2}, we denote
the underlying $W (\rZ_G (s)^\circ,T)$-representation by $M_{y,\rho^\circ}$.
Next we replace $\rZ_G (s)^\circ$ by $\rZ_G (s)$ and $\rho^\circ$ by $\rho =
\rho^\circ \rtimes \tau^*$, obtaining the $(W_{q\cE})_s$-representation
\begin{equation}\label{eq:4.22}
M^{Q^\circ}_{y,-\log |s|,0,\rho^\circ} \rtimes \tau = M_{y,\rho^\circ} \rtimes \tau .
\end{equation}
Applying the Iwahori--Matsumoto involution and Theorem \ref{thm:4.7}, we get
\begin{equation}\label{eq:4.23}
(\exp_{s |s|^{-1}} )_* \IM^* ( M^{Q^\circ}_{y,-\log |s|,0,\rho^\circ} \rtimes \tau ) .
\end{equation}
The previous $S(\mf t^*)$-weight space \eqref{eq:4.22} for $-\log |s|$ has now been
transformed into the $X^* (T)$-weight space for $s$ in the representation
$\bar{M}_{s,u,\rho,1}$ with respect to the group $\rZ_G (s)$. To land inside $\bar{M}_{s,u,\rho,1}$
with respect to $G$, we must still apply Theorem \ref{thm:4.5}. But that does not
change the $X^* (T)$-weight space for $s$, so we can stick to \eqref{eq:4.23}.

For $r=0, z=1$ the map $(\exp_{s |s|^{-1}})_*$ becomes the identity on $\C [W_\cE]$,
see \cite[(2.5) and (1.25)]{Sol3}. It remains to compare the $\C [W_\cE]$-modules
\begin{equation}\label{eq:4.24}
\IM^* ( M_{y,\rho^\circ} \rtimes \tau ) \quad \text{and} \quad
\Hom_{\pi_0 (\rZ_{\rZ_G (s)^\circ} (u))} \big( \rho^\circ, H_{d(u)}
(\mc B_{\rZ_G (s)^\circ}^u,\C) \big) \rtimes \tau .
\end{equation}
By definition \cite[Section 3.2]{AMS2} $M_{y,\rho^\circ}$ is the
$W(\rZ_G (s)^\circ,T)$-representation associated to $(y,\rho^\circ)$ by the generalized
Springer correspondence from \cite{Lus2}. It differs from the classical Springer
correspondence by the sign representation, so
\[
M_{y,\rho^\circ} = \text{sign } \otimes \Hom_{\pi_0 (\rZ_{\rZ_G (s)^\circ} (u))}
\big( \rho^\circ, H_{d(u)} (\mc B_{\rZ_G (s)^\circ}^u,\C) \big) .
\]
On both sides of \eqref{eq:4.24} the actions underlying $\rtimes \tau$ come from the
action of $\rZ_G (s,u)$ on $H_* (\mc B_{\rZ_G (s)^\circ}^u \times \Gamma_{\mr{Ad}(G^\circ) s},\C)
\cong H_* (\mc P_u ,\C)$. Moreover $\IM (w) = \mr{sign}(w) w$ for $w \in W(\rZ_G (s)^\circ,T)$
and IM is the identity on the group $\cR$ for $\rZ_G (s)$. We conclude that the two
representations in \eqref{eq:4.24} are equal.

This proves that $\bar{M} (t_1 = s,u,\rho)$ and $\bar{M}_{s,u,\rho,1}$ have the same
$X (T)$-weight space for the weight $s$. Since both representations are irreducible, that
implies that they are isomorphic.
\end{proof}

\section{Langlands parameters}
\label{sec:Langlands}

Let $F$ be a non-archimedean local field and let $\cG$ be a connected reductive group defined
over $F$. In this section we construct a bijection between enhanced Langlands parameters for
$\cG (F)$ and a certain collection of irreducible representations of twisted Hecke algebras.

We have to collect several notions about L-parameters, for which we
follow \cite{AMS}. For the background we refer to that paper, here we do little
more than recalling the necessary notations.
Let $\cG^\vee$ be the complex dual group of $\cG$. It is endowed with an action of
the Weil group $\mb W_F$, which preserves a pinning of $\cG^\vee$. The Langlands
dual group is ${}^L \cG = \cG^\vee \rtimes \mb W_F$.

\begin{defn}\label{defn:5.1}
A Langlands parameter for ${}^L \cG$ is a continuous group homomorphism
\[
\phi \colon \mb W_F \times \SL_2 (\C) \to \cG^\vee \rtimes \mb W_F
\]
such that:
\begin{itemize}
\item $\phi (w) \in \cG^\vee w$ for all $w \in \mb W_F$;
\item $\phi (\mb W_F)$ consists of semisimple elements;
\item $\phi |_{\SL_2 (\C)}$ is algebraic.
\end{itemize}
We call a L-parameter:
\begin{itemize}
\item bounded, if $\phi (\Fr_F) = (c,\Fr_F)$ with $c$ in a compact subgroup of $\cG^\vee$;
\item discrete, if $\rZ_{\cG^\vee}(\phi)^\circ = \rZ(\cG^\vee)^{\mb W_F,\circ}$.
\end{itemize}
\end{defn}
With \cite[\S 3]{Bor} it is easily seen that this definition of discreteness is equivalent
to the usual definition with proper Levi subgroups.

Let $\cG_\sc^\vee$ be the simply connected cover of the derived group $\cG^\vee_\der$.
Let $\rZ_{\cG^\vee_\ad}(\phi)$ be the image of $\rZ_{\cG^\vee}(\phi)$ in the adjoint
group $\cG^\vee_\ad$. We define
\[
Z^1_{\cG^\vee_\sc}(\phi) = \text{ inverse image of } \rZ_{\cG^\vee_\ad}(\phi)
\text{ under } \cG^\vee_\sc \to \cG^\vee_\ad .
\]
Notice that the conjugation action of $\cG^\vee_\sc \rtimes \mb W_F$ on $\cG^\vee_\sc$
descends to an action of $\mc G^\vee \rtimes \mb W_F$ on $\cG^\vee_\sc$.

\begin{defn}\label{defn:5.2}
To $\phi$ we associate the finite group $\cS_\phi := \pi_0 (Z^1_{\cG^\vee_\sc}(\phi))$.
An enhancement of $\phi$ is an irreducible representation of $\cS_\phi$.

The group $\cG^\vee$ acts on the collection of enhanced L-parameters for ${}^L \cG$ by
\[
g \cdot (\phi,\rho) =  (g \phi g^{-1}, g \cdot \rho), \quad
\text{where } g \cdot \rho(a)=\rho(g^{-1}a g) \text{ for } a\in \cS_\phi.
\]
Let $\Phi_\re ({}^L \cG)$ be the collection of $\cG^\vee$-orbits of enhanced L-parameters.
\end{defn}

Let us consider $\cG (F)$ as an inner twist of a quasi-split group. Via the Kottwitz
isomorphism it is parametrized by a character of $\rZ(\cG^\vee_\sc)^{\mb W_F}$, say
$\zeta_\cG$. We say that $(\phi,\rho) \in \Phi_\re ({}^L \cG)$ is relevant for $\cG (F)$
if $\rZ(\cG^\vee_\sc)^{\mb W_F}$ acts on $\rho$ as $\zeta_\cG$. The subset of
$\Phi_\re ({}^L \cG)$ which is relevant for $\cG (F)$ is denoted $\Phi_\re (\cG (F))$.

As is well-known, $(\phi,\rho) \in \Phi_\re ({}^L \cG)$ is already determined by
$\phi |_{\mb W_F}$ (the restriction to the first factor of $\mb W_F \times \SL_2 (\C)$), 
the unipotent element $u_\phi := \phi \big(1, \matje{1}{1}{0}{1} \big)$ and the
enhancement $\rho$. Sometimes we will also consider $\cG^\vee$-conjugacy classes 
of such triples $(\phi |_{\mb W_F}, u_\phi, \rho)$ as enhanced L-parameters. 
An enhanced L-parameter $(\phi |_{\mb W_F}, v, q\epsilon)$ will often be abbreviated 
to $(\phi_v, q \epsilon)$. We will study enhanced Langlands parameters via their 
cuspidal support, as introduced in \cite{AMS}.

\begin{defn}\label{defn:5.3}
For $(\phi,\rho) \in \Phi_\re ({}^L \cG)$ we write $G_\phi =
Z^1_{\cG^\vee_\sc}(\phi |_{\mb W_F})$, a complex reductive group. We say that $(\phi,\rho)$
is cuspidal if $\phi$ is discrete and $(u_\phi = \phi \big(1, \matje{1}{1}{0}{1} \big),\rho)$
is a cuspidal pair for $G_\phi$ in
the sense of \cite[\S 3]{AMS}. (This means that $\rho = \cF_{u_\phi}$, for a
$G_\phi$-equivariant cuspidal local system $\cF$ on $\cC^{G_\phi}_{u_\phi}$.)
We denote the collection of cuspidal L-parameters for ${}^L \cG$ by $\Phi_\cusp ({}^L \cG)$,
and the subset which is relevant for $\cG (F)$ by $\Phi_\cusp (\cG (F))$.
\end{defn}

We denote the cuspidal quasi-support of $(u_\phi,\rho)$, in the sense of \cite[\S 5]{AMS},
by $[M, v,q\epsilon]_{G_\phi}$. In particular $v \in M \subset G_\phi \subset \cG^\vee_\sc$.

\begin{prop}\label{prop:5.4} \cite[Proposition 7.3]{AMS} \ \\
Let $(\phi,\rho) \in \Phi_\re (\cG (F))$. Upon replacing $(\phi,\rho)$ by 
$\cG^\vee$-conjugate and replacing $(M,v,q\epsilon)$ by a $G_\phi$-conjugate, there exists 
a Levi subgroup $\cL (F) \subset \cG (F)$ such that $(\phi |_{\mb W_F},v,q\epsilon)$
is a cuspidal L-parameter for $\cL (F)$. Moreover
\[
\cL^\vee \rtimes \mb W_F = \rZ_{\cG^\vee \rtimes \mb W_F}(\rZ(M)^\circ) ,
\]
and this group is uniquely determined by $(\phi,\rho)$ up to $\cG^\vee$-conjugation.
\end{prop}

Inside ${}^L \mc G$, we can conjugate $\cL^\vee \rtimes \mb W_F$ with elements of 
$\mc G^\vee$. A subgroup of the form $g (\cL^\vee \rtimes \mb W_F) g^{-1}$ projects naturally 
onto $\mb W_F$, but unlike ${}^L \mc L$ it does not necessarily contain $\mb W_F$. If 
$(\phi',\rho') \in \Phi_e ({}^L \mc L)$, then $g \cdot (\phi',\rho')$ is an enhanced L-parameter
for $g (\cL^\vee \rtimes \mb W_F) g^{-1}$.

Suppose that $(\phi,\rho)$ is as in Proposition \ref{prop:5.4}. We define its modified
cuspidal support as
\[
{}^L \Psi (\phi,\rho) = (\cL^\vee \rtimes \mb W_F, \phi |_{\mb W_F}, v,q\epsilon)
/ \mc G^\vee\text{-conjugacy} .
\]
The right hand side consists of a Langlands dual group and a cuspidal enhanced L-parameter for
that (up to $\mc G^\vee$-conjugacy). Every enhanced L-parameter for ${}^L \cG$ is conjugate
to one as above, so ${}^L \Psi$ can be considered as a well-defined map from $\Phi_\re ({}^L \cG)$
to $\mc G^\vee$-conjugacy of pairs consisting of a $\mb W_F$-stable Levi subgroup of $\mc G^\vee$ 
and a cuspidal L-parameter for the associated L-group.
Notice that ${}^L \Psi$ preserves boundedness of enhanced L-parameters.

We also need Bernstein components of enhanced L-parameters. Recall from \cite[\S 3.3.1]{Hai} 
that the group of unramified characters of $\cL (F)$ is naturally isomorphic to 
$\big( (Z (\cL^\vee)^{\mb I_F})_{\mb W_F} \big)^\circ$. We consider this as an object on the 
Galois side of the local Langlands correspondence and with Lemma \ref{lem:5.17} we write
\begin{equation}\label{eq:Xnr}
X_\nr ({}^L \cL) = \big( (Z (\cL^\vee)^{\mb I_F})_{\mb W_F} \big)^\circ = 
( Z (\cL^\vee \rtimes \mb I_F )_{\mb W_F} )^\circ.
\end{equation}
Given $(\phi',\rho') \in \Phi_\re (\cL (F))$ and $z \in  Z (\cL^\vee \rtimes \mb I_F )_{\mb W_F}$,
we define $(z \phi',\rho') \in \Phi_\re (\cL (F))$ by
\[
z \phi' = \phi' \text{ on } \mb I_F \times \SL_2 (\C) \text{ and }
(z \phi')(\Fr_F) = \tilde z \phi' (\Fr_F) ,
\]
where $\tilde z \in Z (\cL^\vee \rtimes \mb I_F )$ represents $z$.

\begin{defn}\label{defn:5.5}
An inertial equivalence class for $\Phi_\re (\cG (F))$ is the $\cG^\vee$-conjugacy class
$\mf s^\vee$ of a pair $(\cL^\vee\rtimes \bW_F,\mf s^\vee_\cL)$, where $\cL (F)$
is a Levi subgroup of $\cG (F)$ and $\mf s^\vee_\cL$ is a $X_\nr ({}^L \cL)$-orbit
in $\Phi_\cusp (\cL (F))$.

The Bernstein component of $\Phi_\re (\cG (F))$ associated to  $\mf s^\vee$ is
\begin{equation} \label{eqn:Bcomp}
\Phi_\re (\cG (F))^{\mf s^\vee} := {}^L \Psi^{-1} (\cL^\vee \rtimes \mb W_F, \mf s^\vee_\cL).
\end{equation}
We denote the set of inertial equivalence classes for $\Phi_\re (\cG (F))$ by $\fB^\vee(\cG(F))$.
\end{defn}

In this way, we obtain a partition of the set $\Phi_\re (\cG (F))$ analogous
to the partition of $\Irr(\cG (F))$ induced by its Bernstein decomposition:
\begin{equation} \label{eqn:Bdec}
\Phi_\re (\cG (F)) =
\bigsqcup\nolimits_{\fs^\vee\in\fB^\vee(\cG(F))}\Phi_\re (\cG (F))^{\mf s^\vee},
\end{equation}
We note that $\Phi_\re (\cL (F))^{\mf s^\vee_\cL}$ is a torsor for the quotient of
the complex torus $X_\nr ({}^L \cL)$ by a finite subgroup. In particular
$\Phi_\re (\cL (F))^{\mf s^\vee_\cL}$ is isomorphic to a torus as complex algebraic variety,
albeit not in a canonical way.

With an inertial equivalence class $\fs^\vee$ for $\Phi_\re (\cG (F))$ we associate the
finite group
\begin{equation} \label{eqn:Ws}
W_{\fs^\vee} := \text{ stabilizer of } \fs^\vee_\cL \text{ in }
\rN_{\cG^\vee}(\cL^\vee \rtimes \mb W_F) / \cL^\vee .
\end{equation}
Let $W_{\fs^\vee,\phi_v,q\epsilon}$ be the isotropy group in $W_{\fs^\vee}$ of 
$(\phi_v, q\epsilon) \in \fs^\vee_\cL$. With the generalized Springer correspondence 
\cite[Theorem 5.5]{AMS} we can attach to any element of\\
${}^L \Psi^{-1} (\cL^\vee \rtimes \mb W_F, \phi_v, q\epsilon)$ an irreducible projective
representation of $W_{\fs^\vee,\phi_v,q\epsilon}$.
More precisely, consider the cuspidal quasi-support
\[
q \ft = [G_\phi \cap \cL^\vee_c, v ,q \epsilon ]_{G_\phi} ,
\]
where $\cL^\vee_c \subset \cG^\vee_\sc$ is the preimage of $\cL^\vee$ under $\cG^\vee_\sc
\to \cG^\vee$. In this setting we write the group $W_{q\cE}$ from \eqref{eq:0.20}
as $W_{q \ft}$. By \cite[Lemma 8.2]{AMS} $W_{q \ft}$ is canonically isomorphic to
$W_{\fs^\vee,\phi_v,q\epsilon}$. According to \cite[Proposition 9.1]{AMS} there exist
a 2-cocycle $\kappa_{q \ft}$ of $W_{q \ft}$ and a bijection (canonical up to the choice
of $\kappa_{q \ft}$ in its cohomology class)
\[
{}^L \Sigma_{q \ft} \colon{}^L \Psi^{-1} (\cL^\vee \rtimes \mb W_F, \phi_v ,q\epsilon) \to
\Irr (\C [W_{q \ft}, \kappa_{q\ft}]) .
\]
It is given by applying the generalized Springer correspondence for $(G_\phi, q \ft)$ to
$(u_\phi,\rho)$.

\begin{thm}\label{thm:5.6} \cite[Theorem 9.3]{AMS} \ \\
There exists a bijection
\[
\begin{array}{ccc}
\Phi_\re (\cG (F))^{\fs^\vee} & \longleftrightarrow &
\big( \Phi_\re (\cL (F))^{\mf s^\vee_\cL} \q W_{\fs^\vee} \big)_\kappa ,\\
(\phi,\rho) & \mapsto & \big( {}^L \Psi (\phi,\rho), {}^L \Sigma_{q \ft} (\phi,\rho) \big) .
\end{array}
\]
It is almost canonical, in the sense that it depends only on the choices of 2-cocycles
$\kappa_{q \ft}$ as above.
\end{thm}

\subsection{Graded Hecke algebras} \
\label{par:LGHA}

In Theorem \ref{thm:4.11} we saw that the irreducible representations of a (twisted)
affine Hecke algebra can be parametrized with a (twisted) extended quotient of a torus by
a finite group. Motivated by the analogy with Theorem \ref{thm:5.6}, we want to associate
to any Bernstein component $\Phi_\re (\cG (F))^{\mf s^\vee}$ a twisted affine Hecke algebra,
whose irreducible representations are naturally parametrized by $\Phi_\re (\cG (F))^{\mf s^\vee}$.
As this turns out to be complicated, we first do something similar with twisted graded Hecke
algebras. From a Bernstein component we will construct a family of algebras, such that a
suitable subset of their irreducible representations is canonically in bijection with
$\Phi_\re (\cG (F))^{\mf s^\vee}$. Of course this will be based on the cuspidal quasi-support
$[M,v,q\epsilon]_{G_\phi}$ for the group
\begin{equation} \label{eqn:Gphi}
G_\phi := Z^1_{\cG^\vee_\sc}(\phi |_{\mb W_F}).
\end{equation}
As before, we abbreviate $T = \rZ(M)^\circ$. Let $\cL_c^\vee$ be the preimage of $\cL^\vee$
in $\cG_\sc^\vee$. We record that by \cite[(99)]{AMS}
\begin{equation}\label{eq:5.24}
M = G_\phi \cap \cL_c^\vee .
\end{equation}
One problem is that $\rZ(\cG^\vee)^\circ$ was left out of $\cG^\vee_\sc$, so we can never
see it when working in $G_\phi$. We resolve this in a crude way, replacing $G_\phi$ by
$G_\phi \times X_\nr ({}^L \cG)$. Although that is not a subgroup of $\cG^\vee$ or
$\cG^\vee_\sc$, the next result implies that the real split part of its centre has
the desired shape.

\begin{lem}\label{lem:5.7}
We use the notations from Proposition \ref{prop:5.4}. The natural map
\[
T \times X_\nr ({}^L \cG) \to X_\nr ({}^L \cL)
\]
is a finite covering of complex tori.
\end{lem}
\begin{proof}
In Proposition \ref{prop:5.4} we saw that
\begin{equation}\label{eq:5.2}
\cL^\vee \rtimes \mb W_F = \rZ_{\cG^\vee \rtimes \mb W_F}(T) .
\end{equation}
Hence the image of $M^\circ$ under the covering $\cG^\vee_\sc \to \cG^\vee_\der$ is
contained in $\cL^\vee$. It also shows that $\mb W_F$ fixes $T$ pointwise, so
\[
T = (\rZ(M)^{\mb I_F} )^\circ_{\mb W_F} .
\]
As $\cL^\vee$ is a Levi subgroup of $\cG^\vee$, it contains $Z (\cG^\vee)^\circ$.
Hence there exists a natural map
\begin{equation}\label{eq:5.1}
T \times X_\nr (\cG) = \big( \rZ(M)^{\mb I_F} \times \rZ(\cG^\vee)^{\mb I_F}
\big)^\circ_{\mb W_F} \to (Z (\cL^\vee)^{\mb I_F} )^\circ_{\mb W_F} = X_\nr ({}^L \cL) .
\end{equation}
The intersection of $\rZ(\cG^\vee)^\circ$ and $\cG^\vee_\der$ is finite and $T$
lands in $\cG^\vee_\der \cap \cL^\vee$, so the kernel of \eqref{eq:5.1} is finite.

Recall from Proposition \ref{prop:5.4} that $\phi (\mb W_F) \subset \cL^\vee
\rtimes \mb W_F$. Hence
\[
\rZ(\cL^\vee \rtimes \mb W_F) \subset \rZ_{\cG^\vee}(\phi (\mb W_F)) \quad \text{and} \quad
\rZ(\cL^\vee_c \rtimes \mb W_F)^\circ \subset \rZ_{\cG^\vee_\sc}(\phi (\mb W_F))^\circ .
\]
Since $M^\circ$ is a Levi subgroup of $\rZ_{\cG^\vee_\sc}(\phi (\mb W_F))^\circ$ and by
\eqref{eq:5.2}, $T$ equals $\rZ(\cL^\vee_c \rtimes \mb W_F )^\circ$.
In particular
\begin{align*}
\dim T = \dim \rZ(\cL^\vee_c \rtimes \mb W_F )^\circ & =
\dim \rZ(\cL^\vee_c \rtimes \mb I_F )^\circ_{\mb W_F} \\ 
& = \dim \rZ(\cL^\vee \rtimes
\mb I_F )^\circ_{\mb W_F} - \dim \rZ(\cG^\vee \rtimes \mb I_F )^\circ_{\mb W_F} ,
\end{align*}
showing that both sides of \eqref{eq:5.1} have the same dimension. As the map is an
algebraic homomorphism between complex tori and has finite kernel, it is surjective.
\end{proof}

Recall that $\mf s^\vee_\cL$ came from the cuspidal quasi-support $(M,v,q\epsilon)$.
For $(\phi_b |_{\mb W_F},v,q\epsilon) \in \Phi_\re (\cL (F))^{\mf s^\vee_\cL}$ we can
consider the group
\[Z^1_{\cG^\vee_\sc}(\phi_b |_{\mb W_F})
 \times X_\nr ({}^L \cG) =
G_{\phi_b} \times X_\nr ({}^L \cG) ,
\]
which contains $M \times X_\nr ({}^L \cG)$ as a quasi-Levi subgroup. We choose an
almost direct factorization for $G_\phi \times X_\nr ({}^L \cG)$ as in \eqref{eq:0.1}
and we put
\begin{equation}\label{eq:5.11}
\begin{aligned}
\mh H (\phi_b,v,q\epsilon ,\vec{\mb r}) := & \; \mh H \big(G_{\phi_b}
\times X_\nr ({}^L \cG), M \times X_\nr ({}^L \cG), q\cE ,\vec{\mb r} \big) \\
= & \; \mh H \big( \Lie(X_\nr ({}^L \cL)), W_{\mf s^\vee, (\phi_b)_v,q\epsilon},
c \vec{\mb r}, \natural_{q\cE} \big) ,
\end{aligned}
\end{equation}
where $q\cE$ is the $M$-equivariant cuspidal local system on $\cC^M_{\log v}$ with
$q\cE_{\log v} = q \epsilon$ as representations of $\pi_0 (\rZ_M (v)) =
\pi_0 (\rZ_M (\log v))$. From Lemma \ref{lem:5.7} we see that
\begin{align*}
\mh H (\phi_b,v,q\epsilon,\vec{\mb r}) & = \mh H (Z^1_{\cG^\vee_\sc}
(\phi_b |_{\mb W_F}) ,M,q\cE,\vec{\mb r}) \otimes S \big(\Lie (X_\nr ({}^L \cG))^* \big) \\
& = \mh H (G_{\phi_b} ,M,q\cE,\vec{\mb r}) \otimes
S \big(\Lie (\rZ(\cG^\vee \rtimes \mb I_F)^\circ_{\mb W_F})^* \big) .
\end{align*}
We say that a representation of $\mh H (\phi_b,v,q\epsilon,\vec{\mb r})$ is essentially
discrete series if its restriction to $\mh H (G_{\phi_b} ,M,q\cE,\vec{\mb r})$ is so,
in the sense of \cite[Definition 3.27]{AMS2}. That means that the real parts of its
weights (as $\mh H (G_{\phi_b} ,M,q\cE,\vec{\mb r})$-representation) must lie in
$\mr{Lie} (\rZ(G_{\phi_b})^\circ) \oplus \mf t_\R^{--}$.

Let $X_\nr ({}^L \cL) = X_\nr ({}^L \cL)_\uni \times X_\nr ({}^L \cL)_{\rs}$
be the polar decomposition of the complex torus $X_\nr ({}^L \cL)$.
Let $(\phi_b |_{\mb W_F},v,q \epsilon) \in \Phi_\re (\cL (F))^{\mf s^\vee_\cL}$ with
$\phi_b$ bounded. Suppose that $(\phi,\rho) \in \Phi_\re (\cG (F))^{\mf s^\vee}$ with:
\begin{equation}\label{eq:5.9}
\begin{array}{@{\bullet \quad}l}
\phi |_{\mb I_F} = \phi_b |_{\mb I_F}; \\
\phi (\Fr_F) \phi_b (\Fr_F)^{-1} \in X_\nr ({}^L \cL^\vee)_{\rs}; \\
\textup{d} \phi|_{\SL_2 (\C)} \matje{1}{0}{0}{-1} \in \Lie (M).
\end{array}
\end{equation}
For such $(\phi,\rho)$ and $\vec{r} \in \C^d$ we define
\begin{align*}
& E (\phi,\rho,\vec{r}) \; = \IM^* \, E_{\log (u_\phi), \log (\phi (\Fr_F)^{-1} \phi_b (\Fr_F))
+ \textup{d}\vec{\phi} \matje{\vec{r}}{0}{0}{-\vec{r}} , \vec{r},\rho}
\in \Mod (\mh H (\phi_b,v,q\epsilon,\vec{\mb r})) , \\
& M (\phi,\rho,\vec{r}) = \IM^* M_{\log (u_\phi), \log (\phi (\Fr_F)^{-1} \phi_b (\Fr_F))
+ \textup{d}\vec{\phi} \matje{\vec{r}}{0}{0}{-\vec{r}} , \vec{r},\rho}
\in \Irr (\mh H (\phi_b,v,q\epsilon,\vec{\mb r})) .
\end{align*}
If in addition d$\phi \matje{1}{0}{0}{-1} \in \mr{Lie}(T) + \sigma_v$, as can always be
arranged by Proposition \ref{prop:0.7}.c, then we define an algebraic cocharacter
$\chi_{\phi,v} = \chi_{u_\phi,v}$ of $T$ by
\begin{equation}\label{eq:5.4}
\chi_{\phi,v}(z) = \phi \big( 1, \matje{z}{0}{0}{z^{-1}} \big) \gamma_v \matje{z^{-1}}{0}{0}{z}. 
\end{equation}
We note that $\chi_{\phi,v}$ stems from \cite[Lemma 7.6]{AMS} and that
\[
\textup{d}\vec{\chi}_{\phi,v}(\vec{r}) = 
\textup{d} \vec{\phi} \matje{\vec{r}}{0}{0}{-\vec{r}} - \vec{r} \sigma_v .
\]

\begin{thm}\label{thm:5.8}
Fix $\vec{r} \in \C^d$ and $(\phi_b |_{\mb W_F}, v,q\epsilon ) \in
\Phi_\re (\cL (F))^{\mf s^\vee_\cL}$ with $\phi_b$ bounded.
\enuma{
\item The map $(\phi,\rho) \mapsto M (\phi,\rho,\vec{r})$ defines a canonical
bijection between
\begin{itemize}
\item ${}^L \Psi^{-1}(\cL^\vee \rtimes \mb W_F, X_\nr ({}^L \cL )_\rs
\phi_b |_{\mb W_F}, v , q \epsilon)$;
\item the irreducible representations of $\mh H (\phi_b,v,q\epsilon,\vec{\mb r})$
with central character in
$\Lie (X_\nr ({}^L \cL)_{\rs}) / W_{\fs^\vee,\phi_b,v,q\epsilon} \times \{\vec{r}\}$.
\end{itemize}
\item Assume that $\Re (\vec{r}) \in \R_{\geq 0}^d$. The following are equivalent:
\begin{itemize}
\item $\phi$ is bounded;
\item ${}^L \Psi (\phi,\rho) = (\cL^\vee \rtimes \mb W_F, \phi_b |_{\mb W_F}, v ,
q \epsilon)$;
\item $E (\phi,\rho,\vec{r})$ is tempered;
\item $M (\phi,\rho,\vec{r})$ is tempered.
\end{itemize}
\item Suppose that $\Re (\vec{r}) \in \R_{>0}^d$. Then $\phi$ is discrete if and only if
$M (\phi,\rho,\vec{r})$ is essentially discrete series and the rank of $R(G^\circ_{\phi_b},
T)$ equals $\dim_\C (T)$.

In this case $\phi (\Fr_F) \phi_b (\Fr_F)^{-1}$ comes from an element of
$Z (G_{\phi_b}^\circ) \times X_\nr ({}^L \cG)$ via Lemma \ref{lem:5.7}.
\item Let $\zeta \in X_\nr ({}^L \mc G)_{rs}$. Then
\[
M (\zeta \phi,\rho,\vec{r}) = \log (\zeta) \otimes M (\phi,\rho,\vec{r}) \quad \text{and}
\quad E (\zeta \phi,\rho,\vec{r}) = \log (\zeta) \otimes E (\phi,\rho,\vec{r}) .
\]
\item Suppose that $\Re (\vec{r}) \in \R_{>0}^d$ and that $\phi (\Fr_F) \phi_b (\Fr_F)^{-1}$
comes from\\ $Z (G_{\phi_b}^\circ) \times X_\nr ({}^L \cG)$ via Lemma \ref{lem:5.7}.
Then $M (\phi,\rho,\vec{r}) = E (\phi,\rho,\vec{r})$.
\item If $\textup{d}\phi \matje{1}{0}{0}{-1} \in \mr{Lie}(T) + \sigma_v$, then 
$E(\phi,\rho,\vec{r})$ and $M(\phi,\rho,\vec{r})$ admit the central character
$W_{\mf s^\vee, (\phi_b)_v,q\epsilon} (\log (\phi (\Fr_F) \phi_b (\Fr_F)^{-1}) \pm 
\textup{d}\vec{\chi}_{\phi,v} (\vec{r}), \vec{r})$.
}
\end{thm}
\begin{proof}
(a) By Theorem \ref{thm:5.6} every element of ${}^L \Psi^{-1}(\cL^\vee \rtimes
\mb W_F, X_\nr ({}^L \cL )_\rs \phi_b |_{\mb W_F}, v , q \epsilon)$ has a
representative $(\phi,\rho)$ with $\phi |_{\mb W_F}$ in
$X_\nr ({}^L \cL )_\rs \phi_b |_{\mb W_F}$. Then $\phi |_{\mb I_F}$ is fixed, so
$\phi |_{\mb W_F}$ can be described by the single element
$\phi (\Fr_F) \phi_b (\Fr_F)^{-1} \in X_\nr ({}^L \cL^\vee)_{\rs}$. Since
$X_\nr ({}^L \cL^\vee)_{\rs}$ is the real split part of a complex torus, there is a
unique logarithm
\begin{equation}\label{eq:5.13}
\sigma_0 = \log \big( \phi (\Fr_F) \phi_b (\Fr_F)^{-1} \big) \in
\Lie (X_\nr ({}^L \cL^\vee)_{\rs}) .
\end{equation}
Clearly $(\phi_b,v)$ is the unique bounded L-parameter in
$X_\nr ({}^L \cL)_{\rs}(\phi_b,v)$. Hence every element of $\mc G^\vee_\ad$ that
centralizes $\phi$ also centralizes $\phi_b$, which implies
\[
G_\phi = Z^1_{\cG^\vee_\sc}(\phi |_{\mb W_F}) \subset
Z^1_{\cG^\vee_\sc}(\phi_b |_{\mb W_F}) = G_{\phi_b} .
\]
In particular $\phi (\SL_2 (\C)) \subset G_{\phi_b}$ and
\[
\pi_0 (\rZ_{G_\phi}(u_\phi)) = \pi_0 \big( \rZ_{G_{\phi_b}} (\sigma_0 ,\log(u_\phi)) \big) .
\]
By assumption $q \Psi_{G_\phi}(u_\phi,\rho) = (v,q\epsilon)$, and by
\cite[Proposition 3.7]{AMS2} this cuspidal quasi-support is relevant for
\[
\mh H (\phi_b,v,q\epsilon,\vec{\mb r}) =
\mh H \big( G_{\phi_b} \times X_\nr ({}^L \cG) ,M \times X_\nr ({}^L \cG),q\cE,\vec{\mb r} \big) .
\]
By Proposition \ref{prop:0.7}.c, $(\phi,\rho)$ is conjugate to an enhanced
L-parameter with all the above properties, which in addition satisfies
\[
\textup{d} \phi|_{\SL_2 (\C)} \matje{1}{0}{0}{-1} \in \Lie (M).
\]
Consequently $(\log (u_\phi),\sigma_0,\vec{r},\rho)$ is a parameter of the kind considered in 
Section \ref{sec:0}, and $\phi |_{\SL_2 (\C)}$ can play the role of $\gamma$ from 
\eqref{eq:0.16}. By reversing the above procedure every parameter $(y,\sigma',\vec{r},\rho')$
for $\mh H (\phi_b,v,q\epsilon,\vec{\mb r})$ gives rise to an element of 
\[
{}^L \Psi^{-1} (\cL^\vee \rtimes \mb W_F, 
X_\nr ({}^L \cL )_\rs \phi_b |_{\mb W_F}, v , q \epsilon) .
\]
The equivalence relations on these two sets of parameters agree, for both come from
conjugation by $G_{\phi_b}$.

Now it follows from Theorem \ref{thm:0.4}, Proposition \ref{prop:0.7} and Proposition
\ref{prop:0.5}.f that 
\[
{}^L \Psi^{-1}(\cL^\vee \rtimes \mb W_F, X_\nr ({}^L \cL )_\rs \phi_b 
|_{\mb W_F}, v ,q \epsilon)
\] 
parametrizes the part of $\Irr_r (\mh H (\phi_b,v,q\epsilon))$ with central character in 
\[
\Lie (X_\nr ({}^L \cL)_{\rs}) / W_{\fs^\vee,\phi_b,v,q\epsilon} \times \{\vec{r}\}.
\]
As in \cite[Theorem 3.29]{AMS2} and Proposition \ref{prop:0.5}, we compose this
parametrization with the Iwahori--Matsumoto involution from \eqref{eq:0.19}. Then the
representation associated to $(\phi,\rho)$ becomes $\pi (\phi,\rho,r)$.\\
(b) By \cite[Theorem 3.25]{AMS2} and \cite[(84)]{AMS2}
the third and the fourth statements are both equivalent to
\[
\phi (\Fr_F) \phi_b (\Fr_F)^{-1} \in \Lie (X_\nr ({}^L \cL)_\uni ).
\]
But by construction this lies in Lie$ (X_\nr ({}^L \cL)_{\rs} )$, so the
statement becomes $\phi (\Fr_F) = \phi_b (\Fr_F)$. As $(\phi_b,v)$ is the only bounded
L-parameter in $X_\nr ({}^L \cL)_{\rs} (\phi_b,v)$, this holds if and only if $\phi$ is
bounded. Since the map ${}^L \Psi$ preserves $\phi |_{\mb W_F}$, the statement
$\phi (\Fr_F) = \phi_b (\Fr_F)$ is also equivalent to
\[
{}^L \Psi (\phi,\rho) = (\cL^\vee \rtimes \mb W_F, \phi_b |_{\mb W_F}, v , q \epsilon) .
\]
Knowing these equivalences, the equality $M (\phi,\rho,\vec{r}) = E (\phi,\rho,\vec{r})$
is given in Proposition \ref{prop:0.5}.b.\\
(c) Suppose that $\phi$ is discrete. Then
\[
G_\phi^\circ = \rZ_{\cG^\vee_\sc}(\phi (\mb W_F))^\circ =
\rZ_{\cG^\vee_\sc}(\phi_b (\mb W_F),\sigma)^\circ
\]
is a reductive group in which $\phi (\SL_2 (\C))$ has finite centralizer. This implies
that $G_\phi^\circ$ is semisimple and that $u_\phi$ is distinguished in it.
The first of these two properties implies that $G_\phi^\circ$ is a full rank subgroup
of $G_{\phi_b}$, and that $G_{\phi_b}^\circ$ is also semisimple. In other words,
$R(G^\circ_{\phi_b},T)$ has rank equal to the dimension of $T$.
Then $u_\phi$ is distinguished in $G_{\phi_b}^\circ$ as well, and \cite[(85)]{AMS2}
says that $M (\phi,\rho,\vec{r})$ is essentially discrete series.

Conversely, suppose that $M (\phi,\rho,\vec{r})$ is essentially discrete series and that
the rank of $R(G^\circ_{\phi_b},T)$ equals $\dim_\C (T)$.
Then $G_{\phi_b}^\circ$ is semisimple and by \cite[(85)]{AMS2} $u_\phi \in G_\phi^\circ$
is distinguished in $G_{\phi_b}^\circ$. Hence $\rZ_{G_\phi}(u_\phi)^\circ$ is contained
in the unipotent group $\rZ_{G_{\phi_b}}(u_\phi)^\circ$, and itself unipotent. It is known
(see for example \cite[\S 4.3]{Ree}) that
\[
\rZ_{\cG^\vee_\sc}(\phi )^\circ = \rZ_{G_\phi} \big( \phi (\SL_2 (\C)) \big)^\circ
\]
is the maximal reductive quotient of $\rZ_{G_\phi}(u_\phi)^\circ$. Hence
$\rZ_{\cG^\vee_\sc}(\phi )^\circ$ is trivial, which means that $\phi$ is discrete.

In this case Proposition \ref{prop:0.6}.c says that $\sigma_0 \in Z \big( \mr{Lie}(
G_{\phi_b} \times X_\nr ({}^L \mc G) ) \big)$. Via the exponential map, that translates
to the statement about $\phi (\Fr_F) \phi_b (\Fr_F)^{-1}$.\\
(d) This is a direct consequence of Proposition \ref{prop:0.5}.d (and, for
$E (\phi,\rho,\vec{r})$, also the proof thereof).\\
(e) Via \eqref{eq:5.13}, the condition becomes $\sigma_0 \in Z \big( \mr{Lie}(
G_{\phi_b} \times X_\nr ({}^L \mc G) ) \big)$. Apply Proposition \ref{prop:0.5}.e.\\
(f) This follows from Proposition \ref{prop:0.7}.e with $\gamma = \phi |_{SL_2 (\C)}$.
\end{proof}

We conclude this paragraph with some remarks about parabolic induction.
Suppose that $\mc Q (F) \subset \cG (F)$ is a Levi subgroup such that $\phi$ has
image in ${}^L \mc Q$. Let ${\mc Q}^\vee_c$ be the inverse image of $\mc Q^\vee$ in
$\mc G^\vee_\sc$, by \cite[\S 3]{Bor} it equals
$\rZ_{\mc G^\vee_\sc} (\rZ({\mc Q}^\vee_c \rtimes \mb W_F)^\circ)$. Therefore
\begin{equation}\label{eq:5.14}
\begin{aligned}
Z^1_{{\mc Q}^\vee_c}(\phi_b |_{\mb W_F}) & = Z^1_{{\mc G}^\vee_\sc}(\phi_b |_{\mb W_F})
\cap \rZ_{\mc G^\vee_\sc} (\rZ({\mc Q}^\vee_c \rtimes \mb W_F)^\circ) \\
& = G_{\phi_b} \cap \rZ_{\mc G^\vee_\sc} (\rZ({\mc Q}^\vee_c \rtimes \mb W_F)^\circ) .
\end{aligned}
\end{equation}
This in turn shows that
\[
G_{\phi_b}^\circ \cap Z^1_{{\mc Q}^\vee_c}(\phi_b |_{\mb W_F}) =
\rZ_{{\mc Q}^\vee_c}(\phi_b (\mb W_F))^\circ
\]
is a Levi subgroup of $G_{\phi_b}^\circ$.
Furthermore $Z^1_{{\mc Q}^\vee_c}(\phi_b |_{\mb W_F})$ contains $M$, for the cuspidal
quasi-support of $(\phi,\rho)$ with respect to ${}^L \cG$ is the same as the
cuspidal quasi-support of $(\phi,\rho^Q)$ with respect to ${}^L \mc Q$, for a
suitable $\rho^Q \in \Irr (\mc S_\phi^{\mc Q})$ \cite[Proposition 5.6.a]{AMS}.

Let $\zeta$ be the character of $\rZ(\mc G^\vee_\sc)$ determined by $\rho$, an extension
of the character $\zeta_\cG \in \Irr (\rZ(\mc G^\vee_\sc)^{\mb W_F})$ which was used
to define $\mc G (F)$-relevance. Let $\zeta^{\mc Q} \in \Irr (\rZ(\mc Q^\vee_\sc))$ be
derived from $\zeta$ as in \cite[Lemma 7.4]{AMS}. Let $p_\zeta \in \C [\mc S_\phi]$ and
$p_{\zeta^{\mc Q}} \in \C [\mc S_\phi^{\mc Q}]$ be the central idempotents associated to
these characters.

Let $\mc S_{\phi,\mc Q}$ be the component group of the centralizer of $u_\phi$ in
$Z^1_{{\mc Q}^\vee_c}(\phi |_{\mb W_F})$, or equivalently the component group of the
centralizer of $(\phi (\Fr),u_\phi)$ in \eqref{eq:5.14}. By \cite[Lemma 7.4.c]{AMS}
there exist a canonical isomorphism and a canonical injection
\[
p_{\zeta^{\mc Q}} \C [\mc S_\phi^{\mc Q}] \cong p_\zeta \C [\mc S_{\phi,\mc Q}]
\to p_\zeta \C [\mc S_\phi] .
\]
This enables to restrict representations of $\mc S_\phi$ to $\mc S_\phi^{\mc Q}$, and
it shows that enhancements for $\phi \in \Phi (\mc Q (F))$ can just as well be constructed
via \eqref{eq:5.14} and $\mc S_{\phi,Q}$.

That is, $G_{\phi_b} \times X_\nr ({}^L \mc G)$ and $Z^1_{{\mc Q}^\vee_c}(\phi_b |_{\mb W_F})
\times X_\nr ({}^L \mc G)$ fulfill the conditions of \cite[Proposition 3.22]{AMS2}
and Corollary \ref{cor:4.15}. It follows that the families of representations
\begin{align*}
& (\phi,\rho,\vec{r}) \mapsto E_{\log (u_\phi), \log (\phi (\Fr_F)^{-1} \phi_b (\Fr_F))
+ \textup{d}\vec{\phi} \matje{\vec{r}}{0}{0}{-\vec{r}} , \vec{r},\rho}
\in \Mod (\mh H (\phi_b,v,q\epsilon,\vec{\mb r})) , \\
& (\phi,\rho,\vec{r}) \mapsto M_{\log (u_\phi), \log (\phi (\Fr_F)^{-1} \phi_b (\Fr_F))
+ \textup{d}\vec{\phi} \matje{\vec{r}}{0}{0}{-\vec{r}} , \vec{r},\rho}
\in \Irr (\mh H (\phi_b,v,q\epsilon,\vec{\mb r}))
\end{align*}
are compatible with parabolic induction in the same sense as \cite[Proposition 3.22]{AMS2}
and Corollary \ref{cor:4.15}. In view of \cite[(83)]{AMS2} this does
not change upon applying the Iwahori--Matsumoto involution, so it also goes for
the representations $E (\phi,\rho,\vec{r})$ and $M (\phi,\rho,\vec{r})$
considered in Theorem \ref{thm:5.8}.

\subsection{Root systems} \
\label{par:roots}

We fix an inertial equivalence class $\mf s^\vee$ for $\Phi_\re (\mc G (F))$, represented 
by a cuspidal L-parameter $(\phi |_{\mb W_F}, v,q\epsilon)$ for $\mc L (F)$. 
We use the notations from Proposition \ref{prop:5.4} and \eqref{eq:5.24}, in particular
$T = \rZ(M)^\circ = \rZ(\cL_c^\vee)^{\mb W_F, \circ}$.  We define
\begin{equation} \label{eqn:J}
J :=  \rZ^1_{\mc G^\vee_\sc}(\phi |_{\mb I_F}),
\end{equation}
a variation on $G_\phi$ from \eqref{eqn:Gphi}. The groups $T, J$ and $M = G_\phi \cap \cL_c^\vee$ 
depend only on $(\mc L^\vee, \mf s_\cL^\vee)$. We note that $J$ is reductive, 
possibly disconnected and
\begin{equation}\label{eq:3.12}
G_\phi \subset J \quad \text{and} \quad G_\phi^\circ = \rZ_J (\phi (\Fr_F))^\circ .
\end{equation}
In this paragraph, we use the convention that a root system is a finite and integral root system.

\begin{prop}\label{prop:5.9}
Define $R (J^\circ,T)$ as the set of $\alpha \in X^* (T) \setminus \{0\}$
which appear in the adjoint action of $T$ on Lie$ (J^\circ)$.
\enuma{
\item $R(J^\circ,T)$ is a root system.
\item There exists a $(\phi_1 |_{\mb W_F}, v, q\epsilon)$ such that
$R(G^\circ_{\phi_1},T)_\red = R(J^\circ,T)_\red$.
\item If $t \in T$ commutes with $\mc G^\vee$, then it lies in the kernel of
every $\alpha \in R(J^\circ,T)$.
}
\end{prop}
\begin{rem}\label{rem:5.25}
This result does not imply that $R(G^\circ_{\phi_1},T)$ equals $R(J^\circ,T)$. 
For example, suppose that $\mc G = \rU_{2n+1}$ is an unramified unitary group and 
$\phi_1 (\mb I_F) = 1$. Let $\mc L^\vee$ be the diagonal torus in $\mc G^\vee = \GL_{2n+1}(\C)$.
For $\mf s_\cL^\vee$ we can take the set of enhanced L-parameters 
corresponding to the unramified characters of $\mc L (F)$. Then 
\[
J^\circ = \SL_{2n+1}(\C) ,\; G^\circ_{\phi_1} = \SO_{2n+1}(\C) \text{ and } 
T = \mc L^\vee \cap SO_{2n+1}(\C).
\] 
In this case $R(G^\circ_{\phi_1},T)$ has type $\rB_n$ while $R(J^\circ,T)$ has type $\rBC_n$.
\end{rem}
\begin{proof}
(a) From \cite[Proposition 2.2]{Lus3} we know that every $R(G_\phi^\circ,T)$ is a
root system. However, this result does not apply to our current $J^\circ$, as
$(M,v,q\epsilon)$ need not be a cuspidal quasi-support for a group with
neutral component $J^\circ$.

We will check that $R(J^\circ,T)$ satisfies the axioms of a root system. We fix a 
$\rN_J (T)$-invariant inner product on $X^* (T) \otimes_\Z \R$ (which exists because 
$\rN_J (T) / \rZ_J (T)$ is finite). For $\alpha \in R(J^\circ,T)$ we define 
$\alpha^\vee \in X_* (T) \otimes_\Z \R$ as the unique element which is orthogonal to 
\[
\{ x \in X^* (T) \otimes_\Z \R : \inp{x}{\alpha} = 0\}
\] 
and satisfies $\inp{\alpha^\vee}{\alpha} = 2$. Every single $\alpha \in R(J^\circ,T)$ 
appears in $R(G_\phi^\circ,T)$ for a suitable choice of $\phi$ (see the construction 
of $\phi_t$ below), which entails that $\alpha^\vee \in X_* (T)$. 

For arbitrary $\alpha,\beta \in R(J^\circ,T)$ we have to show that
\begin{enumerate}
\item $\inp{\alpha^\vee}{\beta} \in \Z$;
\item $s_\alpha (\beta) \in R(J^\circ,T)$, where $s_\alpha \colon X^* (T) \to
X^* (T)$ is the reflection associated to $\alpha$ and $\alpha^\vee$.
\end{enumerate}
Assume first that $\alpha$ and $\beta$ are linearly independent in $X^* (T)$.
The element $\phi (\Fr_F ) \! \in \mc L^\vee \rtimes \mb W_F$ centralizes $T$
and normalizes $J^\circ$, so it stabilizes each of the root subspaces
$\mf g_\alpha \subset \Lie (J^\circ)$. Let $\lambda_\alpha$ (respectively
$\lambda_\beta$) be an eigenvalue of Ad$(\phi (\Fr_F)) |_{\mf g_\alpha}$
(respectively Ad$(\phi (\Fr_F)) |_{\mf g_\beta}$). Since $\alpha$ and $\beta$
are linearly independent, we can find a $t \in T$ with $\alpha (t^{-1}) =
\lambda_\alpha$ and $\beta (t^{-1}) = \lambda_\beta$. Define
$(\phi_t |_{\mb W_F},v,q \epsilon) \in \mf s^\vee_{\mc L}$ by
\begin{equation} \label{eqn:phi_t}
\phi_t |_{\mb I_F} = \phi |_{\mb I_F} \quad \text{and} \quad
\phi_t (\Fr_F) = \phi (\Fr_F) (\text{image of } t \text{ in } \mc G^\vee_\der) .
\end{equation}
Clearly $\alpha, \beta \in R (G^\circ_{\phi_t}, T)$. Since this is a root system,
(i) and (ii) hold for $\alpha$ and $\beta$ inside $R(G^\circ_{\phi_t}, T)$.
Then they are also valid in the larger set $R(J^\circ,T)$.

Next we consider linearly dependent $\alpha,\beta$. Then $s_\alpha (\beta) =
-\beta$, so (ii) is automatically fulfilled.

Suppose that there exists a $\gamma \in R(J^\circ,T) \setminus \Q \alpha$ which
is not orthogonal to $\alpha$. As before, we can find $\phi_2, \phi_3$ such that
$\alpha, \gamma \in R (G^\circ_{\phi_2},T)$ and
$\beta, \gamma \in R (G^\circ_{\phi_3},T)$.
Hence both $\{ \alpha, \gamma \}$ and $\{\beta, \gamma \}$ generate rank two irreducible
root systems in $X^* (T)$, and these root systems have the same $\Q$-span. From
the classification of rank two root systems we see that $\Q \alpha \cap R(J^\circ,T)$
is either $\{\pm \tilde \alpha\}$ or $\{ \pm \tilde \alpha, \pm 2 \tilde \alpha \}$
for a suitable $\tilde \alpha$. In particular (i) holds, because
\[
\inp{\alpha^\vee}{\beta} \in \pm \{1,2,4\} \subset \Z .
\]
Finally we suppose that $\Q \alpha \cap R(J^\circ,T)$ is orthogonal to
$R(J^\circ,T) \setminus \Q \alpha$. As above, we may pick $\phi$ such that
$\alpha \in R (G_\phi^\circ, T)$. By assumption $\beta = c \alpha$ for some
$c \in \Q^\times$. Pick $\phi_t$ so that $\beta \in R (G_{\phi_t}^\circ, T)$. As
\[
\inp{\beta^\vee}{\beta} = 2 = \inp{\alpha^\vee}{\alpha} ,
\]
we have $c \alpha = \beta \in X^* (T)$ and $c^{-1} \alpha^\vee = \beta^\vee \in X_* (T)$.
It follows that $c \in \pm \{1/2,1,2\}$ and $\inp{\alpha^\vee}{\beta} \in \pm\{1,2,4\}$.\\
(b) Let $\Delta$ be a basis of the reduced root system $R(J^\circ,T)_\red$ --
which is well-defined by part (a). Let $\lambda_\alpha \in \C \; (\alpha \in \Delta)$
be an eigenvalue of Ad$(\phi (\Fr_F))$ on $\mf g_\alpha$. Since $\Delta$ is linearly
independent, we can find $t_1 \in T$ with $\alpha (t_1^{-1}) = \lambda_\alpha$ for
all $\alpha \in \Delta$.
We put $\phi_1:=\phi_{t_1}$, where $\phi_{t_1}$ is defined by (\ref{eqn:phi_t}).
Then $\Delta$ is contained in the root system $R(G^\circ_{\phi_1},T)$. The Weyl
group of $(J^\circ,T)$ is generated by the reflections $s_\alpha$ with $\alpha \in
\Delta$, so it equals the Weyl group of $(G^\circ_{\phi_1},T)$. In particular it
stabilizes $R(G^\circ_{\phi_1},T)$. Every element of $R(J^\circ,T)_\red$
is in the Weyl group orbit of some $\alpha \in \Delta$, so $R(G^\circ_{\phi_1},T)$
contains $R(J^\circ,T)_\red$. \\
(c) Such a $t$ commutes with $\mc G_\sc^\vee$ and with $J$, so its image
under the adjoint representation of $J^\circ$ is trivial.
\end{proof}

By \cite[Theorem 9.2]{Lus2} and Proposition \ref{prop:5.9}, 
\begin{equation}\label{eq:3.6}
\rN_{G_{\phi_1}^\circ} (T) / \rZ_{G_{\phi_1}^\circ} (T) = W(R(G_{\phi_1}^\circ,T)) 
= W( R (J^\circ,T)) .
\end{equation}
The group $\rN_J (T)$ acts naturally on $R(J^\circ,T)$. Let $\rN_{J'}(T)$ be the preimage
of $W(R(J^\circ,T))$ in $\rN_{J^\circ}(T)$, by \eqref{eq:3.6} it surjects onto
$W(R(J^\circ,T))$. We write
\begin{equation} \label{eqn:Wscirc}
W_{\mf s^\vee}^\circ := W(R(J^\circ,T)) = \rN_{J'} (T) / \rZ_{J^\circ} (T) .
\end{equation}
Since $\mc L_c^\vee = \rZ_{\mc G^\vee_\sc}(T)$ and 
$J^\circ = \rZ_{\mc G^\vee_\sc}(\phi |_{\mb I_F})^\circ$
\[
W_{\mf s^\vee}^\circ = \rN_{J^\circ} (T) / \rZ_{\mc L_c^\vee}(\phi (\mb I_F))^\circ =
\rN_{J^\circ} (\mc L_c^\vee) / (J^\circ \cap \mc L_c^\vee) .
\]
Any element of $G_{\phi_1}^\circ$ which normalizes $T = T^{\mb W_F}$
will also normalize $\mc L^\vee \rtimes \mb W_F = \rZ_{\mc G^\vee \rtimes \mb W_F} (T)$
and $M = \rZ_{G_{\phi_1}} (T) = \rZ_{G_\phi}(T)$, while by \cite[Lemma 2.1]{AMS2} 
it stabilizes $\cC_v^M$ and $q\cE$. The group
\begin{equation}\label{eq:3.3}
W_{\mf s^\vee} \subset \rN_{\mc G^\vee}(\cL^\vee \rtimes \mb W_F) / \cL^\vee =
\rN_{\mc G^\vee}(T) / \cL^\vee
\end{equation}
from \eqref{eqn:Ws} stabilizes the $\cL^\vee$-conjugacy classes of
$X_\nr (\cL^\vee) (\phi |_{\mb W_F},v, q\epsilon)$ and of 
\begin{equation}\label{eq:3.11}
M = G_{\phi_1} \cap \cL_c^\vee = G_\phi \cap \cL_c^\vee. 
\end{equation} 
Further, the group $Z_{\mc G^\vee}(\phi |_{\mb I_F})$ automatically normalizes 
$J = \rZ^1_{\mc G^\vee_\sc} (\phi |_{\mb I_F})$. Hence we can express $W_{\mf s^\vee}$ as
\begin{equation}\label{eq:3.7}
W_{\mf s^\vee} \cong \big( \rN_{\mc G^\vee}(T) \cap \rZ_{\mc G^\vee}(\phi |_{\mb I_F}) \big)
\big/ \rZ_{\cL^\vee}( \phi |_{\mb I_F}) .
\end{equation}
As $\cL^\vee = \rZ_{\mc G^\vee}(T)$ and $J / \rZ (\mc G_\sc^1) = \rZ_{\mc G_\ad}(\phi |_{\mb I_F})$,
we deduce from \eqref{eq:3.7} that there is a canonical isomorphism
\begin{equation}\label{eq:3.10}
\rN_J (T) / \rZ_J (T) \to W_{\mf s^\vee} .
\end{equation}
In particular $W_{\mf s^\vee}$ acts on $R(J^\circ,T)$ and naturally contains $W_{\mf s^\vee}^\circ$. 
We choose a $\phi_1$ as in Proposition \ref{prop:5.9}, which will play the role
of a basepoint on $\mf s^\vee_{\mc L}$. Then $W_{\mf s^\vee}^\circ = W(R(G^\circ_{\phi_1},T))$ 
fixes $(\phi_1 |_{\mb W_F},v,q\epsilon) \in \mf s^\vee_\cL$, 
but $W_{\mf s^\vee}$ need not fix $\phi_1 |_{\mb W_F}$. 

Clearly the set
\begin{equation}\label{eq:3.1}
X_\nr ({}^L \cL )_{\phi_1} := \{ z \in X_\nr ({}^L \cL) : z \phi_1 \equiv_{\cL^\vee} \phi_1 \}
\end{equation}
only depends on $\mf s_\cL^\vee$, not on $\phi_1$. Moreover it is finite, for it consists 
of elements coming from the finite group $\cL^\vee_\der \cap \rZ(\cL^\vee)$. Writing
\[
T_{\mf s^\vee} = X_\nr ({}^L \cL) / X_\nr ({}^L \cL)_{\phi_1} ,
\]
we obtain a bijection
\begin{equation}\label{eq:2.3}
T_{\mf s^\vee} \to \mf s_\cL^\vee : z \mapsto [z \phi_1 |_{\mb W_F}, v,q\epsilon] .
\end{equation}
Via this bijection we can retract the action of $W_{\mf s^\vee}$ on $\mf s_\cL^\vee$
to $T_{\mf s^\vee}$. Then $W_{\mf s^\vee}^\circ$ fixes $1 \in T_{\mf s^\vee}$.
If $\phi_0 = t_0 \phi_1$ is another basepoint, like $\phi_1$, then also $W(R(G^\circ_{\phi_0},T))
\cong W_{\mf s^\vee}^\circ$, so $t_0 \in (T )^{W_{\mf s^\vee}^\circ}$.
Consequently the action of $W_{\mf s^\vee}^\circ$ on $T_{\mf s^\vee}$ is independent of 
the choice of $\phi_1$. On the other hand, the action of $W_{\mf s^\vee}$ on $T_{\mf s^\vee}$
may very well depend on the choice of the basepoint $\phi_1$.

Analogous to \eqref{eq:3.1}, we consider the finite group
\begin{align*}
T_{\phi_1} & := \{ t \in T : t \phi_1 \equiv_{\cL^\vee} \phi_1 \} \\
& \; = T \cap \{ l \phi_1 (\Fr_F) l^{-1} \phi_1 (\Fr_F)^{-1} : 
l \in \rZ_{\cL_c^\vee} (\phi_1 (\mb I_F), v, q\epsilon) \} .
\end{align*}
From Lemma \ref{lem:5.7} we get a natural, finite covering of tori
\begin{equation}\label{eq:3.2}
T / T_{\phi_1} \times X_\nr ( {}^L \mc G) \to T_{\mf s^\vee} ,
\end{equation}
which is injective on $T / T_{\phi_1}$. In general the elements of $R(J^\circ,T)$ do not descend 
to characters of $X_\nr ({}^L \mc L)$, and even if they do, they need not descend further to
characters of $T_{\mf s^\vee}$. The former problem arises in the setting of Remark 
\ref{rem:5.25}, and the latter problem already occurs for the
Levi subgroup $\GL_2 (F) \times \GL_2 (F)$ of $\GL_4 (F)$.

To set things up properly, extend $\phi_1 |_{\mb W_F}$ to
\[
(\phi_1 |_{\mb W_F},u_{\phi_1},\rho) \in \Phi_e ({}^L \mc G) \quad \text{with} \quad 
q \Psi_{G_{\phi_1}} (u_{\phi_1},\rho) = [M,v,q\epsilon]_{G_{\phi_1}}.
\]
We consider $\phi_1 (\Fr_F)$ as a semisimple automorphism of $J$. By \cite[Theorem 7.5]{Ste} 
$\phi_1 (\Fr_F)$ stabilizes a Borel subgroup $B_J$ of $J^\circ$, and a maximal torus $T_J$ 
thereof. Then $B_J^{\phi_1 (\Fr_F),\circ}$ is a Borel subgroup of $G^\circ_{\phi_1} = 
J^{\phi_1 (\Fr_F), \circ}$, containing $T_J^{\phi_1 (\Fr_F),\circ}$ 
as maximal torus. By conjugation in $G^\circ_{\phi_1}$ 
we may assume that $u_{\phi_1} \in B_J^{\phi_1 (\Fr_F),\circ}$, and then
\[
M \supset T_J^{\phi_1 (\Fr_F)} \supset T = \rZ(M)^\circ .
\]
We recall that $T$ is the centre of the Levi subgroup $M^\circ$ of $G_\phi^\circ$. Hence
the restriction map
\begin{equation}\label{eq:3.8}
R \big( G_{\phi_1}^\circ, T_J^{\phi_1 (\Fr_F),\circ} \big) \cup \{0\} \to
R(G_{\phi_1}^\circ, T) \cup \{0\} 
\end{equation}
has the property that the full preimage of any $\alpha \in R(G_{\phi_1}^\circ, T)$ is contained
in a single irreducible component of $R \big( G_{\phi_1}^\circ, T_J^{\phi_1 (\Fr_F),\circ} \big)$.
As $\phi_1 (\Fr_F)$ stabilizes $(B_J,T_J)$, the restriction map
\begin{equation}\label{eq:3.9}
R(J^\circ,T_J) \to R \big( G_{\phi_1}^\circ, T_J^{\phi_1 (\Fr_F),\circ} \big)
\end{equation}
induces a bijection between the $\phi_1 (\Fr_F)$-orbits of irreducible components of\\ 
$R(J^\circ,T_J)$ and the set of irreducible components of 
$R \big( G_{\phi_1}^\circ, T_J^{\phi_1 (\Fr_F),\circ} \big)$.
From \eqref{eq:3.8} and \eqref{eq:3.9} we see that the preimage in $R(J^\circ,T_J)$ of any
$\alpha \in R(G_{\phi_1}^\circ, T)$ is contained in single $\phi_1 (\Fr_F)$-orbit of irreducible 
components of $R(J^\circ,T_J)$. That paves the way for the following definition.

\begin{defn}\label{def:3}
For each $\alpha \in R(J^\circ,T)_\red$, we define $m_\alpha \in \Z_{>0}$ by the following
requirements:
\begin{itemize}
\item Suppose that the preimage of $\alpha$ in $R (J^\circ, T_J)$ lies in
a single irreducible component of that root system. Then $m_\alpha$ is the smallest positive
integer such that $T_{\phi_1} \subset \ker (m_\alpha \alpha)$.
\item Suppose that the preimage of $\alpha$ in $R (J^\circ, T_J)$ meets $k > 1$ irreducible 
components of that root system, permuted transitively by the action of $\phi_1 (\Fr_F)$.
Then $m_\alpha = m_\alpha (\Fr_F)$ equals $k$ times the number $m_\alpha (\Fr_F^k)$ 
computed (as in the first bullet) with respect to the action of $\phi_1 (\Fr_F^k)$. Equivalently, 
$m_\alpha$ is $k$ times the analogous number obtained by replacing $\mb W_F$ with the Weil 
group of the degree $k$ unramified extension of $F$.
\end{itemize}
\end{defn}

These conditions guarantee that $m_\alpha \alpha$ descends to a character of $T / T_{\phi_1}$.
Moreover $m_\alpha$ is the minimal such natural number, unless maybe when $\alpha \in 2 X^* (T)$ 
and $k$ even. Extend $m_\alpha \alpha$ to a character of $T/T_{\phi_1} \times X_\nr ({}^L \cG)$, 
trivial on the second factor. In view of Proposition \ref{prop:5.9}.c, $m_\alpha \alpha$ is trivial 
on the kernel of \eqref{eq:3.2}, and hence descends naturally to a character of $T_{\mf s^\vee}$.
We define
\[
R_{\mf s^\vee} = \{ m_\alpha \alpha : \alpha \in R(J^\circ,T)_\red \}
\subset X^* (T_{\mf s^\vee}) .
\]
Recall that in the proof of Proposition \ref{prop:5.9} we fixed an inner product on 
$X^* (T) \otimes_\Z \R$, invariant under $\rN_J (T) / \rZ_J (T) \cong W_{\mf s^\vee}$.

\begin{lem}\label{lem:5.10}
\enuma{
\item $R_{\mf s^\vee}$ is a reduced root system, and it is stable under the
action of $W_{\mf s^\vee}$ on $X^* (T_{\mf s^\vee})$.
\item For $\alpha$ and $\beta$ in the same irreducible component of $R(J^\circ,T)_\red$,
$m_\alpha = m_\beta$ or $m_\alpha = \norm{\alpha}^{-2} \norm{\beta}^2 m_\beta$.
\item For each $\alpha \in R(J^\circ,T)_\red$, $\alpha^\vee / m_\alpha \in X_* (T) \otimes_\Z \Q$
defines a cocharacter of $T_{\mf s^\vee}$.
}
\end{lem}
\begin{proof}
(a) The realization of $W_{\mf s^\vee}$ in \eqref{eq:3.7} makes that it normalizes $T_{\phi_1}$. 
From $T_J \supset T$ and \eqref{eq:3.10} we see that $W_{\mf s^\vee}$ can be represented by 
elements of $\rN_J (T)$ that normalize $T_J$. Then $W_{\mf s^\vee}$ stabilizes all the data 
that go into the definition of $m_\alpha$, so the map $\alpha \mapsto m_\alpha$ is 
constant on $W_{\mf s^\vee}$-orbits. Consequently $R_{\mf s^\vee}$ is $W_{\mf s^\vee}$-stable.
In particular $R_{\mf s^\vee}$ is stable under all the reflections
$s_{m_\beta \beta} = s_\beta$ with $m_\beta \beta \in R_{\mf s^\vee}$. Hence 
$R_{\mf s^\vee}$ is a (possibly non-integral) root system with Weyl group 
\[
W(R_{\mf s^\vee}) = W (R(J^\circ,T)) = W_{\mf s^\vee}^\circ .
\]
By construction $R_{\mf s^\vee}$ is reduced, it remains to see that it is integral.
The lattice $\Z R_{\mf s^\vee}$ is stable under $W (R_{\mf s^\vee})$, so we can form the
semidirect product $W_a = W(R_{\mf s^\vee}) \ltimes \Z R_{\mf s^\vee}$. We let any
$x \in \Z R_{\mf s^\vee}$ act on $X^* (T) \otimes_\Z \R$ by the translation $t_x$.
Then $W_a$ is generated by the affine reflections $s_\alpha t_{n \alpha}$ with 
$\alpha \in R(J^\circ,T)_\red$ and $n \in m_\alpha \Z$. If we consider $W_a$ as a group of
affine transformations of $\Z R_{\mf s^\vee} \otimes_\Z \R$, we are in the setting of
\cite[Proposition VI.2.5.8]{Bou}. It says that $W_a$ is the affine Weyl group of a reduced 
integral root system, namely $R_{\mf s^\vee}^\vee$. Hence $R_{\mf s^\vee}$ is integral as well.\\
(b) By definition, in $R(J^\circ,T)_\red$:
\begin{equation}\label{eq:3.4}
s_\alpha (m_\beta \beta) = m_\beta \beta - m_\beta \inp{\alpha^\vee}{\beta} \alpha .
\end{equation}
On the other hand, in the integral root system $R_{\mf s^\vee}$:
\begin{equation}\label{eq:3.5}
s_{m_\alpha \alpha} (m_\beta \beta) =
m_\beta \beta - \inp{(m_\alpha \alpha)^\vee}{m_\beta \beta} m_\alpha \alpha
\end{equation}
Comparing \eqref{eq:3.4} and \eqref{eq:3.5}, we see that $m_\beta
\inp{\alpha^\vee}{\beta} \in m_\alpha \Z$. 

By the $W_{\mf s^\vee}$-invariance of $\alpha \mapsto m_\alpha$, it suffices to consider 
simple roots $\alpha, \beta$ in one irreducible component of $R(J^\circ,T)_\red$. If they
have the same length, then they are $W_{\mf s^\vee}$-associate and $m_\alpha = m_\beta$.
That leaves the case where their lengths differ, say $\alpha$ is longer. Replacing $\alpha$
and $\beta$ by $W (R(J^\circ,T))$-associate simple roots, we can achieve that they are not 
orthogonal and $\inp{\alpha^\vee}{\beta} = -1$. Then \eqref{eq:3.4} and \eqref{eq:3.5} 
entail that $m_\beta \geq m_\alpha$. More precisely, in that case
\[
\inp{\beta^\vee}{\alpha} = - \norm{\alpha}^2 \norm{\beta}^{-2} \in \{-1,-2,-3\}
\]
and $m_\alpha \in m_\beta \inp{\beta^\vee}{\alpha}^{-1} \Z$, so
$m_\alpha = m_\beta$ or $m_\alpha = \norm{\alpha}^{-2} \norm{\beta}^2 m_\beta$.\\
(c) In view of the covering map \eqref{eq:3.2}, it suffices to show that 
$\alpha^\vee / m_\alpha$ defines a cocharacter of $T / T_{\phi_1}$. From the finite
covering $T \to T / T_{\phi_1}$ we obtain inclusions
\[
\begin{array}{lllll}
X_* (T) & \subset & X_* (T/T_{\phi_1}) & \subset & X_* (T) \otimes_\Z \Q , \\
X^* (T / T_{\phi_1}) & \subset & X^* (T) & \subset & X^* (T / T_{\phi_1}) \otimes_\Z \Q . 
\end{array}
\]
The reflection $s_\alpha$ acts on $X_* (T) \otimes_\Z \Q$ by
\begin{equation}\label{eq:3.13}
x \mapsto s_\alpha (x) = x - \inp{x}{\alpha} \alpha^\vee = 
x - \inp{x}{m_\alpha \alpha} \alpha^\vee / m_\alpha .
\end{equation}
Since $s_\alpha \in W_{\mf s^\vee}$ normalizes $T_{\phi_1}$, the action \eqref{eq:3.13}
stabilizes $X_* (T / T_{\phi_1})$.

Suppose that $m_\alpha$ is the smallest positive integer such that $m_\alpha \alpha$ 
descends to a character of $T / T_{\phi_1}$. Then $m_\alpha \alpha$ is indivisible in
$X_* (T / T_{\phi_1})$, so there exists an $x \in X^* (T/T_{\phi_1})$ such that
$\inp{x}{m_\alpha \alpha} = 1$. For that $x$, \eqref{eq:3.13} shows that 
\[
\alpha^\vee / m_\alpha = x - s_\alpha (x) \in X_* (T / T_{\phi_1}) .
\]
When this characterization of $m_\alpha$ does not hold, then $\alpha / 2 \in X^* (T)$
and $m_\alpha \alpha / 2$ is the smallest multiple of $\alpha$ that 
descends to a character of $T / T_{\phi_1}$ (as remarked after Definition \ref{def:3}). 
As $\alpha \in 2 X^* (T)$, the irreducible component $R_1$ of $R(J^\circ,T)_\red$ that contains
$\alpha$ has type $C_n$, and it is contained in a direct summand $\Z^n$ of $X^* (T)$. 
Then $X^* (T / T_{\phi_1})$ contains $m_\alpha \alpha' / 2$ for every long root $\alpha' \in R_1$,
so it contains $m_\alpha \beta$ for every short root $\beta \in R_1$. From this and part (b)
we deduce that $m_\beta = m_\alpha$ (provided that $C_n$ has rank $>1$ so that it has short 
roots). It follows that $\Q R_1 \cap X^* (T/T_{\phi_1}) = m_\alpha \Z^n$. In particular 
\[
X^* (T/T_{\phi_1}) = (\alpha^\vee)^\perp \oplus m_\alpha \alpha / 2 .
\]
Since the pairing between $X_* (T/T_{\phi_1})$ and $X^* (T/T_{\phi_1})$ is 
perfect, there exists $y \in X_* (T / T_{\phi_1}) \cap (\alpha^\vee)^{\perp \perp}$ with 
$\inp{y}{m_\alpha \alpha /2} = 1$. This $y$ is $\alpha^\vee / m_\alpha$.
\end{proof}

Lemma \ref{lem:5.10} implies that
\[
\mc R_{\mf s^\vee} := \big( R_{\mf s^\vee}, X^* (T_{\mf s^\vee}),
R_{\mf s^\vee}^\vee, X_* (T_{\mf s^\vee}) \big)
\]
is a root datum with an action of $W_{\mf s^\vee}$.

\begin{lem}\label{lem:5.16}
Let $\alpha \in R(J^\circ,T)_\red$ and $t \in T$.
\enuma{
\item If $(m_\alpha \alpha)(t) = 1$, then $\alpha \in R (G^\circ_{t \phi_1},T)$.
\item Suppose that $R(G^\circ_{t \phi_1},T)$ contains $\alpha$ or $2 \alpha$. Then 
$(m_\alpha \alpha)(t) = 1$ or $(m_\alpha \alpha)(t) = -1$ and
$(m_\alpha \alpha)^\vee \in 2 X_* (T _{\mf s^\vee})$.
}
\end{lem}
\begin{proof}
(a) Suppose that $m_\alpha \alpha$ is the smallest multiple of $\alpha$ that descends to
a character of $T / T_{\phi_1}$. From $\alpha (t)^{m_\alpha} = 1$ we see that 
$\alpha (t) \in \alpha (T_{\phi_1})$. In particular there exists a $t' \in t T_{\phi_1}$ with 
$\alpha (t') = 1$. By the definition of $T_{\phi_1}$, $t \phi_1$ and $t' \phi_1$ are
$\mc L^\vee$-conjugate. Hence $\alpha \in R (G^\circ_{t' \phi_1},T) = R (G^\circ_{t \phi_1},T)$.

When this characterization of $m_\alpha$ does not hold, we need a more involved argument.
Following Definition \ref{def:3}, we write $m_\alpha = k m_\alpha (\Fr_F^k)$. This means that
$k$ irreducible components of $R(J^\circ,T_J)$ are relevant for $\alpha$, and they are permuted 
transitively by $\phi_1 (\Fr_F)$. Now $\alpha$ is a root for $(G_{t \phi_1}^\circ,T)$ if and 
only if $\alpha$ is a root for $Z_J ( (t \phi_1 (\Fr_F))^k)^\circ$, the version of 
$G_{t \phi_1}^\circ$ with $\Fr_F^k$ instead of $\Fr_F$. More precisely, the root subspaces for 
$\alpha$ in these two groups are naturally in bijection. Since $T$ centralizes 
$\phi_1 (\Fr_F) \in \cL^\vee \rtimes \mb W_F$, 
\[
(t \phi_1 (\Fr_F))^k = t^k \phi_1 (\Fr_F^k) .
\]
The root subspace for $\alpha$ and $Z_J ( (t \phi_1 (\Fr_F))^k)^\circ$ depends only $t^k$
(regarding $\phi_1$ as fixed). By assumption 
\[
(m_\alpha (\Fr_F^k) \alpha) (t^k) = (k^{-1} m_\alpha \alpha)(t^k) = (m_\alpha \alpha)(t) = 1.
\]
This brings us back to a situation analogous to the first part of the proof, 
and we conclude as over there.\\
(b) The reflection $s_\alpha$ stabilizes $t \phi_1 \in \mf s_L^\vee \cong T_{\mf s^\vee}$. 
Thus $s_\alpha$ fixes $t$ considered as element of $T_{\mf s^\vee}$. As $\mc R_{\mf s^\vee}$ is
a root datum, we have reduced to the well-known setting of Weyl groups acting on complex tori
associated to root data. In that setting we conclude with \cite[Lemma 3.15]{Lus4}.
\end{proof}

We endow $\mc R_{\mf s^\vee}$ with the set of simple roots determined by the Borel subgroup 
$B_J \subset J^\circ$. We look for parameter functions $\lambda$ and $\lambda^*$ on 
$\mc R_{\mf s^\vee}$ which are compatible with specialization to the graded Hecke algebras from 
Paragraph \ref{par:LGHA}. Recall from \eqref{eq:4.2} that $\lambda^* (\alpha)$ is defined to be 
$\lambda (\alpha)$ unless $\alpha$ is a short root in a type $\rB$ root subsystem of $R_{\mf s^\vee}$.

\begin{prop}\label{prop:5.15}
\enuma{
\item There exist unique $W_{\mf s^\vee}$-invariant parameter functions
\[
\lambda \colon R_{\mf s^\vee} \to \Q_{>0}, \qquad
\lambda^* \colon\{ m_\alpha \alpha \in R_{\mf s^\vee} : (m_\alpha \alpha )^\vee \in
2 X_* (T_{\mf s^\vee}) \} \to \Q
\]
such that, for every $(\phi_b |_{\mb W_F},v,q\epsilon) \in \mf s_{\cL}^\vee$ with
$\phi_b$ bounded, the reduction via Theorems \ref{thm:4.5} and \ref{thm:4.7} gives the graded 
Hecke algebra $\mh H (\phi_b, v, q\epsilon, \vec{\mb r})$ from \eqref{eq:5.11}.
\item The basepoint $\phi_1$ of $\mf s^\vee_{\mc L}$ can be chosen so that 
$\lambda$ has image in $\Z_{>0}$ and $\lambda^*$ has image in $\Z_{\geq 0}$.
}
\end{prop}
\begin{proof}
(a) The aforementioned reduction produces graded Hecke algebras with the roots
$m_\alpha \alpha$, whereas in \eqref{eq:5.11} the root system is contained in
$R(J^\circ,T)$. We reconcile this by imposing $c(m_\alpha \alpha) = m_\alpha c(\alpha)$,
which is allowed because it preserves the braid relations in a graded Hecke algebra
(Proposition \ref{prop:0.1}).

For $\phi_b = \phi_1$, \eqref{eq:4.37} imposes the conditions
\begin{equation}\label{eq:5.17}
\lambda (m_\alpha \alpha) + \lambda^* (m_\alpha \alpha) =
m_\alpha c(\alpha) \qquad \alpha \in R (J^\circ,T)_\red ,
\end{equation}
where $c(\alpha) \in \Z_{>0}$ is computed as in Proposition \ref{prop:4.1},
with respect to $G^\circ_{\phi_1}$. Given $\phi_1 \big|_{\mb I_F},v$ and $q\epsilon$,
the value of $c(\alpha)$ depends only on the root subspaces for $\alpha$ and $2 \alpha$ in 
$G_{t \phi_1}$. The proof of Lemma \ref{lem:5.16}.a shows that these root subspaces depend
(up the isomorphism) only on $(m_\alpha \alpha)(t)$.

In view of Lemma \ref{lem:5.16}.b, we need to consider at most two values of $c(\alpha)$
for $\phi_b \in \mf s_{\mc L}^\vee$: one for $\phi_1$ and maybe another one, say $c^* (\alpha)$, 
for a $t \phi_1$ with $(m_\alpha \alpha)(t) = -1$. When $R(G^\circ_{t \phi_1},T)$ contains
$2\alpha$ but not $\alpha$, we must rescale $c^* (\alpha) = c (2 \alpha) / 2$ so that it
really refers to $\alpha$ like $c(\alpha)$.

When $(m_\alpha \alpha)^\vee \notin 2 X_* (T_{\mf s^\vee})$, Lemma \ref{lem:5.16} says that 
$R(G^\circ_{t \phi_1},T)$ contains $\alpha$ or $2 \alpha$ if and only if 
$(m_\alpha \alpha )(t) = 1$. 
By convention $\lambda^* (m_\alpha \alpha) = \lambda (m_\alpha \alpha)$,
and the only way to solve \eqref{eq:5.17} is setting
\begin{equation}\label{eq:5.23}
\lambda (m_\alpha \alpha ) = c(\alpha) m_\alpha / 2 \in \Q_{>0}.
\end{equation}
Next consider an $\alpha \in R(J^\circ,T)_\red$ with $(m_\alpha \alpha)^\vee \in
2 X_* (T_{\mf s^\vee})$. Then $s_{m_\alpha \alpha}$ fixes $t \phi_1$ if $(m_\alpha \alpha)(t) = -1$,
so we have to consider $c^* (\alpha) \in \Z_{\geq 0}$. If $\alpha$ and $2 \alpha$ do not belong to 
$R(G^\circ_{t \phi_1},T)$ for one such $t$, then the above argument shows that they do
not lie in $R(G^\circ_{t \phi_1},T)$ for any such $t$. In that case, in the twisted graded Hecke 
algebra $\mh H (t \phi_1, v, q \epsilon, \vec{\mb r})$ the element $N_{s_\alpha}$ satisfies a
braid relation with trivial parameter $c^* (\alpha) := 0$.

For any $t \in T$ with $(m_\alpha \alpha) (t) = -1$, \eqref{eq:4.37} imposes the new condition
\begin{equation}\label{eq:5.16}
\lambda (m_\alpha \alpha) - \lambda^* (m_\alpha \alpha) = m_\alpha c^* (\alpha) .
\end{equation}
Clearly \eqref{eq:5.17} and \eqref{eq:5.16} admit the unique solution
\begin{equation}\label{eq:5.22}
\lambda (m_\alpha \alpha) = (c(\alpha) + c^* (\alpha)) m_\alpha /2 ,\quad
\lambda^* (m_\alpha \alpha) = (c(\alpha) - c^* (\alpha)) m_\alpha /2 .
\end{equation}
We address the $W_{\mf s^\vee}$-invariance. Represent $\gamma \in W_{\mf s^\vee}$
in $\rN_J (T)$ as in \eqref{eq:3.10}. Then it acts on the entire setting by
conjugation, so $\lambda \circ \gamma$ and $\lambda^* \circ \gamma$ are parameter
functions which also fulfill the requirements with respect to reduction to graded Hecke
algebras. With the uniqueness of the solutions to the above equations, we find that
$\lambda \circ \gamma = \lambda$ and $\lambda^* \circ \gamma = \lambda^*$.\\
(b) If $c^* (\alpha) > c(\alpha)$, then we exchange them. This can be achieved with the
method from the proof of Proposition \ref{prop:5.9}.b: take a new basepoint $\phi_{t'}$ 
such that $(m_\alpha \alpha)(t') = -1$ while $t'$ lies in the kernel of every other simple 
roots of $R(J^\circ,T)$. This assures that $\lambda^*$ takes values in $\Q_{\geq 0}$.\\
\textbf{Case 1:} $(m_\alpha \alpha)^\vee \notin 2 X_* (T_{\mf s^\vee})$\\
When $2 \alpha \notin R(J^\circ,T)$, Proposition \ref{prop:4.1}.a ensures that
$c(\alpha)$ is even. When $2 \alpha \in R(J^\circ,T)$ and still $(m_\alpha \alpha)^\vee
\notin 2 X_* (T_{\mf s^\vee})$, Lemma \ref{lem:5.10} shows that the relevant irreducible 
components of $R(J^\circ,T)$ and $R_{\mf s^\vee}$ have type $\rBC_n$ and $\rC_n$, respectively. 
In particular $m_\alpha = 2 m_\beta$ for any other simple root in the same component of 
$R(J^\circ,T)$, and $m_\alpha$ is even. Hence \eqref{eq:5.23} is always an integer.\\
\textbf{Case 2:} $(m_\alpha \alpha)^\vee \in 2 X_* (T_{\mf s^\vee}), 2\alpha \notin
R (G^\circ_{\phi_1},T)$\\
In view of \eqref{eq:5.22}, we need to show that 
\begin{equation}\label{eq:5.3}
(c(\alpha) \pm c^* (\alpha)) m_\alpha \qquad \text{is even.}
\end{equation}
By Proposition \ref{prop:4.1}.a, $c(\alpha)$ is even. Select $t \in T$ with
$(m_\alpha \alpha)(t) = -1$. 
\begin{itemize}
\item If $\alpha$ lies in $R(G^\circ_{t\phi_1},T)$ but $2 \alpha$
does not, then $c^* (\alpha)$ is also even. 
\item If $\alpha, 2 \alpha \notin R(G^\circ_{t\phi_1},T)$ 
then we argued in the proof of part (a) that $c^* (\alpha) = 0$. 
\item Suppose $2 \alpha$ lies in $R(G^\circ_{t\phi_1},T)$. If $m_\alpha$ 
would be odd, we could arrange that $\alpha (t) = -1$. Then $(2 \alpha) (t) = 1$, 
so $2 \alpha$ would lie in both $R(G^\circ_{t\phi_1},T)$ and $R(G^\circ_{\phi_1},T)$. 
That contradicts our assumptions, so $m_\alpha$ is even. By Proposition \ref{prop:4.1} either 
$c^* (\alpha) = c_t (2\alpha) / 2$ or $c^* (\alpha) = c(\alpha)$ and it is always an integer. 
\end{itemize}
In all these three instances \eqref{eq:5.3} holds.\\
\textbf{Case 3:} $(m_\alpha \alpha)^\vee \in 2 X_* (T_{\mf s^\vee}), 2\alpha \in
R (G^\circ_{\phi_1},T)$\\
Again we need to verify \eqref{eq:5.3}, and we pick a $t \in T$ with
$(m_\alpha \alpha)(t) = -1$. By Proposition \ref{prop:4.1}.b, $c(\alpha)$ is odd.
\begin{itemize}
\item If $\alpha, 2 \alpha \in R (G^\circ_{t\phi_1},T)$, then $c^* (\alpha)$ is also odd.
\item Suppose $m_\alpha$ is even and not $\alpha, 2 \alpha \in R (G^\circ_{t\phi_1},T)$. 
If $\alpha, 2 \alpha \notin R (G^\circ_{t\phi_1},T)$, then we argued in the proof of part (a) 
that $c^* (\alpha) = 0$. Otherwise, by Proposition \ref{prop:4.1} either $c^* (\alpha) = 
c(\alpha)$ or $c^* (\alpha) = c_t (2\alpha) / 2$, and this is always an integer. 
\item Suppose $m_\alpha$ is odd and not $\alpha, 2 \alpha \in R (G^\circ_{t\phi_1},T)$.
Here we can arrange that $\alpha (t) = -1$, so that $(2 \alpha)(t) = 1$. Then the root
subspace $\mf g_{2 \alpha}$ is the same for $G^\circ_{\phi_1}$ as for $G^\circ_{t \phi_1}$, so
$2 \alpha \in  R (G^\circ_{t\phi_1},T) \not\ni \alpha$ and $c^* (\alpha)$ can be computed from 
$G^\circ_{\phi_1}$ alone. By Proposition \ref{prop:4.1}.b 
$c^* (\alpha) = c_t (2 \alpha) / 2 = 1$, which is odd. 
\end{itemize}
In these three instances, \eqref{eq:5.3} is indeed valid.
\end{proof}

\subsection{Affine Hecke algebras} \
\label{par:LAHA}

Recall that $W_{\mf s^\vee}$ acts naturally on the root system $R(J^\circ,T)$.
Let $R^+(J^\circ,T)$ be the positive system defined by the $\phi (\Fr_F)$-stable Borel 
subgroup $B_J$ of $J^\circ$. Any two such $B_J$ are $J^\circ$-conjugate, so the choice 
is inessential. Since $W_{\mf s^\vee}^\circ$ acts simply transitively on the collection 
of positive systems for $R(J^\circ,T)$, we obtain a semi-direct factorization
\begin{align*}
& W_{\mf s^\vee} = W_{\mf s^\vee}^\circ \rtimes \cR_{\mf s^\vee} , \\
& \cR_{\mf s^\vee} = \{ w \in W_{\mf s^\vee} : w R^+ (J^\circ,T) = R^+ (J^\circ,T) \}.
\end{align*}
To $\mf s^\vee$ we can associate the affine Hecke algebra
$\mc H (\mc R_{\mf s^\vee}, \lambda, \lambda^*, \vec{\mb z})$, where $\phi_1$ is as in
Proposition \ref{prop:5.15} and $\lambda$ and $\lambda^*$ satisfy \eqref{eq:5.17} and
\eqref{eq:5.16}. However, this algebra takes only the subgroup $W_{\mf s^\vee}^\circ$ of
$W_{\mf s^\vee}$ into account. To see
$W_{\mf s^\vee, \phi_1, v, q\epsilon}$, we can enlarge it to
\begin{align}\label{eq:5.5}
& \mc H (\mc R_{\mf s^\vee}, \lambda, \lambda^*, \vec{\mb z}) \rtimes \C [\cR_{\mf s^\vee,
\phi_1, v, q\epsilon},\natural_{\mf s^\vee, \phi_1, v, q\epsilon}] \; = \\
& \mc H (\mc R_{\mf s^\vee}, \lambda, \lambda^*, \vec{\mb z}) \rtimes
\End^+_{\mathcal P_{G_{\phi_1}} \Lie (G_{\phi_1})_{\RS}} \big( q\pi_* (\widetilde{q\cE}) \big) .
\end{align}
But $W_{\mf s^\vee}$ can also contain elements that do not fix $\phi_1$.
In fact, in some cases $W_{\mf s^\vee}$ even acts freely on $T_{\mf s^\vee}$.

\begin{prop}\label{prop:5.11}
Assume that the almost direct factorization \eqref{eq:0.1} of $J^\circ$ induces a
decomposition of $R(J^\circ,T)$ which is $W_{\mf s^\vee}$-stable.
\enuma{
\item The group $\mf R_{\mf s^\vee}$ acts naturally on $\mc H (\mc R_{\mf s^\vee},
\lambda, \lambda^*, \vec{\mb z})$, by algebra automorphisms.
\item This can be realized in a twisted affine Hecke algebra
\[
\mc H (\mc R_{\mf s^\vee}, \lambda, \lambda^*, \vec{\mb z}) \rtimes \C [\cR_{\mf s^\vee},
\natural_{\mf s^\vee}] =
\mc H (\mc R_{\mf s^\vee}, \lambda, \lambda^*, \vec{\mb z}) \rtimes
\End^+_{\mathcal P_J \Lie (J)_{\RS}} \big( q\pi_* (\widetilde{q\cE}) \big)
\]
in which \eqref{eq:5.5} is canonically embedded.
}
\end{prop}
\begin{proof}
(a) The action of $\cR_{\mf s^\vee}$ on $T_{\mf s^\vee}$ comes from \eqref{eq:2.3}.
This determines an action on $\mc O (T_{\mf s^\vee}) \cong \C [ X^* (T_{\mf s^\vee})]$.
Any $\gamma \in \cR_{\mf s^\vee}$ maps $\theta_x$ to an invertible element of
$\C [ X^* (T_{\mf s^\vee})]$. That is,
\[
\gamma \cdot \theta_x = \theta_{\gamma x} \lambda_{\gamma,x}
\quad \text{with} \quad \lambda_{\gamma,x} \in \C^\times .
\]
The linear part $x \mapsto \gamma x$ is an automorphism of $X^* (T_{\mf s^\vee})$,
and the translation part of $\gamma \colon T_{\mf s^\vee} \to T_{\mf s^\vee}$ is given
by $\lambda_{\gamma,x}^{-1} = x (\gamma (1))$. Since $W_{\mf s^\vee}^\circ$ is normal
in $W_{\mf s^\vee}$,
\[
(W_{\mf s^\vee}^\circ )_{\gamma (1)} = \big( \gamma W_{\mf s^\vee}^\circ \gamma^{-1}
\big)_1 = (W_{\mf s^\vee}^\circ )_{1} = W_{\mf s^\vee}^\circ .
\]
In other words, the translation part of $\gamma$ commutes with all the reflections
$s_\alpha \; (\alpha \in R_{\mf s^\vee})$.
According to \cite[Lemma 9.2]{AMS} there exists a canonical algebra isomorphism
\[
\psi_{\gamma,\phi_1,v,q\epsilon} \colon\C [W_{\mf s^\vee,\phi_1,v,q\epsilon},
\kappa_{\phi_1,v,q\epsilon}] \to \C [W_{\mf s^\vee,\gamma(\phi_1),v,q\epsilon},
\kappa_{\gamma{\phi_1},v,q\epsilon}] .
\]
Let us recall its construction. There is a $G_{\phi_1}$-equivariant local system
$q\pi_* (\widetilde{q \cE})$ on $(G_{\phi_1})_{\RS}$, an analogue of $K$ and $K^*$.
It satisfies
\begin{equation}\label{eq:5.6}
\C [W_{\mf s^\vee,\phi_1,v,q\epsilon}, \kappa_{\phi_1,v,q\epsilon}] \cong
\End_{\mc D (G_{\phi_1})_{\RS}} (q\pi_* (\widetilde{q\cE})) .
\end{equation}
Choosing a lift $n_\gamma \in \rN_{G_{\phi_1}} (M)$ of $\gamma$ and following the
proof of \cite[Lemma 5.4]{AMS}, we find an isomorphism
\begin{equation}\label{eq:5.7}
q b_\gamma \colon q\pi_* (\widetilde{q\cE}) \to
q\pi_* \big( \widetilde{\Ad (n_\gamma )^* q\cE} \big) .
\end{equation}
Then $\psi_{\gamma,\phi_1,v,q\epsilon}$ is conjugation with $q b_\gamma$.

In this context \cite[Lemma 5.4]{AMS} says that there are canonical elements
$q b_w \in \End_{\mc D (G_{\phi_1})_{\RS}} (q\pi_* (\widetilde{q\cE})) \;
(w \in W_{\mf s^\vee}^\circ )$ which via \eqref{eq:5.6} become a basis of
$\C [W_{\mf s^\vee}^\circ]$. Since $W_{\mf s^\vee}^\circ$ is normal in
$W_{\mf s^\vee} ,\; \psi_{\gamma,\phi_1,v,q\epsilon}$ stabilizes the set
$\{ qb_w \colon W_{\mf s^\vee}^\circ \}$. Moreover $\gamma \in \cR_{\mf s^\vee}$,
so $\psi_{\gamma,\phi_1,v,q\epsilon}$ permutes the set of simple reflections
in $W_{\mf s^\vee}^\circ$.

By Proposition \ref{prop:5.15} the parameter functions $\lambda$ and $\lambda^*$ are
$W_{\mf s^\vee}$-invariant. Hence the map $N_{s_\alpha} \mapsto N_{\gamma s_\alpha
\gamma^{-1}}$ extends uniquely to an automorphism of the Iwahori--Hecke algebra
$\cH (W_\cE^\circ,\vec{\mb z}^{2\lambda})$ which fixes $\vec{\mb z}$.

Now we have canonical group actions of $\cR_{\mf s^\vee}$ on the algebras
\[
\mc O (T_{\mf s^\vee} \times (\C^\times)^d ) = \C [X^* (X_\nr ({}^L \cL))]
\otimes \C [\vec{\mb z}, \vec{\mb z}^{-1}]
\]
and $\cH (W_\cE^\circ,\vec{\mb z}^{2\lambda})$, and as vector spaces
\[
\cH (\mc R_{\mf s^\vee},\lambda,\lambda^*,\vec{\mb z}) = \mc O ( T_{\mf s^\vee} 
\times (\C^\times)^d ) \otimes \cH (W_\cE^\circ,\vec{\mb z}^{2\lambda}) .
\]
The relation involving $\theta_x N_{s_\alpha} - N_{s_\alpha} \theta_{s_\alpha (x)}$
in Proposition \ref{prop:4.2} is also preserved by $\gamma$, because
$x(\gamma (1)) = s_\alpha (x) (\gamma (1))$. So $\cR_{\mf s^\vee}$ acts canonically
on $\cH (\mc R_{\mf s^\vee},\lambda,\lambda^*,\vec{\mb z})$ by algebra automorphisms.\\
(b) The same construction as in the proof of Proposition \ref{prop:4.2} yields
an algebra
\begin{equation}\label{eq:5.8}
\cH (\mc R_{\mf s^\vee},\lambda,\lambda^*,\vec{\mb z})
\rtimes \C [\cR_{\mf s^\vee}, \kappa_{\mf s^\vee}] ,
\end{equation}
in which the action of $\cR_{\mf s^\vee}$ on $\cH (\mc R_{\mf s^\vee},\lambda,
\lambda^*,\vec{\mb z})$ has become an inner automorphism. This works for any 2-cocycle
$\kappa_{\mf s^\vee}$. It only remains to pick it in a good way, such that
$\kappa_{\mf s^\vee} |_{(W_{\mf s^\vee, \phi,v,q\epsilon})^2}$ equals
$\kappa_{\mf s^\vee,\phi,v,q\epsilon}$. For this we, again, use the maps
$q b_\gamma$ from \eqref{eq:5.7}. The cuspidal local system Ad$(n_\gamma)^*
q \cE$ does not depend on the choice of $n_\gamma$, because $q\cE$ is
$M$-equivariant. Furthermore $q b_\gamma$ is unique up to scalars, so
\[
q b_\gamma \cdot q b_{\gamma'} = \lambda_{\gamma,\gamma'} q b_{\gamma \gamma'}
\text{ for a unique } \lambda_{\gamma,\gamma'} \in \C^\times .
\]
We define $\kappa_{\mf s^\vee}$ by $\kappa_{\mf s^\vee}(\gamma,\gamma') =
\lambda_{\gamma,\gamma'}$. This is a slight generalization of the construction
in Section \ref{sec:0} and in \cite[Lemma 5.4]{AMS}. As over there,
\[
\begin{array}{lll}
\End_{\mathcal P_J \Lie (J)_{\RS}} \big( q\pi_* (\widetilde{q\cE}) \big) &
\cong & \C [W_{\mf s^\vee},\kappa_{\mf s^\vee}] , \\
\End^+_{\mathcal P_J \Lie (J)_{\RS}} \big( q\pi_* (\widetilde{q\cE}) \big) &
\cong & \C [\cR_{\mf s^\vee},\kappa_{\mf s^\vee}] .
\end{array}
\]
As the $J$-equivariant sheaf $q\pi_* (\widetilde{q\cE})$ on Lie$(J)_{\RS}$
contains the $G_{\phi}$-equivariant sheaf $q\pi_* (\widetilde{q\cE})$ on
Lie$(G_\phi )_{\RS}$,
\[
\kappa_{\mf s^\vee} \colon(W_{\mf s^\vee} )^2 \to (W_{\mf s^\vee} /
W_{\mf s^\vee}^\circ)^2 = \cR_{\mf s^\vee}^2 \to \C^\times
\]
extends $\kappa_{\mf s^\vee,\phi,v,q\epsilon} \colon(W_{\mf s^\vee,\phi,v,q\epsilon} )^2
\to \C^\times$, for every $(\phi |_{\mb W_F},v,q\epsilon) \in \mf s_\cL^\vee$.
For $\phi = \phi_1$ this means that
\[
\mc H (\mc R_{\mf s^\vee}, \lambda, \lambda^*, \vec{\mb z}) \rtimes
\C [W_{\mf s^\vee, \phi_1, v, q\epsilon},\natural_{\mf s^\vee, \phi_1, v, q\epsilon}] .
\]
is canonically embedded in \eqref{eq:5.8}.
\end{proof}

The algebra from Proposition \ref{prop:5.11}.b is attached to $\mf s^\vee$
and the basepoint $\phi_1$ of $\mf s^\vee_\cL$. To remove the dependence on the
basepoint, we reinterpret $\cH (\mc R_{\mf s^\vee},\lambda,\lambda^*,\vec{\mb z})$.
Recall that $W_{\mf s^\vee}$ acts naturally on $\mf s^\vee_\cL$ (which is diffeomorphic to
$T_{\mf s^\vee}$). Every $\alpha \in R_{\mf s^\vee}$ is by definition a character of
$T_{\mf s^\vee}$ and by Proposition \ref{prop:5.9}.b it does not depend on the choice of
$\phi_1$, so it canonically determines a function on $\mf s^\vee_\cL$.
In the same way as in Proposition \ref{prop:4.2}, we can define an algebra structure on
\[
\mc O (\mf s^\vee_\cL) \otimes
\C[\vec{\mb z}, \vec{\mb z}^{-1}] \otimes \C [W_{\mf s^\vee}^\circ] .
\]
It becomes an algebra $\cH (\mf s^\vee_\cL,W_{\mf s^\vee}^\circ,\lambda,\lambda^*,\vec{\mb z})$
which is isomorphic to $\cH (\mc R_{\mf s^\vee},\lambda,\lambda^*,\vec{\mb z})$, but only via
the choice of a basepoint of $\mf s^\vee_\cL$. In Proposition \ref{prop:5.11}.a we showed
that $\mf R_{\mf s^\vee}$ acts naturally on $\mc H (\mf s^\vee_\cL, W_{\mf s^\vee}^\circ,
\lambda,\lambda^*,\vec{\mb z})$. Applying Proposition \ref{prop:5.11}.b, we obtain an algebra
\begin{equation}\label{eq:5.15}
\cH (\mf s^\vee_\cL,W_{\mf s^\vee}^\circ,\lambda,\lambda^*, \vec{\mb z}) \rtimes
\End^+_{\mathcal P_J \Lie (J)_{\RS}} \big( q\pi_* (\widetilde{q\cE}) \big)
\text{, where } J = Z^1_{\mc G^\vee_\sc}(\phi |_{\mb I_F}).
\end{equation}
Now we suppose that the almost direct factorization of $J^\circ$ induces a $W_{\mf s^\vee}$-stable
decomposition of $R(J^\circ,T)$ (and, equivalently, of $R_{\mf s^\vee}$).
We focus on two algebras obtained in this way:
\begin{itemize}
\item $\cH (\mf s^\vee,\mb z)$, the algebra \eqref{eq:5.15} when $J_1 = J^\circ_\der$, with
only one variable $\mb z$;
\item $\cH (\mf s^\vee,\vec{\mb z})$, the algebra \eqref{eq:5.15} when \eqref{eq:0.1} induces
the finest possible $W_{\mf s^\vee}$-stable decomposition of $R(J^\circ,T)$.
\end{itemize}

\begin{lem}\label{lem:5.12}
The algebras $\cH (\mf s^\vee,\mb z)$ and $\cH (\mf s^\vee,\vec{\mb z})$ depend only on
$\mf s^\vee$, up to canonical isomorphisms.
\end{lem}
\begin{proof}
The above construction shows that $\cH (\mf s^\vee,\mb z)$ and $\mc H (\mf s^\vee,\vec{\mb z})$
are uniquely determined by $(\mf s^\vee_\cL, B_J)$, where $B_J$ serves only to determine the
positive roots in $R(J^\circ,T)$. Up to $\mc G^\vee$-conjugation, this pair
is completely determined by $\fs^\vee$. 

The $\mc G_\sc^\vee$-normalizer of $\mf s^\vee_\cL$ is contained in $J$, and the pointwise 
stabilizer of $\mf s_\cL^\vee$ in $J$ is just $M$, see \eqref{eq:3.12} and \eqref{eq:3.11}. 
Given $\mf s^\vee_\cL$ and $M$, \cite[Lemma 2.1.a]{AMS2} shows that all possible choices for 
$B_J$ are conjugate by unique elements of $\rN_{J^\circ}(M^\circ) / M^\circ$. Thus all possible 
$(\mf {s'}^\vee_\cL,B'_J)$ underlying $\mf s^\vee$ are conjugate to $(\mf s^\vee_\cL,B_J)$ 
in a canonical way. Any element of $\mc G^\vee_\sc$ which realizes such a conjugation 
provides a canonical isomorphism between $\cH (\mf s^\vee,\mb z)$ (respectively 
$\cH (\mf s^\vee,\vec{\mb z})$) and its version based on $(\mf {s'}^\vee_\cL, B'_J)$.
\end{proof}

\begin{ex}
Suppose that $(\phi,\rho)$ is itself cuspidal, so $\cL^\vee = \cG^\vee$ and $q\epsilon = \rho$.
Then $J^\circ = M^\circ$, $v$ is distinguished in that group, $T = 1$ and
$R(J^\circ,T)$ is empty. Furthermore $W_{\fs^\vee} = 1$ because
$\rN_{\cG^\vee}(\cL^\vee \rtimes \mb W_F) / \cL^\vee = 1$. Consequently
\[
\cH (\mf s^\vee,\mb z) = \mc O (T_{\mf s^\vee}) \otimes \C [\mb z,\mb z^{-1}]
\quad \text{and} \quad \cH (\mf s^\vee,\vec{\mb z}) = \mc O (T_{\mf s^\vee}) \otimes
\C [\mb z_1,\mb z_1^{-1},\ldots ,\mb z_d, \mb z_d^{-1}] ,
\]
where $d$ is the number of simple factors of $J^\circ_\der$.
\end{ex}

For $(\phi,\rho)$ as in \eqref{eq:5.9}, let $\bar{M} (\phi,\rho,\vec{z})$ be the irreducible
$\cH (\mf s^\vee,\vec{\mb z})$-module obtained from
$M (\phi,\rho,\log \vec{z}) \in \Irr (\mh H (\phi_b,v,q\epsilon,\vec{\mb r}))$
via Theorems \ref{thm:4.5} and \ref{thm:4.7}. Up to $\mc G^\vee$-conjugation, every element
of $\Phi_\re (\mc G (F) )^{\mf s^\vee}$ is of the form described in \eqref{eq:5.9},
so this definition extends naturally to all possible $(\phi,\rho)$. Similarly we define
$\bar{E} (\phi,\rho,\vec{z})$ as the ``standard" $\cH (\mf s^\vee,\vec{\mb z})$-module obtained
from $E (\phi,\rho,\log \vec{z}) \in \Mod (\mh H (\phi_b,v,q\epsilon,\vec{\mb r}))$
via Theorems \ref{thm:4.5} and \ref{thm:4.7}.

We formulate the next result only for $\cH (\mf s^\vee,\vec{\mb z})$, but there is also a
version for $\cH (\mf s^\vee,\mb z)$. In view of \eqref{eq:4.35}, the latter can be obtained
by assuming that all $z_j$ are equal.

\begin{thm}\label{thm:5.13}
\enuma{
\item For every $\vec{z} \in \R_{>0}^d$ there exists a canonical bijection
\[
\Phi_\re (\mc G(F) )^{\mf s^\vee} \to \Irr_{\vec{z}} (\cH (\mf s^\vee,\vec{\mb z})) :
\quad (\phi,\rho) \mapsto \bar{M} (\phi,\rho,\vec{z}) .
\]
\item Both $\bar{M} (\phi,\rho,\vec{z})$ and $\bar{E} (\phi,\rho,\vec{z})$ admit the
central character $W_{\mf s^\vee}(\tilde \phi |_{\mb W_F}, v, q\epsilon) \in \Phi_\re (\cL(F) 
)^{\mf s_\cL^\vee} / W_{\mf s^\vee}$, where $\tilde \phi |_{\mb I_F} = \phi |_{\mb I_F}$ 
and $\tilde{\phi} (\Fr_F) = \phi (\Fr_F) \vec{\chi}_{\phi,v}(\vec{z})$ with $\chi_{\phi,v}$ 
as in \eqref{eq:5.4}. We may also take $\chi_{\phi,v}^{-1}$ instead of $\chi_{\phi,v}$.
\item Suppose that $\vec{z} \in \R_{\geq 1}^d$. Equivalent are:
\begin{itemize}
\item $\phi$ is bounded;
\item $\bar{E} (\phi,\rho,\vec{z})$ is tempered;
\item $\bar{M} (\phi,\rho,\vec{z})$ is tempered.
\end{itemize}
\item Suppose that $\vec{z} \in \R_{>1}^d$. Then $\phi$ is discrete if and only if
$\bar{M} (\phi,\rho,\vec{z})$ is essentially discrete series and the rank of
$R_{\mf s^\vee}$ equals $\dim_\C (T_{\mf s^\vee} / X_\nr ({}^L \mc G))$.

In this case $\phi (\Fr_F) \phi_b (\Fr_F)^{-1}$ comes from an element of
$\rZ(J^\circ) \times X_\nr ({}^L \mc G)$ via Lemma \ref{lem:5.7} and \eqref{eq:2.3}.
\item Suppose that $\zeta \in \rZ(\mc G^\vee \rtimes \mb I_F)_{\mb W_F}$ stabilizes
$\Phi_\re (\mc G (F))^{\mf s^\vee}$. Via \eqref{eq:2.3} $\zeta$ determines a unique
element $t_\zeta \in T_{\mf s^\vee}$. (For instance $\zeta \in X_\nr ({}^L \mc G)$,
in which case $t_\zeta = \zeta X_\nr ({}^L \cL)_{\mf s^\vee}$.) Then
\[
\bar M (\zeta \phi,\rho,\vec{z}) = t_\zeta \otimes \bar M (\phi,\rho,\vec{z})
\quad \text{and} \quad
\bar E (\zeta \phi,\rho,\vec{z}) = t_\zeta \otimes \bar E (\phi,\rho,\vec{z}) .
\]
\item Suppose that $\vec{z} \in \R_{>1}^d$ and that $\phi (\Fr_F) \phi_b (\Fr_F)^{-1}$
comes from an element of $\rZ(J^\circ) \times X_\nr ({}^L \mc G)$ via Lemma \ref{lem:5.7}
and \eqref{eq:2.3}. Then $\bar{E} (\phi,\rho,\vec{z}) = \bar{M} (\phi,\rho,\vec{z})$.
}
\end{thm}
\begin{proof}
(a) Let us fix the bounded part $\phi_b$ and
consider only $\phi$ in $X_\nr ({}^L \cL)_{\rs} \phi_b$. We need to construct a bijection
between such $(\phi,\rho)$ and the set of irreducible $\cH (\mf s^\vee,\vec{\mb z})$-modules
on which $\vec{\mb z}$ acts as $\vec{z}$ and with $\mc O (\mf s^\vee_\cL)$-weights in
\[
W_{\mf s^\vee} (X_\nr ({}^L \cL)_{\rs} \phi_b,v,q\epsilon) \subset \mf s^\vee_\cL.
\]
We want to apply Theorem \ref{thm:4.5}.a here, although $\cH (\mf s^\vee,\vec{\mb z})$
and $\mc H (\mc R_{\mf s^\vee}, \lambda, \lambda^*, \vec{\mb z})$
need not be of the form $\cH (G,M,q\cE)$. To see that this is allowed, pick a basepoint
$\phi_1$ as in Proposition \ref{prop:5.9}. Then $\cH (\mf s^\vee,\vec{\mb z})$ becomes
a twisted affine Hecke algebra associated to a root datum, parameters, a finite
group and a 2-cocycle. For such an algebra the proof of Theorem \ref{thm:4.5} works,
it does not matter that the parameters can be different and that $\cR_{\mf s^\vee}$
need not fix the basepoint of $T_{\mf s^\vee}$.

Consider the twisted affine Hecke algebra $\mc H (\mf s^\vee,\phi_b)$ with as data the torus
$\mf s^\vee_\cL$, roots $\{ \alpha \in R_{\mf s^\vee} : s_\alpha (\phi_b) = \phi_b \}$,
the finite group $W_{\mf s^\vee, \phi_b,v,q\epsilon}$, parameters $\lambda,\lambda^*$ as
in \eqref{eq:5.17} and \eqref{eq:5.16} and
2-cocycle $\natural_{q\cE}$. The upshot of Theorem \ref{thm:4.5}.a is a canonical bijection
between the above irreducible $\cH (\mf s^\vee,\vec{\mb z})$-modules and the irreducible
modules of $\cH (\mf s^\vee,\phi_b)$ with central character in
$(X_\nr ({}^L \cL)_{\rs} \phi_b,v,q\epsilon) \times \{\vec{z}\}$.

With respect to the new basepoint $\phi_b$, $\cH (\mf s^\vee,\phi_b)$ becomes isomorphic
to a twisted affine Hecke algebra of the form described in Proposition \ref{prop:4.2}.
Then we can apply Theorem \ref{thm:4.7} to it, which relates its modules to those
over a twisted graded Hecke algebra. Again it does not matter that the parameters of the
affine Hecke algebra can differ from those in Theorem \ref{thm:4.7}, this result applies
to all possible parameters. The parameters of the resulting graded Hecke algebra are given
by \eqref{eq:4.37} and \eqref{eq:4.38}. Comparing that with \eqref{eq:5.17},
\eqref{eq:5.16} and \eqref{eq:5.11}, we see that that graded Hecke algebra is none other
than $\mh H (\phi_b,v,q\epsilon,\vec{\mb r})$.

Thus Theorem \ref{thm:4.5}.a yields a bijection between the above set of
irreducible modules and the irreducible $\mh H (\phi_b,v,q\epsilon,\vec{\mb r})$-modules
with central character in $\Lie(X_\nr ({}^L \cL)_{\rs}) \times \{\log \vec{z}\}$.
By Theorem \ref{thm:5.8} this last collection is canonically in bijection with
\begin{equation}\label{eq:5.12}
{}^L \Psi^{-1}(\cL^\vee \rtimes \mb W_F, X_\nr ({}^L \cL )_\rs
\phi_b |_{\mb W_F}, v , q \epsilon) .
\end{equation}
The resulting bijection between \eqref{eq:5.12} and the subset of
$\Irr (\cH (\mf s^\vee,\vec{\mb z}))$ with the appropriate central character could
depend on the choice of an element in the $W_{\mf s^\vee}$-orbit of $\phi_b$.
Fortunately, the proof of Lemma \ref{lem:4.6} applies also in this setting, and it
entails that the bijection does not depend on such choices. Now we combine all these
bijections, for the various $\phi_b$. This gives a canonical bijection between
$\Phi_\re (\mc G^\vee )^{\mf s^\vee}$ and $\Irr_{\vec{z}} (\cH (\mf s^\vee,\vec{\mb z}))$.\\
(b) By Theorem \ref{thm:5.8}.f $E (\phi,\rho,\log \vec{z})$ admits the central character
\[
W_{\mf s^\vee,\phi_b,v,q\epsilon} \big( \sigma_0 \pm d \vec{\chi}_{\phi,v}(\log \vec{z}),
\log \vec{z} \big) ,
\] 
where $\sigma_0$ is given by \eqref{eq:5.13}. Applying Theorems \ref{thm:4.7} and \ref{thm:4.5} 
produces the representation $\bar{E} (\phi,\rho,\vec{z})$, with the central character that sends 
$\Fr_F$ to $\phi (\Fr_F) \vec{\chi}_{\phi,v}(\vec{z})^{\pm 1}$.
That is just $W_{\mf s^\vee}(\tilde \phi |_{\mb W_F}, v, q\epsilon)$. The same holds for the 
quotient $\bar{M} (\phi,\rho,\vec{z})$ of $\bar{E} (\phi,\rho,\vec{z})$.\\
(c) This follows from Theorem \ref{thm:5.8}.b, Theorem \ref{thm:4.7}.d and Proposition
\ref{prop:4.9}.a.\\
(d) Notice that by the very definition of $R_{\mf s^\vee}$, it has the same rank as
$R(J^\circ,T)$.

Suppose that $\phi$ is discrete. By Theorem \ref{thm:5.8}.c
$\bar{M} (\phi,\rho,\log \vec{z})$ is essentially discrete series as a module for
$\mh H (\phi_b ,v, q\epsilon, \log \vec{z})$, and the rank of $R(G^\circ_{\phi_b}, T)$
equals $\dim_\C (T)$. Now Theorem \ref{thm:4.7}.d and Proposition \ref{prop:4.9}.c say that
$\bar{M} (\phi,\rho,\vec{z})$ is essentially discrete series. The root system
$R(J^\circ, T)$ contains $R(G^\circ_{\phi_b}, T)$, so its rank is at least $\dim_\C (T)$ --
and hence precisely that, for it obviously cannot be strictly larger.
By Lemma \ref{lem:5.7} $T$ is a finite cover of
$T_{\mf s^\vee} / X_\nr ({}^L \mc G)$, so both these tori have the same dimension.

Conversely, suppose that $\bar{M} (\phi,\rho,\vec{z})$ is essentially discrete series
and that the rank of $R(J^\circ,T)$ equals $\dim_\C (T_{\mf s^\vee} /
X_\nr ({}^L \mc G))$. By Proposition \ref{prop:4.9}.c the root system
$R(G^\circ_{\phi_b}, T)$ has the same rank, which we already saw equals
$\dim_\C (T)$. In combination with Theorem \ref{thm:4.7} we also obtain that
the $\mh H (\phi_b ,v, q\epsilon, \log \vec{z})$-module
$\bar{M} (\phi,\rho,\log \vec{z})$ is essentially discrete series. Now Theorem
\ref{thm:5.8}.c tells us that $\phi$ is discrete.

By Theorem \ref{thm:4.10}.d $\phi (\Fr_F) \phi_b (\Fr_F)^{-1}$ lies in
$\rZ(G^\circ)$ for a complex reductive group $G^\circ$ with maximal torus
$T_{\fs^\vee}$ and Weyl group $W^\circ_{\fs^\vee}$. That is the Weyl group of
$(J^\circ,T)$, so via Lemma \ref{lem:5.7} $\phi (\Fr_F) \phi_b (\Fr_F)^{-1}$
must come from an element of $G_{\phi_1}^\circ \times X_\nr ({}^L \mc G)$ which
is centralized by $J^\circ$.\\
(e) As ${}^L \Psi (\zeta \phi, \rho) = \zeta {}^L \Psi (\zeta,\rho) \in
\mf s_{\cL}^\vee$, $\zeta$ determines a unique element of $T_{\mf s^\vee}$.
It is invariant under $\mc G^\vee$ and $\mc G^\vee_\sc$, because $\zeta$ comes
from $\rZ(\mc G^\vee)$. Now the claim follows from Theorem \ref{thm:5.8}.d in the
same way as Theorem \ref{thm:4.10}.e was derived from Proposition \ref{prop:0.5}.d.\\
(f) Reasoning as in the last lines of the proof of part (d), we see that \\
$\phi (\Fr_F) \phi_b (\Fr_F)^{-1} \in \rZ(G^\circ)$. Apply Theorem \ref{thm:4.10}.f.
\end{proof}

Comparing Theorem \ref{thm:5.13}.b with \cite[Definition 7.7]{AMS} we see that, when 
$\vec{z} = q_F^{\pm 1/2}$, the central character of $\bar M(\phi, \rho, q_F^{\pm 1/2})$ 
equals the cuspidal support of $(\phi,\rho)$. Part (e) says
that Theorem \ref{thm:5.13} is equivariant with respect to twists by $X_\nr ({}^L \mc G)$,
that is, equivariant with respect to twisting by unramified characters of $\mc G(F)$.

The bijection obtained in part (a) is compatible with parabolic induction in the same
sense as Corollary \ref{cor:4.15}. For reference, we formulate this precisely.
We use the notations as in \eqref{eq:5.14} and after that. Recall from pages
\pageref{cor:4.15} and \pageref{prop:0.6} that
\[
\epsilon_{u_\phi,j} \big( \phi (\Fr_F) \phi_b (\Fr_F)^{-1}, \vec{z} \big) =
\epsilon_{\log (u_\phi),j} \Big(  \textup{d}\vec{\phi} \matje{\vec{r}}{0}{0}{-\vec{r}} +
\log (\phi (\Fr_F)^{-1} \phi_b (\Fr_F)) , \vec{r} \Big)
\]
is a function which detects parameters for which parabolic induction could behave
undesirably.

\begin{lem}\label{lem:5.14}
Let $Q = \mc Q (F)$ be a Levi subgroup of $\mc G (F)$ and assume that \\
$\epsilon_{u_\phi,j} \big( \phi (\Fr_F) \phi_b (\Fr_F)^{-1}, \vec{z} \big) \neq 0$
for each $j = 1,\ldots,d$. 
\enuma{
\item There is a natural isomorphism of $\mc H (\fs^\vee,\vec{\mb z})$-modules
\[
\mc H (\fs^\vee,\vec{\mb z}) \underset{\mc H (\fs^\vee_Q,\vec{\mb z})}{\otimes}
\bar{E}^Q (\phi,\rho^Q,\vec{z}) \cong
\bigoplus\nolimits_\rho \Hom_{\mc S_\phi^Q} (\rho^Q ,\rho) \otimes
\bar{E} (\phi,\rho,\vec{z}) ,
\]
where the sum runs over all $\rho \in \Irr \big( \mc S_\phi \big)$
with ${}^L \Psi^Q (\phi,\rho^Q) = {}^L \Psi (\phi,\rho)$.
\item The multiplicity of $\bar{M} (\phi,\rho,\vec{z})$ in
$\mc H (\fs^\vee,\vec{\mb z}) \underset{\mc H (\fs^\vee_Q,\vec{\mb z})}{\otimes}
\bar{E}^Q (\phi,\rho^Q,\vec{z})$ is
$[\rho^Q : \rho]_{\mc S_\phi^Q}$. It already appears that many times as a
quotient of $\mc H (\fs^\vee,\vec{\mb z}) \underset{\mc H (\fs^\vee_Q,
\vec{\mb z})}{\otimes} \bar{M}^Q (\phi,\rho^Q,\vec{z})$.
}
\end{lem}
\begin{proof}
As observed after \eqref{eq:5.14}, the bijection in Theorem \ref{thm:5.8}.a is
compatible with parabolic induction in the sense of Corollary \ref{cor:4.15}.
The bijection in Theorem \ref{thm:5.13}.a is obtained from Theorem \ref{thm:5.8}
by means of the reduction Theorems \ref{thm:4.5} and \ref{thm:4.7}. Since these
reduction theorems respect parabolic induction, Corollary \ref{cor:4.15} remains
valid in the setting of Theorem \ref{thm:5.8}, and it gives the desired results.
\end{proof}

\section{The relation with the stable Bernstein center}
\label{sec:stable}

Let $\Phi({}^L \cG)$ be the collection of $\cG^\vee$-orbits of L-parameters for ${}^L \cG$.
Recently, inspired by \cite{Vog}, Haines has considered the stable Bernstein
center in \cite{Hai}. We will explore below the relation of the latter with the
Bernstein components $\Phi_\re({}^L\cG)^{\fs^\vee}$.

The notion of stable Bernstein center which we employ here naturally lives on the
Galois side. In principle it should be related to stable distributions on $\cG (F)$
\cite[\S 5.5]{Hai}, but that connection is currently highly conjectural.
Because of that, we will consider it for all inner twists of a given reductive
connected $p$-adic group $\cG(F)$ simultaneously.
Let $\cG^*(F)$ be a quasi-split $F$-group which is an inner twist of $\cG (F)$.
The equivalence classes of inner twists of $\cG^*$ are parametrized by the
Galois cohomology group $H^1(F,\cG^*_\ad)$. For every
$\alpha\in H^1(F,\cG^*_\ad)$, we will denote by $\cG_\alpha(F)$ an inner twist of
$\cG^*(F)$ which is parametrized by $\alpha$. By construction
\[
\Phi_\re({}^L\cG) = \bigsqcup\nolimits_{\alpha\in H^1(F,\cG^*_\ad)} \Phi_\re (\cG_\alpha (F)) .
\]

\begin{defn} \label{defn:inft char}
The infinitesimal character of an L-parameter $\phi\in\Phi({}^L \cG)$ (or an enhanced L-parameter
$(\phi,\rho)\in\Phi_\re({}^L \cG)$) is the $\cG^\vee$-conjugacy class of the admissible morphism
$\lambda_\phi\colon \bW_F\to\cG^\vee\rtimes \bW_F$ (trivial on $\SL_2(\C)$) defined by
\[
\lambda_\phi(w) := \phi \left( w ,\matje{\norm{w}^{1/2}}{0}{0}{\norm{w}^{-1/2}} \right)
\qquad w \in \bW_F .
\]
\end{defn}

With this notion we can reinterpret Theorem \ref{thm:5.13}.b as: the infinitesimal character
of $(\phi,\rho)$ equals the infinitesimal character of the central character of 
$\bar M (\phi,\rho,q_F^{\pm 1/2})$.

\begin{rem}
As noticed in \cite[\S 5]{Hai}, if $\phi$ is relevant for $\cG(F)$, it may happen that
$\lambda_\phi$ is not relevant for $\cG(F)$ anymore. This is why $\lambda_\phi$ is called 
an admissible morphism, {\it i.e.} an $L$-parameter without the relevance condition. 
In contrast, for every $\phi\in\Phi({}^L\cG)$, we have $\lambda_\phi\in\Phi({}^L\cG)$, for 
$\lambda_\phi$ is relevant for $\cG^*(F)$.
\end{rem}

\begin{defn} \label{defn:inft datum}
An inertial infinitesimal datum $\ffi$ for $\Phi({}^L \cG)$ is a pair $({}^L \cM,\ffi_{{}^L \cM})$,
where ${}^L \cM$ is a Levi L-subgroup of ${}^L \cG$, {\it i.e.} ${}^L \cM= \cM^\vee \rtimes \bW_F$
with $\cM^\vee$ a $\mb W_F$-stable Levi subgroup of $\cG^\vee$ and $\ffi_{{}^L \cM}$ is the 
$\cM^\vee$-conjugacy class of the $X_\nr ({}^L \cM)$-orbit of a discrete admissible morphism 
$\lambda \colon \bW_F\to \cM^\vee \rtimes \bW_F$ (trivial on $\SL_2(\C)$). Another such object 
is regarded as equivalent if the two are conjugate by an element of $\cG^\vee$. The equivalence 
class is denoted 
\[
\ffi=(\cM^\vee \rtimes \bW_F,\ffi_{{}^L \cM})_{\cG^\vee} = 
[\cM^\vee \rtimes \bW_F,\lambda]_{\cG^\vee} .
\] 
We will write $\fB^\vee_\stab({}^L \cG)$ for the set of inertial infinitesimal equivalence classes.

For every inertial infinitesimal datum $\ffi=(\cM^\vee \rtimes \bW_F,\ffi_{{}^L \cM})_{\cG^\vee}$,
$\ffi_{{}^L \cM}$ has the structure of an affine variety over $\C$ (see \cite[\S~5.3]{Hai}). 
The stable Bernstein center for ${}^L \cG$ is the ring of regular functions on the disjoint union
$\bigsqcup_{\ffi=({}^L \cM,\ffi_{{}^L \cM}) \in \fB^\vee_\stab({}^L \cG)} \ffi_{{}^L \cM}$.
\end{defn}

We will attach to each inertial equivalence class for $\Phi_\re (\cG (F))$ 
an inertial infinitesimal datum, as follows:

\begin{defn} \label{defn:inft support}
For every cuspidal inertial equivalence class \\
$\fs^\vee = (\cL\rtimes \bW_F,X_\nr ({}^L \cL)\cdot (\phi,\rho)) \in \fB^\vee(\cG(F))$, we set
\[
\inf(\fs^\vee) := (\cM^\vee \rtimes \bW_F , (X_\nr ({}^L \cM)\cdot\lambda_\phi)_{\cM} )_{\cG^\vee},
\]
where $\cM^\vee \rtimes \bW_F $ is a Levi L-subgroup of ${}^L \cG$ which  minimally contains
$\lambda_\phi(\bW_F)$.
\end{defn}

We remark that if $\phi$ has nontrivial restriction to $\SL_2(\C)$, then we may have
$\cM^\vee\rtimes \bW_F \subsetneq \cL^\vee\rtimes \bW_F$ and 
$X_\nr ({}^L \cL) \subsetneq X_\nr ({}^L \cM)$.

For every $\ffi=[\cM^\vee \rtimes \bW_F,\lambda]_{\cG^\vee}\in\fB^\vee_\stab({}^L \cG)$ we set:
\[
\Phi_\re ({}^L\cG)^{\ffi} := \left\{(\phi,\rho) \in\Phi_\re ({}^L\cG)\,:\, \begin{array}{ll}
\lambda_\phi \,\, \text{is minimally contained in } \cM^\vee\rtimes \bW_F  \\
\text{and} \; \lambda_\phi\in (X_\nr ({}^L \cM)\cdot\lambda)_{\cM^\vee}
\end{array}
\right\}.
\]
In this way, we obtain a partition of the set $\Phi_\re({}^L \cG)$ (a "stable Bernstein decomposition"):
\begin{equation} \label{eqn:infBdec}
\Phi_\re ({}^L\cG) = \bigsqcup\nolimits_{\ffi\in\fB^\vee_\stab({}^L \cG)}\Phi_\re ({}^L\cG)^{\ffi}.
\end{equation}
It is worth to observe that, in contrast with Section \ref{sec:Langlands}, the above
definitions involve only the Langlands parameter $\phi\in\Phi({}^L\cG)$ and not the enhancement
of $\phi$. In particular $(\phi,\rho)$ and $(\phi,\rho')$ are always contained in the same
$\Phi_\re ({}^L\cG)^{\ffi}$. Consequently the decomposition \eqref{eqn:infBdec} is coarser than
the Bernstein decomposition of $\Phi_\re ({}^L \cG)$ from \eqref{eqn:Bdec}. However, under the 
local Langlands conjecture, it is a union of L-packets. Indeed, let 
$\ffi=[\cM^\vee \rtimes \bW_F,\lambda]_{\cG^\vee}\in\fB^\vee_\stab({}^L \cG)$. 
From the definition of $\ffi$ we see that 
\[
\Phi_\re ({}^L\cG)^{\ffi}=\bigsqcup_{\alpha \in H^1(F,\cG^*_\ad)} \bigsqcup_{
(\lambda\chi)_{\cM^\vee} \in \ffi_{{}^L \cM}}  \bigsqcup_{\phi \in \Phi({}^L \cG) , 
(\lambda_{\phi})_{\cG^\vee}=(\lambda\chi)_{\cG^\vee}} \Pi_{\phi}(\cG_{\alpha}(F)). 
\]
Define
\[
\fB^\vee({}^L\cG) \,:=\,\bigsqcup\nolimits_{\alpha\in H^1(F,\cG^*_\ad)} \fB^\vee(\cG_\alpha (F)) .
\]

\begin{thm} \label{thm:coarser partition}
For $\ffi\in\fB^\vee_\stab({}^L \cG)$, we write $\fB^\vee({}^L \cG)_\ffi :=
\left\{\fs^\vee\in\fB^\vee({}^L\cG)\,:\, \inf(\fs^\vee) = \ffi\right\}$. Then
\[
\Phi_\re ({}^L \cG)^{\ffi} = \bigsqcup\nolimits_{\fs^\vee \in \fB^\vee({}^L \cG)_\ffi}
\Phi_\re({}^L \cG)^{\fs^\vee}.
\]
\end{thm}
\begin{proof}
Use that for any enhanced Langlands parameter $(\phi,\rho) \in \Phi_\re({}^L \cG)$, the 
infinitesimal character $\lambda_\phi$ of $\phi$ coincides with the infinitesimal 
character $\lambda_\varphi$ of its cuspidal support $(\varphi,q\epsilon)$ \cite[(108)]{AMS}.
\end{proof}

This theorem implies that \eqref{eqn:infBdec} is a partition of $\Phi_\re ({}^L \cG)$
in subsets which are at the same time unions of Bernstein components and unions of L-packets.

Combining Theorems~\ref{thm:coarser partition} and \ref{thm:5.13}, we obtain:
\begin{cor} \label{cor:Heckes}
For every $\ffi\in\fB^\vee_\stab({}^L \cG)$ and every $\vec{z} \in \R_{>0}^d$, there
is a canonical bijection
\[
\Phi_\re ({}^L \cG)^{\ffi} \longleftrightarrow
\bigsqcup\nolimits_{\fs^\vee \in \fB^\vee({}^L \cG)_\ffi} \Irr_{\vec{z}} (\cH (\mf s^\vee,\vec{\mb z})).
\]
\end{cor}

\begin{rem}
It is natural to expect that a certain compatibility should exist between the algebras
$\cH (\mf s^\vee,\vec{\mb z})$ when $\fs^\vee$ runs over the set $\fB^\vee({}^L \cG)_\ffi$,
for a fixed $\ffi=[\cM^\vee \rtimes \bW_F,\lambda]_{\cG^\vee}$.
A naive guess would be that there exist "spectral transfer morphisms" (as introduced for
affine Hecke algebras by Opdam \cite{Opd2}) between the algebras $\cH (\mf s^\vee,\vec{\mb z})$ for
$\fs^\vee\in \fB^\vee({}^L \cG)_\ffi$, the role of the lowest algebra being played by an algebra
$\cH (\mf s_1^\vee,\vec{\mb z})$, with $\fs_1^\vee=[\cM^\vee \rtimes \bW_F,\lambda,1]_{\cG^\vee}$.
\end{rem}

\section{Examples}
\label{sec:examples}

In this section we will work out some affine Hecke algebras attached to Bernstein
components of Langlands parameters. In the examples that we consider the local
Langlands correspondence is known, and it matches Bernstein components on the
Galois side with Bernstein components on the $p$-adic side. We will compare the
Hecke algebras associated to Bernstein components that correspond under the LLC.

All our examples are inner forms of split groups, so $X_\nr ({}^L \mc L) =
\rZ(\mc L^\vee)^\circ$ and we may replace ${}^L \mc G$ by $\mc G^\vee$.

\subsection{Inner twists of $\GL_n (F)$} \label{subsec:GL(D)} \
\label{par:innGLn}

Recall that $F$ is a local non-archimedean field, and let $q_F$ be the cardinality
of its residue field.
Let $D$ be a division algebra with centre $F$ and $\dim_F (D) = d^2$. Take
$m \in N$ and consider $\cG (F) = \GL_m (D)$. It is an inner form of $\GL_n (F)$
with $n = m d$. In fact $\cG (F)$ becomes an inner twist if we regard $D$, the
Hasse invariant $h (D) \in \{z \in \C^\times : z^d = 1 \}$ or
the associated character $\chi_D$ of $\rZ (\SL_n (\C))$ as part of the data.
Up to conjugacy every Levi subgroup of $\cG (F)$ is of the form
\[
\cL (F) = \prod\nolimits_j \GL_{m_j} (D) \quad \text{with } \sum\nolimits_j m_j = m .
\]
Let $(\phi = \bigoplus_j \phi_j, \rho = \otimes_j \rho_j ) \in \Phi_\cusp (\cL (F))$.
In \cite[Example 6.11]{AMS} we worked out the shape of cuspidal Langlands
parameters $(\phi_j,\rho_j)$ for $\GL_{m_j} (D)$. Namely
\begin{itemize}
\item $\phi_j = \phi_j |_{\mb W_F} \otimes S_{d_j}$ where
$S_{d_j}$ is the irreducible $d_j$-dimensional representation of $\SL_2 (\C)$
and $\phi_j |_{\mb W_F}$ is an irreducible representation of dimension $m_j d / d_j$.
(This says that $\phi_j$ is discrete.)
\item $\cS_{\phi_j} = \rZ(\SL_{m_j d}(\C))$ and $\rho_j$ is the character associated
to $\GL_{m_j}(D)$, that is, $\rho_j (\exp ( 2\pi i k / (m_j d)) I_{m_j d}) = h(D)^k$.
(So $(\phi_j,\rho_j)$ is relevant for $\GL_{m_j}(D)$.)
\item lcm $(d, m_j d / d_j) = m_j d$, or equivalently gcd$(d,m_j d / d_j) = d / d_j$.
(This gua\-ran\-tees cuspidality.)
\end{itemize}
It is known that two irreducible representation $\phi_j$ and $\phi_k$ of $\mb W_F$
are isomorphic up to an unramified character twist if and only if their restrictions to
$\mb I_F$ are isomorphic. Hence we can adjust the indexing so that
$\phi |_{\mb I_F}= \bigoplus_{i} \phi_i^{\oplus e_i} |_{\mb I_F}$. Because the restriction
of each $\phi_i$ to $\mb I_F$ decomposes as sum of irreducible representations of $\mb I_F$
with multiplicity one, we find that $R(J^{\circ},T)\cong \prod\nolimits_i A_{e_i - 1}$.
To determine the Hecke algebra of the associated Bernstein component
$\fs^\vee$ of $\Phi_\re (\cG (F))$, we make a simplifying assumption: if $m_i = m_j$ and
$\phi_i$ differs from $\phi_j$ by an unramified twist, then $\phi_i = \phi_j$.

We adjust the indexing so that
\[
\cL (F) = \prod\nolimits_i \GL_{m_i}(D)^{e_i} ,\quad \phi = \bigoplus\nolimits_i
\phi_i^{\oplus e_i}, \quad \rho = \bigotimes\nolimits_i \rho_i^{\otimes e_i} ,
\]
where $\phi_i$ and $\phi_j$ are not inertially equivalent if $i \neq j$. Let $\fs_i^\vee$
be the Bernstein component of $\Phi_\re (\GL_{m_i e_i}(D))$ determined by
$(\phi_i^{\oplus e_i}, \rho_i^{\otimes e_i})$.
Choose an isomorphism $M_{d e_i m_i}(\C) \cong M_{m_i d / d_i}(\C) \otimes M_{d_i e_i}(\C)$
and let $1_m$ be the multiplicative unit of the matrix algebra $M_m (\C)$. Then
\begin{align*}
& G_\phi = \rZ_{\SL_n (\C)} (\phi (\mb W_F)) \cong \SL_n (\C) \cap \prod\nolimits_i
( 1_{m_i d /d_i} \otimes \GL_{d_i e_i}(\C) ) = \SL_n (\C) \cap \prod\nolimits_i G_{\phi,i} ,\\
& M \cong \SL_n (\C) \cap \prod\nolimits_i ( 1_{m_i d /d_i} \otimes \GL_{d_i}(\C)^{e_i} ), \\
& T \cong \SL_n (\C) \cap \prod\nolimits_i ( 1_{m_i d /d_i} \otimes \rZ(\GL_{d_i}(\C) )^{e_i} )
,\quad R(G_\phi,T) \cong \prod\nolimits_i A_{e_i - 1}, \\
& T_i = \big\{ \phi_i \otimes \chi_i \in \Phi (\GL_{m_i}(D)) :
\chi_i \in X_\nr ({}^L \GL_{m_i}(D)) \big\}  \big/ \mu_{t_{\phi_i}}(\C) , \\
& T_{\fs^\vee} = \prod\nolimits_i T_{\fs_i^\vee} = \prod\nolimits_i T_i^{e_i} ,\quad
W_{\fs^\vee} = W_{\fs^\vee, \phi} \cong \prod\nolimits_i S_{e_i} .
\end{align*}
Here $\mu_k$ denotes the functor of taking $k$-th roots of unity and $t_{\phi_i}$ 
denotes the number of unramified twists $z_i \in X_\nr ({}^L \GL_{m_i}(D))$
such that $z_i \phi_i \cong \phi_i$ in $\Phi_\cusp (\GL_{m_i}(D))$. The cyclic group
$\mu_{t_{\phi_i}}(\C)$ is naturally embedded in the onedimensional complex torus
$X_\nr ({}^L \GL_{m_i}(D))$.
Furthermore we can decompose $u_\phi = \prod_i u_{\phi,i}$, where $u_{\phi,i}$ belongs to
the unique distinguished unipotent class of $1_{m_i d /d_i} \otimes \GL_{d_i}(\C)^{e_i}$.
By \cite[2.13]{Lus3} this implies $c(\alpha) = 2 d_i$ for all $\alpha \in R(G_{\phi,i}T,T)$.
Then $\lambda (\alpha) = t_{\phi_i} d_i$ on $R(G_{\phi,i}T,T)$,
whereas $\lambda^*$ does not occur. We conclude that
\begin{equation}\label{eq:6.13}
\cH (\fs^\vee, \vec{\mb z}) = \cH (\mc R_{\fs^\vee}, \lambda, \vec{\mb z}) \cong
\bigotimes\nolimits_i \cH (\GL_{e_i d_i}(\C), \GL_{d_i}(\C)^{e_i},
v_i, \rho_i^{\otimes e_i}, \mb{z}_i) ,
\end{equation}
a tensor product of affine Hecke algebras of type $\GL_{e_i}$ with parameters
$\mb{z}_i^{t_{\phi_i} d_i}$. The most appropriate specialization of \eqref{eq:6.13} is at
$\mb{z}_i = q_F^{1/2}$. Indeed this recovers the exact parameters found by S\'echerre in
\cite[Th\'eor\`eme 4.6]{Sec3}, see \eqref{eq:6.16}.\\

Now we consider Hecke algebras on the $p$-adic side. By the local Langlands correspondence
for $\GL_{m_i}(D)$ (see \cite[\S 11]{HiSa} and \cite[\S 2]{ABPS3}), $(\phi_i,\rho_i)$
is associated to a unique essentially square-integrable representation $\sigma_i \in
\Irr (\GL_{m_i}(D))$. Moreover the condition lcm$(d,m_i d / d_i) = m_i d$ guarantees that
$\sigma_i$ is supercuspidal, by \cite[Th\'eor\`eme B.2.b]{DKV}. (This is a formal consequence
of the Jacquet--Langlands correspondence, so in view of \cite{Bad} it also holds in positive
characteristic.) Hence
\[
(\phi_i^{\oplus e_i}, \rho_i^{\otimes e_i}) \in \Phi_\cusp (\GL_{m_i}(D)^{e_i}) \quad
\text{corresponds to} \quad \sigma_i^{\otimes e_i} \in \Irr_\cusp (\GL_{m_i}(D)^{e_i}) .
\]
Let $\fs_i$ denote the inertial equivalence class for $\GL_{m_i e_i}(D)$ determined by\\
$(\GL_{m_i}(D)^{e_i}, \sigma_i^{\otimes e_i})$. In \cite[Th\'eor\`eme 5.23]{SeSt4} a
$\fs_i$-type $(J_i,\tau_i)$ was constructed. The Hecke algebra for $(J_i,\tau_i)$
was analysed in \cite[Th\'eor\`eme 4.6]{Sec3}, S\'echerre found an isomorphism
\begin{equation}\label{eq:6.4}
\mc H (\GL_{m_i e_i}(D), J_i, \tau_i) \cong \mc H (\GL_{e_i},q_F^{f_i}),
\end{equation}
where the right hand side denotes an affine Hecke algebra of type $\GL_{e_i}$ with parameter
$q_F^{f_i}$ (for a suitable $f_i \in \N$ depending only on $\sigma_i$ or $\phi_i$, see below).
From the explicit description in \cite[\S 4]{Sec3} one sees readily that the isomorphism
\eqref{eq:6.4} respects the natural Hilbert algebra structures on both sides.

\begin{rem}
Let $t_{\sigma_i}$ denote the torsion number of $\sigma_i$, {\it i.e.,} the number of unramified
characters $\chi_i$ of $\GL_{m_i}(D)$ such that $\chi_i \otimes \sigma_i \cong \sigma_i$.
It equals $t_{\phi_i}$. \end{rem}

If $D = F$, then $f_i=t_{\sigma_i}$. In general, $f_i=s_{\sigma_i}\,t_{\sigma_i}$, where 
$s_{\sigma_i}$ is the reducibility number of $\sigma_i$, as defined in \cite[Introduction]{SeSt6}
(see also \cite[Theorem~4.6]{SecU0}). The number $s_{\sigma_i}$ coincides with the invariant
introduced in \cite[Th\'eor\`eme~B.2.b]{DKV} (as it follows for instance from \cite[Eqn.~(1.1) and
Definition~2.2]{BHLS}), itself equal to the integer $d_i$. Hence $f_i$ admits the following
description in terms of Langlands parameters:
\begin{equation}\label{eq:6.16}
f_i = s_{\sigma_i} t_{\sigma_i} = d_i t_{\phi_i}.
\end{equation}
Write $\mc M (F) = \prod_i \GL_{m_i}(D)^{e_i} ,\; \sigma = \bigotimes_i
\sigma_i^{\otimes e_i}$ and let $\fs$ be the inertial equivalence class of
$(\mc M (F), \sigma)$ for $\GL_m (D)$. In \cite[Theorem C]{SeSt6} a $\fs$-type $(J,\tau)$
was constructed, as a cover of the product of the types $(J_i,\tau_i)$ for $\fs_i$.
Moreover it was shown that
\begin{equation}\label{eq:6.5}
\mc H (\GL_m (D), J, \tau) \cong \bigotimes\nolimits_i \mc H (\GL_{e_i},q_F^{f_i}).
\end{equation}
Since \eqref{eq:6.4} was an isomorphism of Hilbert algebras, so is \eqref{eq:6.5}.
Notice that the right hand side is also the specialization of $\mc H(\fs^\vee,\vec{z})$
at $\mb {z}_i = q_F^{1/2}$. Thus there are equivalences of categories
\begin{equation}\label{eq:6.1}
\Rep (\GL_m (D))^\fs \cong \Mod \Big( \bigotimes\nolimits_i
\mc H (\GL_{e_i},q_F^{f_i}) \Big) \cong \Mod \Big( \mc H (\fs^\vee,\vec{\mb z}) /
\big( \{ \mb{z}_i - q_F^{1/2} \}_i \big) \Big). \hspace{-3mm}
\end{equation}
It was shown in \cite[\S 5.4]{BaCi} that, since these equivalences come from
isomorphisms of Hilbert algebras, they preserve temperedness of representations. Then
\cite[Lemma 16.5]{ABPS5} proves that \eqref{eq:6.1} maps essentially square-integrable
representations to essentially discrete series representations and conversely.

The torus underlying $\bigotimes\nolimits_i \mc H (\GL_{e_i},q_F^{f_i})$ is
$T_{\fs} = [ \mc M (F) , \sigma ]_{\mc M (F)}$, which by the LLC for $\GL_{m_i}(D)$
is naturally isomorphic to the torus $T_{\fs^\vee}$ underlying
$\mc H (\fs^\vee,\vec{\mb z})$.
Then \cite[Theorem 4.1]{ABPS4} shows that, with the interpretation as
in Lemma \ref{lem:5.12} (which highlights the tori in these affine Hecke algebras),
the equivalences \eqref{eq:6.1} become canonical. This means in essence that we use
the local Langlands correspondence for supercuspidal representations as input.
With Theorem \ref{thm:5.13} we obtain canonical bijections
\begin{equation}\label{eq:6.2}
\Irr (\GL_m (D))^\fs \longleftrightarrow \Irr \big( \mc H (\fs^\vee,\vec{\mb z}) /
\big( \{ \mb{z}_i - q_F^{1/2} \}_i \big) \big)
\longleftrightarrow \Phi_\re (\GL_m (D))^{\fs^\vee}.
\end{equation}

\begin{prop}\label{lem:6.1}
The union of the bijections \eqref{eq:6.2} over all Bernstein components for
$\GL_m (D)$ equals the local Langlands correspondence for $\GL_m (D)$.
\end{prop}
\begin{proof}
In \cite[\S 2]{ABPS3} the LLC for $\GL_m (D)$ was constructed by starting with
irreducible essentially square-integrable representations of Levi subgroups, then
applying parabolic induction and finally taking Langlands quotients. In the context of
types and covers thereof, \cite[Corollary 8.4]{BuKu} shows that the maps \eqref{eq:6.1}
commute with parabolic induction. They also commute with taking Langlands quotients,
because for these groups every Langlands quotient is the unique irreducible
quotient of a suitable representation.

Thus we have reduced the claim to the case of irreducible essentially square-integrable
representations. From \cite[\S B.2]{DKV} we see that $\Rep (\GL_m (D))^\fs$ only
contains such representations if $m_1 e_1 = m$. We may just as well consider the group
$\GL_{m_i e_i}(D)$, which we prefer because then we can stick to the above notation.
All its irreducible essentially square-integrable representations are generalized
Steinberg representations built from $T_{\fs_i}$. By construction the bijection
\eqref{eq:6.2} for $\GL_{m_i}(D)^{e_i}$ sends $T_{\fs_i}$ to $T_{\fs_i^\vee}$.

Let $\chi_i \in X_\nr (\GL_{m_i}(D))$, with Langlands parameter
$t_i \in X_\nr ({}^L \GL_{m_i}(D))$. The generalized Steinberg representation
St$(\sigma')$ based on $\sigma' = (\chi_i \sigma_i)^{\otimes e_i}$ is the irreducible
essentially square-integrable subrepresentation of the parabolic induction of
\begin{equation}\label{eq:6.3}
\nu_i^{(1-e_i)/2} \chi_i \sigma_i \otimes \cdots \otimes \nu_i^{(e_i - 1)/2} \chi_i \sigma_i
\end{equation}
to $\prod_i \GL_{m_i e_i}(D)$, where $\nu_i$ denotes the absolute value of
reduced norm map for $\GL_{m_i}(D)$. There is a unique such subrepresentation by
\cite[Th\'eor\`eme B.2.b]{DKV}. By definition \cite[(12)]{ABPS3} St$(\sigma')$
has Langlands parameter $t_i \phi_i \otimes S_{e_i}$.

Now we plug St$(\sigma')$ in \eqref{eq:6.2} and we use the property discussed under
\eqref{eq:6.1}. Thus we end up with an essentially discrete series representation of
$\mc H (\fs^\vee,\vec{\mb z}) / \big( \{ \mb{z}_i - q_F^{1/2} \}_i \big)$. By
Theorem \ref{thm:5.13} it corresponds to a discrete element of
$\Phi_\re (\GL_{m_i e_i}(D))^{\fs_i^\vee}$. Its enhancement $\rho_i$ is uniquely determined
by the requirement that it is relevant for $\GL_{m_i e_i}(D)$, so we can ignore that and
focus on the L-parameter. The image of $\mb W_F$ under this
L-parameter is contained in $\GL_{m_i}(D)^{e_i,\vee} = \GL_{m_i d}(\C)^{e_i}$, so it can
only be discrete if it is of the form $\psi_i \otimes \pi_{e_i,\SL_2 (\C)}$ for
some irreducible $m_i d$-dimensional representation of $\mb W_F$. Since the cuspidal
support of the enhanced L-parameter lies in $T_{\fs_i^\vee}$, $\psi_i$ must be an
unramified twist of $\phi_i$. From \eqref{eq:6.3} and the expression for the central
character of $M(\psi_i \otimes \pi_{e_i,\SL_2 (\C)}, \rho_i, z_i)$ given in
Theorem \ref{thm:5.13}.b we deduce that $\psi_i = t_i \phi_i$. Thus \eqref{eq:6.2} agrees
with the local Langlands correspondence for essentially square-integrable representations.
\end{proof}

\subsection{Inner twists of $\SL_n (F)$} \
\label{par:innSLn}

This paragraph is largely based on \cite{ABPS3,ABPS4}.
We keep the notations from the previous paragraph. For any subgroup of $\GL_m (D)$,
we indicate the subgroup of elements of reduced norm 1 by a $\sharp$. Thus
\[
\mc G^\sharp (F) = \GL_m (D)^\sharp = \{ g \in \GL_m (D) : \mathrm{Nrd}(g) = 1\}
= \SL_m (D) .
\]
The inner twists of $\GL_n (F)$ are in bijection with the inner twists of $\SL_n (F)$, via 
\[
\GL_m (D) \leftrightarrow \GL_m (D)^\sharp = \SL_m (D) .
\] 
The L-parameters for $\GL_m (D)^\sharp$ are the same as for $\GL_m (D)$, only their image is
considered in $\PGL_n (\C)$. In particular every discrete L-parameter
\[
\phi^\sharp \colon\mb W_F \times \SL_2 (\C) \to \PGL_n (\C)
\]
lifts to an irreducible $n$-dimensional representation of $\mb W_F \times \SL_2 (\C)$.
The local Langlands correspondence for these groups was worked out in \cite{HiSa,ABPS3}.
It provides a bijection between the Bernstein components on both sides of the LLC,
which will use implicitly as $\fs^\sharp \leftrightarrow \fs^{\sharp \vee}$.

Let $\phi = \otimes_i \phi_i^{\otimes e_i}$ be as before, and let $\phi^\sharp \in
\Phi (\mc L^\sharp (F))$ be the obtained by composition with the projection
$\GL_n (\C) \to \PGL_n (\C)$. Every Bernstein component contains L-parameters of this form.
There is a central extension
\[
1 \to \mc Z_{\phi^\sharp} \to \mc S_{\phi^\sharp} \to \mc R_{\phi^\sharp} \to 1
\]
where $\mf R_{\phi^\sharp} = \pi_0 (\rZ_{\PGL_n (\C)}(\mathrm{im} \phi^\sharp))$ and
\[
\mc Z_{\phi^\sharp} = \rZ(\SL_n (\C)) / \rZ (\SL_n (\C)) \cap \rZ_{\SL_n (\C)}(\phi^\sharp)^\circ .
\]
Let $\rho^\sharp$ be an enhancement of $\phi^\sharp$. The restriction $\rho =
\rho^\sharp |_{\mc Z_{\phi^\sharp}}$ is an enhancement of $\phi$, so as before we may
assume that it has the form $\rho = \otimes_i \rho_i^{\otimes e_i}$. Cuspidality
of $(\phi^\sharp,\rho^\sharp)$ depends only $(\phi,\rho)$, it holds whenever
$\rho_i$ is associated to the inner twist $\GL_{m_i}(D)$ of $\GL_n (F)$ via the
Kottwitz isomorphism. We assume that this is the case, and that
$(\phi^\sharp,\rho^\sharp) \in \Phi_\cusp (\mc L^\sharp (F))$. We note that
$\mc G_\sc^\vee$ is the same for $\GL_m (D)$ and $\SL_m (D)$, and that $\phi$ and
$\phi^\sharp$ have the same connected centralizer. Consequently
\begin{align*}
& G_{\phi^\sharp}^\circ = G_\phi^\circ, \quad G_{\phi^\sharp} / G_{\phi^\sharp}^\circ
\cong \mf R_{\phi^\sharp} ,\quad M_{\phi^\sharp}^\circ = M_\phi^\circ ,\\
& R(G_{\phi^\sharp}^\circ, T) = \prod\nolimits_i A_{e_i - 1} ,\quad \lambda (\alpha) =
t_{\phi_i} d_i \; \forall \alpha \in R(G_{\phi,i}T, T) \subset R (G_{\phi^\sharp}^\circ, T) .
\end{align*}
Let $\mf s^{\sharp \vee}$ be the inertial equivalence class for $\Phi_\re (\GL_m (D)^\sharp)$
determined by $(\phi^\sharp,\rho^\sharp)$. (In spite of the notation $\fs^\vee$ does
not determine it uniquely.) Then
\[
T_{\fs^{\sharp \vee}} = \big( \prod\nolimits_i T_{\phi_i}^{e_i} \big) \big/ \rZ(\GL_n (\C)),
\quad W_{\fs^{\sharp \vee}}^\circ \cong \prod\nolimits_i S_{e_i} .
\]
The cuspidal local system $q\cE$ associated to $(\phi^\sharp,\rho^\sharp)$ satisfies
\[
\mf R_{q\cE} \cong W_{\fs^{\sharp \vee}} / W_{\fs^{\sharp \vee}}^\circ =
\mf R_{\fs^{\sharp \vee}} \cong \mf R_{\phi^\sharp} .
\]
The algebra
\begin{equation}\label{eq:6.14}
\mc H (\mc R_{\fs^{\sharp \vee}}, \lambda, \vec{\mb z}) = \mc H (
G_{\phi^\sharp}^\circ, M_{\phi^\sharp}^\circ, v, \rho, \vec{\mb z})
\end{equation}
is a subalgebra of $\mc H (\mc R_{\fs^{\vee}}, \lambda, \vec{\mb z})$, corresponding
to the projection $T_{\fs^\vee} \to T_{\fs^{\sharp \vee}}$. It is contained in
\[
\mc H (\fs^{\vee \sharp}, \vec{\mb z}) = \mc H (\mc R_{\fs^{\sharp \vee}}, \lambda,
\vec{\mb z}) \rtimes \C [\mf R_{\phi^\sharp}, \natural_{\phi^\sharp}] .
\]
Here the twisted group algebra and the 2-cocycle
$\natural_{\phi^\sharp} = \natural_{\fs^{\sharp \vee}}$ are given by
\[
\C [\mf R_{\phi^\sharp}, \natural_{\phi^\sharp}] = p_\rho \C[\mc S_{\phi^\sharp}] ,
\]
while the action of $\mf R_{\phi^\sharp}$ on \eqref{eq:6.14} comes from its natural
action on $\mc R_{\fs^{\sharp \vee}}$.

For better comparison with the $p$-adic side we also determine the graded Hecke
algebras attached to $\mf s^{\sharp \vee}$. Let $(\phi_b^\sharp,\rho^\sharp) \in
\Phi_\cusp (\mc L^\sharp (F))$ be an unramified twist of $(\phi^\sharp,\rho^\sharp)$
which is bounded. Let $W_{\phi_b^\sharp}$ be the stabilizer of $\phi_b^\sharp$ in
$W_{\fs^{\sharp \vee}}$.
Then $W_{\phi_b^\sharp}^\circ = W(G_{\phi_b^\sharp}^\circ,T)$ is the subgroup of
$W_{\phi_b^\sharp} \cap W_{\fs^{\sharp \vee}}^\circ$ generated by the reflections it
contains. The parabolic subgroup of $G_{\phi_b^\sharp}^\circ$ generated by
$M_{\phi_b^\sharp}^\circ$ and upper triangular matrices determines a group
$\mf R_{\phi_b^{\sharp \vee}}$ such that
\[
W_{\phi_b^\sharp} = W_{\phi_b^\sharp}^\circ \rtimes \mf R_{\phi_b^\sharp} .
\]
The 2-cocycle $\natural_{\phi_b^\sharp}$ on $W_{\phi_b^\sharp}$ is the restriction of
$\natural_{\fs^{\sharp \vee}} \colon W_{\fs^{\sharp \vee}}^2 \to \C^\times$. The root system
$R_{\phi_b^\sharp}$ is again a product of systems of type A, namely
$\prod_j A_{\epsilon_j - 1}$ if $\phi_b^\sharp = \otimes_j \phi_j^{\epsilon_j}$.
Then 
\[
W_{\phi_b^\sharp}^\circ \cong\prod\nolimits_j S_{\epsilon_j} \quad \text{and} \quad
\mf t_{\fs^{\sharp \vee}} = \Lie (T_{\fs^{\sharp \vee}}) =
\big( \sum\nolimits_i \Lie (T_{\phi_i}^{e_i}) \big) \big/ \rZ(\mf{gl}_n (\C)) .
\]
It follows that
\begin{equation}\label{eq:6.15}
\mh H (\phi_b, v, q\epsilon, \vec{\mb r}) \cong \mh H (\mf t_{\fs^{\sharp \vee}},
W_{\phi_b^\sharp}, \vec{\mb r}, \natural_{\phi_b^\sharp}) \cong \mh H (\mf t_{\fs^{\sharp \vee}} ,
W_{\phi_b^\sharp}, \vec{\mb r}) \rtimes \C [\mf R_{\phi_b^\sharp}, \natural_{\phi_b^\sharp}] .
\end{equation}
The Hecke algebras for Bernstein components of $\SL_m (D)$ were computed in
\cite{ABPS4}. They are substantially more complicated than their counterparts for
$\GL_m (D)$, and in particular do not match entirely with the above affine
Hecke algebras for Langlands parameters. To describe them, we need some
notations. Let $\mc P$ be a parabolic subgroup of $\GL_m (D)$, with Levi
factor $\mc M$. Consider the inertial equivalence classes $\fs_{\mc M} =
[\mc M,\sigma]_{\mc M}$ and $\fs = [\mc M ,\sigma]_{\GL_m (D)}$.
Recall from \eqref{eq:6.5} that $\mc H (\GL_m (D))^\fs$ is Morita equivalent with
\[
\mc H (\mc R_\fs, \lambda, q_\fs) =
\bigotimes\nolimits_i \mc H (\GL_{e_i}, q_F^{f_i}) .
\]
We need the groups
\begin{align*}
& X^{\mc M}(\fs) = \big\{ \gamma \in \Irr \big( \mc M / \mc M^\sharp \rZ(\GL_m (D)) \big) :
\gamma \otimes \sigma \in \fs_{\mc M} \big\} , \\
& X^{\GL_m (D)}(\fs) = \big\{ \gamma \in \Irr \big( \GL_m (D) / \GL_m (D)^\sharp \rZ(\GL_m (D))
\big) : \gamma \otimes I_{\mc P}^{\GL_m (D)}(\sigma) \in \fs \big\} , \\
& W_\fs^\sharp = \big\{ w \in \rN_{\GL_m (D)}(\mc M) / (\mc M) : \exists \gamma \in
\Irr \big( \mc M / \mc M^\sharp \rZ(\GL_m (D)) \big) : w (\gamma \otimes \sigma) \in
\fs_{\mc M} \big\} .
\end{align*}
By \cite[Lemma 2.3]{ABPS4} $W_\fs^\sharp = W_\fs \rtimes \mf R_\fs^\sharp$ for
a suitable subgroup $\mf R_\fs^\sharp$, and 
\[
X^{\GL_m (D)}(\fs) / X^{\mc M}(\fs) \cong \mf R_\fs^\sharp
\]
by \cite[Lemma 2.4]{ABPS4}. The group $X^{\GL_m (D)}(\fs)$ acts naturally on 
$T_\fs \rtimes W_\fs$.

Let $\sigma^\sharp$ be an irreducible constituent of $\sigma |_{\mc M^\sharp}$.
Every inertial equivalence class for $\SL_m (D) = \GL_m (D)^\sharp$ is of the
form $\fs^\sharp = [\mc M^\sharp, \sigma^\sharp]_{\GL_m (D)^\sharp}$.
By \cite[Theorem 1]{ABPS4} there exists a finite dimensional projective
representation $V_\mu$ of $X^{\GL_m} \fs)$ such that
$\mc H (\GL_m (D)^\sharp )^{\fs^\sharp}$ is Morita equivalent with one direct summand of
\begin{equation}\label{eq:6.6}
\big( \mc H (\mc R_\fs, \lambda, q_\fs ) \otimes \End_\C (V_\mu)
\big)^{X^{\mc M}(\fs) X_\nr (\mc M / \mc M^\sharp)} \rtimes \mf R_\fs^\sharp .
\end{equation}
The other direct summands correspond to different constituents of $\sigma |_{\mc M^\sharp}$.
In \eqref{eq:6.6} the group
\[
X_\nr (\mc M / \mc M^\sharp) = \{ \chi \in X_\nr (\mc M) : \mc M^\sharp \subset \ker \chi \}
\]
acts only via translations of $T_\fs$. We denote the quotient torus
$T_\fs / X_\nr (\mc M / \mc M^\sharp)$ by $T_\fs^\sharp$ and its Lie algebra by
$\mf t_\fs^\sharp$.

From now on we will be more sketchy. The below can be made precise, but for that one would
have to delve into some of the technicalities of \cite{ABPS4}, which are not so
relevant for this paper. Although it is not so easy to write down all direct summands of
\eqref{eq:6.6} explicitly, we can say that they look like
\begin{equation}\label{eq:6.7}
\big( \mc H (X^* (T_\fs^\sharp), R_\fs, X_* (T_\fs^\sharp), R_\fs^\vee, \lambda, q_\fs)
\otimes \End_\C (V_{\mu^\sharp}) \big)^{X^{\mc M}(\fs,\sigma^\sharp)} \rtimes \mf R_{\fs^\sharp}
\end{equation}
for suitable $X^{\mc M}(\fs,\sigma^\sharp) \subset X^{\mc M}(\fs)$ and
$V_{\mu^\sharp} \subset V_\mu$. (From the below argument for graded Hecke algebras one sees
approximately how \eqref{eq:6.7} arises from \eqref{eq:6.6}.)
This algebra need not be Morita equivalent to a twisted affine Hecke algebra as studied
in this paper. The problem comes from the simultaneous action of $X^{\mc M}(\fs,\sigma^\sharp)$
on $T_\fs^\sharp$ and $V_{\mu^\sharp}$: if that is complicated, it prevents \eqref{eq:6.7} from
being Morita equivalent to a similar algebra without $\End_\C (V_{\mu^\sharp})$. If we consider
\eqref{eq:6.7} as a kind of algebra bundle over $T_\fs^\sharp$, then these remarks mean that
$V_{\mu^\sharp}$ could introduce some extra twists in this bundle, which take the algebra
outside the scope of this paper. Examples can be constructed by combining the ideas in
\cite[Examples 5.2 and 5.5]{ABPS4}.

That being said, the other data involved in \eqref{eq:6.7} are as desired. It was checked
in \cite[Lemma 5.5]{ABPS7} that:
\begin{enumerate}
\item[(i)] The underlying torus $T_{\fs^\sharp} = T_\fs^\sharp / X^{\mc M}(\fs,\sigma^\sharp)$
is naturally isomorphic to $T_{\fs^{\sharp \vee}} = \Phi_\re (\mc M^\sharp )^{
[\mc M^\sharp,\sigma^\sharp]_{\mc M^\sharp}}$.
\item[(ii)] $W_\fs \rtimes \mf R_{\fs^\sharp} = W_{\fs^\sharp}$ is isomorphic to
$W_{\fs^{\sharp \vee}} = W_{\fs^\vee} \rtimes \mf R_{\fs^{\sharp \vee}}$.
\item[(iii)] The space of irreducible representations of \eqref{eq:6.7} is isomorphic to a
twisted extended quotient
\[
( T_{\fs^\sharp} \q W_{\fs^\sharp} )_{\kappa_{\sigma^\sharp}}
\cong (T_{\fs^{\sharp \vee}} \q W_{\fs^{\sharp \vee}} )_{\kappa_{\sigma^\sharp}} ,
\]
and the 2-cocycle $\kappa_{\sigma^\sharp}$ of $W_{\fs^\sharp}$ is equivalent to the
2-cocycle $\natural_{\fs^{\sharp \vee}}$ of $W_{\fs^{\sharp \vee}}$.
\end{enumerate}
Let us also discuss the graded Hecke algebras which can be derived from \eqref{eq:6.6}
and \eqref{eq:6.7}.
The algebra $\mc O (T_\fs^\sharp )^{X^{\mc M}(\fs) W_\fs}$ is naturally contained
in the centre of \eqref{eq:6.6}. This entails that we can localize at suitable subsets
of $T_\fs^\sharp / W_\fs^\sharp X^{\mc M}(\fs)$. Fix $t \in (T_\fs^\sharp)_\uni$.
By localization at a small neighborhood of $U$ of $W_\fs^\sharp X^{\mc M}(\fs) t
(T_\fs^\sharp )_{\rs}$, we can effectively replace $X^{\mc M}(\fs)$ by the stabilizer
of $X^{\mc M}(\fs)_t$, and $\mf R_\fs^\sharp$ by the stabilizer $\mf R_\fs^\sharp (t)$
of $W_\fs X^{\mc M}(\fs) t$. Then \eqref{eq:6.6} is transformed into the algebra
\begin{equation}\label{eq:6.8}
C_{an}(U)^{X^{\mc M}(\fs) W_\fs^\sharp}
\underset{\mc O (T_\fs^\sharp)^{X^{\mc M}(\fs) W_\fs^\sharp}}{\otimes}
\big( \mc H (\mc R_\fs^\sharp, \lambda, q_\fs)
\otimes \End_\C (V_\mu) \big)^{X^{\mc M}(\fs)_t} \rtimes \mf R_{\fs^\sharp}(t)
\end{equation}
where $\mc R_\fs^\sharp = (X^* (T_\fs^\sharp), R_\fs, X_* (T_\fs^\sharp), R_\fs^\vee)$.
But $X^{\mc M}(\fs)$ acts by translations on $T_\fs^\sharp$, so $X^{\mc M}(\fs)_t$
consists of all the elements that fix $T_\fs^\sharp$ entirely. From the description of
the actions on \eqref{eq:6.6} in \cite[Lemma 4.11]{ABPS4} we see that $X^{\mc M}(\fs)_t$
acts only on $\End_\C (V_\mu)$. Then
\begin{equation}\label{eq:6.9}
\End_\C (V_\mu)^{X^{\mc M}(\fs)_t} = \End_{X^{\mc M}(\fs)_t} (V_\mu) \cong
\bigoplus\nolimits_{\mu^\sharp} \End_\C (V_{\mu^\sharp})
\end{equation}
is a finite dimensional semisimple algebra. The direct summands of \eqref{eq:6.6} and
of \eqref{eq:6.8} are
in bijection with the $\mf R_\fs^\sharp (t)$-orbits on the set of direct summands
of \eqref{eq:6.9}. That holds for any $t \in (T_\fs^\sharp)_\uni$, in particular for s
ome $t$ with $\mf R_\fs^\sharp (t) = 1$, so in fact the direct summands
$\End_\C (V_{\mu^\sharp})$ of \eqref{eq:6.9} parametrize the direct summands of
\eqref{eq:6.6} and of \eqref{eq:6.8}. Thus \eqref{eq:6.8} is a direct sum of algebras
\begin{equation}\label{eq:6.10}
C_{an}(U)^{X^{\mc M}(\fs) W_\fs^\sharp}
\underset{\mc O (T_\fs^\sharp)^{X^{\mc M}(\fs) W_\fs^\sharp}}{\otimes}
\big( \mc H (\mc R_\fs^\sharp, \lambda, q_\fs)
\otimes \End_\C (V_{\mu^\sharp}) \big) \rtimes \mf R_{\fs^\sharp}(t) .
\end{equation}
Here $(\mu^\sharp,V_{\mu^\sharp})$ is a projective representation of
$\mf R_{\fs^\sharp}(t)$. In such situations there is a Morita equivalent algebra embedding
\[
\begin{array}{ccc}
\C [\mf R_{\fs^\sharp}(t) ,\natural] & \to &
\End_\C (V_{\mu^\sharp}) \rtimes \mf R_{\fs^\sharp}(t) \\
r & \mapsto & \mu^\sharp (r)^{-1} r ,
\end{array}
\]
for a suitable 2-cocycle $\natural$. Via this method \eqref{eq:6.10}
is Morita equivalent with
\begin{equation}\label{eq:6.11}
C_{an}(U)^{X^{\mc M}(\fs) W_\fs^\sharp}
\underset{\mc O (T_\fs^\sharp)^{X^{\mc M}(\fs) W_\fs^\sharp}}{\otimes}
\mc H (\mc R_\fs^\sharp, \lambda, q_\fs) \rtimes \C [\mf R_{\fs^\sharp}(t) , \natural] .
\end{equation}
From the property (iii) of the algebra \eqref{eq:6.7} we see that $\natural$ has to be
the restriction of $\natural_{\fs^{\sharp \vee}}$ to $\mf R_{\fs^\sharp}(t)^2$.
By Theorems \ref{thm:4.5}.a and \ref{thm:4.7}.a the algebra \eqref{eq:6.11} is Morita
equivalent with
\begin{equation}
C_{an}(U)^{X^{\mc M}(\fs) W_\fs^\sharp}
\underset{\mc O (\mf t_\fs^\sharp)^{X^{\mc M}(\fs) W_\fs^\sharp}}{\otimes}
\mh H (\mf t_\fs^\sharp, W(R_\fs)_t , q_\fs)
\rtimes \C [\mf R_{\fs^\sharp}(t) , \natural_{\fs^{\sharp \vee}}] .
\end{equation}
Hence the equivalence between $\Rep (\SL_m (D))^{\fs^\sharp} \cong
\Mod \big( \mc H (\GL_m (D)^\sharp)^{\fs^\sharp} \big)$ and the module category of
\eqref{eq:6.7} restricts to an equivalence between
\begin{align*}
& \Mod_{f,W_\fs^\sharp X^{\mc M}(\fs) t (T_\fs^\sharp)_{\rs}}
\big( \mc H \big(\GL_m (D)^\sharp)^{\fs^\sharp} \big) \big) \quad \text{and} \\
& \Mod_{f,(\mf t_\fs^\sharp)_{\rs}} \big( \mh H (\mf t_\fs^\sharp, W(R_\fs)_t , q_\fs)
\rtimes \C [\mf R_{\fs^\sharp}(t) , \natural_{\fs^{\sharp \vee}}] \big) .
\end{align*}
Every finite length representation in $\Rep (\SL_m (D))^{\fs^\sharp}$ decomposes
canonically as a direct sum of generalized weight spaces for
$\mc O (T_\fs^\sharp)^{X^{\mc M}(\fs) W_\fs^\sharp}$, so by varying $t$ in
$(T_\fs^\sharp )_\uni$ we can describe all such representations in terms of these
equivalences of categories. In this sense
\begin{equation}\label{eq:6.12}
\mh H (\mf t_\fs^\sharp, W(R_\fs)_t , q_\fs)
\rtimes \C [\mf R_{\fs^\sharp}(t) , \natural_{\fs^{\sharp \vee}}]
\end{equation}
is the graded Hecke algebra attached to $(\fs^\sharp,t)$. Suppose that $t$
corresponds to $(\phi_b^\sharp, \rho^\sharp) \in \Phi_\cusp (\mc L^\sharp (F))$,
where $\mc M = \mc L (F)$. Then we can compare \eqref{eq:6.12} with \eqref{eq:6.15}.
Using the earlier comparison results (i), (ii) and (iii), we see that \eqref{eq:6.12}
is the specialization of \eqref{eq:6.15} at $\vec{\mb r} = \log (q_\fs)$.

We conclude that, for a Bernstein component $\fs^\sharp$ of $\SL_m (D)$, corresponding
to a Bernstein component $\fs^{\sharp \vee}$ of enhanced L-parameters:
\begin{itemize}
\item The twisted graded Hecke algebras attached to $\fs^\sharp$ and to $\fs^{\sharp \vee}$
are isomorphic.
\item The twisted affine Hecke algebras attached to $\fs^\sharp$ and to $\fs^{\sharp \vee}$
need not be isomorphic, but they are sufficiently close, so that their categories
of finite length modules are equivalent.
\end{itemize}

\subsection{Pure inner twists of classical groups} \
\label{par:classical}

Take $n\in\N$ and let $\cG^*_n$ be a $F$-split connected classical group of rank $n$.
That is, $\cG^*_n$ is one the following groups:
\begin{enumerate}
\item[(i)]
$\Sp_{2n}$, the symplectic group in $2n$ variables defined over $F$,
\item[(ii)]
$\SO_{2n+1}$, the split special orthogonal group in $2n \!+\! 1$ variables defined over $F$,
\item[(iii)]
$\SO_{2n}$, the split special orthogonal group in $2n$ variables defined over $F$,
\end{enumerate}
Let $V^*$ be a finite dimensional $F$-vector space equipped with a non-degenerate symplectic
or orthogonal form such that $\cG^*_n(F)$ equals $\Sp(V^*)$ or $\SO(V^*)$.
The pure inner twists $\cG_n$ of $\cG^*_n$ correspond bijectively to forms $V$ of the space
$V^*$ with its bilinear form $\langle\,,\,\rangle$ \cite[\S 29D--E]{KMRT}.
If $\cG^*_n(F)=\Sp(V^*)$, then the pointed set $H_1(F,\cG^*_n)$ has only one element and there
are no nontrivial pure inner twists of $\cG^*_n$.  If $\cG^*_n(F)=\SO(V^*)$, then elements of
$H_1(F,\cG^*_n)$ correspond bijectively to the isomorphism classes of orthogonal
spaces $V$ over $F$ with $\dim(V)=\dim(V^*)$ and $\disc(V)=\disc(V^*)$.
The corresponding pure inner twist of $\cG^*_n(F)$ is the special orthogonal group $\SO(V)$.

Let $\cG_n(F)$ be a pure inner twist of $\cG_n^*(F)$ (we allow $\cG_n(F)=\cG_n^*(F)$).
It is known (see for instance \cite{ChGo}), that up to conjugacy every Levi subgroup
of $\cG_n(F)$ is of the form
\begin{equation}\label{eq:6.19}
\cL(F) = \cG_{n^{-}}(F)\,\times\,\prod\nolimits_j \GL_{m_j}(F),
\end{equation}
where $\sum_j m_j + n^{-} = n$ and
$\cG_{n^{-}}(F)$ is an inner twist of the split connected classical group
$\cG_{n^{-}}^*$ defined over $F$, of rank $n^{-}$, which has the same type as $\cG_n^*(F)$.
There is a natural embedding $\Std_{{}^L\cG_{}}$ of ${}^L\cG_{}$
into $\GL_{N^{\vee}}(\C) \rtimes \mb W_F$, where $N^\vee=2n+1$ if $\cG_{n}^*=\Sp_{2n}$,
and $N^\vee=2n$ otherwise.

Let $(\phi,\rho)\in\Phi_\cusp(\cL (F))$. The factorization \eqref{eq:6.19} leads to
\begin{equation}\label{eq:6.17}
\Std_{{}^L\cG} \circ \phi = \varphi \,\oplus \,\bigoplus\nolimits_j (\phi_j \oplus \phi_{j}^{\vee}).
\end{equation}
Because we consider only pure inner twists in this section, it would be superfluous to
replace $\mc G^\vee$ by $\mc G^\vee_\sc$. We refrain from doing so in this section,
and we use the objects, which before where defined in terms of $\mc G^\vee_\sc$, now
with the same definition involving just $\mc G^\vee$. For instance, instead of the group
$\mc S_\phi$ defined in Definition \ref{defn:5.2}, we will take the component group
$\pi_0 (\rZ_{\cL^\vee}(\phi))$ and we use a variation on $\Phi_\re ({}^L \cG)$ with that component
group. The restriction of an enhancement $\rho$ to the center of $\cL^{\vee}$ still determines
the relevance. For instance, if the restriction to $\rZ(\cL^{\vee})$ is
trivial, then it corresponds to the split form, otherwise it corresponds to a non-split form.
Hence, we can decompose $\rho = \varrho\,\otimes\,\bigotimes\nolimits_j \rho_j,$ where
$(\varphi,\varrho)\in\Phi_\cusp(\cG_{n^{-}}(F))$ and
$(\phi_j,\rho_j)\in \Phi_\cusp (\GL_{m_j}(F))$ for each $j$.

Let $I_{\phi}^{+}$ (resp. $I_{\phi}^{-}$) be the set of (classes of) self-dual irreducible
representations of $\bW_F$ which occur in $\Std_{{}^L \cG} \circ \phi$ and which factor
through a group of the type of $\cG^{\vee}$ (resp. of opposite type of $\cG^{\vee}$).
Let $I_{\phi}^{0}$ be a set of (classes of) non self-dual irreducible representations of
$\bW_F$ which occur in $\Std_{{}^L \cG} \circ \phi$, such that if $\tau \in I_{\phi}^{0}$
then $\tau^{\vee} \not \in I_{\phi}^{0}$, and maximal for this property.
We denote the irreducible $a$-dimensional representation of $\SL_n (\C)$ by $S_a$.

On the one hand $(\phi_j,\rho_j)$ satisfy the conditions stated in Paragraph \ref{subsec:GL(D)},
i.e. $\phi_{j}$ is an irreducible representation of $\bW_F$ and $\rho_{j}$ is the trivial
representation of $\pi_0(\rZ_{\GL_{m_j}(\C)}(\phi_j))$.
On the other hand, by \cite[Proposition 3.6]{Mou} we have
\begin{equation}\label{eq:6.18}
\Std_{{}^L \cG_{n^{-}}} \circ \varphi=\bigoplus_{\tau \in I_{\varphi}^{+}}
\bigoplus_{a \text{ odd}, a=1}^{a_\tau}(\tau\otimes S_a)\,\oplus\,
\bigoplus_{\tau\in I_{\varphi}^{-}}\bigoplus_{a \text{ even},a=2}^{a_\tau}(\tau\otimes S_a ),
\end{equation}
where $a_\tau \in \Z_{\geq 0}$. As introduced by M\oe glin, let $\Jord(\varphi)$ be the set 
of pairs $(\tau,a)$ with $\tau \in \Irr(\bW_F)$, $a \in \Z_{>0}$ such that $\tau\boxtimes S_{a}$
is an irreducible subrepresentation of $\Std_{{}^L \cG_{n^{-}}} \circ \varphi$.

The group $\mathcal{S}_{\phi}$ is isomorphic to $(\Z/2\Z)^p$ for some integer $p$. It is
generated by elements of order two, by $z_{\tau,a} z_{\tau',a'}$ where
$(\tau,a),(\tau',a') \in \Jord(\phi)$ without hypothesis on the
parity of $a$ and by $z_{\tau,a}$ when $a$ is even. The character $\rho$ satisfies
$\rho(z_{\tau,2i-1}z_{\tau,2i+1})=-1$ for all $\tau \in I_{\varphi}^{+}$ and
$i \in \llbracket 1, \frac{a_{\tau}-1}{2} \rrbracket$ and $\rho(z_{\tau,2i})=(-1)^i$
for all $\tau\in I_{\varphi}^{-} $ and $i \in \llbracket 1, \frac{a_{\tau}}{2} \rrbracket$.

If $\tau$ is an irreducible representation of $\mb W_F$
and of dimension $m$ such that $\tau |_{\mb I_F} \cong \tau^{\vee} |_{\mb I_F} $, then
$\tau \cong \tau^{\vee}z$ with $z \in X_\nr ({}^L \GL_{m}(F))$. Replacing $\tau$ by $\tau z^{1/2}$
(where $z^{1/2}$ is any square root of $z$), we can assume that $\tau \cong \tau^{\vee}$.
In the following, for all $j$  we assume that, if $\phi_{j}^{\vee}$ is inertially equivalent to
$\phi_j$, then $\phi_{j}^{\vee} \cong \phi_{j}$. Note that a self-dual irreducible 
representation of $\mb W_F$ is necessarily of symplectic-type or of orthogonal-type.

We choose a basepoint $\phi$ (inside its inertial equivalence class) as follows:
\begin{itemize}
\item if $m_i = m_j$ and $\phi_i$ differs from $\phi_j$ by an unramified twist, then $\phi_i = \phi_j $;
\item if $\phi_{i}^{\vee}$ is an unramified twist of $\phi_{i}$,
then we can assume that $\phi_{i}^{\vee} \cong \phi_{i} $;
\item if $\phi_{i}^{\vee} \cong \phi_{j}$, then $i=j$.
\end{itemize}
For an irreducible representation
$\tau$ of $\mb W_F$, we will denote by $e_{\tau}$ the number of times that $\tau$ appears in a
$\GL$ factor of $\mathcal{L}^{\vee}$, by $\ell_{\tau}$ the multiplicity of $\tau$ in
$\varphi |_{\mb W_F}$. With the above choice of $\phi$, 
for all $\tau \in I_{\phi}^{+} \sqcup I_{\phi}^{-} \sqcup I_{\phi}^{0}$:
\begin{align*}
& \phi=\bigoplus_{\tau \in I_{\phi}^{+} \sqcup I_{\phi}^{-}} 2e_{\tau} \tau \oplus
\bigoplus_{\tau \in I_{\phi}^{0}} e_{\tau} (\tau \oplus \tau^{\vee}) \oplus \varphi  ,\\
& \restriction{\phi}{\bW_F}=\bigoplus_{\tau \in I_{\phi}^{+} \sqcup I_{\phi}^{-}}
(2e_{\tau}+\ell_{\tau}) \tau \oplus \bigoplus_{\tau \in I_{\phi}^{0}} e_{\tau}
(\tau \oplus \tau^{\vee}) , \\
& G_\phi \cong  \prod_{\tau \in I_{\phi}^{-}} \Sp_{2e_{\tau}+\ell_{\tau}}(\C) \times
\prod_{\tau \in I_{\phi}^{+},\, \dim \tau \text{ even} \hspace{-15mm}} \O_{2e_{\tau}+\ell_{\tau}}(\C) 
\times S\left( \prod_{\tau \in I_{\phi}^{+}, \,\dim \tau \text{ odd} \hspace{-15mm}}
\O_{2e_{\tau}+\ell_{\tau}}(\C) \right) \times
\prod_{\tau \in I_{\phi}^{0}} \GL_{e_{\tau}}(\C) , \\
& M \cong \prod_{\tau \in I_{\phi}^{-}} \big( (\C^{\times})^{e_{\tau}} \times \Sp_{\ell_{\tau}}(\C) \big)
\times \prod_{\tau \in I_{\phi}^{+}}  (\C^{\times})^{e_{\tau}} \times \\
& \hspace{8mm} \prod_{\tau \in I_{\phi}^{+}, \dim \tau \text{ even} \hspace{-12mm}} 
\O_{\ell_{\tau}}(\C) \times S\left( \prod_{\tau \in I_{\phi}^{+}, \dim \tau \text{ odd} \hspace{-12mm}}
\O_{\ell_{\tau}}(\C) \right) \times \prod_{\tau \in I_{\phi}^{0}} (\C^{\times})^{e_{\tau}} .
\end{align*}
Here $S(H)$, for a matrix group $H$, means the elements of determinant 1 in $H$.
The above expression for $G_\phi^\circ$ naturally factors as
$\prod_{\tau \in I_{\phi}^{-} \sqcup I_{\phi}^{+} \sqcup I_{\phi}^{0}} G_\tau^\circ$,
and similarly for $M^\circ$. This is an almost direct factorization of $G_\phi^\circ$ in the 
sense of \eqref{eq:0.1}. With that we can write
\begin{equation}\label{eq:2.1}
T \cong \prod_{\tau \in I_{\phi}^{-} \sqcup I_{\phi}^{+} \sqcup I_{\phi}^{0}}
(\C^{\times})^{e_{\tau}}, \quad R(G_{\phi}^{\circ},T) \cong 
\prod_{\tau \in I_{\phi}^{-} \sqcup I_{\phi}^{+} \sqcup I_{\phi}^{0}} R(G_{\tau}^{\circ}T,T) .
\end{equation}
Let us record the root systems $R_\tau = R (G_\tau^\circ T,T)$:
\[
\renewcommand{\arraystretch}{1.3}
\begin{array}{|*4{>{\displaystyle}c|}}
\cline{1-4}
  & \text{condition} & R_\tau & R_{\tau,\red} \\
\cline{1-4}
 \multirow{3}*{$\tau \in I_{\phi}^{-}$} & e_\tau =0 &
\varnothing & \varnothing \\
\cline{2-4}
  & e_\tau \neq 0, \ell_\tau =0 & C_{e_\tau} & C_{e_\tau}  \\
 \cline{2-4}
  & e_\tau \neq 0, \ell_\tau \neq 0 & BC_{e_\tau} & B_{e_\tau}  \\
\cline{1-4}
 \multirow{3}*{$\tau \in I_{\phi}^{+}$} & e_\tau =0 & \varnothing &
\varnothing \\
\cline{2-4}
  & e_\tau \neq 0, \ell_\tau =0 & D_{e_\tau} & D_{e_\tau}  \\
\cline{2-4}
 & e_\tau \neq 0, \ell_\tau \neq 0 & B_{e_\tau} & B_{e_\tau} \\
\cline{1-4}
 \multirow{2}*{$\tau \in I_{\phi}^{0}$} & e_\tau \leqslant 1 & \varnothing &
\varnothing \\
\cline{2-4}
 & e_\tau \geqslant 2 & A_{e_\tau-1} & A_{e_\tau -1}  \\
\cline{1-4}
\end{array} 
\]
To justify the above choice of a basepoint $\phi$, we need to check that $G_\phi^\circ$ detects
as many roots as possible. Let us consider the restriction $\phi |_{\mb I_F}$:
\begin{align*}
\Std_{{}^L \cG} \,\circ\, \phi |_{\mb I_F} & = \varphi |_{\mb I_F} \,\oplus \,\bigoplus\nolimits_j
(\phi_j |_{\mb I_F} \oplus \phi_{j}^{\vee} |_{\mb I_F}) \\
& = \bigoplus_{\tau \in I^{+}_{\phi} \sqcup I^{-}_{\phi}} (2e_{\tau}+\ell_{\tau}) \tau|_{\mb I_F} \oplus
\bigoplus_{\tau \in I_{\phi}^{0}} e_{\tau} (\tau|_{\mb I_F} \oplus \tau^{\vee}|_{\mb I_F}) .
\end{align*}
We have assumed that for $\tau \in I_{\phi}^{0}$, 
$\tau |_{\mb I_F} \not\cong \tau^{\vee}|_{\mb I_F}$ and we know that an irreducible 
representation $\tau$ of $\mb W_F$ decomposes upon restriction to $\mb I_F$ as
\begin{equation}\label{eq:6.21}
\tau |_{\mb I_F}=\theta \oplus \theta^{\Fr_F} \oplus \ldots \oplus \theta^{\Fr_F^{t_{\tau}-1}},
\end{equation}
for some irreducible representation  $\theta$ of $\mb I_F$. Here for all
$w \in \mb I_F, \,\, \theta^{\Fr_F^k}(w)=\theta(\Fr_F^{-k} w \Fr_F^k)$. If we assume
$\tau |_{\mb I_F} \cong \tau^{\vee} |_{\mb I_F}$, then $\theta^{\vee} \cong \theta^{\Fr_F^{i}}$
for some integer $i$ between $0$ and $t_{\tau}-1$. Then we have
$\theta \cong {\theta^{\Fr_F^{i}}}^{\vee} \cong \theta^{\Fr_F^{2i}}$. This implies that $i=0$ or
$t_{\tau}$ is even and $i=t_{\tau}/2$. In the first case, $\theta^{\vee} \cong \theta$ and in the 
second case $\theta^{\vee} \cong \theta^{\Fr_F^{t_{\tau}/2}}$. We denote by $I_\phi^{+ +}$
(resp. $I_\phi^{- -}$) the subset of $I_\phi^+$ (resp. $I_\phi^-$) corresponding to the 
first case, and define $I_\phi^{+-}$ as the remaining subset of $I_\phi^+ \cup I_\phi^-$.
For any $\tau$, let $\tau'$ be a twist of $\tau$ by an unramified character of $\mb W_F$, such
that $\tau'$ is self-dual but not isomorphic to $\tau$. For $\tau \in I_\phi^{+-}$ the type
of $\tau'$ is opposite to that of $\tau$, which motivates the superscipt $+-$. The three sets 
$I_\phi^{+ +}, I_\phi^{- -}, I_\phi^{+-}$ are considered modulo the relation $\tau \sim \tau'$.
We find that
\begin{equation}\label{eq:2.2}
\begin{aligned}
J^\circ = Z_{\mathcal{G}^{\vee}}(\phi |_{\mb I_F})^{\circ} \cong & \prod_{\tau \in 
I^{--}_{\phi}} \Sp_{2e_{\tau}+\ell_{\tau}+\ell_{\tau'}}(\C)^{t_{\tau}} \times \
\prod_{\tau \in I^{++}_{\phi}} \SO_{2e_{\tau}+\ell_{\tau}+\ell_{\tau'}}(\C)^{t_{\tau}} \times \\
& \prod_{\tau \in I^{+-}_{\phi}} \GL_{2e_{\tau}+\ell_{\tau}+\ell_{\tau'}}(\C)^{t_{\tau}/2} 
\times \prod_{\tau \in I_{\phi}^{0}} \GL_{e_{\tau}}(\C)^{t_{\tau}} .
\end{aligned}
\end{equation}
For all $\tau \in I_{\phi}^{+ +}$, we have an embedding of $(\C^{\times})^{e_{\tau}}$ into
$(\C^{\times})^{e_{\tau}} \times \SO_{\ell_{\tau}+\ell_{\tau'}}(\C)^{t_\tau}$ and the latter is 
embedded diagonally as Levi subgroup in  $\SO_{2e_{\tau}+\ell_{\tau}+\ell_{\tau'}}(\C)^{t_{\tau}}$. 
We have the same kind of embedding for $\tau \in I_{\phi}^{- -}$ and $\tau \in I_{\phi}^{0}$. 
For $\tau \in I_{\phi}^{+-}$, the embedding of $(\C^{\times})^{e_{\tau}}$ in 
$\GL_{2e_{\tau}+\ell_{\tau}+\ell_{\tau'}}(\C)$ is given by
\[
(z_1,\ldots,z_{e_{\tau}}) \mapsto
\diag(z_1,\ldots,z_{e_{\tau}},1,\ldots,1,z_{e_{\tau}}^{-1},\ldots,z_1^{-1}),
\]
with $\ell_{\tau} + \ell_{\tau'}$ times $1$ in the middle.

From \eqref{eq:2.2} we see that $R(J^\circ ,T)$ is a union of irreducible components 
$R(J^\circ ,T)_\tau$. Comparing these data with the earlier description from \eqref{eq:2.1} and 
the subsequent table, we deduce that $R(J^\circ,T)_{\tau,\red} = R (G_\phi^\circ,T)_{\tau,\red}$ 
for all $\tau$. Hence $R(J^\circ,T)_\red = R (G_\phi^\circ,T)_\red$, as required for a good 
basepoint $\phi$. In particular $W_{\mf s^\vee}^\circ \cong W(G_\phi^\circ,T)$.

We note that $\rZ( \cL^\vee )^\circ = T$, see \eqref{eq:2.1}. Since $\phi_j \colon\mb W_F \times 
\SL_2 (\C) \to \GL_{m_j}(\C)$ is cuspidal, it is irreducible and trivial on $\SL_2 (\C)$. 
Thus we can write
\[
\cL^\vee = \cG_{n^-}^\vee (\C) \times \prod_j \GL_{m_j} (\C) =
\mc G_{n^-}^\vee (\C) \times \prod_{\tau \in I_{\phi}^{-} \sqcup I_{\phi}^{+}
\sqcup I_{\phi}^{0}} \GL_{\dim(\tau)}(\C)^{e_\tau} .
\]
It follows from \eqref{eq:6.21} that $\dim (\tau) = t_{\tau} \dim (\theta)$ with
$\theta \in \Irr (\mb I_F)$, and that
\[
X_\nr ( {}^L \mc L)_\phi \cong \prod\nolimits_{\tau \in I_{\phi}^{-} \sqcup I_{\phi}^{+}
\sqcup I_{\phi}^{0}} \mu_{t_{\tau}}(\C)^{e_\tau} .
\]
Here $\mu_k$ denotes the functor of taking the $k$-th roots of unity in ring.
In particular $t_{\tau}$ equals $|X_\nr ({}^L \GL_m )_\tau|$, the number of unramified characters
$\chi$ such that $\chi \tau \equiv \tau$. 
In the following table, which stems largely from \cite[\S 4.1]{Mou}, we describe the root systems
and the Weyl groups. We may omit the cases $e_\tau = 0$, because there all the root systems and 
Weyl groups are trivial.
\small
\[
\renewcommand{\arraystretch}{1.3}
\begin{array}{|*6{>{\displaystyle}c|}}
\cline{1-6}
\tau \in & M_{\tau}, M_{\tau'} & \text{condition} & \! R (J^\circ,T)_\tau \!\! & W
_{M_{\tau}^{\circ}}^{G_{\tau}^{\circ}} & W_{M_{\tau}}^{G_{\tau}} \\
\cline{1-6} \cline{1-6}
\multirow{2}*{$I_\phi^{- -}$} & \multirow{2}*{$(\C^{\times})^{e_\tau}\times 
\Sp_{\ell}(\C)$} & \ell_\tau =0 & C_{e_\tau} & S_{e_\tau} \rtimes (\Z/2\Z)^{e_\tau} & 
S_{e_\tau} \rtimes (\Z/2\Z)^{e_\tau} \\
 \cline{3-6}
 &  & \ell_\tau \neq 0 & BC_{e_\tau} & S_{e_\tau} \rtimes (\Z/2\Z)^{e_\tau} & 
 S_{e_\tau} \rtimes (\Z/2\Z)^{e_\tau} \\
\cline{1-6}
\multirow{2}*{$I_\phi^{+-}$} & (\C^{\times})^{e_\tau} \times \Sp_\ell (\C) , & \ell_\tau =0 & 
C_{e_\tau} & S_{e_\tau} \rtimes (\Z/2\Z)^{e_\tau} & 
S_{e_\tau} \rtimes (\Z/2\Z)^{e_\tau} \\
 \cline{3-6}
 & (\C^{\times})^{e_\tau} \times \O_{\ell}(\C) & \ell_\tau \neq 0 & 
 BC_{e_\tau} & S_{e_\tau} \rtimes (\Z/2\Z)^{e_\tau} & S_{e_\tau} \rtimes (\Z/2\Z)^{e_\tau} \\
\cline{1-6}
\multirow{2}*{$I_\phi^{+ +}$} & 
\multirow{2}*{$(\C^{\times})^{e_\tau} \times \O_{\ell}(\C)$} & \ell_\tau =0 & D_{e_\tau} & 
\! S_{e_\tau} \rtimes (\Z/2\Z)^{e_\tau-1} \!\!\! & S_{e_\tau} \rtimes (\Z/2\Z)^{e_\tau} \\
 \cline{3-6}
 &  & \ell_\tau \neq 0 & B_{e_\tau} & S_{e_\tau} \rtimes (\Z/2\Z)^{e_\tau} & 
S_{e_\tau} \rtimes (\Z/2\Z)^{e_\tau} \\
\cline{1-6}
\multirow{2}*{$I_\phi^{0}$} & \multirow{2}*{$(\C^{\times})^{e_\tau}$} & 
e_\tau \leqslant 1 & \varnothing & \{1\} & \{1\} \\
\cline{3-6}
&  & e_\tau \geqslant 2 & A_{e_\tau-1} & S_{e_\tau} & S_{e_\tau} \\
\cline{1-6}
\end{array} 
\]
\normalsize
For all $\tau \in I_{\phi}^{+ +}$ such that $\ell_{\tau}=0 \neq e_\tau$, take
\[
r_\tau = \text{diag} \big( 1,\ldots,1, \matje{0}{1}{1}{0}, 1, \ldots, 1 \big) \in 
O_{2 e_\tau}(\C) \setminus SO_{2 e_\tau}(\C) .
\]
It normalizes $M,T,\phi$ and generates $W_{M_{\tau}}^{G_{\tau}} \big/ 
W_{M_{\tau}^{\circ}}^{G_{\tau}^{\circ}}$. The finite group $\cR_{\mf s^\vee}$ is 
generated by such elements $r_\tau$. More precisely, let
\begin{align*}
& C := \{ \tau \in I_{\phi}^{+} \mid \ell_{\tau}=0 \}, \\
& C_{\text{even}} := \{ \tau  \in C \mid \dim \tau \text{ even} \}, \\
& C_{\text{odd}} := \{ \tau  \in C \mid \dim \tau \text{ odd} \}.
\end{align*}
It was shown in \cite[\textsection 4.1]{Mou} that:
\begin{itemize}
\item if $\cG=\Sp_{N}$ or $\cG=\SO_{N}$ with $N$ odd, then
\[
\cR_{\mf s^\vee} \cong \prod\nolimits_{\tau \in C} \langle r_{\tau} \rangle ;
\]
\item if $\cG=\SO_{N}$ and $\cL=\GL_{d_1}^{\ell_1} \times \ldots \times \GL_{d_r}^{\ell_r}
\times \SO_{N'}$ with $N$ even and $N' \geqslant 4$, then
\[
\cR_{\mf s^\vee} \cong \prod\nolimits_{\tau \in C} \langle r_{\tau} \rangle ;
\]
\item if $\cG=\SO_{N}$ and $\cL=\GL_{d_1}^{\ell_1} \times \ldots \times \GL_{d_r}^{\ell_r}$
with $N$ even, then
\[
\cR_{\mf s^\vee} \cong \prod\nolimits_{\tau \in C_{\text{even}}} \langle r_{\tau} \rangle \times
\langle r_{\tau}r_{\tau'} \mid \tau,\tau' \in C_{\text{odd}} \rangle .
\]
\end{itemize}
From the shape of $M_{\tau}^{\circ}$ we can describe the unipotent element $v_{\tau}$:
\[
\renewcommand{\arraystretch}{1.3}
\begin{array}{|c|c|c|}
\hline
M_{\tau}^{\circ} & v_{\tau} & \ell \\ \hline
(\C^{\times})^{e} \times \Sp_{\ell_\tau}(\C) & 
(1^{e}) \times (2,4,\ldots, 2d-2,2d) & \ell_\tau =d(d+1) \\ \hline
(\C^{\times})^{e} \times \SO_{\ell_\tau}(\C) & 
(1^{e}) \times (1,3,\ldots, 2d-3,2d-1) & \ell_\tau =d^{2} \\ \hline
(\C^{\times})^{e}  & (1^{e})  &  \\\hline
\end{array}
\]
To be complete, let us describe the cuspidal representations of $A_{M_{\tau}^{\circ}}(v_{\tau})$.
We have
\[
A_{M_{\tau}^{\circ}}(v_{\tau}) \cong \left\{ \begin{array}{ll}
(\Z/2\Z)^{d}=\langle z_{\tau,2a}, a \in \llbracket 1,d \rrbracket  \rangle &
\text{if $\tau \in I_{\phi}^{-}$} \\
(\Z/2\Z)^{d-1}=\langle z_{\tau,2a-1}z_{\tau,2a+1}, a \in \llbracket 1,d-1 \rrbracket \rangle &
\text{if $\tau \in I_{\phi}^{+}$}
\end{array} \right. .
\]
Moreover, the cuspidal irreducible representation $\epsilon_{\tau}$ of $A_{M_{\tau}^{\circ}}(v_{\tau})$
satisfies
\[
\epsilon_{\tau}(z_{\tau,2a}) = (-1)^a \text{ if } \tau \in I_{\phi}^{-} \quad \text{and} \quad
\epsilon_{\tau}(z_{\tau,2a-1}z_{\tau,2a+1}) = -1 \text{ if } \tau \in I_{\phi}^{+}.
\]
For all $\tau \in I_{\phi}^{+} \sqcup I_{\phi}^{-}$, denote by $a_{\tau}$ the biggest part of the
partition of $v_{\tau}$ and by $a_{\tau}'$ the biggest part of the partition of $v_{\tau'}$.
In case $v_{\tau'}=1$, we will assume that $a_{\tau}'=0$ if
$\tau \in I_{\phi}^{-}$ and $a_{\tau}'=-1$ if $\tau \in I_{\phi}^{+}$
(this is compatible with Proposition \ref{prop:5.15}).

Finally, we consider the parameter functions. The number $m_\alpha$ from Definition \ref{def:3}
equals $t_\tau$ unless $\tau \in \Irr (\mb W_F )_\phi^{+-}, \ell_\tau = 0$ and $\alpha$ is a long
root in a type $C$ root system, then $m_\alpha = t_\tau / 2$. Recall that $R_{\mf s^\vee}$ consists 
of the roots $m_\alpha \alpha$ with $\alpha \in R(J^\circ,T)_\red$. Multiplication by $m_\alpha$
does not change the type of $R(J^\circ,T)_\tau$, only in the exceptional case, there $C_{e_\tau}$ is
turned into $B_{e_\tau}$. 

If $\alpha \in R_{\tau,\red}$ is not a short root in a type $B$ root system, then by 
\cite[2.13]{Lus3} $c(\alpha)=2$, so  $\lambda(\alpha) = m_\alpha$. 
For the simple short root $\alpha_{\tau} \in R_{\tau,\red}$ we have 
$c(\alpha_{\tau})=a_{\tau}+1, c^*(\alpha_{\tau})=a_{\tau}'+1$ and $m_\alpha = t_\tau$. Hence
\[
\lambda(\alpha_{\tau}) = (a_{\tau}+a_{\tau}' + 2) t_{\tau} / 2 \quad \text{and} \quad
\lambda^{*}(\alpha_{\tau})= |a_{\tau}-a_{\tau}'| t_{\tau} / 2 .
\]
We conclude that
\begin{equation}\label{eq:6.20}
\cH (\fs^\vee, \vec{\mb z}) = 
\cH (\mc R_{\fs^\vee}, \lambda,\lambda^{*}, \vec{\mb z}) \rtimes \C[\cR_{\mf s^\vee}] .
\end{equation}
Via the specialization of $\mb{z}_{\tau}$ at $q_{F}^{1/2}$, \eqref{eq:6.20}
becomes the extended affine Hecke algebra given in \cite{Hei}.
Moreover, it was shown in \cite{Hei} that there is an equivalence of categories between
$\Rep(\cG(F))^{\fs}$ and the right modules over
$\cH (\fs^\vee, \vec{\mb z}) / \big( \{ \mb z_\tau- q_{F}^{1/2}\}_\tau \big)$.
Together with the LLC for $\cG (F)$ we get bijections
\begin{equation}
\Irr \Big( \cH (\fs^\vee,\vec{\mb z}) / \big( \{ \mb z_\tau- q_{F}^{1/2}\}_\tau \big) \Big)
\longleftrightarrow \Irr(\cG(F))^{\fs} \longleftrightarrow \Phi_\re(\cG(F))^{\fs^{\vee}} .
\end{equation}
It does not seem unlikely that this works out to the same bijection as in Theorem \ref{thm:5.13}.a.
But at present that is hard to check, because the LLC is not really explicit.

\begin{ex}
We consider an example that illustrates many of the above aspects.
Let $\tau \colon \mb W_F \rightarrow \GL_4(\C)$ be an irreducible representation of $\mb W_F$, self-dual
of symplectic type and let $\varphi \colon \mb W_F \times \SL_2(\C) \rightarrow \SO_{37}(\C)$ be defined by
\[
\Std_{GL_{37}} \circ \varphi = 1 \boxtimes (S_5 \oplus S_3 \oplus S_1)
\oplus \xi \boxtimes (S_3 \oplus S_1) \oplus \tau \boxtimes (S_4 \oplus S_2),
\]
with $\xi \colon\mb W_F \rightarrow \C^{\times}$ an unramified quadratic character. We have
\[
\rZ_{\SO_{37}(\C)}(\varphi |_{\mb W_F} )^{\circ} \cong \SO_{9}(\C) \times \SO_4(\C) \times \Sp_6(\C),
\]
and $\varphi$ defines a $L$-packet $\Pi_\varphi (\Sp_{36}(F))$ with
$2^5$ elements, of which two are supercuspidal.
Let $\sigma \in \Pi_\varphi (\Sp_{36}(F))$ be supercuspidal, corresponding to
an enhanced Langlands parameter $(\varphi,\varepsilon)$ with $\varepsilon$ cuspidal.
Consider $\mathcal{G}(F)=\Sp_{58}(F)$, the Levi subgroup
\[
\mathcal{L}(F)=\GL_4(F)^{2} \times \GL_1(F)^3 \times \Sp_{36}(F)
\]
and an irreducible supercuspidal representation
$\pi_\tau^{\otimes 2} \boxtimes 1^{\otimes 3} \boxtimes \sigma$ of $\mathcal{L}(F)$. The cuspidal pair
$\mathfrak{s}=[\mathcal{L}(F),\pi_\tau^{\otimes 2} \boxtimes 1^{\otimes 3} \boxtimes \sigma]$ of
$\mathcal{G}(F)$ admits $\mathfrak{s}^{\vee}=[\mathcal{L}^{\vee},\phi,\varepsilon]$ as dual inertial
equivalence class, where $\phi \colon \mb W_F \times \SL_2(\C) \rightarrow \mathcal{L}^{\vee}$,
\[
\cL^\vee = \GL_4 (\C)^2 \times \GL_1 (\C)^3 \times \SO_{37}(\C) 
\quad \text{and} \quad
\Std_{\mathcal{L}^{\vee}} \circ \phi=(\tau \oplus \tau^{\vee})^{\oplus 2}\oplus
(1\oplus 1^{\vee})^{\oplus 3} \oplus \varphi .
\]
We assume that $\tau|_{\mb I_F}=\theta \oplus \theta^{\Fr_F} $ with $\theta^{\vee} \cong \theta$,
so $t_\tau = 2$. We first compute $W_{\mathfrak{s}^{\vee}}^{\circ} $:
\begin{align*}
& \phi |_{\mb I_F} = \tau |_{\mb I_F}^{\oplus 4} \oplus 1|_{\mb I_F}^{\oplus 6} \oplus
1|_{\mb I_F}^{\oplus 9} \oplus \xi|_{\mb I_F}^{\oplus 4} \oplus \tau|_{\mb I_F}^{\oplus 6}=
\theta^{\oplus 10} {\oplus \theta^{\Fr_F}}^{\oplus 10} \oplus 1^{\oplus 19} , \\
& J^\circ = \rZ_{\mathcal{G}^{\vee}}(\phi |_{\mb I_F})^{\circ} \cong \Sp_{10}(\C)^{2} \times \SO_{19}(\C) .
\end{align*}
The torus $T$ is decomposed as $T=(\C^{\times})^{2} \times (\C^{\times})^{3}$. The first part
$(\C^{\times})^{2}$ is embedded in an obvious way in $(\C^{\times})^{2} \times \Sp_{6}(\C)$
and then in $\Sp_{10}(\C)^{2}$ diagonally as Levi subgroup. The second part $(\C^{\times})^{3}$
is embedded in $(\C^{\times})^{3} \times \SO_{13}(\C)$ and then in $\SO_{19}(\C)$ as Levi subgroup
as well. The root system $R(J^{\circ},T)$ (resp. $R(J^{\circ},T)_\red$) is
$BC_{2}\times B_{3}$ (resp. $B_2 \times B_3$), so $W_{\mathfrak{s}^{\vee}}^{\circ}=W_{B_2} \times W_{B_3}$.

From the above discussion, we can see that $\phi$ is already a basepoint. If we denote by
$\phi'$ the parameter defined by $\phi'=(\tau' \oplus \tau^{'\vee})^{\oplus 2}\oplus
(\xi \oplus \xi^{\vee})^{\oplus 3} \oplus \varphi$, then $\phi'$ is another basepoint.
Indeed, we have:
\begin{align*}
\phi |_{\mb W_F} &= \tau^{\oplus 10} \oplus 1^{\oplus 15} \oplus \xi^{\oplus 4}\\
G_{\phi}^{\circ}= \rZ_{\mathcal{G}^{\vee}}(\phi |_{\mb W_F})^{\circ} &
\cong \Sp_{10}(\C) \times \SO_{15}(\C) \times \SO_4(\C) \\
M_{\phi}^{\circ}= \rZ_{\mathcal{L}^{\vee}}(\phi |_{\mb W_F})^{\circ} & \cong \big( (\C^{\times})^{2}
\times \Sp_{6}(\C) \big) \times \big( (\C^{\times})^{3}\times \SO_9(\C) \big) \times \SO_4(\C) \\
\phi' |_{\mb W_F} &= \tau'^{\oplus 4} \oplus \tau^{\oplus 6} \oplus 1^{\oplus 9} \oplus \xi^{\oplus 10}\\
G_{\phi'}^{\circ}= \rZ_{\mathcal{G}^{\vee}}(\phi' |_{\mb W_F})^{\circ} &
\cong \Sp_4 (\C) \times \Sp_6 (\C) \times \SO_{9}(\C) \times \SO_{10}(\C) \\
M_{\phi'}^{\circ}= \rZ_{\mathcal{L}^{\vee}}(\phi' |_{\mb W_F})^{\circ} & \cong (\C^{\times})^{2}\times 
\Sp_{6}(\C)  \times \SO_9(\C) \times \big( (\C^{\times})^{3} \times \SO_4(\C) \big) .
\end{align*}
Here $\mf R_{\mathfrak{s}^{\vee}}$ is trivial, so $W_{\mathfrak{s}^{\vee}}=W_{\mathfrak{s}}^{\circ}$.
Denote by $\alpha_1,\alpha_2$ (resp. $\beta_1,\beta_2, \beta_3$) the simple roots of $B_2$
(resp. $B_3$) with $\alpha_2$ (resp. $\beta_3$) the short root. Then $a_1 = a'_\xi = 5$, 
$a_\xi = a'_1 = 3$, $a_\tau = 4$ and $a'_\tau = 0$. The parameters are given by
$\lambda(\alpha_1) = t_{\tau} = 2$, $\lambda(\beta_1) = \lambda(\beta_2) = 1$ and
\[
\lambda(\alpha_2) = t_{\tau} \frac{4 + 2}{2} = 6 ,\,
\lambda(\beta_3)=\frac{5+3 + 2}{2} = 5 ,\, \lambda^*(\alpha_2) = t_{\tau} \frac{4}{2} = 4 ,\, 
\lambda^{*}(\beta_3)=\frac{5-3}{2}=1.
\]
Specializing $\vec{\mb z}$ to $q_F^{1/2}$, the quadratic relations in the Hecke algebra become
\begin{align*}
& (N_{s_{\alpha_1}} - q_F^2 )(N_{s_{\alpha_1}}+ q_F^{-2}) = 0, \,\,
(N_{s_{\alpha_2}} - q_F^3 )(N_{s_{\alpha_2}}+ q_F^{-3} ) = 0, \\
& (N_{s_{\beta_3}} - q_F^{5/2})(N_{s_{\beta_3}} + q_F^{-5/2}) =0 , \,\,
(N_{s_{\beta_i}} - q_F^{1/2})(N_{s_{\beta_i}} + q_F^{-1/2}) = 0 \quad (i = 1,2).
\end{align*}
\end{ex}

\appendix
\section{}

In this appendix we prove a number theoretic result which probably has been known for a 
long time, but for which we could not find a reference. Let $\mb W_F$ be the Weil group of 
the non-archimedean local field $F$, $\mb I_F$ the 
inertia subgroup and $\mb P_F$ the wild inertia subgroup of $\mb W_F$. Let $\Fr_F \in \mb W_F$ 
a geometric Frobenius element and let $q_F$ be the cardinality of the residue field of $F$. 

\begin{lem}\label{lem:5.17}
$Z(\mb W_F) = Z(\mb I_F) = Z(\mb P_F) = \{\mr{id}\}$.
\end{lem}
\begin{proof}
According to \cite[\S 3]{Jan}, $\mb P_F$ is a free pro-$p$ group on more than one generator.
In particular its centre is trivial. 

It follows from \cite[Corollary 1 to Proposition IV.2.9]{Ser} that an arbitrary element $x$ of 
$\mb I_F \setminus \mb P_F$ does not commute with some elements of 
$\mb P_F$. Namely, we apply [ibid] to the Galois group of some finite Galois extension $E/F$, 
which we choose so large that $x$ ends up in the ramification group $\mr{Gal}(E/F)_0$
but not in $\mr{Gal}(E/F)_1$. Then [ibid] says that $x$ does not commute with most elements of
$\mr{Gal}(E/F)_1$, and we can lift that noncommutativity back to $\mb I_F = \mr{Gal}(F_s/F)_0$.
Hence $Z(\mb I_F) = \{\mr{id}\}$.

The group $\mb I_F / \mb P_F$ is isomorphic to $\hat \Z / \Z_p$ and the conjugation action of
$\Fr_F^{-1}$ on it equals raising elements to the power $q_F$ \cite{Iwa}. Hence $\Fr_F^n x$ with
$x \in \mb I_F, n \in \Z \setminus \{0\}$ does not commute (in $\mb W_F / \mb P_F$) with most
elements of $\mb I_F / \mb P_F$. As $\mb W_F = \cup_{n \in \Z} \Fr_F^n \mb I_F$, we find that
$Z(\mb W_F) = \{ \mr{id} \}$.
\end{proof}

Lemma \ref{lem:5.17} is used in the definition of $X_\nr ({}^L \mc G)$ in \eqref{eq:Xnr},
and was already used in the same way in \cite{AMS}.


\begin{thebibliography}{99}

\bibitem[ABPS1]{ABPS2} A.-M. Aubert, P.F. Baum, R.J. Plymen, M. Solleveld,
``Geometric structure in smooth dual and local Langlands correspondence'',
Japan. J. Math. {\bf 9} (2014), 99--136.

\bibitem[ABPS2]{ABPS3} A.-M. Aubert, P.F. Baum, R.J. Plymen, M. Solleveld,
``The local Langlands correspondence for inner forms of $\SL_n$",
Res. Math. Sci. {\bf 3:32} (2016).

\bibitem[ABPS3]{ABPS4} A.-M. Aubert, P.F. Baum, R.J. Plymen, M. Solleveld,
``Hecke algebras for inner forms of $p$-adic special linear groups",
J. Inst. Math. Jussieu {\bf 16.2} (2016), 351--419.

\bibitem[ABPS4]{ABPS5} A.-M. Aubert, P.F. Baum, R.J. Plymen, M. Solleveld,
``The principal series of $p$-adic groups with disconnected centre'',
Proc. London Math. Soc. {\bf 114.5} (2017), 798--854.

\bibitem[ABPS5]{ABPS7} A.-M. Aubert, P.F. Baum, R.J. Plymen, M. Solleveld,
``Conjectures about $p$-adic groups and their noncommutative geometry'',
pp. 15--51 in: \emph{Around Langlands Correspondences}, Contemp. Math. {\bf 691},
American Mathematical Society, 2017.

\bibitem[AMS1]{AMS} A.-M. Aubert, A. Moussaoui, M. Solleveld,
``Generalizations of the Springer correspondence and cuspidal Langlands parameters'',
Manus. Math. {\bf 157} (2018), 121--192

\bibitem[AMS2]{AMS2} A.-M. Aubert, A. Moussaoui, M. Solleveld,
``Graded Hecke algebras for disconnected reductive groups",
\emph{Geometric aspects of the trace formula, W. M\"uller, S. W. Shin, N. Templier (eds.)},
Simons Symposia, Springer, 2018, 23--84, and arxiv:1607.02713v2

\bibitem[Bad]{Bad} A. I. Badulescu, ``Correspondance de Jacquet-Langlands pour
les corps locaux de caract\'eristique non nulle",
Ann. Sci. \'Ec. Norm. Sup. (4) {\bf 35} (2002), 695--747.

\bibitem[BHLS]{BHLS} A. I. Badulescu, G. Henniart, B. Lemaire, V. S\'echerre,
``Sur le dual unitaire de $\GL_r(D)$", Amer. J. Math. \textbf{132} (2010), 1365--1396.

\bibitem[BaCi]{BaCi} D. Barbasch, D. Ciubotaru,
``Unitary equivalences for reductive $p$-adic groups"
Amer. J. Math. {\bf 135.6} (2013), 1633--1674.

\bibitem[BaMo]{BaMo} D. Barbasch, A. Moy,
``Reduction to real infinitesimal character in affine Hecke algebras'',
J. Amer. Math. Soc. \textbf{6.3} (1993), 611--635.

\bibitem[Bor]{Bor} A. Borel,
``Automorphic L-functions'',
Proc. Symp. Pure Math {\bf 33.2} (1979), 27--61.
 
\bibitem[Bou]{Bou} N. Bourbaki, 
\emph{Groupes et algèbres de Lie. Chapitres IV, V et VI},
\'El\'ements de math\'ematique {\bf XXXIV}, Hermann, 1968.

\bibitem[BuKu1]{BuKu} C.J. Bushnell, P.C. Kutzko, ``Smooth representations of
reductive $p$-adic groups: structure theory via types",
Proc. London Math. Soc. {\bf 77.3} (1998), 582--634.

\bibitem[BuKu2]{BuKu2} C.J. Bushnell, P.C. Kutzko,
``Semisimple types in $\GL_n$",
Compositio Math. {\bf 119.1} (1999), 53--97.

\bibitem[ChGo]{ChGo} K. Choiy, D. Goldberg,
``Invariance of $R$-groups between $p$-adic inner forms of quasi-split classical groups",
Trans. Amer. Math. Soc. \textbf{368} (2016), 1387--1410.

\bibitem[Dat]{Dat} J.-F. Dat,
``A functoriality principle for blocks of $p$-adic linear groups",
pp. 103--132 in: \emph{Around Langlands Correspondences}, Contemp. Math. {\bf 691},
American Mathematical Society, 2017.

\bibitem[DKV]{DKV} P. Deligne, D. Kazhdan, M.-F.~Vigneras,
``Repr\'esentations des alg\`ebres centrales simples $p$-adiques", pp. 33--117 in:
\emph{Repr\'esentations des groupes r\'eductifs sur un corps local},
Travaux en cours, Hermann, 1984.

\bibitem[EvMi]{EvMi} S. Evens, I. Mirkovi\'c,
``Fourier transform and the Iwahori--Matsumoto Involution",
Duke Math. J. {\bf 86.3} (1997), 435--464.

\bibitem[Hai]{Hai} T.J. Haines,
``The stable Bernstein center and test functions for Shimura varieties'',
pp. 118--186 in: \emph{Automorphic forms and Galois representations},
London Math. Soc. Lecture Note Ser. {\bf 415}, Cambridge University Press, 2014.

\bibitem[Hei1]{Hei1} V. Heiermann,
``Param\`etres de Langlands et alg\`ebres d'entrelacement",
Int. Math. Res. Not. {\bf 2010.9} (2010), 1607--1623.

\bibitem[Hei2]{Hei} V. Heiermann,
``Local Langlands correspondence for classical groups and affine Hecke algebras",
Math. Z. {\bf 287} (2017), 1029--1052.

\bibitem[HiSa]{HiSa} K. Hiraga, H. Saito, ``On $L$-packets for inner forms of $\SL_n$",
Mem. Amer. Math. Soc. {\bf 1013}, Vol. {\bf 215} (2012).

\bibitem[Hum]{Hum} J.E. Humphreys,
\emph{Reflection groups and Coxeter groups},
Cambridge Studies in Advanced Mathematics {\bf 29},
Cambridge University Press, 1990.

\bibitem[Iwa]{Iwa} K. Iwasawa, 
``On Galois groups of local fields",
Trans. Amer. Math. Soc. {\bf 80} (1955), 448--469.

\bibitem[Jan]{Jan} U. Jannsen,
``\"Uber Galoisgruppen lokaler K\"orper",
Invent. Math. {\bf 70.1} (1982), 53--69.

\bibitem[Kat]{Kat} S.-I. Kato,
``A realization of irreducible representations of affine Weyl groups",
Indag. Math. {\bf 45.2} (1983), 193--201.

\bibitem[KaLu]{KaLu} D. Kazhdan, G. Lusztig,
``Proof of the Deligne--Langlands conjecture for Hecke algebras'',
Invent. Math. {\bf 87} (1987), 153--215.

\bibitem[KMRT]{KMRT} M-A. Knus, A. Merkujev, M. Rost, J-P. Tignol,
{\sl The book of involutions},
Amer. Math. Soc. Coll. Publications \text{\bf 44}, 1998.

\bibitem[Lus1]{Lus2} G. Lusztig,
``Intersection cohomology complexes on a reductive group'',
Invent. Math. {\bf 75.2} (1984), 205--272.

\bibitem[Lus2]{Lus3} G. Lusztig,
``Cuspidal local systems and graded Hecke algebras'',
Publ. Math. Inst. Hautes \'Etudes Sci. {\bf 67} (1988), 145--202.

\bibitem[Lus3]{Lus4} G. Lusztig,
``Affine Hecke algebras and their graded version'',
J. Amer. Math. Soc {\bf 2.3} (1989), 599--635.

\bibitem[Lus4]{Lus5} G. Lusztig,
``Cuspidal local systems and graded Hecke algebras. II'',
pp. 217--275 in: \emph{Representations of groups},
Canadian Mathematical Society Conference Proceedings {\bf 16}, 1995.

\bibitem[Lus5]{Lus6} G. Lusztig,
``Classification of unipotent representations of simple $p$-adic groups'',
Int. Math. Res. Notices {\bf 11} (1995), 517--589.

\bibitem[Lus6]{Lus7} G. Lusztig,
``Cuspidal local systems and graded Hecke algebras. III'',
Represent. Theory {\bf 6} (2002), 202--242.

\bibitem[Lus7]{Lus8} G. Lusztig,
``Classification of unipotent representations of simple $p$-adic groups. II'',
Represent. Theory {\bf 6} (2002), 243--289.

\bibitem[Mou]{Mou} A. Moussaoui, ``Centre de Bernstein dual pour les groupes classiques",
Represent. Theory \textbf{21} (2017), 172--246.

\bibitem[Opd1]{Opd} E.M. Opdam,
``On the spectral decomposition of affine Hecke algebras",
J. Inst. Math. Jussieu {\bf 3.4} (2004), 531--648.

\bibitem[Opd2]{Opd2} E.M. Opdam,
``Spectral correspondences for affine Hecke algebras",
Adv. Math. \textbf{286} (2016), 912--957.

\bibitem[Ree]{Ree} M. Reeder,
``Isogenies of Hecke algebras and a Langlands correspondence
for ramified principal series representations",
Representation Theory {\bf 6} (2002), 101--126.

\bibitem[Roc]{Roc} A. Roche, Types and Hecke algebras for principal
series representations of split reductive $p$-adic groups, Ann.
scient. \'Ec. Norm. Sup. {\bf 31} (1998), 361--413.

\bibitem[Sec1]{Sec3} V. S\'echerre,
``Repr\'esentations lisses de $\GL_m (D)$ III: types simples",
Ann. Scient. \'Ec. Norm. Sup. {\bf 38} (2005), 951--977.

\bibitem[Sec2]{SecU0}  V. S\'echerre,
``Proof of the Tadi\'c conjecture ($\mathrm{U}_0$) on the unitary dual of $\GL_m(D)$",
J. reine angew. Math. \textbf{626} (2009), 187--203.

\bibitem[SeSt1]{SeSt4} V. S\'echerre, S. Stevens,
``Repr\'esentations lisses de $\GL_m (D)$ IV: repr\'esentations supercuspidales",
J. Inst. Math. Jussieu {\bf 7.3} (2008), 527--574.

\bibitem[SeSt2]{SeSt6} V. S\'echerre, S. Stevens,
``Smooth representations of $\GL(m,D)$ VI: semisimple types",
Int. Math. Res. Notices (2011).

\bibitem[Ser]{Ser} J.-P. Serre,
\emph{Local fields}, Springer Verlag, New York NJ, 1979

\bibitem[Slo]{Slo} K. Slooten,
``Generalized Springer correspondence and Green functions
for type B/C graded Hecke algebras'',
Adv. Math. \textbf{203} (2005), 34--108.

\bibitem[Sol1]{Sol2} M. Solleveld,
``Parabolically induced representations of graded Hecke algebras",
Algebras and Representation Theory {\bf 15.2} (2012), 233--271.

\bibitem[Sol2]{Sol1} M. Solleveld,
``Homology of graded Hecke algebras'',
J. Algebra {\bf 323} (2010), 1622--1648.

\bibitem[Sol3]{Sol3} M. Solleveld, ``On the classification of irreducible
representations of affine Hecke algebras with unequal parameters",
Representation Theory {\bf 16} (2012), 1--87.

\bibitem[Ste]{Ste} R. Steinberg,
\emph{Endomorphisms of linear algebraic groups},
Mem. Amer. Math. Soc. {\bf 80},
American Mathematical Society, Providence RI, 1968

\bibitem[Vog]{Vog} D. Vogan	
``The local Langlands conjecture'',
pp. 305--379 in: \emph{Representation theory of groups and algebras},
Contemp. Math. {\bf 145}, American Mathematical Society, 1993.

\bibitem[Wal]{Wal} J.-L. Waldspurger,
``Repr\'esentations de r\'eduction unipotente pour SO(2n+1):
quelques cons\'equences d'un article de Lusztig",
pp. 803--910 in: \emph{Contributions to automorphic forms, geometry, and number theory},
Johns Hopkins Univ. Press, Baltimore MD, 2004.

\end{thebibliography}
\end{document}